\documentclass[Preprint, 12pt]{elsarticle}
\usepackage{amsmath,amssymb,amsthm, graphicx}
\usepackage{tikz}
\usepackage{fancyhdr}
\usepackage{epsfig}
\usepackage{mathrsfs}
\usepackage[toc,page]{appendix}

\numberwithin{equation}{section}
\newtheorem{theorem}{Theorem}[section]

\newtheorem{definition}{Definition}[section]
\newtheorem{lemma}{Lemma}[section]
\newtheorem{remark}{Remark}[section]
\newtheorem{proposition}{Proposition}[section]

\makeatletter \@addtoreset{figure}{section} \makeatother

\newcommand{\dd}{\mathrm{d}}

\newcommand{\R}{{\mathbb R}}

\journal{Advances in Mathematics}

\begin{document}

\begin{frontmatter}
\title {Hypersonic Similarity for Steady Compressible Full Euler Flows over Two-Dimensional Lipschitz Wedges}

\author[Ox,Am,Fu]{Gui-Qiang G. Chen\corref{Cor1}}
\ead{chengq@maths.ox.ac.uk}

\author[In]{Jie Kuang\corref{Cor1}}
\ead{jkuang@wipm.ac.cn}

\author[Ci]{Wei Xiang\corref{Cor1}}
\ead{weixiang@cityu.edu.hk}

\author[Fu]{Yongqian Zhang\corref{Cor1}}
\ead{yongqianz@fudan.edu.cn}

\cortext[Cor1]{The corresponding author.}

\affiliation[Ox]{organization={Mathematical Institute, University of Oxford},
 	           city={Oxford},
             postcode={OX2 6GG},
             country={UK}}

             \affiliation[Am]{organization={Academy of Mathematics and Systems Science, Chinese Academy of Sciences},
             	city={Beijing},
             	postcode={100190},
             	country={China}}

 \affiliation[In]{organization={Innovation Academy for Precision Measurement Science and Technology, and Wuhan Institute of Physics and Mathematics, Chinese Academy of Sciences},
             city={Wuhan},
             postcode={430071},
             country={China}}

 \affiliation[Ci]{
organization={Department of Mathematics,
             		City University of Hong Kong},
             	addressline={Kowloon},
             	city={Hong Kong},
             	country={China}
             }

              \affiliation[Fu]{
              	organization={School of Mathematical Sciences, Fudan University},
             	city={Shanghai},
             	postcode={200433},
             	country={China}
             	}


\begin{abstract}
We establish the optimal convergence rate to the hypersonic similarity law,
which is also called the Mach number independence principle,
for steady compressible full Euler flows over two-dimensional slender Lipschitz wedges.
Mathematically, it can be formulated as the comparison of the entropy solutions in
$BV\cap L^{1}$
between the two initial-boundary value problems for
the compressible full Euler equations with parameter $\tau>0$ and
the hypersonic small-disturbance equations (the scaled compressible full Euler equations with parameter $\tau=0$)
with curved characteristic boundaries.
We establish the $L^1$--convergence estimate of these two solutions
with the optimal convergence rate,
which justifies the Van Dyke's similarity theory rigorously for the compressible full Euler flows.
This is the first mathematical result on the comparison of two solutions
of the compressible Euler equations with characteristic boundary conditions.
To achieve this, we first employ the special structures of the two systems
and establish the global existence and the $L^1$--stability of the entropy solutions via the wave-front tracking scheme
under the smallness assumption on the total variation of both the initial data and the tangential slope function of the wedge boundary.
Based on the $L^1$--stability properties of the approximate solutions
to the scaled  equations
with parameter $\tau>0$, a uniform Lipschtiz continuous map
with respect to the initial data
and the wedge boundary is obtained,
which is the first time for the characteristic boundary conditions.
Next, we compare the solutions given by the Riemann solvers of the two systems by taking the boundary perturbations
into account case by case.
Then, for a given fixed hypersonic similarity parameter,
as the Mach number tends to infinity,
by employing the Lipschitz continuous properties of the map,
we establish the desired $L^1$--convergence estimate with the optimal convergence rate.
Finally, we show the optimality of the convergence rate by investigating a special solution.
\end{abstract}

\end{frontmatter}

\date{\today}


\section{Introduction and Main Theorems}\setcounter{equation}{0}
We are concerned with the mathematical validation of the hypersonic similarity to
the problem of stationary hypersonic Euler flows
over a two-dimensional Lipschitz wedge (see Fig. \ref{fig1.1'}).
In gas dynamics, we call the supersonic flow is hypersonic when its Mach number is greater than five.
The study of the hypersonic flow can date back to the 1940s due to many important applications in aerodynamics and
engineering (see {\it e.g.}, \cite{dyke,tsien}), which is different from the study of the two-dimensional supersonic
problems (see {\it e.g.},
\cite{ChenFeldman, ChenFeldmanXiang, chen-kuang-zhang, chen-li, chen-zhang-zhu-1,chen-zhang-zhu-2, huang-kuang-wang-xiang, xiang-zhang-zhao, zhang-1, zhang-2}).
It bears important properties of the solution structures,
such as the \emph{hypersonic similarity law}, which is our main focus of this paper.

One of the main reasons for verifying the \emph{hypersonic similarity law} rigorously is to overcome the difficulty
that the density is very small, compared to the flow speed for the hypersonic flow.
Thus, like the fluid behavior near the vacuum, all the characteristics are close to each other,
and the shock layer is very thin.
The validation of the hypersonic similarity law allows us to study a scaled problem in which such a difficulty does not {occur}
(see \cite{anderson} for more details).
In general, we say that two or more different flows satisfy a similarity law if the solution structures
have the similarity with respect to certain nondimensional parameters.

As a paradigm,
the hypersonic similarity can be expressed as follows:
Let $\theta$ be the half opening angle of the wedge, and let $M_{\infty}$ be the Mach number of the uniform
upstream flow (see Fig.~\ref{fig1.1'}).
Define the nondimensional similarity parameter
as
\begin{equation}\label{eq:1.6}
K=\theta M_{\infty}
\end{equation}
(see also (127.3) in Landau-Lifschitz \cite[page 482]{landau-lifschitz} for more details).
Then the hypersonic similarity is that, for a fixed similarity parameter $K$,
the flow structures are similar under scaling if the Mach number $M_{\infty}$ is sufficiently large.
Actually, after scaling, the flows with the same similarity parameter $K$ are governed approximately by the same equations,
which are called the hypersonic small-disturbance equations (see Tsien \cite{tsien} and Van Dyke \cite{dyke}).
Recently, Tsien's hypersonic similarity theory was rigorously justified for the two-dimensional potential flow past over a straight wedge
with large data in \cite{kuang-xiang-zhang-1} and the optimal convergence rate was obtained in \cite{chen-kuang-xiang-zhang}, as well as
over a Lipschitz curved wedge with small data in \cite{kuang-xiang-zhang-2}.
The main purpose of this paper is to verify the Van Dyke's similarity law \cite{dyke} for the two-dimensional steady compressible full Euler equations
for ideal polytropic {gas with the optimal convergence rate}.

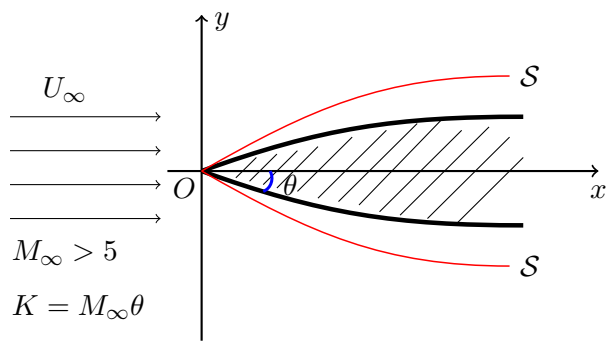
\begin{figure}[ht]
\begin{center}
\begin{tikzpicture}[scale=0.9]
\draw [thick][->](-2.5,1)--(3.8,1);
\draw [thick][->](-2,-1.5)--(-2,3.3);
\draw [line width=0.06cm] (-2,1)to[out=20, in=180](2.7,1.8);
\draw [line width=0.06cm] (-2,1)to[out=-20, in=-180](2.7,0.2);
\draw [line width=0.02cm][red] (-2,1)to [out=30, in=180](2.5,2.4);
\draw [line width=0.02cm][red] (-2,1)to [out=-30, in=-180](2.5,-0.4);
\draw [thin][->](-4.8,1.8)--(-2.6,1.8);
\draw [thin][->](-4.8,1.3)--(-2.6,1.3);
\draw [thin][->](-4.8,0.8)--(-2.6,0.8);
\draw [thin][->](-4.8,0.3)--(-2.6,0.3);

\draw [thin](-1.5,0.9)--(-1.2,1.2);
\draw [thin](-1.3,0.85)--(-0.9,1.25);
\draw [thin](-1.1,0.80)--(-0.6,1.3);
\draw [thin](-0.9,0.75)--(-0.3,1.36);
\draw [thin](-0.6,0.7)--(0.2,1.5);
\draw [thin](-0.3,0.6)--(0.7,1.6);
\draw [thin](0,0.55)--(1.15,1.70);
\draw [thin](0.3,0.5)--(1.5,1.75);
\draw [thin](0.6,0.40)--(1.8,1.6);
\draw [thin](0.9,0.35)--(2.2,1.65);
\draw [thin](1.3,0.3)--(2.5,1.5);
\draw [thin](1.7,0.20)--(2.7,1.2);
\draw [line width=0.04cm][blue] (-1,1)to [out=-60, in=20](-1.1,0.7);
\node at (1.5, 1.5) {$$};
\node at (3.8,0.7) {$x$};
\node at (-1.7,3.2) {$y$};
\node at (2.8,2.4) {$\mathcal{S}$};
\node at (2.8,-0.4) {$\mathcal{S}$};
\node at (-4.0,2.2) {$U_{\infty}$};
\node at (-4.0,-0.2) {$M_{\infty}>5$};
\node at (-0.7,0.8) {$\theta$};
\node at (-2.25,0.75) {$O$};
\node at (-3.8, -1.0) {$K=M_{\infty}\theta $};
\end{tikzpicture}
\end{center}
\caption{Hypersonic similarity for the Euler flows over a two-dimensional slender wedge}\label{fig1.1'}
\end{figure}

We remark that a related problem on the hypersonic limit as the Mach number $M_{\infty}$ of
the upstream flow tends to infinity
with a fixed wedge was studied successfully in \cite{jin-qu-yuan-1, qu-wang-yuan, qu-yuan-zhao}.
That problem is different from the problem under our consideration in this paper
because, for a fixed wedge, the similarity parameter is no longer a fixed number.
In fact, $K\rightarrow\infty$ as $M_{\infty}\rightarrow\infty$.
For that problem, the limit solution as {$M_{\infty}\rightarrow\infty$} was constructed,
which is a Radon measure-valued solution over a straight wedge or a ramp,
and the Newton theory or the New-Busemann pressure law of the hypersonic flow was
justified in \cite{jin-qu-yuan-1, qu-wang-yuan, qu-yuan-zhao}.
See also \cite{jin-qu-yuan-2,qu-yuan} for further results for the three-dimensional case.
We remark that there are also some related analyses on the steady supersonic flow past a curved cone with its surface {being} smooth
or Lipschitz continuous {when} the upstream flow is sufficiently
large; see \cite{chen-kuang-zhang-2, chen-xin-yin, hu-zhang,li-witt-yin, lien-liu, wang-zhang} and the references cited therein.

\smallskip
The steady compressible full Euler equations for the ideal polytropic {gas} in $\mathbb{R}^2$ take the form:
\begin{equation}\label{eq:1.1}
\begin{cases}
\mathbf{div}(\rho \mathbf{u})=0, \\[2pt]
\mathbf{div}(\rho \mathbf{u}\otimes\mathbf{u}+p\mathbb{I})=0, \\[2pt]
\mathbf{div}\big(\rho \mathbf{u}(E+\frac{p}{\rho})\big)=0,
\end{cases}
\end{equation}
where $\mathbf{div}$ is the divergence operator with respect to $\mathbf{x}=(x,y)\in \mathbb{R}^2$,
$\mathbb{I}$ is the identity matrix, $\mathbf{u}=(u,v)$ is the velocity with the horizontal and vertical velocity {components} $u$ and $v$ respectively,
$p$ and $\rho$ represent the pressure and the density respectively, and $E$ is the total energy given by
\begin{eqnarray}\label{eq:1.2}
E=\frac{|\mathbf{u}|^2}{2}+e(\rho, p)  \qquad\,\, \mbox{with $|\mathbf{u}|=\sqrt{u^2+v^2}$}.
\end{eqnarray}
The internal energy $e$,
temperature $T$, and
entropy $S$ are given by
\begin{eqnarray}\label{eq:1.4}
T=\frac{p}{\textsc{R}\rho},\quad e=\frac{p}{(\gamma-1)\rho},\quad p=\textsc{A}(S)\rho^{\gamma} \qquad\,\,\,\, \mbox{for $\gamma>1$},
\end{eqnarray}
where $\textsc{R}>0$ is a constant and $\textsc{A}(S)$ is a positive function of $S$. See also \cite{ChenFeldman,courant-friedrich,Dafermos}.

Let $c=\sqrt{\frac{\gamma p}{\rho}}$ be the local sonic speed of the flow. Then the Mach number of the flow is
defined by
\begin{eqnarray}\label{eq:1.5}
M=\frac{|\mathbf{u}|}{c}.
\end{eqnarray}

We now introduce the hypersonic similarity law and derive the corresponding hypersonic small-disturbance equations
corresponding to the full Euler equations \eqref{eq:1.1} mathematically.
Let $\tau>0$ be a constant. For a given two-dimensional symmetric Lipschtiz slender wedge with surface $y=\tau g(x)$,
without loss of generality, we consider only the lower half-space domain, \emph{i.e.},
in the region that $x\geq0$ and $y\leq\tau g(x)$
with $g(x)<0$ as shown in Fig. \ref{fig1.1'}.
The flow satisfies the following impermeable slip boundary condition along the surface of the wedge:
\begin{eqnarray}\label{eq:1.8}
(u, v)\cdot(\tau g'(x),-1)=0.
\end{eqnarray}
For a given uniformly hypersonic
upstream Euler flow $U_{\infty}=(u_{\infty}, 0, p_{\infty}, \rho_{\infty})^{\top}$,
define
\begin{equation}\label{eq:1.9}
a_{\infty}\doteq\tau M_{\infty}=\tau \frac{u_{\infty}}{c_{\infty}}\qquad\,\, \mbox{for $c_\infty=\sqrt{\frac{\gamma p_{\infty}}{\rho_{\infty}}}$}.
\end{equation}

Mathematically, it follows from \eqref{eq:1.6} that parameter $a_\infty$ is equivalent to $K$ when $M_\infty\to \infty$.
Thus, $a_{\infty}$ is now called the \emph{hypersonic similarity parameter} due to its naturality in terms of scaling; see \cite{anderson, dyke}.
Following \cite{anderson, dyke}, we can now define the following scaling:
\begin{equation}\label{eq:1.10}
x=\bar{x},\ \ y=\tau\bar{y}, \ \ u=u_{\infty}(1+\tau^2\bar{u}) ,\ \ v=u_{\infty}\tau \bar{v},\ \
p=\gamma M^{2}_{\infty}p_{\infty}\tau^{2}\bar{p},\ \ \rho=\rho_{\infty}\bar{\rho}.
\end{equation}
See \cite[\S 4.4]{anderson} for more details.

Substituting \eqref{eq:1.10} into \eqref{eq:1.1} by virtue of
\eqref{eq:1.2}--\eqref{eq:1.4}, we obtain
\begin{equation}\label{eq:1.11}
\begin{cases}
\partial_{\bar{x}}\big(\bar{\rho}(1+\tau^2\bar{u}) \big)+\partial_{\bar{y}}(\bar{\rho} \bar{v})=0,\\[5pt]
\partial_{\bar{x}}\big(\bar{\rho}\bar{u}(1+\tau^{2}\bar{u})+\bar{p}\big)+\partial_{\bar{y}}(\bar{\rho}\bar{u}\bar{v})=0,\\[5pt]
\partial_{\bar{x}}\big(\bar{\rho}\bar{v}(1+\tau^{2}\bar{u})\big)+\partial_{\bar{y}}\big(\bar{\rho}\bar{v}^{2}+\bar{p}\big)=0,\\[5pt]
\partial_{\bar{x}}\Big(\bar{\rho}(1+\tau^2\bar{u})
\big(\bar{u}+\frac{1}{2}\bar{v}^2+\frac{\gamma \bar{p}}{(\gamma-1)\bar{\rho}}+\frac{1}{2}\tau^2\bar{u}^2\big)\Big)\\[5pt]
\qquad+\partial_{\bar{y}}\Big(\bar{\rho}\bar{v}
\big(\bar{u}+\frac{1}{2}\bar{v}^2+\frac{\gamma \bar{p}}{(\gamma-1)\bar{\rho}}+\frac{1}{2}\tau^2\bar{u}^2\big)\Big)=0.
\end{cases}
\end{equation}

\vspace{5pt}
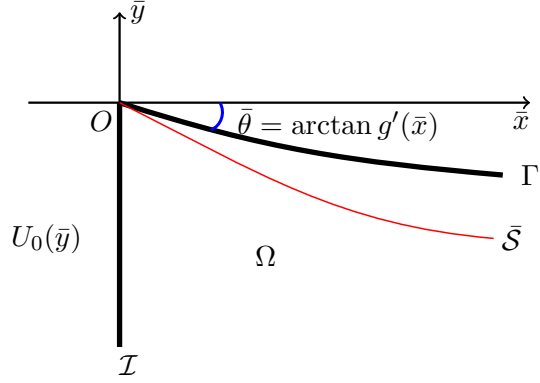
\begin{figure}[ht]
\begin{center}
\begin{tikzpicture}[scale=1.2]
\draw [line width=0.03cm][->] (-2,0)--(3.5,0);
\draw [line width=0.03cm][->] (-1,-2.7) --(-1,1);
\draw [line width=0.07cm](-1,0)to[out=-18, in=-185](3.2,-0.8);
\draw [line width=0.07cm](-1,0)--(-1,-2.7);
\draw [line width=0.02cm][red](-1,0)to[out=-27, in=-185](3.1,-1.5);
\draw [line width=0.04cm][blue](0.1,0)to [out=-60, in=20](0,-0.3);
\node at (3.4,-0.2) {$\bar{x}$};
\node at (-0.8,1) {$\bar{y}$};
\node at (-1.2,-0.2) {$O$};
\node at (3.5,-0.8) {$\Gamma$};
\node at (3.3,-1.5) {$\bar{\mathcal{S}}$};
\node at (-0.9,-2.9) {$\mathcal{I}$};
\node at (1.4,-0.25) {$\bar{\theta}=\arctan g'(\bar{x})$};
\node at (0.6,-1.7) {$\Omega$};
\node at (-1.8,-1.5) {$U_{0}(\bar{y})$};
\end{tikzpicture}
\end{center}
\caption{Initial-boundary value problem \eqref{eq:1.11}--\eqref{eq:1.12}}\label{fig1.2}
\end{figure}

In the $(\bar{x},\bar{y})$--coordinates, the corresponding fluid domain and its boundary are given by
$$
\Omega=\{(\bar{x}, \bar{y})\in \mathbb{R}^{2}\,:\, \bar{x}>0,\ \bar{y}<g(\bar{x}) \},\quad
 \Gamma=\{(\bar{x}, \bar{y})\in \mathbb{R}^{2}\,:\, \bar{x}>0,\ \bar{y}=g(\bar{x}) \},
$$
respectively  (see Fig. \ref{fig1.2}).
Then the boundary condition \eqref{eq:1.9} now becomes
\begin{eqnarray}\label{eq:1.13}
\big((1+\tau^{2}\bar{u}), \bar{v}\big)\cdot \mathbf{n}=0 \qquad \mbox{on $\Gamma$},
\end{eqnarray}
where $\mathbf{n}=\mathbf{n}(\bar{x},g(\bar{x}))=\frac{(g'(\bar{x}),-1)}{\sqrt{1+g'^{2}(\bar{x})}}$ is the interior unit normal vector to $\Gamma$.

Finally, the initial condition at $\bar{x}=0$ is given by
\begin{eqnarray}\label{eq:1.12}
(\bar{\rho}, \bar{u}, \bar{v},\bar{p})=\big(\bar{\rho}_{0}, \bar{u}_{0},\bar{ v}_{0}, \bar{p}\big)(\bar{y})
\qquad\mbox{on $\mathcal{I} \doteq\{\bar{x}=0,\ \bar{y}\leq0\}$}.
\end{eqnarray}

Mathematically, the \emph{hypersonic similarity law} is that, for a fixed similarity parameter $a_{\infty}$,
the structure of solutions of \eqref{eq:1.11}--\eqref{eq:1.12} is persistent if $\tau$ is small (\emph{i.e.}, $M_{\infty}$ is large).
Moreover, {it is conjectured that} the solutions of \eqref{eq:1.11}--\eqref{eq:1.12}
should converge to the solution of the following hypersonic small-disturbance equations
(obtained by neglecting the terms involving $\tau^{2}$) in $\Omega$:
\begin{equation}\label{eq:1.14}
\begin{cases}
\partial_{\bar{x}}\bar{\rho}+\partial_{\bar{y}}(\bar{\rho} \bar{v})=0, \\[4pt]
\partial_{\bar{x}}(\bar{\rho}\bar{u}+\bar{p})+\partial_{\bar{y}}(\bar{\rho}\bar{u}\bar{v})=0,\qquad\qquad\qquad\qquad \qquad\qquad\qquad\qquad\\[5pt]
\partial_{\bar{x}}\big(\bar{\rho}\bar{v}\big)+\partial_{\bar{y}}\big(\bar{\rho}\bar{v}^{2}+\bar{p}\big)=0,\\[4pt]
\partial_{\bar{x}}\Big(\bar{\rho}
\big(\bar{u}+\frac{1}{2}\bar{v}^2+\frac{\gamma \bar{p}}{(\gamma-1)\bar{\rho}}\big)\Big)
+\partial_{\bar{y}}\Big(\bar{\rho}\bar{v}
\big(\bar{u}+\frac{1}{2}\bar{v}^2+\frac{\gamma \bar{p}}{(\gamma-1)\bar{\rho}}\big)\Big)=0,
\end{cases}
\end{equation}
with the initial data \eqref{eq:1.12} and the boundary condition:
\begin{eqnarray}\label{eq:1.15}
\bar{v}=g'(\bar{x}) \qquad\,\, \mbox{on $\Gamma$}.
\end{eqnarray}

If the \emph{hypersonic similarity law} can be justified rigorously, then the study of the steady hypersonic Euler flow over
a two-dimensional Lipschitz slender wedge can be much simplified, because we do not face the difficulty that the density is small
and the characteristics are so close to each other.
In this paper, we justify this rigorously and establish the optimal convergence rate of the solutions
of problem \eqref{eq:1.11}--\eqref{eq:1.12}
to the solutions of problem \eqref{eq:1.12}--\eqref{eq:1.15} in $BV\cap L^1$, as {$\tau\rightarrow 0$}.
This provides more helpful information and insights for the applications.

\smallskip
Denote
$U^{(\tau)}\doteq (\bar{\rho}^{(\tau)},\bar{u}^{(\tau)},\bar{v}^{(\tau)},\bar{p}^{(\tau)})$
as the solution of \eqref{eq:1.11}--\eqref{eq:1.12}, and use $U\doteq (\bar{\rho},\bar{u},\bar{v},\bar{p})$
as the solution corresponding to the case that $\tau=0$ in \eqref{eq:1.12}--\eqref{eq:1.15}.
The definition of the global entropy solutions
$$
U^{(\tau)}=(\bar{\rho}^{(\tau)},\bar{u}^{(\tau)},\bar{v}^{(\tau)},\bar{p}^{(\tau)})^{\top}\in BV \cap L^{1}
$$
of the initial-boundary value problem \eqref{eq:1.11}--\eqref{eq:1.12} is given as follows:

\begin{definition}[Entropy solutions]\label{def:1.1}
A function $U^{(\tau)}\in (BV_{\rm loc}\cap L^{1}_{\rm loc})(\Omega, \mathbb{R}^{4})$ is called an entropy solution
of the initial-boundary value problem \eqref{eq:1.11}--\eqref{eq:1.12}
in $\Omega\subset \R^2_+$ if
$U^{(\tau)}$ satisfies the following{\rm :}

\begin{enumerate}
\item[\rm (i)] For any $\boldsymbol{\varphi}=(\varphi_1, \varphi_2,\varphi_3,\varphi_4) \in C^{\infty}(\bar{\Omega})$ with  $\varphi_2=\varphi_3=0$ on $\Gamma$,
\begin{align}
&\iint_{\Omega}\Big(\bar{\rho}^{(\tau)}(1+\tau^2\bar{u}^{(\tau)})\partial_{\bar{x}}\varphi_1
+\bar{\rho}^{(\tau)}\bar{v}^{(\tau)}\partial_{\bar{y}}\varphi_1\Big)\,\dd\bar{x}\dd\bar{y}\nonumber\\[2pt]
&\qquad +\int_{\mathcal{I}}\bar{\rho}_{0}(\bar{y})\big(1+\tau^2\bar{u}_0(\bar{y})\big)\varphi_1(0,\bar{y})\,\dd\bar{y}=0,\label{eq:1.17a}
\\[4pt]
&\iint_{\Omega}\Big(\big(\bar{\rho}^{(\tau)}\bar{u}^{(\tau)}(1+\tau^2\bar{u}^{(\tau)})+\bar{p}^{(\tau)}\big)\partial_{\bar{x}}\varphi_2
+\bar{\rho}^{(\tau)}\bar{u}^{(\tau)}\bar{v}^{(\tau)}\partial_{\bar{y}}\varphi_2\Big)\,\dd\bar{x}\dd\bar{y}\nonumber\\[2pt]
&\qquad+\int_{\mathcal{I}}\big(\bar{\rho}_{0}u_{0}(1+\tau^2\bar{u}_0)+p_{0}\big)\varphi_2(0,\bar{y})\,\dd\bar{y}=0,\label{eq:1.17b}
\\[5pt]
&\iint_{\Omega}\Big(\bar{\rho}^{(\tau)}\bar{v}^{(\tau)}(1+\tau^2\bar{u}^{(\tau)})\partial_{\bar{x}}\varphi_3
+\big(\bar{\rho}^{(\tau)}(\bar{v}^{(\tau)})^2+\bar{p}^{(\tau)}\big)\partial_{\bar{y}}\varphi_3\Big)\,\dd\bar{x}\dd\bar{y}\nonumber\\[2pt]
&\qquad +\int_{\mathcal{I}}\bar{\rho}_{0}(\bar{y})\bar{v}_{0}(\bar{y})\big(1+\tau^2\bar{u}_0(\bar{y})\big)\varphi_3(0,\bar{y})\, \dd\bar{y}=0,
   \label{eq:1.17c}  \\[4pt]
&\iint_{\Omega}\Big(\bar{\rho}^{(\tau)}\bar{\mathcal{B}}^{(\tau)}\big(1+\tau^2\bar{u}^{(\tau)}\big)\partial_{\bar{x}}\varphi_4
+\bar{\rho}^{(\tau)}\bar{v}^{(\tau)}\bar{\mathcal{B}}^{(\tau)}\partial_{\bar{y}}\varphi_4\Big)\,\dd\bar{x}\dd\bar{y}\nonumber\\[2pt]
&\qquad
+\int_{\mathcal{I}}\bar{\rho}_{0}(\bar{y})\bar{\mathcal{B}}_0(\bar{y})\big(1+\tau^2\bar{u}_0(\bar{y})\big)
\varphi_4(0,\bar{y})\,\dd\bar{y}=0, \label{eq:1.17d}
\end{align}
and $U^{(\tau)}$ satisfies the boundary condition \eqref{eq:1.13} in the trace sense, where $\bar{\mathcal{B}}^{(\tau)}\doteq \bar{u}^{(\tau)}+\frac{1}{2}(\bar{v}^{(\tau)})^2
+\frac{\gamma}{\gamma-1}\frac{\bar{p}^{(\tau)}}{\bar{\rho}^{(\tau)}}+\frac{1}{2}\tau^2(\bar{u}^{(\tau)})^2$
and $\bar{\mathcal{B}}_0\doteq \bar{u}_0+\frac{1}{2}\bar{v}^2_0+\frac{\gamma}{\gamma-1}\frac{\bar{p}_0}{\bar{\rho}_0}
+\frac{1}{2}\tau^2\bar{u}^2_0$.

\smallskip
\item[\rm (ii)] For any test function $\psi\in C^{\infty}(\bar{\Omega})$ with $\psi\geq0$, the following entropy inequality holds{\rm :}
\begin{align}\label{eq:1.18}
&\iint_{\Omega}\Big((\bar{\rho}^{(\tau)})^{1-\gamma}\bar{p}^{(\tau)}\big(1+\tau^2\bar{u}^{(\tau)}\big)\partial_{\bar{x}}\psi+
\bar{v}^{(\tau)}(\bar{\rho}^{(\tau)})^{1-\gamma}\bar{p}^{(\tau)}\partial_{\bar{y}}\psi\Big)\,\dd\bar{x}\dd\bar{y}\nonumber\\[2pt]
&\qquad+\int_{\mathcal{I}}(\bar{\rho}_0(\bar{y}))^{1-\gamma}\bar{p}_{0}(\bar{y})\big(1+\tau^2\bar{u}_0(\bar{y})\big)\psi(0,\bar{y})\,\dd\bar{y}\leq0.
\end{align}
\end{enumerate}
\end{definition}

We first consider the following problem in this paper.

\bigskip
\noindent
{\bf Problem I.}
\emph{For a given hypersonic similarity parameter $a_{\infty}$ and a sufficiently small fixed parameter $\tau>0$,
seek a global entropy { solution} $U^{(\tau)}=(\bar{\rho}^{(\tau)},\bar{u}^{(\tau)},\bar{v}^{(\tau)},\bar{p}^{(\tau)})^{\top}\in BV \cap L^{1}$
of the initial-boundary value problem \eqref{eq:1.11}--\eqref{eq:1.12} in the sense of {\rm Definition \ref{def:1.1}},
and then
show the entropy solution $U^{(\tau)}$ is $L^1$--stable with respect to the initial data $U_{0}$ and the boundary data $g$.
}

\medskip
Denote
\begin{eqnarray}\label{eq:1.16}
\underline{U}\doteq
(1, 0, 0, \frac{1}{\gamma a^{2}_{\infty}})^{\top}, \qquad
\underline{\Omega}\doteq
\{(\bar{x}, \bar{y})\,:\,\bar{x}>0, {\bar{y}<0}\}.
\end{eqnarray}
It is direct to see that, if $g\equiv 0$ and
$U_{0}(\bar{y})=(1, 0, 0, \frac{1}{\gamma a^{2}_{\infty}})^{\top}$,
then the constant state $\underline{U}$ is an entropy solution of Problem I.
We call $\underline{U}$ as a \emph{background\ solution} of Problem I.

Once Problem I
is solved, a natural question is whether we can establish the optimal convergence rate for the entropy
solution $U^{(\tau)}$ as $\tau\rightarrow 0$,
which can be formulated as follows:

\bigskip
\noindent
{\bf Problem II}.
\emph{Let $U^{(\tau)}$ be the entropy solutions of
{\rm Problem I},
and let $U=(\bar{\rho}, \bar{u},\bar{v}, \bar{p})^{\top}$ be the entropy solution of
the initial-boundary value problem \eqref{eq:1.12}--\eqref{eq:1.15}.
Then, for a given hypersonic similarity parameter $a_\infty$, can the optimal convergence rate of
sequence $U^{(\tau)}$ to $U$ with respect to $\tau^2$ be established for sufficiently small $\tau>0$?}

\medskip
Before presenting our main results of this paper to answer both Problems I--II,
we first make some basic assumptions.

\medskip
\noindent
{\bf Assumptions}: Both the initial data $U_0(\bar{y})$ and the boundary function $g(\bar{x})$ satisfy
the following assumptions:

\smallskip
\begin{enumerate}
\item[$\mathbf{(U_{0})}$] $U_{0}(\bar{y})=(\bar{\rho}_0, \bar{u}_{0}, \bar{v}_{0},\bar{p}_{0})^{\top}(\bar{y})$
satisfies $U_{0}-\underline{U}\in (BV\cap L^{1})(\mathcal{I}; \mathbb{R}^{4})$ and {$\bar{\rho}_{0}>0$}.

\smallskip
\item[$\mathbf{(g)}$] $g(\bar{x})\in C^{0,1}(\mathbb{R}_{+}; \mathbb{R})$ and $g'(\bar{x})\in (BV\cap L^{1})(\mathbb{R}_{+}; \mathbb{R})$
such that
\begin{equation}\label{eq:1.20}
g(0)=0, \qquad  g'(0)\leq 0, \qquad g(\bar{x})< 0 \ \ \mbox{for $\bar{x}>0$},
\end{equation}
where $g'(\bar{x})$ is the derivative of $g(\bar{x})$ at the differentiable point $\bar{x}$
on boundary $\Gamma$.
\end{enumerate}

\medskip
The first main result is the existence of solutions of Problem I.

\begin{theorem}\label{thm:1.1}
Under assumptions $\mathbf{(U_0)}$ and $\mathbf{(g)}$, for a given fixed hypersonic similarity parameter
$a_{\infty}$ as defined by \eqref{eq:1.9}, there exist constants $\varepsilon_{0}>0, \tau_{0}>0$, and $C_1>0$
depending only on $(\underline{U},a_{\infty})$ such that,
for any $\varepsilon \in(0, \varepsilon_{0})$ and $\tau\in(0,\tau_{0})$, if
\begin{eqnarray}\label{eq:1.21}
\|U_{0}(\bar{y})-\underline{U}\|_{BV(\mathcal{I})}
+|g'(0)|+\|g'\|_{BV([0,\infty))}\leq \varepsilon,
\end{eqnarray}
then there exists a uniformly Lipschtiz constant $L>0$ depending only on $(\underline{U}, a_{\infty})$
so that the initial-boundary value problem \eqref{eq:1.11}--\eqref{eq:1.12} admits domains $\mathcal{D}_{\bar{x}}$
and a unique map $\mathcal{P}^{(\tau)}$
\begin{eqnarray}\label{eq:1.22}
\mathcal{P}^{(\tau)}(\bar{x},\bar{x}_{0})\,:\,\mathcal{D}_{\bar{x}_{0}}\longmapsto \mathcal{D}_{\bar{x}} \qquad\,\,\,
\mbox{for all $\bar{x}\geq\bar{x}_{0}\geq 0$},
\end{eqnarray}
satisfying the following properties{\rm :}

\begin{enumerate}
\item[\rm (i)]\, For any $\bar{x}\geq 0$ and $U^{(\tau)}(\bar{x},\bar{y})\in \mathcal{D}_{\bar{x}}$,
$$
\mathcal{P}^{(\tau)}(\bar{x},\bar{x})(U^{(\tau)}(\bar{x},\bar{y}))=U^{(\tau)}(\bar{x},\bar{y}),
$$
and, for all $0\leq\bar{x}_{0}\leq\bar{x}' \leq \bar{x}$ and $U^{(\tau)}(\bar{x}_{0},\bar{y})\in \mathcal{D}_{\bar{x}_{0}}$,
\begin{eqnarray}\label{eq:1.23}
\mathcal{P}^{(\tau)}(\bar{x},\bar{x}_{0})(U^{(\tau)}(\bar{x}_{0},\bar{y}))
=\mathcal{P}^{(\tau)}(\bar{x},\bar{x}')\circ \mathcal{P}^{(\tau)}(\bar{x}',\bar{x}_{0})(U^{(\tau)}(\bar{x}_{0},\bar{y}));\qquad
\end{eqnarray}

\item[\rm (ii)]\, $U^{(\tau)}(\bar{x},\bar{y})=\mathcal{P}^{(\tau)}(\bar{x},0)(U_{0}(\bar{y}))$ is the entropy solution
of {\rm Problem I}
for any $U_{0}(\bar{y})\in \mathcal{D}_{0}$
and satisfies
\begin{align}
&\sup_{\bar{x}>0}\|U^{(\tau)}(\bar{x},\cdot)-\underline{U}\|_{L^{\infty}({(-\infty, g(\bar{x})))}}\nonumber\\
&\,\,+ \sup_{\bar{x}>0}\|U^{(\tau)}(\bar{x},\cdot)-\underline{U}\|_{BV((-\infty, g(\bar{x})))}
\leq C_{1}\varepsilon,\label{eq:1.23b}
\end{align}
where constant $C_{1}>0$ depends only on $(\underline{U},a_{\infty})${\rm ;}

\item[\rm (iii)]\, If $\tilde{\mathcal{P}}^{(\tau)}$ and $\tilde{\mathcal{D}}_{\bar{x}}$ are the map and domains corresponding
to another boundary $\bar{y}=\tilde{g}(\bar{x})$, then, for any
$U^{(\tau)}(\bar{x}_{0},{\bar{y}})\in \mathcal{D}_{\bar{x}_{0}}$, $V^{(\tau)}(\bar{x}_{0},{\bar{y}})\in \tilde{\mathcal{D}}_{\bar{x}_0}$,
and $\bar{x}''>\bar{x}'\geq\bar{x}_0$,
\begin{align}\label{eq:1.24}
&\big\|\mathcal{P}^{(\tau)}(\bar{x}',\bar{x}_{0})(U^{(\tau)}(\bar{x}_{0},\cdot))-\tilde{\mathcal{P}}^{(\tau)}(\bar{x}'',\bar{x}_{0})(V^{(\tau)}(\bar{x}_{0},\cdot))
\big\|_{L^{1}((-\infty,\hat{g}(\bar{x}'')))}\nonumber\\[1pt]
&\leq
L\big(|\bar{x}'-\bar{x}''|+\|U^{(\tau)}(\bar{x}_{0},\cdot)-V^{(\tau)}(\bar{x}_{0},\cdot)\|_{L^{1}((-\infty,\hat{g}(\bar{x}')))}
+\|g'-\tilde{g}'\|_{L^{1}((\bar{x}_0,\infty))}\big),\nonumber\\[1mm]
\end{align}
where $\,\hat{g}(\bar{x})=\max\{g(\bar{x}), \tilde{g}(\bar{x})\}$,
$\,\mathcal{P}^{(\tau)}(\bar{x}',\bar{x}_{0})(U^{(\tau)}(\bar{x}_{0},\cdot))\doteq \underline{U}$ on $(g(\bar{x}),\hat{g}(\bar{x}'))$,\\[2mm]
and $\tilde{\mathcal{P}}^{(\tau)}(\bar{x}'',\bar{x}_{0})(V^{(\tau)}(\bar{x}_{0},\cdot))\doteq \underline{U}$ on $(\tilde{g}(\bar{x}),\hat{g}(\bar{x}'))$.
\end{enumerate}
\end{theorem}

\noindent
{In Theorem \ref{thm:1.1} and in the sequel, $\|\cdot\|_{BV}$ represents the total variation of the state on the domain under consideration.}

\smallskip
Based on Theorem \ref{thm:1.1}, our second main result is to provide a positive answer to Problem II.

\begin{theorem}\label{thm:1.2}
Under the same assumptions in {\rm Theorem \ref{thm:1.1}}, let
$U^{(\tau)}(\bar{x}, \bar{y})$
be the entropy solution of {\rm Problem I}
as given in {\rm Theorem \ref{thm:1.1}}.
Let $U(\bar{x},\bar{y})$
be the entropy solution of the initial-boundary value problem \eqref{eq:1.12}--\eqref{eq:1.15}
given in {\rm Proposition A.2}.
Then there exist small parameters $\varepsilon^*_{0}<\varepsilon_{0}$ and $\tau^{*}_{0}<\tau_{0}$
depending only on $(\underline{U},a_{\infty})$ such that, for any $\varepsilon\in (0, \varepsilon_0^*)$ and $\tau\in (0, \tau_0^*)$,
the following optimal $L^1$--convergence rate
holds{\rm :}
\begin{eqnarray}\label{eq:1.25}
\|U^{(\tau)}(\bar{x}, \cdot)-U(\bar{x}, \cdot)\|_{L^{1}((-\infty, g(\bar{x})))}
\leq C_2\bar{x} \tau^{2},
\end{eqnarray}
where $C_2>0$ depends only on $\|U_{0}(\cdot)\|_{BV(\mathcal{I})}$, $|g'(0)|$, $\|g'(\cdot)\|_{BV(\mathbb{R}_{+})}$, and $L$,
but is independent of $\tau$ and $\bar{x}$.
\end{theorem}

Since $U^{(\tau)}=\big(\bar{\rho}^{(\tau)}, \bar{u}^{(\tau)}, \bar{v}^{(\tau)}, \bar{p}^{(\tau)}\big)^{\top}$
and $U=\big(\bar{\rho}, \bar{u}, \bar{v}, \bar{p}\big)^{\top}$, estimate \eqref{eq:1.25} is equivalent to
\begin{eqnarray}\label{eq:1.25a}
\big\|\big(\bar{\rho}^{(\tau)}-\bar{\rho},\bar{u}^{(\tau)}-\bar{u}, \bar{v}^{(\tau)}-\bar{v}, \bar{p}^{(\tau)}-\bar{p}\big)(\bar{x},\cdot)\big\|_{L^{1}((-\infty, g(\bar{x})))}
\leq C_2\bar{x} \tau^{2}.
\end{eqnarray}
The order of $\tau^2$ in Theorem \ref{thm:1.2} for a finite length of the wedge is consistent with
what the Newtonian-Busemann law is predicted, which states that the error term should be of order $\tau^2$
(see \cite[(3.29), pp.\,67]{anderson} for more details).
Moreover, this $L^1$--convergence rate in \eqref{eq:1.25a} is optimal for the infinite length of wedge,
as indicated by a special solution to be demonstrated in the last part of \S4 below.

{
Mathematically, the main difficulties for solving Problems I--II are that the wedge boundary conditions are both characteristic.
One of the boundary conditions is of Neumann-type, while the other one is of Dirichlet-type.
Both systems considered have two characteristics with multiplicity, and the corresponding characteristic fields are linearly degenerate.
They are different from the ones studied in \cite{colombo-guerra, donadello-marson, kuang-xiang-zhang-2} for the non-characteristic cases.
Moreover, due to the curved boundary, the error estimate of the semigroup for the Cauchy problem as stated in \cite{bressan} can not be applied directly
so that the approaches used in \cite{chen-christoforou-zhang-1, chen-christoforou-zhang-2, chen-xiang-zhang, zhang-3}
for the Cauchy problem have to be further developed for the initial-boundary value problem with multiplicity characteristic boundary.

In fact,
the comparison between the solutions of the Cauchy problems governed by two different systems
was considered in \cite{chen-christoforou-zhang-1, chen-christoforou-zhang-2, chen-xiang-zhang, zhang-3},
provided that one of them is the standard Riemann semigroup and the other is a global entropy solution obtained by the wave-front tracking scheme.
As far as we have known,
the result established in this paper is the first rigorous mathematical result on the comparison of two entropy solutions
of the compressible Euler equations with characteristic boundary conditions.

To overcome these difficulties and establish the main theorems, we first need to construct a new Lipschitz continuous map
by carefully piecing together the Riemann solutions with a piecewise straight and characteristic boundary,
and then to establish the error estimate of this map with this boundary.

\smallskip
More precisely, to establish Theorem \ref{thm:1.1}, we first proceed with the following four main steps:

\vspace{2pt}
$\mathbf{Step\ 1.}$ We approximate the boundary: $y=g(x)$, by a piecewise straight curve $y=g_{h}(x)$ with a mesh length $h$,
and then construct a global $(\nu,h)$--approximate solution $U^{(\tau)}_{h,\nu}$ of the initial-boundary value problem \eqref{eq:1.11}--\eqref{eq:1.12}
for $\tau \neq 0$ with boundary $y=g_{h}(x)$ and piecewise constant initial data $U^{\nu}_{0}$ via the wave-front tracking scheme.

\vspace{1pt}
$\mathbf{Step\ 2.}$ To derive the $L^1$--stability of the approximate solution $U^{(\tau)}_{h,\nu}$, we construct a weighted Lyapunov functional $\mathscr{L}^{(\tau)}$ for the two approximate solutions $U^{(\tau)}_{h,\nu}$ and $V^{(\tau)}_{h',\nu}$ corresponding to the initial-boundary data $(U^{\nu}_{0}, g_{h})$ and $(V^{\nu}_0, g_{h'}(x))$, respectively; see \eqref{eq:2.41} below for more details.

\vspace{1pt}
$\mathbf{Step\ 3.}$ We differentiate $\mathscr{L}^{(\tau)}$ with respect to $x$ at the non-interaction point to prove that $\mathscr{L}^{(\tau)}$ decreases along the flow direction $x>0$. The proof is divided into the two cases:

\vspace{2pt}
(i) {\it The point is away from the boundary}. This can be dealt with as done in \cite{bressan-liu-yang}.

(ii) {\it The point is near the boundary}. For this case, since the boundary is a characteristic that corresponds to the linearly degenerate field, the linear combination of ${q^{(\tau)}_i}, 1\leq i\leq 4$, in $\frac{\dd\mathscr{L}^{(\tau)}}{\dd x}$ can not be canceled by using the usual weighting methods.
Fortunately, we observe that the flow angle is unchanged across the linearly degenerate characteristic field.
Based on this observation, we can establish a new relation of the strengths between the waves of the {genuinely} nonlinear fields and the linear degenerate {characteristic} fields; see {Lemma \ref{lem:2.8}} in detail.
By this relation, we can see that
functional $\mathscr{L}^{(\tau)}$ is almost decreasing.
{Moreover, by the invariance property of the flow angle and the pressure when crossing the linearly degenerate characteristic field, we can further
establish the $L^1$-comparison estimate for the flow slope function with respect to the different boundary curves.}

\vspace{1pt}
$\mathbf{Step\ 4.}$ With these, we can show that $U^{(\tau)}_{h,\nu}$ converges to $U^{(\tau)}_{h}$ uniquely as $\nu\rightarrow \infty$. Hence, we establish the global existence of a Lipschitz continuous map $\mathcal{P}^{(\tau)}_{h}$ with $U^{(\tau)}_{h}=\mathcal{P}^{(\tau)}_{h}(\bar{x},0)(U_{0})$ and an $L^{1}$--stability estimate
with boundary function $g_{h}(\bar{x})$.
Then it follows from the limit: $\mathcal{P}^{(\tau)}_{h}(\bar{x},0)(U_{0})\rightarrow \mathcal{P}^{(\tau)}(\bar{x},0)(U_{0})$ as $h\rightarrow 0$
to establish the existence {and $L^1$--stability} of the global entropy solution of Problem I.

Similarly, we can also establish the {existence of the approximate solution} $U_{h,\nu}$ of problem \eqref{eq:1.12}--\eqref{eq:1.15}
for piecewise constant initial data and piecewise straight boundary; see Appendix A for the details.
Therefore, taking the limits: $\nu\rightarrow \infty$ first and $h\rightarrow 0$ second as above,
we can also establish the existence of the global entropy solution of problem \eqref{eq:1.12}--\eqref{eq:1.15}.

Next, in order to prove Theorem \ref{thm:1.2}, we can proceed in the following way: First, we compare the solutions of the Riemann problems between \eqref{eq:1.11}--\eqref{eq:1.12} and \eqref{eq:1.12}--\eqref{eq:1.15} case by case.
Then, applying the Lipschitz continuity of map $\mathcal{P}^{(\tau)}_{h}$ of problem \eqref{eq:1.11}--\eqref{eq:1.12},
we can obtain the local $L^1$--estimates between the approximate solution $U_{h,\nu}(\bar{x}+s,\cdot)$ of problem \eqref{eq:1.12}--\eqref{eq:1.15}
and trajectory $\mathcal{P}^{(\tau)}_{h}(\bar{x}+s,\bar{x})(U_{h,\nu}(\bar{x},\cdot))$ for $s>0$ sufficiently small.
Based on these, to obtain the global $L^1$--estimate between the approximate solution $U_{h,\nu}$ and
trajectory
$\mathcal{P}^{(\tau)}_{h}(\bar{x},0)(U^{\nu}_{0}(\cdot))$, we establish the following formula
with an approximate boundary $g_{h}$ (see Lemma \ref{lem:5.1}):
\begin{align*}
&\big\|\mathcal{P}^{(\tau)}_{h}(\bar{x},0)(U^{\nu}(0, \cdot))-U_{h,\nu}(\bar{x}, \cdot)\big\|_{L^{1}((-\infty, g_{h}(\bar{x})))}\\[2pt]
&\leq L\int^{\bar{x}}_{0}
\lim_{s\rightarrow0^{+}}\inf\frac{\big\|\mathcal{P}^{(\tau)}_{h}(\xi+s, \xi)(U_{h,\nu}(\xi,\cdot))-U_{h,\nu}(\xi+s, \cdot)\big\|_{L^{1}((-\infty, g_{h}(\xi+s)))}}{s}\,\dd\xi.
\end{align*}

Finally,
we let $\nu\rightarrow \infty$ and $h\rightarrow 0$ to show estimate \eqref{eq:1.25}, which justifies the Van Dyke's similarity theory rigorously.
}

\smallskip
The remaining sections of this paper are organized as follows:
In \S 2, we consider the entropy solutions $U^{(\tau)}$ of Problem I via a modified wave-front tracking algorithm
and then {complete} the proof of Theorem \ref{thm:1.1}.
Meanwhile, by a similar argument, we can also establish the existence {theorem}
for the initial-boundary value
problem \eqref{eq:1.12}--\eqref{eq:1.15}
by the modified wave-front tracking scheme, which {is} given in Appendix A.
After that, we compare the Riemann solvers with boundaries between systems \eqref{eq:1.11} and \eqref{eq:1.14},
and then obtain the local $L^1$--estimates between the approximate solution $U_{h,\nu}(\bar{x}+s,\cdot)$
and trajectory $\mathcal{P}^{(\tau)}_{h}(\bar{x}+s,\bar{x})(U_{h,\nu}(\bar{x},\cdot))$ for $s>0$ sufficiently small in \S 3.
In \S 4, we establish the global $L^1$--estimate between $\mathcal{P}^{(\tau)}_{h}(\bar{x},0)(U^{\nu}_{0})$
and the approximate solution $U_{h,\nu}$ of the initial-boundary {value} problem \eqref{eq:1.12}--\eqref{eq:1.15}
to show estimate \eqref{eq:1.25}. Finally, we verify that the convergence rate obtained in \eqref{eq:1.25} is optimal
by considering a special solution.

\smallskip
For notational convenience in the following sections,
$C>0$ always represents a universal constant that depends only on $(\underline{U}, a_\infty)$,
$O(1)$ is a universal function so that $|O(1)|\le C$,
and $(\bar{\rho}^{(\tau)}, \bar{u}^{(\tau)}, \bar{v}^{(\tau)}, \bar{p}^{(\tau)})$,
$(\bar{\rho}, \bar{u}, \bar{v}, \bar{p})$, and $(\bar{x}, \bar{y})$ are denoted as $(\rho^{(\tau)}, u^{(\tau)}, v^{(\tau)}, p^{(\tau)})$,
$(\rho, u, v, p)$, and $(x,y)$, correspondingly.

\section{Modified Wave-Front Tracking Scheme and Well-Posedness of {Problem I}}\setcounter{equation}{0}
In this section, we establish the global existence of entropy solutions of Problem I
by introducing a modified wave-front tracking algorithm
in $BV\cap L^{1}$.
Constructing a weighted entropy functional carefully involving a characteristic boundary,
we further show the $L^{1}$--stability estimate between the two entropy solutions with respect to the initial and boundary data,
and then we complete the proof of Theorem \ref{thm:1.1}.

\subsection{Elementary wave curves of system \eqref{eq:1.11}}

By direct computation,
the eigenvalues of system \eqref{eq:1.11} are
\begin{align}
&\lambda^{(\tau)}_{j}(U^{(\tau)},\tau^{2})\nonumber\\[2pt]
&=\frac{\big(1+\tau^{2}u^{(\tau)}\big)v^{(\tau)}+(-1)^{j}c^{(\tau)}\sqrt{\big(1+\tau^{2}u^{(\tau)}\big)^{2}
+\tau^{2}\big((v^{(\tau)})^{2}-(c^{(\tau)})^{2}\big)}}{\big(1+\tau^{2}u^{(\tau)}\big)^{2}-\tau^{2}\big(c^{(\tau)}\big)^{2}} \label{eq:2.2}
\end{align}
for $j=1, 4$, and
\begin{equation}\label{eq:2.3}
\lambda^{(\tau)}_{i}(U^{(\tau)},\tau^{2})=\frac{v^{(\tau)}}{1+\tau^{2}u^{(\tau)}} \qquad\,\, \mbox{for $i=2, 3$},
\end{equation}
with the corresponding right eigenvectors:
\begin{align}\label{eq:2.4}
\tilde{\mathbf{r}}^{(\tau)}_{j}(U^{(\tau)}, \tau^{2})=(\frac{\big(1+\tau^{2}(\lambda^{(\tau)}_{j})^2\big)\rho^{(\tau)}}{(1+\tau^{2}u^{(\tau)})\lambda^{(\tau)}_{j}-v^{(\tau)}}, -\lambda^{(\tau)}_{j}, 1, \big((1+\tau^{2}u^{(\tau)})\lambda^{(\tau)}_{j}-v^{(\tau)}\big)\rho^{(\tau)})^{\top}
\end{align}
for $j=1, 4$, and
\begin{eqnarray}\label{eq:2.5}
\tilde{\mathbf{r}}^{(\tau)}_{2}(U^{(\tau)}, \tau^{2})=\big(0, 1+\tau^2u^{(\tau)}, \tau^{2}v^{(\tau)}, 0\big)^{\top},
\quad \tilde{\mathbf{r}}^{(\tau)}_{3}(U^{(\tau)}, \tau^{2})=\big(1, 0,0 ,0 \big)^{\top}.\quad
\end{eqnarray}

For $\lambda^{(\tau)}_{j}$ and $\tilde{\mathbf{r}}^{(\tau)}_{j}$, $1\leq j\leq 4$, we have the following lemma
to know that the $1^{\rm st}$ and
$4^{\rm th}$ characteristic fields are genuinely nonlinear, and the $2^{\textnormal{nd}}$ and $3^{\textnormal{rd}}$ characteristic fields
are linearly degenerate.

\begin{lemma}\label{lem:2.1}
For a fixed hypersonic similarity parameter $a_{\infty}$ defined in \eqref{eq:1.9},
there exist positive constants $\tau_{1}$ and $\epsilon_{1}$ depending only on $(\underline{U}, a_{\infty})$
such that, for $\epsilon\in (0,\epsilon_{1})$ and $\tau\in (0,\tau_{1})$,
if $U^{(\tau)}\in \mathcal{O}_{\epsilon}(\underline{U})$, then
\begin{align}
&\nabla_{U^{(\tau)}}\lambda^{(\tau)}_{j}(U^{(\tau)},\tau^{2})\cdot \tilde{\mathbf{r}}^{(\tau)}_{j}(U^{(\tau)}, \tau^{2})>0
\qquad \mbox{for $j=1, 4$},\label{eq:2.6}\\[1mm]
&\nabla_{U^{(\tau)}}\lambda^{(\tau)}_{i}(U^{(\tau)},\tau^{2})\cdot \tilde{\mathbf{r}}^{(\tau)}_{i}(U^{(\tau)}, \tau^{2})=0
\qquad \mbox{for $i=2, 3$}.\label{eq:2.7}
\end{align}
\end{lemma}

\begin{proof}
For $j=1,4$, it is direct to see that $\lambda^{(\tau)}_{j}$ satisfies the following equation:
\begin{align}\label{eq:2.7a}
\big((1+\tau^{2}u^{(\tau)})^2-\tau^{2}(c^{(\tau)})^{2}\big)(\lambda^{(\tau)}_{j})^{2}
-2(1+\tau^{2}u^{(\tau)})v^{(\tau)}\lambda^{(\tau)}_{j}+(v^{(\tau)})^{2}-(c^{(\tau)})^{2}=0,
\end{align}
where $c^{(\tau)}=\sqrt{\frac{\gamma p^{(\tau)}}{\rho^{(\tau)}}}$.
Then, taking the derivative on \eqref{eq:2.7a} with respect to $\rho^{(\tau)}$,
we obtain
\begin{eqnarray*}
&&2\Big(\big((1+\tau^{2}u^{(\tau)})^{2}-\tau^{2}(c^{(\tau)})^{2}\big)\lambda^{(\tau)}_{j}-(1+\tau^{2}u^{(\tau)})v^{(\tau)}\Big)
\frac{\partial\lambda^{(\tau)}_{j}}{\partial \rho^{(\tau)}}\\
&&\,\, +\frac{\big(1+(\lambda^{(\tau)}_{j})^{2}\tau^{2}\big)(c^{(\tau)})^{2}}{\rho^{(\tau)}}=0.
\end{eqnarray*}
Therefore, we have
\begin{eqnarray}\label{eq:2.7b}
\frac{\partial\lambda^{(\tau)}_{j}}{\partial\rho^{(\tau)}}=-\frac{\big(1+\tau^{2}(\lambda^{(\tau)}_{j})^{2}\big)(c^{(\tau)})^{2}}
{2\Big(\big((1+\tau^{2}u^{(\tau)})^{2}-\tau^{2}(c^{(\tau)})^{2}\big)
\lambda^{(\tau)}_{j}-(1+\tau^{2}u^{(\tau)})v^{(\tau)}\Big)\rho^{(\tau)}}.
\end{eqnarray}

Similarly, we can also obtain
\begin{eqnarray}\label{eq:2.7c}
\begin{split}
&\frac{\partial\lambda^{(\tau)}_{j}}{\partial u^{(\tau)}}
=-\frac{\tau^{2}\big((1+\tau^{2}u^{(\tau)})\lambda^{(\tau)}_{j}-v^{(\tau)}\big)\lambda^{(\tau)}_{j}}
{\big((1+\tau^{2}u^{(\tau)})^{2}-\tau^{2}(c^{(\tau)})^{2}\big)\lambda^{(\tau)}_{j}-(1+\tau^{2}u^{(\tau)})v^{(\tau)}},\\[2pt]
&\frac{\partial\lambda^{(\tau)}_{j}}{\partial v^{(\tau)}}=\frac{\big(1+\tau^{2}u^{(\tau)})\lambda^{(\tau)}_{j}-v^{(\tau)}}
{\big((1+\tau^{2}u^{(\tau)})^{2}-\tau^{2}(c^{(\tau)})^{2}\big)\lambda^{(\tau)}_{j}-(1+\tau^{2}u^{(\tau)})v^{(\tau)}},\\[2pt]
&\frac{\partial\lambda^{(\tau)}_{j}}{\partial p^{(\tau)}}=\frac{\gamma\big(1+\tau^{2}(\lambda^{(\tau)}_{j})^2\big)}
{2\Big(\big((1+\tau^{2}u^{(\tau)})^{2}-\tau^{2}(c^{(\tau)})^{2}\big)\lambda^{(\tau)}_{j}-(1+\tau^{2}u^{(\tau)})v^{(\tau)}\Big)\rho^{(\tau)}}.
\end{split}
\end{eqnarray}

Thus, by direct calculation, we have
{\small
\begin{align*}
&\nabla_{U^{(\tau)}}\lambda^{(\tau)}_{1}\cdot \tilde{\mathbf{r}}^{(\tau)}_{1}\\[2pt]
&=-\frac{\big(1+\tau^{2}(\lambda^{(\tau)}_{1})^{2}\big)^{2}(c^{(\tau)})^{2}-(\gamma+2)\big(1+\tau^{2}(\lambda^{(\tau)}_{1})^{2}\big)
\big((1+\tau^{2}u^{(\tau)})\lambda^{(\tau)}_{1}-v^{(\tau)}\big)^{2}}
{2\big((1+\tau^{2}u^{(\tau)})\lambda^{(\tau)}_{1}-v^{(\tau)}\big)\Big(\big((1+\tau^{2}u^{(\tau)})^{2}-\tau^{2}(c^{(\tau)})^{2}\big)\lambda^{(\tau)}_{1}
-(1+\tau^{2}u^{(\tau)})v^{(\tau)}\Big)}.
\end{align*}
}

By equations \eqref{eq:2.2} and \eqref{eq:2.7a}, we have
\begin{eqnarray*}
\nabla_{U^{(\tau)}}\lambda^{(\tau)}_{1}\cdot \tilde{\mathbf{r}}^{(\tau)}_{1}
=-\frac{(\gamma+1)\big(1+\tau^{2}(\lambda^{(\tau)}_{1})^{2}\big)
  \big((1+\tau^{2}u^{(\tau)})\lambda^{(\tau)}_{1}-v^{(\tau)}\big)}
{2c^{(\tau)}\sqrt{\big(1+\tau^{2}u^{(\tau)}\big)^{2}+\big((v^{(\tau)})^{2}-\tau^{2}(c^{(\tau)})^{2}\big)}}.
\end{eqnarray*}

Similarly, we also have
\begin{eqnarray*}
\nabla_{U^{(\tau)}}\lambda^{(\tau)}_{4}\cdot \tilde{\mathbf{r}}^{(\tau)}_{4}
=\frac{(\gamma+1)\big(1+\tau^{2}(\lambda^{(\tau)}_{4})^{2}\big)\big((1+\tau^{2}u^{(\tau)})\lambda^{(\tau)}_{4}-v^{(\tau)}\big)}
{2c^{(\tau)}\sqrt{\big(1+\tau^{2}u^{(\tau)}\big)^{2}+\big((v^{(\tau)})^{2}-\tau^{2}(c^{(\tau)})^{2}\big)}}.
\end{eqnarray*}

Then it implies that, at the background solution $U^{(\tau)}=\underline{U}$,
\begin{eqnarray*}
\nabla_{U^{(\tau)}}\lambda^{(\tau)}_{j}\cdot \tilde{\mathbf{r}}^{(\tau)}_{j}\Big|_{U^{(\tau)}=\underline{U}}
=\frac{(\gamma+1)a^{4}_{\infty}}{2(a^{2}_{\infty}-\tau^{2})^2} \qquad\,\, \mbox{for $j=1,4$}.
\end{eqnarray*}
Therefore, we can choose small constants $\tau_{1}>0$ and $\epsilon_{1}>0$ depending only on $(\underline{U}, a_{\infty})$
such that, for $\tau\in(0,\tau_{1})$ and $\epsilon\in(0,\epsilon_{1})$,
$\nabla_{U^{(\tau)}}\lambda^{(\tau)}_{j}\cdot \tilde{\mathbf{r}}^{(\tau)}_{j}>0$ for $j=1,4$.

In the same way, by direct computation, we see that
\begin{eqnarray*}
\nabla_{U^{(\tau)}}\lambda^{(\tau)}_{2}\cdot \tilde{\mathbf{r}}^{(\tau)}_{2}=\nabla_{U^{(\tau)}}\lambda^{(\tau)}_{3}\cdot \tilde{\mathbf{r}}^{(\tau)}_{3}=0.
\end{eqnarray*}
This completes the proof.
\end{proof}

\begin{remark}\label{rem:2.1}
If we define
\begin{align}\label{eq:2.7d}
\mathbf{r}^{(\tau)}_{j}(U^{(\tau)},\tau^{2})
={e}^{(\tau)}_{j}(U^{(\tau)},\tau^{2})\,\tilde{\mathbf{r}}^{(\tau)}_{j}(U^{(\tau)},\tau^{2})
\qquad\mbox{for $ j=1,2,3, 4$},
\end{align}
where ${e}^{(\tau)}_{j}(U^{(\tau)},\tau^{2})=\big(\nabla_{U^{(\tau)}}\lambda^{(\tau)}_{j}\cdot \tilde{\mathbf{r}}^{(\tau)}_{j}\big)^{-1}$
for $j=1,4$, and ${e}^{(\tau)}_{j}(U^{(\tau)},\tau^{2})=1$ for $j=2,3$,
then
\begin{align*}
&\nabla_{U^{(\tau)}}\lambda^{(\tau)}_{j}(U^{(\tau)}, \tau^{2})\cdot \mathbf{r}^{(\tau)}_j(U^{(\tau)}, \tau^{2})\equiv1
\qquad \mbox{in $\mathcal{O}_{\epsilon}(\underline{U})$  for $j=1,4$},\\[1mm]
&\nabla_{U^{(\tau)}}\lambda^{(\tau)}_{j}(U^{(\tau)}, \tau^{2})\cdot \mathbf{r}^{(\tau)}_j(U^{(\tau)}, \tau^{2})\equiv0 \qquad \mbox{in $\mathcal{O}_{\epsilon}(\underline{U})$ for $j=2,3$},
\end{align*}
for some $\epsilon>0$ and $\tau>0$ sufficiently small.
Moreover,  at the background state $U^{(\tau)}=\underline{U}$,
\begin{eqnarray}\label{eq:2.8}
{e}^{(\tau)}_{1}(\underline{U},\tau^{2})
={e}^{(\tau)}_{4}(\underline{U},\tau^{2})=\frac{2(a^{2}_{\infty}-\tau^{2})^2}{(\gamma+1)a^{4}_{\infty}},
\end{eqnarray}
from the proof of {\rm Lemma \ref{lem:2.1}}.
Then ${e}^{(\tau)}_{1}(\underline{U},\tau^{2})={e}^{(\tau)}_{4}(\underline{U},\tau^{2})>0$
for some $\tau\in(0,\tau_{1})$ with $\tau_{1}$ depending only on $(\underline{U},a_{\infty})$.
\end{remark}

Based on Remark \ref{rem:2.1}, following the ideas in \cite{kong-yang, smoller},
we can parameterize the elementary wave curves.

\begin{lemma}\label{lem:2.2}
{ Let a hypersonic similarity parameter $a_{\infty}$ defined in \eqref{eq:1.9} be given}.
Then, for any given constant state
$U^{(\tau)}_{b}\in \mathcal{O}_{\epsilon_1}(\underline{U})$,
the $j^{\rm th}$ physical admissible wave curves in $\mathcal{O}_{\epsilon_1}(\underline{U})$
can be parameterized by $\alpha_j$ with
$\sigma^{(\tau)}_{\alpha_j}\mapsto \Phi^{(\tau)}_{j}(\sigma^{(\tau)}_{\alpha_j}; U^{(\tau)}_{b})$, where
$\Phi^{(\tau)}_{j}\in C^{2}\big((-\delta_{1},\delta_{1})\times \mathcal{O}_{\epsilon_{1}}(\underline{U})\times(0,\tau_{1})\big)$
satisfies
\begin{eqnarray}\label{eq:2.9}
\left. \Phi^{(\tau)}_{j}\right|_{\sigma^{(\tau)}_{\alpha_j}=0}= U^{(\tau)}_{b},\quad
\left. \frac{\partial\Phi^{(\tau)}_{j}}{\partial\sigma^{(\tau)}_{\alpha_j}}\right|_{\sigma^{(\tau)}_{\alpha_j}=0}=
\mathbf{r}^{(\tau)}_{j}(U^{(\tau)}_{b},\tau^{2})\qquad \mbox{for $1\leq j\leq 4$},\qquad
\end{eqnarray}
for some $\delta_{1}>0$.
Moreover, for $j=1,4$, if $\sigma^{(\tau)}_{\alpha_j}>0$,
we call it the rarefaction wave and denote the corresponding elementary wave curve
by $\mathcal{R}^{(\tau)}_{j}(U^{(\tau)}_{b})\cap \mathcal{O}_{\epsilon_{1}}(\underline{U})${\rm ;}
while, if $\sigma^{(\tau)}_{\alpha_j}<0$, we call it the shock wave and denote the corresponding
elementary wave curve by $\mathcal{S}^{(\tau)}_{j}(U^{(\tau)}_{b})\cap \mathcal{O}_{\epsilon_{1}}(\underline{U})$.
\end{lemma}

From now on, we denote $\boldsymbol{\alpha}=(\alpha_{1},\alpha_{2},\alpha_{3}, \alpha_{4})$, $\boldsymbol{\sigma}^{(\tau)}_{\boldsymbol{\alpha}}
=\big(\sigma^{(\tau)}_{\alpha_{1}},\sigma^{(\tau)}_{\alpha_{2}},\sigma^{(\tau)}_{\alpha_{3}},
\sigma^{(\tau)}_{\alpha_{4}}\big)$, and
{\small
\begin{align}\label{eq:2.10}
\Phi^{(\tau)}(\boldsymbol{\sigma}^{(\tau)}_{\boldsymbol{\alpha}}; U^{(\tau)}_{b},\tau^2)
\doteq\Phi^{(\tau)}_{4}(\sigma^{(\tau)}_{\alpha_{4}};\Phi^{(\tau)}_{3}(\sigma^{(\tau)}_{\alpha_{3}}; \Phi^{(\tau)}_{2}(\sigma^{(\tau)}_{\alpha_{2}};
\Phi^{(\tau)}_{1}(\sigma^{(\tau)}_{\alpha_{1}}; U^{(\tau)}_{b},\tau^2 ),\tau^2\big),\tau^2),\tau^2).\nonumber\\[2pt]
\end{align}
}

\subsection{Solvability of the Riemann problems}
In this subsection, we study the Riemann problems, including the lateral Riemann problems (with boundaries).

We first consider the Riemann problem for system \eqref{eq:1.11} without the boundary as
\begin{equation}\label{eq:2.11}
\begin{cases}
{\rm System}~\eqref{eq:1.11},\\[4pt]
\left.
U^{(\tau)}\right|_{x=\hat{x}_{0}}=
\begin{cases}
U^{(\tau)}_{a} \quad  &\mbox{for $y>\hat{y}_{0}$},\\[4pt]
U^{(\tau)}_{b} \quad  &\mbox{for $y<\hat{y}_{0}$},
\end{cases}
\end{cases}
\end{equation}
where the constant states $U^{(\tau)}_{a}$ and $U^{(\tau)}_{b}$ denote the $above$ state and $below$ state with respect to
line $y=\hat{y}_{0}$, respectively; see Fig. \ref{fig2.3}.
\vspace{5pt}
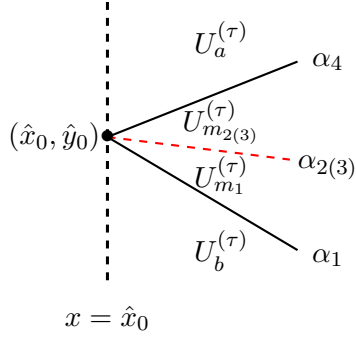
\begin{figure}[ht]
\begin{center}
\begin{tikzpicture}[scale=1.0]
\draw [line width=0.04cm][dashed](-1.5,-1.4) --(-1.5,2.3);
\draw [thick](-1.5,0.5)--(1,1.5);
\draw [thick][dashed][red](-1.5,0.5)--(0.9,0.2);
\draw [thick](-1.5,0.5)--(1,-1);
\node at (1.4,1.5){$\alpha_4$};
\node at (1.4,0.1){$\alpha_{2(3)}$};
\node at (1.4,-1.1){$\alpha_{1}$};
\node at (-1.5, 0.5){$\bullet$};
\node at (-1.5, -1.9){$x=\hat{x}_{0}$};
\node at (0, -1.0){$U^{(\tau)}_{b}$};
\node at (0, 0.7){$U^{(\tau)}_{m_{2(3)}}$};
\node at (0, 0){$U^{(\tau)}_{m_{1}}$};
\node at (0, 1.8){$U^{(\tau)}_{a}$};
\node at (-2.2, 0.5){$(\hat{x}_{0},\hat{y}_{0})$};
\end{tikzpicture}
\end{center}
\caption{The interior Riemann problem }\label{fig2.3}
\end{figure}

\par Then we have the following lemma on the solvability of the Riemann problem \eqref{eq:2.11}.

\begin{lemma}\label{lem:2.3}
For a given hypersonic similarity parameter $a_{\infty}$,
there exist constants $\epsilon_{2}\in (0, \epsilon_{1})$ and $\tau_{2}\in (0, \tau_{1})$
depending only on $(\underline{U},a_{\infty})$ such that,
for any states $U^{(\tau)}_{a}, U^{(\tau)}_{b}\in \mathcal{O}_{\epsilon_{2}}(\underline{U})$
and $\tau\in (0, \tau_{2})$, the Riemann problem \eqref{eq:2.11} admits a unique admissible solution consisting of
four elementary waves $\boldsymbol{\alpha}=(\alpha_{1}, \alpha_{2}, \alpha_{3}, \alpha_{4})$ with strengths
$\boldsymbol{\sigma}^{(\tau)}_{\boldsymbol{\alpha}}=(\sigma^{(\tau)}_{\alpha_{1}},\sigma^{(\tau)}_{\alpha_{2}},
\sigma^{(\tau)}_{\alpha_{3}},\sigma^{(\tau)}_{\alpha_{4}})$.
The constant states $U^{(\tau)}_{m_j}, 0\leq j\leq 4$, in the Riemann solution satisfy
\begin{equation}\label{eq:2.12}
U^{(\tau)}_{m_0}\doteq U^{(\tau)}_{b},\quad U^{(\tau)}_{m_j}
=\Phi^{(\tau)}_j(\boldsymbol{\sigma}^{(\tau)}_{\boldsymbol{\alpha_j}}; U^{(\tau)}_{m_{j-1}},\tau^2)\,\,\,\mbox{for $1\leq j\leq 4$},
\quad U^{(\tau)}_{m_4}\doteq U^{(\tau)}_{a},
\end{equation}
and the properties in \eqref{eq:2.9}.
Moreover, there exists a constant $C_{2,1}>0$ depending only on $\underline{U}$ such that
\begin{equation}\label{eq:2.13}
C^{-1}_{2,1}\sum^{4}_{j=1}|\sigma^{(\tau)}_{\alpha_{j}}|\leq \big|U^{(\tau)}_{a}-U^{(\tau)}_{b}\big|
\leq C_{2,1}\sum^{4}_{j=1}|\sigma^{(\tau)}_{\alpha_{j}}|.
\end{equation}
\end{lemma}

\begin{proof}
To obtain the solvability of the Riemann problem \eqref{eq:2.11}, it suffices to show that
\begin{equation}
\det\bigg(\frac{\partial \Phi^{(\tau)}(\boldsymbol{\sigma}^{(\tau)}; U^{(\tau)}_{b},\tau^2)}{\partial \boldsymbol{\sigma}^{(\tau)}} \bigg)
\Bigg|_{\boldsymbol{\sigma}^{(\tau)}=0}\neq0.
\end{equation}

Notice that

{\small
\begin{align*}
&\det\bigg(\frac{\partial \Phi^{(\tau)}(\boldsymbol{\sigma}^{(\tau)}; U^{(\tau)}_{b},\tau^2)}{\partial \boldsymbol{\sigma}^{(\tau)}} \bigg)
\Bigg|_{\boldsymbol{\sigma}^{(\tau)}=0}\\[5pt]
&=\det\Big(\mathbf{r}^{(\tau)}_{1}(U^{(\tau)}_{b},\tau^{2}),\mathbf{r}^{(\tau)}_{2}(U^{(\tau)}_{b},\tau^{2}),
\mathbf{r}^{(\tau)}_{3}(U^{(\tau)}_{b},\tau^{2}), \mathbf{r}^{(\tau)}_{4}(U^{(\tau)}_{b},\tau^{2})\Big)\\[5pt]
&=e^{(\tau)}_{1}e^{(\tau)}_{4}\begin{vmatrix}
\frac{(1+\tau^{2}\lambda^{(\tau)}_{1})\rho^{(\tau)}}{(1+\tau^{2}u^{(\tau)})\lambda^{(\tau)}_{1}-v^{(\tau)}}
    &&\hspace{-10pt} 0 &&\hspace{-9pt}1 &&\hspace{-9pt} \frac{(1+\tau^{2}\lambda^{(\tau)}_{4})\rho^{(\tau)}}{(1+\tau^{2}u^{(\tau)})\lambda^{(\tau)}_{4}-v^{(\tau)}}  \\[6pt]
  -\lambda^{(\tau)}_{1}&&\hspace{-10pt} 1+\tau^{2}u^{(\tau)}&&\hspace{-9pt} 0 &&\hspace{-9pt}  -\lambda^{(\tau)}_{4}\\[6pt]
1&&\hspace{-10pt} \tau^{2}v^{(\tau)} &&\hspace{-9pt} 0 &&\hspace{-9pt}1 \\[6pt]
\big((1+\tau^{2}u^{(\tau)})\lambda^{(\tau)}_{1}-v^{(\tau)}\big)\rho^{(\tau)} &&\hspace{-10pt} 0 &&\hspace{-9pt} 0 &&\hspace{-9pt} \big((1+\tau^{2}u^{(\tau)})\lambda^{(\tau)}_{4}-v^{(\tau)}\big)\rho^{(\tau)}
\end{vmatrix}
\\[6pt]
&=-e^{(\tau)}_{1}e^{(\tau)}_{4}\big(\lambda^{(\tau)}_{4}-\lambda^{(\tau)}_{1}\big)
\Big((1+\tau^{2}u^{(\tau)})\rho^{(\tau)}-\tau^{2}\rho^{(\tau)}v^{(\tau)}\lambda^{(\tau)}_{2}\Big).
\end{align*}
}
Then, by Remark \ref{rem:2.1}, at the background state $U^{(\tau)}_{b}=\underline{U}$, we have
\begin{eqnarray*}
\det\bigg(\frac{\partial \Phi^{(\tau)}(\boldsymbol{\sigma}^{(\tau)}; U^{(\tau)}_{b},\tau^2)}{\partial \boldsymbol{\sigma}^{(\tau)}}
\bigg)\bigg|_{\boldsymbol{\sigma}^{(\tau)}=0, U^{(\tau)}_{b}=\underline{U}}
=-\frac{8(a^{2}_{\infty}-\tau^{2})^{\frac{7}{2}}}{(\gamma+1)^{2}a^{8}_{\infty}}<0,
\end{eqnarray*}
where $\tau\in(0,\tau_2)$ with choice of sufficiently small $\tau_2>0$.
Therefore, the solvability of the Riemann problem follows from the implicit function theorem directly.
Finally, by \eqref{eq:2.9} and \eqref{eq:2.12}, we can obtain estimates \eqref{eq:2.13}.
\end{proof}

Next, we consider the Riemann problem with a boundary. As shown in Fig. \ref{fig2.5x},
let $\textsc{C}_{l}(x_l, g_l)$, $l=1,2, 3$,
be three points on $\Gamma$ with $x_{1}<x_{2}<x_{3}$ and $g_{l}=g(x_l)$.
Denote
\begin{eqnarray}\label{eq-B1}
\theta_{1}=\arctan\Big(\frac{g_{2}-g_{1}}{x_{2}-x_1}\Big),\quad  \theta_{2}=\arctan\Big(\frac{g_{3}-g_{2}}{x_{3}-x_{2}}\Big),  \quad  \omega=\theta_{2}-\theta_{1}.
\end{eqnarray}

For $l=1,2$, we define
\begin{eqnarray}\label{eq-B2}
\begin{split}
&\Omega_{l}=\{(x,y)\,:\, x_{l}\leq x< x_{l+1}, \ y< g_l+(x-x_{l})\tan(\theta_{l})\},\\[2pt]
&\Gamma_{l}=\{(x,y)\,:\, x_{l}\leq x< x_{l+1}, \ y=g_l+(x-x_{l})\tan(\theta_{l})\}.
\end{split}
\end{eqnarray}
Let $\textbf{n}_{l}$ be the interior unit normal vector to $\Gamma_{l}$, \emph{that is}, $\textbf{n}_{l}=(\sin\theta_{l},-\cos\theta_{l})$.
Consider the following Riemann problem with a boundary $\Gamma_{l}$:
\begin{eqnarray}\label{eq:2.17}
\begin{cases}
{\rm System} \ \eqref{eq:1.11} \quad & \mbox{in $\Omega_{2}$},\\[5pt]
U^{(\tau)}=U^{(\tau)}_b \quad & \mbox{on $\{x=x_{2}\}\cap\Omega_{2}$},\\[5pt]
(1+\tau^{2}u^{(\tau)}_{\Gamma_2},v^{(\tau)}_{\Gamma_2})\cdot\textbf{n}_{2}=0 \quad &\mbox{on $\Gamma_{2}$},
\end{cases}
\end{eqnarray}
where $U^{(\tau)}_b$ is a constant state near $\underline{U}$ and satisfies
\begin{eqnarray}\label{eq:2.18}
(1+\tau^{2}u^{(\tau)}_{b},v^{(\tau)}_{b})\cdot\textbf{n}_{1}=0 \quad & \mbox{on $\Gamma_{1}$}.
\end{eqnarray}

\vspace{5pt}
\begin{figure}[ht]
\begin{center}
\begin{tikzpicture}[scale=1.0]
\draw [line width=0.03cm](-5.0,1.5)--(-3,1)--(-1,1.8);

\draw [line width=0.05cm](-5.0,1.5)--(-5.0,-1.5);
\draw [line width=0.05cm](-3,1)--(-3,-1.5);
\draw [line width=0.05cm](-1.0,1.8)--(-1.0,-1.5);

\draw [thin](-4.8,1.45)--(-4.5, 1.75);
\draw [thin](-4.5,1.38)--(-4.2, 1.68);
\draw [thin](-4.2,1.30)--(-3.9, 1.60);
\draw [thin](-3.9, 1.23)--(-3.6,1.53);
\draw [thin](-3.6, 1.16)--(-3.3,1.46);
\draw [thin](-3.3, 1.08)--(-3.0,1.38);
\draw [thin](-3.0, 1.0)--(-2.8,1.4);
\draw [thin](-2.7,1.12)--(-2.5, 1.52);
\draw [thin](-2.4,1.23)--(-2.2, 1.63);
\draw [thin](-2.1,1.35)--(-1.9, 1.75);
\draw [thin](-1.8,1.47)--(-1.6, 1.87);
\draw [thin](-1.5, 1.59)--(-1.3,1.99);
\draw [thin](-1.2,1.71)--(-1.0, 2.11);

\draw [thick](-3,1)--(-1.8,0.5);
\draw [thick](-3,1)--(-2.0,0.2);

\node at (-5.5, 1.6) {$\textsc{C}_{1}$};
\node at (-3.0, 1.5) {$\textsc{C}_{2}$};
\node at (-0.5, 1.8) {$\textsc{C}_{3}$};
\node at (-1.5, 0.3) {$\alpha_1$};
\node at (-2.2, 1.0) {$U^{(\tau)}_{\Gamma_2}$};
\node at (-2.4, -0.2) {$U^{(\tau)}_{b}$};
\node at (1, 2) {$$};

\node at (-4.0, -0.8) {$\Omega_{1}$};
\node at (-2.0, -0.8) {$\Omega_{2}$};

\node at (-5.0, -1.9) {$x_{1}$};
\node at (-3.0, -1.9) {$x_{2}$};
\node at (-0.9, -1.9) {$x_{3}$};
\end{tikzpicture}
\caption{The Riemann problem near the boundary}\label{fig2.5x}
\end{center}
\end{figure}
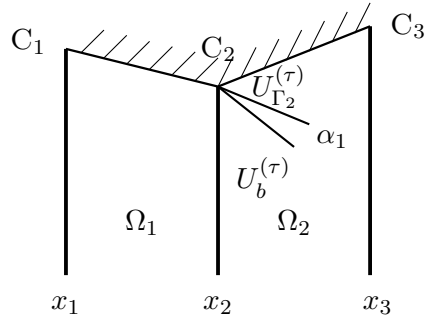

We have the following lemma on the solvability of problem \eqref{eq:2.17}.

\begin{lemma}\label{lem:2.5}
For a given hypersonic similarity parameter $a_{\infty}$,
there exist small constants $\epsilon_{3}\in (0, \epsilon_{1})$ and $\tau_{3}\in (0, \tau_{1})$
depending only on $(\underline{U},a_{\infty})$
such that, if $\tau\in(0,\tau_{3})$, $U^{(\tau)}_{b}\in \mathcal{O}_{\epsilon_{3}}(\underline{U})$ with \eqref{eq:2.18}, and $|\omega|+|\theta_{1}|<\epsilon_{3}$,
then problem \eqref{eq:2.17} admits a unique solution $U^{(\tau)}_{\Gamma_2}$ that
connects $U^{(\tau)}_{b}$ by a weak $1^{\rm st}$ wave $\alpha_1$
with strength $\sigma^{(\tau)}_{\alpha_1}$, {i.e.},
\begin{eqnarray}\label{eq:2.19}
U^{(\tau)}_{\Gamma_2}=\Phi^{(\tau)}_{1}(\sigma^{(\tau)}_{\alpha_1};\,U^{(\tau)}_{b},\tau^2),
\end{eqnarray}
such that
\begin{eqnarray}\label{eq:2.20}
\sigma^{(\tau)}_{\alpha}=K^{(\tau)}_{b}\omega,
\end{eqnarray}
where the bound of
function
$K^{(\tau)}_{b}>0$ depends only on $(\underline{U}, a_{\infty})$, but is particularly independent of $\tau$.
\end{lemma}

\begin{proof}
From the boundary condition on $\Gamma_{2}$ in \eqref{eq:2.17}, we define
\begin{align*}
\mathcal{L}_{b}(\sigma^{(\tau)}_{\alpha_1}, \omega, \theta_{1}, U^{(\tau)}_{b}, \tau^2)
&=\Big(1+\tau^{2}\Phi^{(\tau), (2)}_{1}(\sigma^{(\tau)}_{\alpha_1}; U^{(\tau)}_{b},\tau^2)\Big)\sin(\theta_{1}+\omega)\\[2pt]
&\ \ \ \ -\Phi^{(\tau), (3)}_{1}(\sigma^{(\tau)}_{\alpha_1}; U^{(\tau)}_{b},\tau^2)\cos(\theta_{1}+\omega),
\end{align*}
where $\Phi^{(\tau), (2)}_{1}$ and $\Phi^{(\tau), (3)}_{1}$ represent the second and third components of $\Phi^{(\tau)}_{1}$, respectively.

By \eqref{eq:1.16}, we see that $\mathcal{L}(0, 0, 0, \underline{U}, \tau^{2})= 0$ for any fixed number $\tau>0$.
Notice that
\begin{eqnarray*}
&&\frac{\partial\mathcal{L}_{b}(\sigma^{(\tau)}_{\alpha_1}, \omega ,\theta_{1}, U^{(\tau)}_{b}, \tau^{2}) }{\partial \sigma^{(\tau)}_{\alpha_1}}
\Big|_{\sigma^{(\tau)}_{\alpha_1}=\omega=\theta_{1}=0, U^{(\tau)}_{b}=\underline{U}}\\[4pt]
&&=\Big(1+\tau^{2}\frac{\partial\Phi^{(\tau), (2)}_{1}(\sigma^{(\tau)}_{\alpha_1}; U^{(\tau)}_{b},\tau^2)}{\partial \sigma^{(\tau)}_{\alpha_1}}\Big)\sin(\theta_{1}+\omega)\bigg|_{\sigma^{(\tau)}_{\alpha_1}=\omega_{k}=\theta_{1}=0, U^{(\tau)}_{b}=\underline{U}}\\[4pt]
&&\quad\, -\frac{\partial\Phi^{(\tau), (3)}_{1}(\sigma^{(\tau)}_{\alpha_1}; U^{(\tau)}_{b},\tau^2)}{\partial \sigma^{(\tau)}_{\alpha_1}}
  \cos(\theta_{1}+\omega)\bigg|_{\sigma^{(\tau)}_{\alpha_1}=\omega=\theta_{1}=0, U^{(\tau)}_{b}=\underline{U}}\\[4pt]
&&=-e^{(\tau)}_{1}(\underline{U}_{b},\tau^{2})<0.
\end{eqnarray*}
Thus, by the implicit function theorem, we can show that there exist small constants
$\epsilon_{3}\in(0, \epsilon_{1})$ and $\tau_{3}\in(0, \tau_{1})$ such that the Riemann problem \eqref{eq:2.17}
admits a unique solution $U^{(\tau)}_{\Gamma_2}$ that consists of the $1^{\rm st}$ wave $\sigma^{(\tau)}_{\alpha_1}$ only.
Furthermore, we have
\begin{eqnarray*}
&&\frac{\partial\mathcal{L}_b(\sigma^{(\tau)}_{\alpha_1}, \omega ,\theta_{1}, U^{(\tau)}_{b}, \tau^{2}) }{\partial \omega}
\Big|_{\sigma^{(\tau)}_{\alpha_1}=\omega=\theta_{1}=0, U^{(\tau)}_{b}=\underline{U}}\\[2pt]
&&=\Big(1+\tau^{2}\Phi^{(\tau), (2)}_{1}(\sigma^{(\tau)}_{\alpha_1}; U^{(\tau)}_{b},\tau^2)\Big)
 \cos(\theta_{1}+\omega)\Big|_{\sigma^{(\tau)}_{\alpha_1}=\omega=\theta_{1}=0, U^{(\tau)}_{b}=\underline{U}}\\[2pt]
&&\quad\,  +\Phi^{(\tau), (3)}_{1}(\sigma^{(\tau)}_{\alpha_1}; U^{(\tau)}_{b},\tau^2)\sin(\theta_{1}+\omega)\Big|_{\sigma^{(\tau)}_{\alpha_1}=\omega=\theta_{1}=0, U^{(\tau)}_{b}=\underline{U}}\\[5pt]
&&=1.
\end{eqnarray*}
Therefore, we obtain
\begin{align*}
K^{(\tau)}_{b}\Big|_{\sigma^{(\tau)}_{\alpha_1}=\omega=\theta_{1}=0, U^{(\tau)}_{b}=\underline{U}}
=-\frac{\frac{\partial\mathcal{L}_b(\sigma^{(\tau)}_{\alpha_1}, \omega,\theta_{1}, U^{(\tau)}_{b}, \tau^{2}) }{\partial \omega}
\Big|_{\sigma_{\alpha_1}=\omega=\theta_{1}=0, U^{\tau}_{b}=\underline{U}}}{\frac{\partial\mathcal{L}_b(\sigma_{\alpha_1}, \omega,\theta_{1}, U^{(\tau)}_{b}, \tau^{2}) }{\partial \sigma^{(\tau)}_{\alpha_1}}
\Big|_{\sigma^{(\tau)}_{\alpha_1}=\omega=\theta_{1}=0, U^{(\tau)}_{b}=\underline{U}}}
=\frac{1}{e^{(\tau)}_{1}(\underline{U},\tau^{2})}.
\end{align*}
This indicates that, for any $\tau\in(0,\tau_3)$, the bound of function $K^{(\tau)}_{b}>0$ depends only on $(\underline{U},a_{\infty})$.
This completes the proof.
\end{proof}

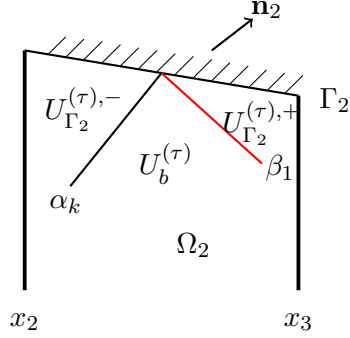
\begin{figure}[ht]
\begin{center}
\begin{tikzpicture}[scale=0.6]
\draw [line width=0.05cm] (-3.5,-3.8) --(-3.5,1.5);
\draw [line width=0.05cm] (2.5,-3.8) --(2.5,0.5);
\draw [thick] (-3.5,1.5)--(2.5,0.5);

\draw [thin] (-3, 1.4) --(-2.6, 1.8);
\draw [thin] (-2.6, 1.35) --(-2.2, 1.75);
\draw [thin] (-2.2, 1.30) --(-1.8, 1.70);
\draw [thin] (-1.8, 1.23) --(-1.4, 1.63);
\draw [thin] (-1.4, 1.16) --(-1.0, 1.56);
\draw [thin] (-1.0, 1.10) --(-0.6, 1.50);
\draw [thin] (-0.6, 1.03) --(-0.2, 1.43);
\draw [thin] (-0.2, 0.97) --(0.2, 1.37);
\draw [thin] (0.2, 0.9) --(0.6, 1.30);
\draw [thin] (0.6, 0.83) --(1, 1.23);
\draw [thin] (1, 0.76) --(1.4, 1.16);
\draw [thin] (1.4, 0.67) --(1.8, 1.07);
\draw [thin] (1.8, 0.60) --(2.2, 1.0);
\draw [thin] (2.2, 0.55) --(2.6, 0.95);

\draw [thick](-2.5,-1.5)--(-0.5,1);
\draw [thick][red](-0.5,1)--(1.7,-1.0);

\draw [thick][<-] (1.5,2.2)--(0.6,1.5);

\node at (1.8, 2.4){$\mathbf{n}_{2}$};

\node at (0.2, -2.8){$\Omega_{2}$};
\node at (3.3, 0.4){$\Gamma_{2}$};

\node at (-2.6, -1.9){$\alpha_{k}$};
\node at (2.1, -1.1){$\beta_{1}$};

\node at (-0.4, -1.0){$U^{(\tau)}_{b}$};
\node at (-2.2, 0.2){$U^{(\tau),-}_{\Gamma_2}$};
\node at (1.7, -0.1){$U^{(\tau),+}_{\Gamma_2}$};

\node at (-3.5, -4.5){$x_{2}$};
\node at (2.5, -4.5){$x_{3}$};
\end{tikzpicture}
\end{center}
\caption{ Weak waves hit on the boundary and reflect}\label{fig2.5}
\end{figure}

Next, we consider the reflection of the weak waves on the approximate boundary.
\begin{lemma}\label{lem:2.6}
For a given hypersonic similarity parameter $a_{\infty}$ with $\tau\in(0,\tau_{3})$,
assume that $U^{(\tau),-}_{\Gamma_2}$ and $U^{(\tau)}_b$ are two constant states
in $\mathcal{O}_{\epsilon_{3}}(\underline{U})$, as shown in {\rm Fig. \ref{fig2.5}}.
Assume that there is $k\in \{2,3,4\}$ such that
\begin{equation}\label{eq:2.21}
U^{(\tau),-}_{\Gamma_2}=\Phi^{(\tau)}_{k}(\sigma^{(\tau)}_{\alpha_k}; U^{\tau}_b,\tau^2),
\quad\,
((1+\tau^{2}u^{(\tau),-}_{\Gamma_2}), v^{(\tau),-}_{\Gamma_2})\cdot\mathbf{n}_2=0.
\end{equation}
Then there exists a constant state $U^{(\tau),+}_{\Gamma_2}\in \mathcal{O}_{\epsilon_{3}}(\underline{U})$ satisfying
\begin{equation}\label{eq:2.22}
U^{(\tau),+}_{\Gamma_2}=\Phi^{(\tau)}_{1}(\sigma^{(\tau)}_{\beta_1};U^{(\tau)}_b,\tau^2), \quad\,
((1+\tau^{2}u^{(\tau),+}_{\Gamma_2}), v^{(\tau),+}_{\Gamma_2})\cdot\mathbf{n}_2=0
\end{equation}
with
\begin{equation}\label{eq:2.23}
\sigma^{(\tau)}_{\beta_1}=K^{(\tau)}_{r,k} \sigma^{(\tau)}_{\alpha_k},
\end{equation}
where $K^{(\tau)}_{r,k}$ is a $C^{2}$-function of $(\sigma^{(\tau)}_{\alpha_k}, \tau^{2}, U^{(\tau)}_b)$ and satisfies
\begin{equation}\label{eq:2.24}
K^{(\tau)}_{r,4}|_{\sigma^{(\tau)}_{\alpha_4}=0,\ U^{(\tau)}_{b}=\underline{U}}=1\quad\,\,\mbox{or}\quad\,\,
K^{(\tau)}_{r,k}|_{\sigma^{(\tau)}_{\alpha_k}=0,\  U^{(\tau)}_{b}=\underline{U}}=0 \,\,\,\,\mbox{if $k=2,3$}.
\end{equation}
\end{lemma}

\begin{proof}
It follows from \eqref{eq:2.21}--\eqref{eq:2.22} that, for a fixed integer $k\in\{2,3,4\}$,
\begin{align}\label{eq:2.28x}
&\big(1+\tau^{2}\Phi^{(\tau),(2)}_{1}(\sigma^{(\tau)}_{\beta_1}; U^{(\tau)}_{b},\tau^2)\big)\Phi^{(\tau),(3)}_{k}(\sigma^{(\tau)}_{\alpha_k}; U^{(\tau)}_{b},\tau^2)\nonumber\\[2pt]
&\, -\big(1+\tau^{2}\Phi^{(\tau),(2)}_{k}(\sigma^{(\tau)}_{\alpha_k}; U^{(\tau)}_{b},\tau^2)\big)\Phi^{(\tau),(3)}_{1}(\sigma^{(\tau)}_{\beta_1}; U^{(\tau)}_{b},\tau^2)=0.
\end{align}
Then the existence proof of
$U^{(\tau),+}_{\Gamma_2}$ is similar to the one in the proof of Lemma \ref{lem:2.5}.

Then, to establish the estimate for $\sigma^{(\tau)}_{\beta_1}$,
we take the derivative with respect to $\sigma^{(\tau)}_{\alpha_k}$ on both sides of \eqref{eq:2.28x} and use Remark \ref{rem:2.1} to obtain that, for and $k=2,3,4$,
\begin{align*}
&K^{(\tau)}_{r, k}\Big|_{\sigma^{(\tau)}_{\beta_1}=\sigma^{(\tau)}_{\alpha_k}=0,\ U^{(\tau)}_{b}=\underline{U}}\\[2pt]
&=-\frac{\partial_{\sigma^{(\tau)}_{\alpha_k}}\mathcal{L}_{r, k}(\sigma^{(\tau)}_{\beta_1}, \sigma^{(\tau)}_{\alpha_k},
U^{(\tau)}_{b},\tau^2)}{\partial_{\sigma^{(\tau)}_{\beta_1}}\mathcal{L}_{r,k}(\sigma^{(\tau)}_{\beta_1}, \sigma^{(\tau)}_{\alpha_k}, U^{\tau}_{b},\tau^2)}\Bigg|_{\sigma^{(\tau)}_{\beta_1}=\sigma^{(\tau)}_{\alpha_k}=0,\ U^{(\tau)}_{b}=\underline{U}}\\[2pt]
&=\frac{(1+\tau^{2}u^{(\tau)}_{b})\mathbf{r}^{(\tau),(2)}_{k}(U^{(\tau)}_{b},\tau^2)
-\tau^{2}v^{(\tau)}_{b}\mathbf{r}^{(\tau),(1)}_{k}(U^{(\tau)}_{b},\tau^2)}
{(1+\tau^{2}u^{(\tau)}_{b})\mathbf{r}^{(\tau),(2)}_{1}(U^{(\tau)}_{b},\tau^2)
-\tau^{2}v^{(\tau)}_{b}\mathbf{r}^{(\tau),(1)}_{1}(U^{(\tau)}_{b},\tau^{2})}\Bigg|_{\sigma^{(\tau)}_{\beta_1}=\sigma^{(\tau)}_{\alpha_k}=0,\ U^{(\tau)}_{b}=\underline{U}}\\[4pt]
&=\frac{\mathbf{r}^{(\tau),(2)}_{k}(\underline{U},\tau^{2})}{\mathbf{r}^{(\tau),(2)}_{1}(\underline{U},\tau^{2})}\\[4pt]
&=\delta_{k 4},
\end{align*}
where
\begin{align*}
\mathcal{L}_{r,k}(\sigma^{(\tau)}_{\beta_1}, \sigma^{(\tau)}_{\alpha_{k}}, U^{(\tau)}_{b},\tau^{2})
=&\,\big(1+\tau^{2}\Phi^{(\tau),(2)}_{1}(\sigma^{(\tau)}_{\beta_1}; U^{(\tau)}_{b},\tau^2)\big)\Phi^{(\tau),(3)}_{k}(\sigma^{(\tau)}_{\alpha_k}; U^{(\tau)}_{b},\tau^2)\\[2pt]
&\,-\big(1+\tau^{2}\Phi^{(\tau),(2)}_{k}(\sigma^{(\tau)}_{\alpha_k}; U^{(\tau)}_{b},\tau^2)\big)\Phi^{(\tau),(3)}_{1}(\sigma^{(\tau)}_{\beta_1}; U^{(\tau)}_{b},\tau^2),
\end{align*}
and $\delta_{jk}$ is the Kronecker symbol. This completes the proof.
\end{proof}

\subsection{Wave-front tracking algorithm for {Problem I}}\label{sec:3.1}
Based on the results obtained in \S 2.2, we now construct the $(\nu, h)$--approximate
solutions $U^{(\tau)}_{\nu,h}$ of Problem I via the wave-front tracking scheme
in the following three steps  (see also Fig. \ref{fig2.6}):

\medskip
{\bf 1.} Define mesh-length $h=\Delta x>0$, and denote $\{\textsc{C}_{k}\}^{\infty}_{k=0}$ for $\textsc{C}_{k}:=(x_{k}, g_k)=(kh,g(kh))$
for the points on boundary $\Gamma$ for integers $k\geq0$.
Denote
\begin{equation}\label{eq:2.25}
\theta_{k}=\arctan(\frac{g_{k+1}-g_{k}}{h}),\,\,\, \omega_{0}=\arctan(\frac{g_1-g_0}{h}),
\,\,\, \omega_{k}=\theta_{k+1}-\theta_k \quad\,\,\, \mbox{for $k\geq 1$}.
\end{equation}

By \eqref{eq:1.21}, there exists a constant $g'_{\infty}$ such that
$\lim_{x\rightarrow \infty}g'(x)=g'_{\infty}$.
Therefore, for any given $h>0$, we can choose an integer $k_{*}\in \mathbb{N}_{+}$ satisfying
that $\|g'(x)-g'_{\infty}\|_{L^{\infty}(\{x>k_{*}h\})}<h$.

Now, for sufficiently small $h>0$, we can construct a piecewise straight line $y=g_{h}(x)$ with finite corners
to the approximate boundary $y=g(x)$ as
\begin{equation}\label{eq:2.26}
g_{h}(x)=\begin{cases}
g_k+\tan(\theta_k)(x-x_k) \quad &\mbox{for $x\in[x_{k-1}, x_{k})$ and $k\in[1,k_{*})\cap\mathbb{N}_+$},\\[1pt]
g_{k_{*}}+g'_{\infty}(x-x_{k_{*}}) & \mbox{for $x\in[x_{k_{*}}, \infty)$},
\end{cases}
\end{equation}
so that
\begin{equation}\label{eq:2.27}
\begin{aligned}
&\|g'_{h}(\cdot)-g'(\cdot)\|_{L^{1}(\mathbb{R}_{+})}\leq h,\qquad \lim_{x\rightarrow \infty}(g'_{h}(x)-g'(x))=0,\\
&\|g'_{h}(\cdot)\|_{BV(\mathbb{R}_{+})}\leq \|g'(\cdot)\|_{BV(\mathbb{R}_{+})}.
\end{aligned}
\end{equation}
Then we define the approximate domain:
\begin{eqnarray*}
\Omega_{h}={\displaystyle\cup^{k_{*}}_{k=1}\Omega_{h,k}}\cup \Omega_{h,k_*},
\end{eqnarray*}
where
\begin{eqnarray*}
&&\Omega_{h,k}=\{(x,y)\,:\,y<g_{h}(x), x\in[x_{k-1}, x_{k})\} \qquad \mbox{for $k\in[1,k_{*})\cap\mathbb{N}_+$}, \\[2pt]
&&\Omega_{h,k_{*}}=\{(x,y)\,:\, y<g_{h}(x), x\in[x_{k_*},\infty)\},
\end{eqnarray*}
and the corresponding approximate boundary:
\begin{eqnarray}\label{eq:2.27b}
\Gamma_{h}=\cup^{k_{*}}_{k=1}\Gamma_{h,k}\cup\Gamma_{h,k_{*}},
\end{eqnarray}
where
\begin{eqnarray}\label{eq:2.27c}
\begin{split}
&\Gamma_{h,k}=\{(x,y)\,:\,y=g_{h}(x), x\in[x_{k-1}, x_{k})\}\qquad \mbox{for $k\in[0,k_{*})\cap\mathbb{N}_+$}, \\[4pt]
&\Gamma_{h,k_{*}}=\{(x,y)\,:\,y=g_{h}(x), x\in[x_{k_{*}}, \infty)\}.
\end{split}\quad
\end{eqnarray}

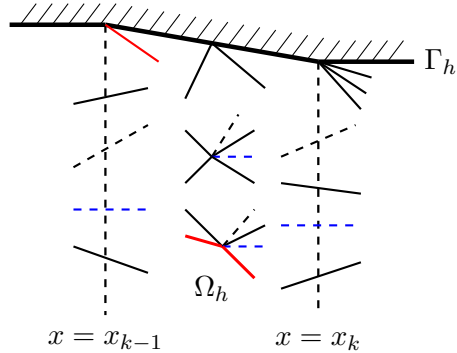
\begin{figure}[ht]
\begin{center}
\begin{tikzpicture}[scale=0.7]
\draw [line width=0.06cm](-5.8,2)--(-4,2);
\draw [line width=0.06cm](-4,2)--(0,1.3);
\draw [line width=0.06cm](0,1.3)--(1.8,1.3);
\draw [line width=0.03cm][dashed](-4,2)--(-4,-3.5);
\draw [line width=0.03cm][dashed](0,1.3)--(0,-3.5);

\draw [thin](-5.6,2)--(-5.3,2.4);
\draw [thin](-5.3,2)--(-5.0,2.4);
\draw [thin](-5,2)--(-4.7,2.4);
\draw [thin](-4.7,2)--(-4.4,2.4);
\draw [thin](-4.4,2)--(-4.1,2.4);
\draw [thin](-4.1,2)--(-3.8,2.4);
\draw [thin](-3.8,1.95)--(-3.5,2.35);
\draw [thin](-3.5, 1.9)--(-3.2,2.3);
\draw [thin](-3.2,1.85)--(-2.9,2.25);
\draw [thin](-2.9,1.80)--(-2.6,2.2);
\draw [thin](-2.6,1.75)--(-2.3,2.15);
\draw [thin](-2.3,1.70)--(-2.0,2.1);
\draw [thin](-2.0,1.65)--(-1.70,2.05);
\draw [thin](-1.7,1.60)--(-1.4,2.00);
\draw [thin](-1.4,1.55)--(-1.1,1.95);
\draw [thin](-1.1,1.50)--(-0.8,1.90);
\draw [thin](-0.8,1.45)--(-0.5,1.85);
\draw [thin](-0.5,1.40)--(-0.2,1.8);
\draw [thin](-0.2,1.35)--(0.1,1.75);
\draw [thin](0.1,1.30)--(0.4,1.70);
\draw [thin](0.4,1.30)--(0.7,1.70);
\draw [thin](0.7,1.30)--(1.0,1.70);
\draw [thin](1.0,1.30)--(1.3,1.70);
\draw [thin](1.3,1.30)--(1.6,1.70);

\draw [thick][red](-4,2)--(-3,1.3);

\draw [thick](-4.6,0.5)--(-3.2,0.8);
\draw [thick][dashed](-4.6,-0.7)--(-3.2,0.1);
\draw [thick][dashed][blue](-4.6,-1.5)--(-3.2,-1.5);
\draw [thick](-4.6,-2.2)--(-3.2,-2.7);

\draw [thick](-2.5,0.6)--(-2,1.65);
\draw [thick](-2,1.65)--(-1,0.8);

\draw [thick](-2.5,0)--(-2,-0.5);
\draw [thick](-2.5,-1)--(-2,-0.5);

\draw [thick][dashed][black](-2,-0.5)--(-1.5,0.3);
\draw [thick](-2,-0.5)--(-1.2,0);
\draw [thick][dashed][blue](-2,-0.5)--(-1.2,-0.5);
\draw [thick](-2,-0.5)--(-1.2,-1);

\draw [thick](-2.5,-1.5)--(-1.8,-2.2);
\draw [line width=0.04cm][red](-2.5,-2)--(-1.8,-2.2);

\draw [thick][dashed][black](-1.8,-2.2)--(-1.2,-1.5);
\draw [thick](-1.8,-2.2)--(-1.0,-1.8);
\draw [thick][dashed][blue](-1.8,-2.2)--(-1.0,-2.2);
\draw [line width=0.04cm][red](-1.8,-2.2)--(-1.2,-2.8);

\draw[thick](0,1.3)--(1,1);
\draw[thick](0,1.3)--(0.9,0.7);
\draw[thick](0,1.3)--(0.8,0.4);

\draw [thick][dashed](-0.7,-0.5)--(0.7,0.1);
\draw [thick](-0.7,-1)--(0.8,-1.2);
\draw [thick][dashed][blue](-0.7,-1.8)--(0.8,-1.8);
\draw [thick](-0.7,-3.0)--(0.8,-2.5);

\node at (-2.0, -3) {$\Omega_h$};
\node at (2.3, 1.3) {$\Gamma_h$};
\node at (2.6, 2) {$$};
\node at (-4, -4.0) {$x=x_{k-1}$};
\node at (0, -4.0) {$x=x_{k}$};
\end{tikzpicture}
\end{center}
\caption{Wave-front tracking scheme for {Problem I}}\label{fig2.6}
\end{figure}

Finally, for a sufficiently large parameter $\nu \in\mathbb{N}_{+}$, we can construct a piecewise constant function $U^{\nu}_{0}$ satisfying
\begin{eqnarray}\label{eq:2.28}
\|U^{\nu}_{0}-U_{0}\|_{L^{1}(\mathcal{I})}<2^{-\nu},\qquad  \|U^{\nu}_{0}-\underline{U}\|_{BV(\mathcal{I})}\leq
\|U_{0}-\underline{U}\|_{BV(\mathcal{I})}.
\end{eqnarray}

\smallskip
{\bf 2}.
By Lemmas \ref{lem:2.3}--\ref{lem:2.5}, at each discontinuity point of $U^{\nu}_{0}$ and the boundary corner $\textrm{C}_0$,
the corresponding Riemann problem can be solved, whose solution consists of shocks, vortex sheet/entropy wave,
or rarefaction waves. To control the number of the wave-fronts, two types of Riemann solvers are introduced as in \cite{bressan}.
One is called the accurate Riemann solver, denoted by $(ARS)$, and the other is called the simplified Riemann solver, denoted by $(SRS)$.
For $(ARS)$, we further partition the rarefaction waves into several small central rarefaction fans (piecewise constant solution)
with the speed of each fan closing to the characteristic speed and the strength of each fan less than $\nu^{-1}$ (see \cite[pp. 129--130]{bressan} for more details).
For $(SRS)$, all the new waves are lumped into a single non-physical wave with a fixed traveling speed that is larger than all the characteristics speed,
and the strength of the non-physical wave is defined by the Euclidean distance between the states on both sides (see \cite[pp. 131--132]{bressan} for more details).

Then the piecewise constant approximate solution $U^{(\tau)}_{h, \nu}$ exists along the $x$--direction up to line $x=x_*$,
on which there is an interaction point inside or on the boundary (\emph{i.e.}, $x_*=x_k$ for an integer $k\leq k_*$, or there is a reflection on the boundary).
If the interaction occurs on the boundary,
we adopt $(ARS)$.
If the interaction occurs inside, to decide which Riemann solver is used after the interaction,
we introduce a threshold parameter $\varrho>0$ to
be determined later.
If the two wave-fronts $\alpha$ and $\beta$ with strengths $\sigma^{(\tau)}_{\alpha}$ and $\sigma^{(\tau)}_{\beta}$ respectively
satisfy that $|\sigma^{(\tau)}_{\alpha}\sigma^{(\tau)}_{\beta}|>\varrho$, $(ARS)$ is used. Otherwise, $(SRS)$
is used.
Finally, in each construction, we may change the speed of a front slightly, so that no more than two wave-fronts interact,
no more than one wave hits the boundary at the same non-corner point, and no wave-front hits the boundary corner $\textrm{C}_k$.

\vspace{5pt}
\begin{figure}[ht]
\begin{center}
\begin{tikzpicture}[scale=1.3]
\draw [thick](-4.5,1.9)--(-3,0.5);
\draw [thick](-4.5,0.3)--(-3,0.5);

\draw [thick][dashed][red](-3,0.5)--(-1.4,1.7);
\draw [thick](-3,0.5)--(-1.3,1.2);
\draw [thick][dashed](-3,0.5)--(-1.3,0.3);
\draw [thick](-3,0.5)--(-1.5,-0.7);

\draw [thin][line width=0.04cm][dashed](-3,2)--(-3, -1);

\node at (-4.7, 1.9) {$\alpha_i$};
\node at (-4.7, 0.3) {$\beta_k$};

\node at (-1.1, 1.8) {$\gamma_{\mathcal{NP}}$};
\node at (-1.1, 1.2) {$\gamma_{4}$};
\node at (-0.9, 0.3) {$\gamma_{2,3}$};
\node at (-1.2, -0.8) {$\gamma_{1}$};

\node at (-3.7, -0.1) {$U^{(\tau)}_{b}$};
\node at (-3.8, 0.7) {$U^{(\tau)}_{m}$};
\node at (-3.7, 1.7) {$U^{(\tau)}_{a}$};

\node at (-1.9, -0.9) {$U^{(\tau)}_{b}$};
\node at (-1.9, 0) {$\hat{U}^{(\tau)}_{m_{1}}$};
\node at (-1.9, 0.7) {$\hat{U}^{(\tau)}_{m_{2(3)}}$};
\node at (-1.9, 1.8) {$U^{(\tau)}_{a}$};
\node at (-3, -1.5) {$x=\hat{x}$};
\end{tikzpicture}
\caption{Interactions between weak waves}\label{fig2.4}
\end{center}
\end{figure}
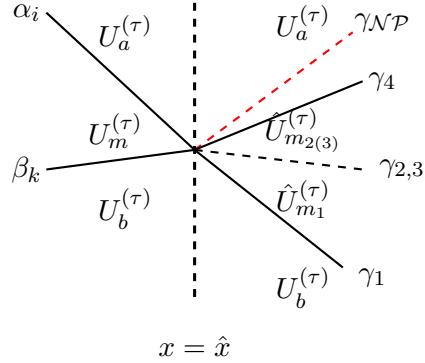

Based on Lemma \ref{lem:2.3}, following the standard procedure as done in \cite{bressan, smoller},
we have the following local interaction estimate between two weak waves.

\begin{lemma}[Local interaction estimate between two weak waves]\label{lem:2.4}
For a given hypersonic similarity parameter $a_{\infty}$, there exist constants $\epsilon_{2}>0$
and $\tau_{2}>0$ sufficiently small, depending only on $(\underline{U},a_{\infty})$,
such that, if three constant states $U^{(\tau)}_a, U^{(\tau)}_m, U^{(\tau)}_b\in \mathcal{O}_{\epsilon_{2}}(\underline{U})$
for $\tau\in(0, \tau_{2})$ satisfy
\begin{eqnarray*}
&&U^{(\tau)}_a=\Phi^{(\tau)}_{i}(\sigma^{(\tau)}_{\alpha_i}; U^{(\tau)}_m,\tau^2), \qquad
U^{(\tau)}_m=\Phi^{(\tau)}_{k}(\sigma^{(\tau)}_{\beta_k}; U^{(\tau)}_b,\tau^2),\\
&&U^{(\tau)}_a=\Phi^{(\tau)}(\boldsymbol{\sigma}^{(\tau)}_{\boldsymbol{\gamma}}; U^{(\tau)}_b,\tau^2)\qquad
\mbox{for $\, \boldsymbol{\sigma}^{(\tau)}_{\boldsymbol{\gamma}}=(\sigma^{(\tau)}_{\gamma_{1}},\sigma^{(\tau)}_{\gamma_{2}}.
\sigma^{(\tau)}_{\gamma_{3}},\sigma^{(\tau)}_{\gamma_{4}})$},
\end{eqnarray*}
then
\begin{eqnarray}\label{eq:2.14}
\sigma^{(\tau)}_{\gamma_{j}}=\delta_{ji}\sigma^{(\tau)}_{\alpha_i}+\delta_{jk}\sigma^{(\tau)}_{\beta_k}
+O(1)Q(\sigma^{(\tau)}_{\alpha_i},\sigma^{(\tau)}_{\beta_k}) \qquad \mbox{for $1\leq i,j,k\leq 4$}.
\end{eqnarray}
Moreover, if the non-physical wave $\gamma_{\mathcal{NP}}$ is constructed after the interaction by $(SRS)$, then
\begin{eqnarray}\label{eq:2.15}
\sigma^{(\tau)}_{\gamma_{\mathcal{NP}}}=O(1)Q(\sigma^{(\tau)}_{\alpha_i},\sigma^{(\tau)}_{\beta_k}),
\end{eqnarray}
where $Q(\sigma^{(\tau)}_{\alpha},\sigma^{(\tau)}_{\beta})$ is given by
\begin{equation}\label{eq:2.16}
Q(\sigma^{(\tau)}_{\alpha_i},\sigma^{(\tau)}_{\beta_k})=
\begin{cases}
 0 \ \ \  &\emph{if}\ i>k, \ \emph{or}\ i=k \ \emph{and}\ \min\{\sigma^{(\tau)}_{\alpha_i},\sigma^{(\tau)}_{\beta_k}\}>0, \\[5pt]
|\sigma^{(\tau)}_{\alpha_i}||\sigma^{(\tau)}_{\beta_k}|\ \ \ & \emph{otherwise},
\end{cases}
\end{equation}
and the bound of $O(1)$ depends only on $(\underline{U}, a_{\infty})$.
\end{lemma}

{\bf 3}. We denote all the fronts in $U^{(\tau)}_{h, \nu}$
by $\mathcal{J}^{(\tau)}(U^{(\tau)}_{h, \nu})=\mathcal{S}^{(\tau)}\cup\mathcal{C}^{(\tau)}\cup\mathcal{R}^{(\tau)}\cup\mathcal{NP}^{(\tau)}$,
where $\mathcal{S}^{(\tau)}$, $\mathcal{C}^{(\tau)}$, $\mathcal{R}^{(\tau)}$, and $\mathcal{NP}^{(\tau)}$ denote the shock wave-fronts,
vortex sheet/entropy wave-fronts, rarefaction fans, and nonphysical waves, respectively.
For each wave-front, a generation order is introduced to count how many interactions needed to produce such a front inductively as follows:

\smallskip
\rm(i) All wave-fronts generated by the corner points and the initial data are of order one.

\smallskip
\rm (ii) A wave-front $\alpha$ of order $k_{\alpha}$ hits the boundary at a non-corner point $(\hat{x}, g_{h}(\hat{x}))$.
Then the order of the new wave from $(\hat{x}, g_{h}(\hat{x}))$ is also set to be $k_{\alpha}$.

\smallskip
\rm(iii) The $i^{\rm th}$ wave-front $\alpha_i$ of order $k_{\alpha}$ and $j^{\rm th}$ wave-front $\beta_j$ of order $k_{\beta}$
interact at $(\hat{x}, \hat{y})$ with $\hat{y}<g_h(\hat{x})$.
Here we denote the non-physical wave-front as the $5^{\rm th}$ wave-front.
Assume that $\alpha_i$ lies below $\beta_j$.
Then the orders of the new wave-fronts are given as below:

\smallskip
$\quad$ $\rm (iii)_{1}\,$ If $i\neq j$, then the order of the outgoing $i^{\rm th}$ wave-front is $k_{\alpha}$ and the order of the outgoing $j^{\rm th}$ wave-front is $k_{\beta}$.
The order of the outgoing fronts of every other family is $\max\{k_{\alpha},k_{\beta}\}+1$.

\smallskip
$\quad$ $\rm (iii)_{2}\,$ If $i=j$, then the order of the outgoing $i^{\rm th}$ wave-front is $\min\{k_{\alpha},k_{\beta}\}$.
The order of the outgoing fronts of every other family is $\max\{k_{\alpha},k_{\beta}\}+1$.

\smallskip
With these, we can now construct the approximate solution $U^{(\tau)}_{h, \nu}$ globally, provided that the uniform bound of the total variation of the approximate solution
$U^{(\tau)}_{h, \nu}$ can be obtained and the total number of the wave-fronts is finite. These will be achieved
in \S 2.4 below.

\subsection{Uniform bounds and compactness of the approximate solutions}
We now introduce the weighted Glimm-type functional and apply the functional to obtain the uniform bound
of $U^{(\tau)}_{h, \nu}$ and to show that the total number of the fronts is finite for each fixed $(\tau, h, \nu)$.
First, let us introduce the notation of the approaching waves.

\begin{definition}[Approaching waves]\label{def:2.1}
The $i^{\rm th}$ wave $\alpha_i$ of strength $\sigma^{(\tau)}_{\alpha_i}$ and the $j^{\rm th}$ wave $\beta_j$ of strength $\sigma^{(\tau)}_{\beta_j}$,
which locate at $(x,y_{\alpha_i})$ and $(x,y_{\beta_j})$ respectively with $y_{\alpha_i}<y_{\beta_j}$, are called approaching if one of
the following cases occurs{\rm :}

\smallskip
\par \rm (i)\ $\alpha\in \mathcal{NP}^{(\tau)}$ and $\beta\in \mathcal{S}^{(\tau)}\cup\mathcal{C}^{(\tau)}\cup\mathcal{R}^{(\tau)}$;
\par \rm (ii)\ $i>j$;
\par \rm (iii)\  $i=j$ and $\min\{\sigma^{(\tau)}_{\alpha_i}, \sigma^{(\tau)}_{\beta_j}\}<0$.

\noindent
Denote by $\mathcal{A}^{(\tau)}(\hat{x})$ the set of all these approaching waves on $x=\hat{x}$.
\end{definition}

Assume that the approximate solution $U^{(\tau)}_{h,\nu}$, constructed in \S 2.3, satisfy the following assumptions up to $x=\hat{x}$:
\begin{itemize}
\item[$\mathbf{(P1)}$]\ $U^{(\tau)}_{h, \nu}$ has been defined for $x<\hat{x}$ and satisfies
$U^{(\tau)}_{h, \nu}(\hat{x}-,\cdot)\in \mathcal{O}_{\epsilon_*}(\underline{U})$,
where $\epsilon_{*}=\min\{\epsilon_{2},\epsilon_{3}\}$,  $\tau_{*}=\min\{\tau_{2},\tau_{3}\}$, and $\tau\in(0,\tau_*)$;

\item[$\mathbf{(P2)}$] \ For $x<\hat{x}$, the strength $\sigma^{(\tau)}_{\alpha}$ of each front $\alpha$ satisfies
$|\sigma^{(\tau)}_{\alpha}|\leq C\nu^{-1}$.
\end{itemize}
Then we define the modified Glimm-type functional for $x>0$ by
\begin{eqnarray}\label{2.30}
\mathcal{G}^{(\tau)}(x)=\mathcal{V}^{(\tau)}(x)+\mathcal{K}^{(\tau)}\mathcal{Q}^{(\tau)}(x),
\end{eqnarray}
where
\begin{eqnarray}
&&\mathcal{V}^{(\tau)}(x)=\mathcal{V}^{(\tau)}_{1}(x)+\sum^{4}_{ i=2}\mathcal{K}^{(\tau)}_{i}\mathcal{V}^{(\tau)}_{i}(x)
+\mathcal{V}^{(\tau)}_{\mathcal{NP}^{(\tau)}}(x)+\mathcal{K}^{(\tau)}_c\mathcal{V}^{(\tau)}_{c}(x),\qquad\quad \label{2.31}\\
&&\mathcal{Q}^{(\tau)}(x)=\sum_{(\alpha_i,\beta_j)\in \mathcal{A}^{(\tau)}(x)}|\sigma^{(\tau)}_{\alpha_i}||\sigma^{(\tau)}_{\beta_j}|,
\end{eqnarray}
with
\begin{eqnarray}
&&\mathcal{V}^{(\tau)}_{i}(x)=\sum_{\alpha_{i}\in \mathcal{J}^{(\tau)}}|\sigma^{(\tau)}_{\alpha_i}|
\qquad  \mbox{for $1\leq i\leq 4$},\label{eq:2.32}\\
&&\mathcal{V}^{(\tau)}_{\mathcal{NP}^{(\tau)}}(x)=\sum_{\alpha_{\mathcal{NP}}\in \mathcal{NP}^{(\tau)}}|\sigma^{(\tau)}_{\alpha_{\mathcal{NP}}}|, \qquad \mathcal{V}^{(\tau)}_{c}(x)=\sum_{k>[\frac{x}{h}]}|\omega_{k}|. \label{eq:2.33}
\end{eqnarray}
These constants $\mathcal{K}^{(\tau)}_{c}$, $\mathcal{K}^{(\tau)}$, and  $\mathcal{K}^{(\tau)}_{i}, i=2, 3, 4$,  are chosen such that
\begin{eqnarray}\label{eq:2.34}
\begin{split}
&\mathcal{K}^{(\tau)}_{c}>\max\big\{|K^{(\tau)}_{b}|+\frac{1}{2}, 1\big\},
\quad\, \mathcal{K}^{(\tau)}_{i}>\max\big\{|K^{(\tau)}_{r, i}|+\frac{1}{4}, 1\big\} \,\,\,\mbox{for $i=2,3,4$},\\[2pt]
&\mathcal{K}^{(\tau)}>4C_{2,1}\max\big\{\mathcal{K}^{(\tau)}_{2}, \ \mathcal{K}^{(\tau)}_{3},\ \mathcal{K}^{(\tau)}_{4}\big\}+1,
\end{split}\nonumber\\
\end{eqnarray}
where $C_{2,1}$ is the bound in \eqref{eq:2.13} in Lemma \ref{lem:2.3} depending only on $(\underline{U},a_{\infty})$.

We now show that functional $\mathcal{G}^{(\tau)}(x)$ is decreasing with respect to $x$ such that $U^{(\tau)}_{h, \nu}$ can be defined
and satisfies
assumptions ($\mathbf{P1}$) and ($\mathbf{P2}$) for all $x>0$.

\begin{lemma}\label{lem:2.7}
Let $\mathcal{K}^{(\tau)}_{c}$, $\mathcal{K}^{(\tau)}$, and {$\mathcal{K}^{(\tau)}_{\ell}, 2\leq \ell\leq 4$}, satisfy \eqref{eq:2.34}.
Under assumptions $\mathbf{(P1)}$ and $\mathbf{(P2)}$,
there exists a small parameter $\varepsilon_{1}>0$ such that, for $\varepsilon\in(0, \varepsilon_{1})$ and $\tau \in(0,\tau_{*})$,
if $\mathcal{G}^{(\tau)}(\hat{x}-)<\varepsilon$, then
\begin{equation}\label{eq:2.35}
\mathcal{G}^{(\tau)}(\hat{x}+)-\mathcal{G}^{(\tau)}(\hat{x}-)\leq -\frac{1}{8}\mathcal{E}^{(\tau)}(\hat{x}),\qquad
U^{(\tau)}_{h,\nu}(\hat{x}+,\cdot)\in \mathcal{O}_{\epsilon_{*}}(\underline{U}),
\end{equation}
where
\begin{eqnarray*}
\mathcal{E}^{(\tau)}(\hat{x})=\left\{
\begin{array}{llll}
|\sigma^{(\tau)}_{\alpha_i}||\sigma^{(\tau)}_{\beta_j}|
& \emph{if $(\alpha_i, \beta_j)\in \mathcal{A}^{(\tau)}(\hat{x})$ and $y_{\alpha_i}=y_{\beta_j}<g_h(\hat{x})$},\\[5pt]
|\sigma^{(\tau)}_{\alpha_i}|
& \emph{if the $i^{\rm th}$ wave hits the boundary $\Gamma_{h}$ at point $(\hat{x},g_{h}(\hat{x}))$}, \\[5pt]
|\omega_{k}|
 & \emph{if $\hat{x}=x_k$},
\end{array}
\right.
\end{eqnarray*}
for $i,j=2, 3$ or $4$, and $k=[\frac{\hat{x}}{h}]$.
\end{lemma}

\begin{proof}
We prove estimate \eqref{eq:2.35} case by case.

\smallskip
\par \emph{Case 1}: {\it $(\alpha_i, \beta_j)\in \mathcal{A}^{(\tau)}(\hat{x})$ and $y_{\alpha_i}=y_{\beta_j}<g_h(\hat{x})$}.
Let the strengths of wave-fronts $\alpha_i$ and $\beta_j$ are $\sigma^{(\tau)}_{\alpha_i}$ and
$\sigma^{(\tau)}_{\beta_j}$, respectively.
Let $\boldsymbol{\sigma}^{(\tau)}_{\boldsymbol{\gamma}}=(\sigma^{(\tau)}_{\gamma_{1}},\sigma^{(\tau)}_{\gamma_{2}},
\sigma^{(\tau)}_{\gamma_{3}},\sigma^{(\tau)}_{\gamma_{4}})$
(see Fig. \ref{fig2.4}). Then, by Lemma \ref{lem:2.4}, we have
\begin{eqnarray*}
\mathcal{V}^{(\tau)}(\hat{x}+)-\mathcal{V}^{(\tau)}(\hat{x}-)
\leq C_{2,0}\max\{\mathcal{K}^{(\tau)}_{2}, \mathcal{K}^{(\tau)}_{3},
\mathcal{K}^{(\tau)}_{4}\}|\sigma^{(\tau)}_{\alpha_i}||\sigma^{(\tau)}_{\beta_j}|.
\end{eqnarray*}
Meanwhile, for $\mathcal{Q}^{(\tau)}(x)$,
\begin{eqnarray*}
\mathcal{Q}^{(\tau)}(\hat{x}+)-\mathcal{Q}^{(\tau)}(\hat{x}-)
\leq \big(C_{2,1}\mathcal{V}^{(\tau)}(\hat{x}-)-1\big)|\sigma^{(\tau)}_{\alpha_i}||\sigma^{(\tau)}_{\beta_j}|,
\end{eqnarray*}
where constants {$C_{2,j}>0, j=0,1$}, depend only on $(\underline{U},a_{\infty})$.
If $\mathcal{V}^{(\tau)}(\hat{x}-)<\frac{3}{4C_{2,1}}$ is chosen, then
\begin{eqnarray*}
\mathcal{Q}^{(\tau)}(\hat{x}+)-\mathcal{Q}^{(\tau)}(\hat{x}-)\leq
-\frac{1}{4}|\sigma^{(\tau)}_{\alpha_i}||\sigma^{(\tau)}_{\beta_j}|.
\end{eqnarray*}
Thus, it follows from the estimates of $\mathcal{V}^{(\tau)}(\hat{x})$ and $\mathcal{Q}^{(\tau)}(\hat{x})$ and the choice of
$\mathcal{K}^{(\tau)}$ that
\begin{align*}
\mathcal{G}(\hat{x}+)-\mathcal{G}(\hat{x}-)
&\leq \mathcal{V}^{(\tau)}(\hat{x}+)-\mathcal{V}^{(\tau)}(\hat{x}-)
+\mathcal{K}^{(\tau)}\big(\mathcal{Q}^{(\tau)}(\hat{x}+)-\mathcal{Q}^{(\tau)}(\hat{x}-)\big)\\[2pt]
&\leq \Big(C_{2}\max\{\mathcal{K}^{(\tau)}_{2}, \mathcal{K}^{(\tau)}_{3}, \mathcal{K}^{(\tau)}_{4}\}-\frac{\mathcal{K}^{(\tau)}}{4}\Big)
 |\sigma^{(\tau)}_{\alpha_i}||\sigma^{(\tau)}_{\beta_j}|\\[2pt]
&\leq -\frac{1}{4}|\sigma^{(\tau)}_{\alpha_i}||\sigma^{(\tau)}_{\beta_j}|.
\end{align*}

\emph{Case 2}:  $\hat{x}=x_k$.
Then, by Lemma \ref{lem:2.5}, we have
\begin{eqnarray*}
&&\mathcal{V}^{(\tau)}_{1}(\hat{x}+)-\mathcal{V}^{(\tau)}_{1}(\hat{x}-)
=|K^{(\tau)}_{b}||\omega_{k}|,\quad\,\, \mathcal{V}^{(\tau)}_{c}(\hat{x}+)-\mathcal{V}^{(\tau)}_{c}(\hat{x}-)=-|\omega_{k}|,\\[2mm]
&&\mathcal{V}^{(\tau)}_{i}(\hat{x}+)-\mathcal{V}^{(\tau)}_{i}(\hat{x}-)
=0 \qquad \mbox{for $2\leq i\leq 4$}.
\end{eqnarray*}
Then it follows from the choice of $\mathcal{K}^{(\tau)}_{c}$ that
\begin{eqnarray*}
\mathcal{V}^{(\tau)}(\hat{x}+)-\mathcal{V}^{(\tau)}(\hat{x}-)
=-\big(\mathcal{K}^{(\tau)}_{c}-|K^{(\tau)}_{b}|\big)|\omega_{k}|
\leq -\frac{1}{2}|\omega_{k}|.
\end{eqnarray*}
Meanwhile, for $\mathcal{Q}^{(\tau)}(\hat{x})$, we have
\begin{eqnarray*}
\mathcal{Q}^{(\tau)}(\hat{x}+)-\mathcal{Q}^{(\tau)}(\hat{x}-)
\leq|K^{(\tau)}_{b}|\,\mathcal{V}^{(\tau)}(\hat{x}-)|\omega_{k}|.
\end{eqnarray*}

Thus, if $\mathcal{V}^{(\tau)}(\hat{x}-)\leq \frac{1}{4\mathcal{K}^{(\tau)}|K^{(\tau)}_{b}|}$ is chosen, then
\begin{eqnarray*}
\mathcal{G}^{(\tau)}(\hat{x}+)-\mathcal{G}^{(\tau)}(\hat{x}-)
\leq-\frac{1}{2}\big(1-2\mathcal{K}^{(\tau)}|K^{(\tau)}_{b}|\mathcal{V}^{(\tau)}(\hat{x}-)\big)|\omega_{k}|
\leq-\frac{1}{4}|\omega_{k}|.
\end{eqnarray*}

\emph{Case 3}: {\it The $i^{\rm th}$ wave hits boundary $\Gamma_{h}$ at point $(\hat{x},g_{h}(\hat{x}))$}.
Without loss of generality, we only consider the case: $i=4$,
since the other cases can be done analogously.
Let the strength of the $4^{\rm th}$ wave $\alpha_4$ is $\sigma^{(\tau)}_{\alpha_4}$.
From Lemma \ref{lem:2.6}, we obtain
\begin{align*}
&\mathcal{V}^{(\tau)}_{1}(\hat{x}+)-\mathcal{V}^{(\tau)}_{1}(\hat{x}-)
=|K^{(\tau)}_{r,4}||\sigma^{(\tau)}_{\alpha_4}|,\qquad
\mathcal{V}^{(\tau)}_{4}(\hat{x}+)-\mathcal{V}^{(\tau)}_{4}(\hat{x}-)
=-|\sigma^{(\tau)}_{\alpha_4}|,\\[1mm]
&\mathcal{V}^{(\tau)}_{c}(\hat{x}+)-\mathcal{V}^{(\tau)}_{c}(\hat{x}-)=0,\qquad
\mathcal{V}^{(\tau)}_{\ell}(\hat{x}+)-\mathcal{V}^{(\tau)}_{\ell}(\hat{x}-)
=0 \quad\mbox{for $\ell=2, 3$}.
\end{align*}
Then it follows from the choice of $\mathcal{K}^{(\tau)}_{4}$ that
\begin{eqnarray*}
\mathcal{V}^{(\tau)}(\hat{x}+)-\mathcal{V}^{(\tau)}(\hat{x}-)
=-\big(\mathcal{K}^{(\tau)}_{4}-|K^{(\tau)}_{r, 4}|\big)|\sigma^{(\tau)}_{\alpha_4}|
\leq -\frac{1}{4}|\sigma^{(\tau)}_{\alpha_4}|.
\end{eqnarray*}
Meanwhile,
\begin{eqnarray*}
\mathcal{Q}^{(\tau)}(\hat{x}+)-\mathcal{Q}^{(\tau)}(\hat{x}-)
\leq\big(|K^{(\tau)}_{r, 4}|-1\big)\mathcal{V}^{(\tau)}(\hat{x}-)|\sigma^{(\tau)}_{\alpha_4}|
\leq |K^{(\tau)}_{r, 4}|\,\mathcal{V}^{(\tau)}(\hat{x}-)|\sigma^{(\tau)}_{\alpha_4}|.
\end{eqnarray*}
Therefore, if $\mathcal{V}^{(\tau)}(\hat{x}-)\leq \frac{1}{2\mathcal{K}^{(\tau)}K^{(\tau)}_{r,4}}$ is chosen,
then
\begin{eqnarray*}
\mathcal{G}^{(\tau)}(\hat{x}+)-\mathcal{G}^{(\tau)}(\hat{x}-)
\leq-\frac{1}{4}\big(1-\mathcal{K}^{(\tau)}|K^{(\tau)}_{r, 4}|\mathcal{V}^{(\tau)}(\hat{x}-)\big)|\sigma^{(\tau)}_{\alpha_4}|
\leq -\frac{1}{8}|\sigma^{(\tau)}_{\alpha_4}|.
\end{eqnarray*}
Thus, for all the cases,
if
\begin{eqnarray*}
\mathcal{G}^{(\tau)}(\hat{x}-)<\min\bigg\{\frac{3}{4C_{2,1}}, \frac{1}{2\mathcal{K}^{(\tau)}|K^{(\tau)}_{r,2}|},
\frac{1}{2\mathcal{K}^{(\tau)}|K^{(\tau)}_{r,3}|}, \frac{1}{2\mathcal{K}^{(\tau)}|K^{(\tau)}_{r,4}|},\frac{1}{4\mathcal{K}^{(\tau)}|K^{(\tau)}_{b}|} \bigg\},
\end{eqnarray*}
then
$$
\mathcal{G}(\hat{x}+)<\mathcal{G}(\hat{x}-).
$$

Choose
\begin{eqnarray*}
\varepsilon_1=\min\Big\{\frac{3}{4C_{2,1}}, \frac{1}{2\mathcal{K}^{(\tau)}|K^{(\tau)}_{r,2}|},
\frac{1}{2\mathcal{K}^{(\tau)}|K^{(\tau)}_{r,3}|}, \frac{1}{2\mathcal{K}^{(\tau)}|K^{(\tau)}_{r,4}|},
\frac{1}{4\mathcal{K}^{(\tau)}|K^{(\tau)}_{b}|}, \frac{\epsilon_{*}}{C_{2,1}}\Big\}.
\end{eqnarray*}
If $\mathcal{G}(\hat{x}-)<\varepsilon$ for $\varepsilon\in(0,\varepsilon_1)$, then $\mathcal{G}(\hat{x}+)<\varepsilon$.
This implies
\begin{eqnarray*}
|U^{(\tau)}_{h,\nu}(\hat{x}+,\cdot)-\underline{U}|<C_{2,1}\mathcal{V}^{(\tau)}(\hat{x}+)\leq C_{2,1}\mathcal{G}^{(\tau)}(\hat{x}+)
\leq C_{2,1}\varepsilon.
\end{eqnarray*}
Then we obtain the estimates in \eqref{eq:2.35}.
\end{proof}

Using Lemma \ref{lem:2.7} and following the arguments in \cite{amadori,bressan,colombo-guerra},
we arrive at the following proposition. Because the argument is standard, we omit the proof.

\begin{proposition}\label{prop:2.1}
Under assumptions $(\mathbf{U_0})$ and $(\mathbf{g})$,  there exist constants
$\varepsilon_{2}=\varepsilon_{2}(\varepsilon_{1})>0$ and $\tau_{1}<\tau_{*}$ depending only on $(\underline{U},a_{\infty})$
such that, if the initial data $U_{0}(y)$ and the boundary function $g(x)$ satisfy
\begin{eqnarray}\label{eq:2.36}
\big\|U_{0}(\cdot)-\underline{U}\big\|_{BV(\mathcal{I})}
+ |g'(0)|+\big\|g'(\cdot)\|_{BV((0,\infty))}\leq \varepsilon
\end{eqnarray}
for some $\varepsilon \in (0,\varepsilon_{2})$ and $\tau\in(0,\tau_{1})$,
then there exists a positive threshold $\varrho=\varrho_{\nu}>0$ with $\varrho_{\nu}\rightarrow 0$ as $\nu\rightarrow \infty$
such that the wave-front tracking algorithm introduced in {\rm \S 2.3}
provides an approximate solution $U^{(\tau)}_{h,\nu}(x,y)\in \big(BV_{\rm loc}\cap L^{1}_{\rm loc}\big)(\Omega_{h})$ globally for all $x>0$.
Moreover, the following estimates hold{\rm :}

{\rm (i)} There exist positive constants $C_{2,2}>0$ and $C_{2,3}>0$ independent of $(h, \nu,\tau)$ such that
\begin{equation}\label{eq:2.37}
\sup_{x>0}\|U^{(\tau)}_{h,\nu}(x,\cdot)-\underline{U}\|_{L^{\infty}((-\infty, g_{h}(x)))}+\sup_{x>0}\big\|U^{(\tau)}_{h,\nu}(x,\cdot)-\underline{U}\big\|_{BV((-\infty, g_{h}(x)))}
\leq C_{2,2}\varepsilon,
\end{equation}
and, for any $x', x''>0$,
\begin{equation}\label{eq:2.38}
\big\|U^{(\tau)}_{h,\nu}(x',\cdot+g_{h}(x'))-U^{(\tau)}_{h,\nu}(x'',\cdot+g_{h}(x''))\big\|_{L^1((-\infty,0))}
\leq C_{2,3}|x'-x''|;\quad
\end{equation}

{\rm (ii)} The strength of each rarefaction-front is small, \emph{i.e.},
\begin{eqnarray}\label{eq:2.39}
|\sigma^{(\tau)}_{\alpha}|\leq C_{2,4}\nu^{-1}\qquad \mbox{for $\alpha\in\mathcal{R}^{(\tau)}$},
\end{eqnarray}
and the total strength of the nonphysical waves is small, \emph{i.e.},
\begin{eqnarray}\label{eq:2.40}
\sum_{\alpha\in\mathcal{NP}^{(\tau)}}|\sigma^{(\tau)}_{\alpha}|\leq C_{2,5}2^{-\nu},
\end{eqnarray}
where constants $C_{2,4}>0$ and $C_{2,5}>0$ depend only on $(\underline{U},a_{\infty})$.
\end{proposition}

{
Next, we further investigate properties of the approximate solutions $U^{(\tau)}_{h,\nu}$.
To this end, denote by $y=\mathcal{Y}(x)$ as a Lipschitz continuous curve
that is a linear combination of $g_{h}(x)$ and $\check{g}_{h}(x)$, and
let $\mathcal{Y}(\Sigma)=\{(x, \mathcal{Y}(x))\,:\, y=\mathcal{Y}(x), x\in\Sigma\}$ for an interval $\Sigma$.
Denote $\|f\|_{BV(\mathcal{Y}(\Sigma))}$ to be the total variation of function $f$ along $y=\mathcal{Y}(x)$ on interval $\Sigma$.
Then we have

\begin{lemma}\label{lem-BV-velocity}
Suppose that $U^{(\tau)}_{h,\nu}$ be the approximate solutions
of problem \eqref{eq:1.11}--\eqref{eq:1.12} constructed by the wave-front tracking scheme in {\rm \S 2.4}.
Then, when $\varepsilon>0$ and $\tau>0$ are sufficiently small,
\begin{equation}\label{eq-BV-velocity}
\Big\|(\frac{v^{(\tau)}_{h,\nu}}{1+\tau^2 u^{(\tau)}_{h,\nu}}, p^{(\tau)}_{h,\nu})\Big\|_{BV(\mathcal{Y}(\Sigma_x))}
\le C\Big(\big\|U_{0}\big\|_{BV(\mathcal{I})}+|g'(0)|+\big\|g'(\cdot)\|_{BV((0,\infty))}\Big)
\end{equation}
for $\Sigma_{x}=(0,x)$, where $C>0$ depends only on $(\underline{U}, a_{\infty})$.
\end{lemma}

\begin{proof}
Let $\omega^{(\tau)}_{h,\nu}\doteq \frac{v^{(\tau)}_{h,\nu}}{1+\tau^2 u^{(\tau)}_{h,\nu}}$, and
let $y_{\alpha}(x)$ be the location of each front $\alpha\in\mathcal{J}^{(\tau)}(U^{(\tau)}_{h,\nu})$.
In order to show \eqref{eq-BV-velocity}, we re-define
the modified Glimm-type functional for $x>0$ by
\begin{eqnarray*}
\tilde{\mathcal{G}}^{(\tau)}_{\mathcal{Y}}(x)
=\tilde{\mathcal{V}}^{(\tau)}(x)+\tilde{\mathcal{K}}^{(\tau)}\widetilde{\mathcal{Q}}^{(\tau)}(x),
\end{eqnarray*}
where $\tilde{\mathcal{Q}}^{(\tau)}(x)=\mathcal{Q}^{(\tau)}(x)$,
\begin{align*}
\tilde{\mathcal{V}}^{(\tau)}(x)
=&\sum_{j=a,b}\Big(\sum_{i=1,4}\tilde{\mathcal{K}}^{(\tau),j}_{i}\tilde{\mathcal{V}}^{(\tau),j}_{i}(x)
 +\tilde{\mathcal{K}}^{(\tau),j}_{\mathcal{NP}}\tilde{\mathcal{V}}^{(\tau),j}_{\mathcal{NP}}(x)\Big)\\
  &+\sum_{l=2,3}\tilde{\mathcal{K}}^{(\tau)}_{l}\tilde{\mathcal{V}}^{(\tau)}_{l}(x)
  +\tilde{\mathcal{K}}^{(\tau)}_c\tilde{\mathcal{V}}^{(\tau)}_{c}(x)
   +\tilde{\mathcal{V}}^{(\tau)}_{\mathcal{Y}}(x),
\end{align*}
with $\tilde{\mathcal{V}}^{(\tau)}_{i}(x)=\mathcal{V}^{(\tau)}_{i}(x)$ for $i=2,3$, $\tilde{\mathcal{V}}^{(\tau)}_{c}(x)=\mathcal{V}^{(\tau)}_{c}(x)$,
and
\begin{align*}
&\tilde{\mathcal{V}}^{(\tau),a}_{i}(x)=\sum_{\alpha_{i}\in \mathcal{J}^{(\tau)}(U^{(\tau)}_{h, \nu}),\atop y_{\alpha_i}(x)>\mathcal{Y}(x)}|\sigma^{(\tau)}_{\alpha_i}|,\quad \tilde{\mathcal{V}}^{(\tau),b}_{i}(x)=\sum_{\alpha_{i}\in \mathcal{J}^{(\tau)}(U^{(\tau)}_{h, \nu}), \atop y_{\alpha_i}(x)<\mathcal{Y}(x)}|\sigma^{(\tau)}_{\alpha_i}|
\qquad \mbox{for $i=1, 4$},\\
&\tilde{\mathcal{V}}^{(\tau), a}_{\mathcal{NP}^{(\tau)}}(x)=\sum_{\alpha_{\mathcal{NP}}\in \mathcal{NP}^{(\tau)}, \atop y_{\alpha_{\mathcal{NP}}}(x)>\mathcal{Y}(x)}|\sigma^{(\tau)}_{\alpha_{\mathcal{NP}}}|,\quad  \tilde{\mathcal{V}}^{(\tau), b}_{\mathcal{NP}^{(\tau)}}(x)=\sum_{\alpha_{\mathcal{NP}}\in \mathcal{NP}^{(\tau)},\atop y_{\alpha_{\mathcal{NP}}}(x)<\mathcal{Y}(x)}|\sigma^{(\tau)}_{\alpha_{\mathcal{NP}}}|,\\
&\tilde{\mathcal{V}}^{(\tau)}_{\mathcal{Y}}(x)=\big\|(\omega^{(\tau)}_{h,\nu}, p^{(\tau)}_{h,\nu})\big\|_{BV(\mathcal{Y}(\Sigma_x))},
\end{align*}
constants $\tilde{\mathcal{K}}^{(\tau),j}_{i}$, $\tilde{\mathcal{K}}^{(\tau),j}_{\mathcal{NP}}$, $\tilde{\mathcal{K}}^{(\tau)}_{l}$ for $i=1, 4,\ j=a,b, l=2,3$, and $\tilde{\mathcal{K}}^{(\tau)}_c$
are chosen such that
\begin{eqnarray}\label{eq-coefficient-K}
\begin{split}
&\tilde{\mathcal{K}}^{(\tau),a}_{1}>\tilde{\mathcal{K}}^{(\tau),b}_{1}+C^\natural,\ \ \tilde{\mathcal{K}}^{(\tau),b}_{4}>\tilde{\mathcal{K}}^{(\tau),a}_{4}+C^\natural, \ \
\tilde{\mathcal{K}}^{(\tau),b}_{\mathcal{NP}}>\tilde{\mathcal{K}}^{(\tau),a}_{\mathcal{NP}}+C^\natural,\\[2pt]
&\tilde{\mathcal{K}}^{(\tau),a}_{1}>K^{(\tau)}_{r,i}\tilde{\mathcal{K}}^{(\tau),a}_{i},\quad \tilde{\mathcal{K}}^{(\tau)}_c>\tilde{\mathcal{K}}^{(\tau),a}_{1}K^{(\tau)}_{b}\qquad \mbox{for $i=2,3,4$},
\end{split}
\end{eqnarray}
constant $\tilde{\mathcal{K}}^{(\tau)}>0$ will be determined later,
$C^\natural>0$ is a large constant depending only on $(\underline{U}, a_{\infty})$,
and coefficients $K^{(\tau)}_{b}$, $K^{(\tau)}_{r,i},\ i=2,3,4$, are given by Lemmas \ref{lem:2.5}--\ref{lem:2.6}.

\smallskip
Then we show that $\tilde{\mathcal{G}}^{(\tau)}(x)$ is decreasing by dividing four cases:

\smallskip
\emph{Case 1{\rm :} $\,$Interactions between two physical wave-fronts at $(x^\natural, \mathcal{Y}(x^\natural))$.}
Without loss of generality, consider the $1^{\rm th}$ wave-front $\alpha_1$ from the region above $y=\mathcal{Y}(x)$
interacts with the $4^{\rm th}$ wave-front $\alpha_4$ from the region below $y=\mathcal{Y}(x)$ at $(x^\natural, \mathcal{Y}(x^\natural))$.
If the resulting wave-fronts are all physical, which are denoted by $\beta_j$ for $j=1,2,3,4$,
then, by Lemma \ref{lem:2.6},
\begin{eqnarray*}
\sigma^{(\tau)}_{\beta_j}=\sigma^{(\tau)}_{\alpha_j}+O(1)|\sigma^{(\tau)}_{\alpha_1}\sigma^{(\tau)}_{\alpha_4}| \,\,\,\,\mbox{for $j=1,4$},\quad\,\,\,
\sigma^{(\tau)}_{\beta_i}=O(1)|\sigma^{(\tau)}_{\alpha_1}\sigma^{(\tau)}_{\alpha_4}| \,\,\,\, \mbox{for $i=2,3$}.
\end{eqnarray*}
By direct calculation, we have
\begin{align*}
&\sum_{i=1,4}\tilde{\mathcal{K}}^{(\tau),a}_{i}\tilde{\mathcal{V}}^{(\tau),a}_{i}(x^\natural+)
-\sum_{i=1,4}\tilde{\mathcal{K}}^{(\tau),a}_{i}\tilde{\mathcal{V}}^{(\tau),a}_{i}(x^\natural-)\\
&\quad \leq \tilde{\mathcal{K}}^{(\tau),a}_{4}|\sigma^{(\tau)}_{\alpha_4}|-\tilde{\mathcal{K}}^{(\tau),a}_{1}|\sigma^{(\tau)}_{\alpha_1}|+C|\sigma^{(\tau)}_{\alpha_1}\sigma^{(\tau)}_{\alpha_4}|,\\[2mm]
&\sum_{i=1,4}\tilde{\mathcal{K}}^{(\tau),b}_{i}\tilde{\mathcal{V}}^{(\tau),b}_{i}(x^\natural+)
-\sum_{i=1,4}\tilde{\mathcal{K}}^{(\tau),b}_{i}\tilde{\mathcal{V}}^{(\tau),b}_{i}(x^\natural-)\\
&\quad \leq \tilde{\mathcal{K}}^{(\tau),b}_{1}|\sigma^{(\tau)}_{\alpha_1}|-\tilde{\mathcal{K}}^{(\tau),b}_{4}|\sigma^{(\tau)}_{\alpha_4}|+C|\sigma^{(\tau)}_{\alpha_1}\sigma^{(\tau)}_{\alpha_4}|,\\[2mm]
&\tilde{\mathcal{V}}^{(\tau), a}_{\mathcal{NP}^{(\tau)}}(x^\natural+)-\tilde{\mathcal{V}}^{(\tau), a}_{\mathcal{NP}^{(\tau)}}(x^\natural-)
=\tilde{\mathcal{V}}^{(\tau), b}_{\mathcal{NP}^{(\tau)}}(x^\natural+)-\tilde{\mathcal{V}}^{(\tau), b}_{\mathcal{NP}^{(\tau)}}(x^\natural-)=0,\\[2mm]
&\sum_{i=2,3}\tilde{\mathcal{V}}^{(\tau)}_{i}(x^\natural+)-\sum_{i=2,3}\tilde{\mathcal{V}}^{(\tau)}_{i}(x^\natural-)\le C|\sigma^{(\tau)}_{\alpha_1}\sigma^{(\tau)}_{\alpha_4}|,\\[2mm]
&\tilde{\mathcal{V}}^{(\tau)}_{c}(x^\natural+)-\tilde{\mathcal{V}}^{(\tau)}_{c}(x^\natural-)=0,
\end{align*}
and
\begin{eqnarray*}
\begin{split}
\tilde{\mathcal{V}}^{(\tau)}_{\mathcal{Y}}(x^\natural+)-\tilde{\mathcal{V}}^{(\tau)}_{\mathcal{Y}}(x^\natural-)
&=|\omega^{(\tau)}_{h,\nu}(x^\natural+, \mathcal{Y}(x^\natural+))-\omega^{(\tau)}_{h,\nu}(x^\natural-, \mathcal{Y}(x^\natural-))|\\[2pt]
&\quad\ +|p^{(\tau)}_{h,\nu}(x^\natural+, \mathcal{Y}(x^\natural+))-p^{(\tau)}_{h,\nu}(x^\natural-, \mathcal{Y}(x^\natural-))|\\[2pt]
&\leq C\big(|\sigma^{(\tau)}_{\alpha_1}|+|\sigma^{(\tau)}_{\beta_4}|\big)\\[2pt]
&\leq C\big(|\sigma^{(\tau)}_{\alpha_1}|+|\sigma^{(\tau)}_{\alpha_4}|\big)+ C|\sigma^{(\tau)}_{\alpha_1}\sigma^{(\tau)}_{\alpha_4}|,\\[2pt]
\tilde{\mathcal{Q}}^{(\tau)}(x^\natural+)-\tilde{\mathcal{Q}}^{(\tau)}(x^\natural+)&\leq \big(-1+ C\,\tilde{\mathcal{V}}^{(\tau)}(x^\natural-)\big)|\sigma^{(\tau)}_{\alpha_1}\sigma^{(\tau)}_{\alpha_4}|\leq -\frac{1}{2}|\sigma^{(\tau)}_{\alpha_1}\sigma^{(\tau)}_{\alpha_4}|,
\end{split}
\end{eqnarray*}
when $\tilde{\mathcal{V}}^{(\tau)}(x^\natural-)$ is sufficiently small.

Therefore, choosing $\tilde{\mathcal{K}}^{(\tau)}>0$ sufficiently large and the coefficients as in \eqref{eq-coefficient-K},
we can obtain
\begin{eqnarray*}
\begin{split}
&\tilde{\mathcal{G}}^{(\tau)}(x^\natural+)-\tilde{\mathcal{G}}^{(\tau)}(x^\natural-)\\
&\leq \big(\tilde{\mathcal{K}}^{(\tau),b}_{1}-\tilde{\mathcal{K}}^{(\tau),a}_{1}+ C\big)|\sigma^{(\tau)}_{\alpha_1}|
+\big(\tilde{\mathcal{K}}^{(\tau),a}_{4}-\tilde{\mathcal{K}}^{(\tau),b}_{4}+C\big)|\sigma^{(\tau)}_{\alpha_4}|\\
&\quad +\big(-\frac{\tilde{\mathcal{K}}^{(\tau)}}{2}+C\big)|\sigma^{(\tau)}_{\alpha_1}\sigma^{(\tau)}_{\alpha_4}|<0,
\end{split}
\end{eqnarray*}
by choosing $\tilde{\mathcal{K}}^{(\tau)}>0$ sufficiently large.

If a nonphysical wave-front $\beta_{\mathcal{NP}}$ is generated, then, by Lemma \ref{lem:2.6}, we have
\begin{eqnarray*}
\sigma^{(\tau)}_{\beta_1}=\sigma^{(\tau)}_{\alpha_1},\quad\,\, \sigma^{(\tau)}_{\beta_4}=\sigma^{(\tau)}_{\alpha_4},
\quad\,\, \sigma^{(\tau)}_{\beta_{\mathcal{NP}}}=O(1)|\sigma^{(\tau)}_{\alpha_1}\sigma^{(\tau)}_{\alpha_4}|.
\end{eqnarray*}

Thus, for $\tilde{\mathcal{V}}^{(\tau)}(x^\natural-)$ sufficiently small, we can follow the above procedures to deduce
\begin{eqnarray*}
\begin{split}
\tilde{\mathcal{V}}^{(\tau)}(x^\natural+)-\tilde{\mathcal{V}}^{(\tau)}(x^\natural-)&\leq \big(\tilde{\mathcal{K}}^{(\tau),b}_{1}-\tilde{\mathcal{K}}^{(\tau),a}_{1}+C\,\tilde{\mathcal{K}}^{(\tau)}_{\mathcal{Y}}\big)|\sigma^{(\tau)}_{\alpha_1}|\\[2pt]
&\quad\, +\big(\tilde{\mathcal{K}}^{(\tau),a}_{4}-\tilde{\mathcal{K}}^{(\tau),b}_{4}+C\,\tilde{\mathcal{K}}^{(\tau)}_{\mathcal{Y}}\big)|\sigma^{(\tau)}_{\alpha_4}|
  +C|\sigma^{(\tau)}_{\alpha_1}\sigma^{(\tau)}_{\alpha_4}|,\\[2pt]
\tilde{\mathcal{Q}}^{(\tau)}(x^\natural+)-\tilde{\mathcal{Q}}^{(\tau)}(x^\natural+)&\leq \big(-1 + C\,\tilde{\mathcal{V}}^{(\tau)}(x^\natural-)\big)|\sigma^{(\tau)}_{\alpha_1}\sigma^{(\tau)}_{\alpha_4}|\leq -\frac{1}{2}|\sigma^{(\tau)}_{\alpha_1}\sigma^{(\tau)}_{\alpha_4}|,
\end{split}
\end{eqnarray*}
which implies
\begin{eqnarray*}
\tilde{\mathcal{G}}^{(\tau)}(x^\natural+)-\tilde{\mathcal{G}}^{(\tau)}(x^\natural-)<0,
\end{eqnarray*}
by letting $\tilde{\mathcal{K}}^{(\tau)}$ sufficiently large and the coefficients as in \eqref{eq-coefficient-K}.

\smallskip
\emph{Case 2{\rm :} $\,$ Interactions involving nonphysical wave-fronts at $(x^\natural, \mathcal{Y}(x^\natural))$.}
Consider the $1^{\rm st}$ wave-front $\alpha_1$ from the region above $y=\mathcal{Y}(x)$
interacts with the nonphysical wave-front $\alpha_{\mathcal{NP}}$ from the region below $y=\mathcal{Y}(x)$
at $(x^\natural, \mathcal{Y}(x^\natural))$.
Let $\beta_1$ and $\beta_{\mathcal{NP}}$ be the resulting wave-fronts.
Then, by Lemma \ref{lem:2.6}, we have
\begin{eqnarray*}
\sigma^{(\tau)}_{\beta_1}=\sigma^{(\tau)}_{\alpha_1},\quad
\sigma^{(\tau)}_{\beta_{\mathcal{NP}}}=\sigma^{(\tau)}_{\alpha_{\mathcal{NP}}}+O(1)|\sigma^{(\tau)}_{\alpha_1}\sigma^{(\tau)}_{\alpha_{\mathcal{NP}}}|.
\end{eqnarray*}
Therefore, when $\tilde{\mathcal{V}}^{(\tau)}(x^\natural-)$ is sufficiently small, a direct computation yields
\begin{align*}
&\sum_{i=1,4}\sum_{j=a,b}\big(\tilde{\mathcal{K}}^{(\tau),j}_{i}\tilde{\mathcal{V}}^{(\tau),j}_{i}(x^\natural+)
  -\tilde{\mathcal{K}}^{(\tau),j}_{i}\tilde{\mathcal{V}}^{(\tau),j}_{i}(x^\natural-)\big)
\leq \big(\tilde{\mathcal{K}}^{(\tau),b}_{1}-\tilde{\mathcal{K}}^{(\tau),a}_{1}\big)|\sigma^{(\tau)}_{\alpha_1}|,\\[2pt]
&\sum_{j=a,b}\big(\tilde{\mathcal{V}}^{(\tau), j}_{\mathcal{NP}^{(\tau)}}(x^\natural+)-\tilde{\mathcal{V}}^{(\tau), j}_{\mathcal{NP}^{(\tau)}}(x^\natural-)\big)
\leq \big(\tilde{\mathcal{K}}^{(\tau),a}_{\mathcal{NP}}-\tilde{\mathcal{K}}^{(\tau),b}_{\mathcal{NP}}\big)|\sigma^{(\tau)}_{\alpha_{\mathcal{NP}}}|
  +C|\sigma^{(\tau)}_{\alpha_1}\sigma^{(\tau)}_{\alpha_{\mathcal{NP}}}|,\\[2pt]
&\sum_{i=2,3}\tilde{\mathcal{V}}^{(\tau)}_{i}(x^\natural+)-\sum_{i=2,3}\tilde{\mathcal{V}}^{(\tau)}_{i}(x^\natural-)
 = \tilde{\mathcal{V}}^{(\tau)}_{c}(x^\natural+)-\tilde{\mathcal{V}}^{(\tau)}_{c}(x^\natural-)=0,\\[2pt]
&\tilde{\mathcal{V}}^{(\tau)}_{\mathcal{Y}}(x^\natural+)-\tilde{\mathcal{V}}^{(\tau)}_{\mathcal{Y}}(x^\natural-)
\leq C\big(|\sigma^{(\tau)}_{\alpha_1}|+|\sigma^{(\tau)}_{\alpha_{\mathcal{NP}}}|\big)+C|\sigma^{(\tau)}_{\alpha_1}\sigma^{(\tau)}_{\alpha_{\mathcal{NP}}}|,
\end{align*}
and
\begin{eqnarray*}
\begin{split}
\tilde{\mathcal{Q}}^{(\tau)}(x^\natural+)-\tilde{\mathcal{Q}}^{(\tau)}(x^\natural+)
&\leq \big(-1+ C\tilde{\mathcal{V}}^{(\tau)}(x^\natural-)\big)|\sigma^{(\tau)}_{\alpha_1}\sigma^{(\tau)}_{\alpha_{\mathcal{NP}}}|
\leq -\frac{1}{2}|\sigma^{(\tau)}_{\alpha_1}\sigma^{(\tau)}_{\alpha_{\mathcal{NP}}}|.
\end{split}
\end{eqnarray*}

Choosing $\tilde{\mathcal{K}}^{(\tau)}>0$ sufficiently large and the coefficients as in \eqref{eq-coefficient-K},
we then obtain
\begin{eqnarray*}
\begin{split}
&\tilde{\mathcal{G}}^{(\tau)}(x^\natural+)-\tilde{\mathcal{G}}^{(\tau)}(x^\natural-)\\[2pt]
&\leq \big(\tilde{\mathcal{K}}^{(\tau),b}_{1}-\tilde{\mathcal{K}}^{(\tau),a}_{1}+ C\tilde{\mathcal{K}}^{(\tau)}_{\mathcal{Y}}\big)|\sigma^{(\tau)}_{\alpha_1}|
+\big(\tilde{\mathcal{K}}^{(\tau),a}_{\mathcal{NP}}
  -\tilde{\mathcal{K}}^{(\tau),b}_{\mathcal{NP}}+ C\tilde{\mathcal{K}}^{(\tau)}_{\mathcal{Y}}\big)|\sigma^{(\tau)}_{\alpha_{\mathcal{NP}}}|\\
&\quad +\big(-\frac{\tilde{\mathcal{K}}^{(\tau)}}{2}+C\big)|\sigma^{(\tau)}_{\alpha_1}\sigma^{(\tau)}_{\alpha_{\mathcal{NP}}}|<0.
\end{split}
\end{eqnarray*}
For the other cases of interactions, we can done in the same way as above.

\smallskip
\emph{Case 3{\rm :} $\,$Weak wave-fronts hit on boundary $y=g_{h}(x)$}.
We only consider that a $4^{\rm th}$ wave-front $\alpha_4$ hits the approximate boundary $y=g_{h}(x)$ since the other cases can be done in the same way. Let $\beta_1$ be the reflected wave-front. Then, as shown in Lemma \ref{lem:2.5}, we have the relation:
$\sigma^{(\tau)}_{\beta_1}=K^{(\tau)}_{r,4}\sigma^{(\tau)}_{\alpha_4}$.
By direct computation, we can further obtain
\begin{eqnarray*}
&&\tilde{\mathcal{V}}^{(\tau)}(x^\natural+)-\tilde{\mathcal{V}}^{(\tau)}(x^\natural-)\leq \big(\tilde{\mathcal{K}}^{(\tau),a}_{4}K^{(\tau)}_{r,4}-\tilde{\mathcal{K}}^{(\tau),a}_{1}\big)|\sigma_{\alpha_4}|,\\[2pt]
&&\tilde{\mathcal{Q}}^{(\tau)}(x^\natural+)-\tilde{\mathcal{Q}}^{(\tau)}(x^\natural+)\leq C\,\tilde{\mathcal{V}}^{(\tau)}(x^\natural-)|\sigma^{(\tau)}_{\alpha_4}|.
\end{eqnarray*}
Therefore, when $\tilde{\mathcal{V}}^{(\tau)}(x^\natural-)$ is sufficiently small,
we choose $\tilde{\mathcal{K}}^{(\tau)}$ sufficiently large and the coefficients as in \eqref{eq-coefficient-K} to obtain
\begin{eqnarray*}
\begin{split}
\tilde{\mathcal{G}}^{(\tau)}(x^\natural+)-\tilde{\mathcal{G}}^{(\tau)}(x^\natural-)
\leq \big(\tilde{\mathcal{K}}^{(\tau),a}_{4}K^{(\tau)}_{r,4}-\tilde{\mathcal{K}}^{(\tau),a}_{1}
+ C\,\tilde{\mathcal{K}}^{(\tau)}\tilde{\mathcal{V}}^{(\tau)}(x^\natural-)\big)|\sigma^{(\tau)}_{\alpha_4}|<0.
\end{split}
\end{eqnarray*}

\smallskip
\emph{Case 4{\rm :} $\,$The $1^{\rm st}$ wave-front $\alpha_1$ issues from the corner point $(x_k, g_{h}(x_k))$ of
boundary $y=g_{h}(x)$}. As shown in Lemma \ref{lem:2.4}, we see that
$\sigma^{(\tau)}_{\alpha_1}=K^{(\tau)}_{b}\omega_k$.
Similar to the case above, for $\tilde{\mathcal{V}}^{(\tau)}(x^\natural-)$ sufficiently small,
we choose $\tilde{\mathcal{K}}^{(\tau)}$ sufficiently large and the coefficients as in \eqref{eq-coefficient-K} to obtain
\begin{eqnarray*}
\tilde{\mathcal{G}}^{(\tau)}(x^\natural+)-\tilde{\mathcal{G}}^{(\tau)}(x^\natural-)\leq \big(\tilde{\mathcal{K}}^{(\tau),a}_{1}K^{(\tau)}_{b}-\tilde{\mathcal{K}}^{(\tau)}_{c}
+C\,\tilde{\mathcal{K}}^{(\tau)}\tilde{\mathcal{V}}^{(\tau)}(x^\natural-)\big)|\omega_k|<0.
\end{eqnarray*}
Finally, we remark that the cases of interaction of wave-fronts away from $y=\mathcal{Y}(x)$ can be handled
in the same way as in Lemma \ref{lem:2.7}.

We conclude from the decreasing of $\tilde{\mathcal{G}}^{(\tau)}$ that, for all $x>0$,
$$
\tilde{\mathcal{G}}^{(\tau)}(x)<\tilde{\mathcal{G}}^{(\tau)}(0)\leq C\big(\|U_{0}(\cdot)\|_{BV(\mathcal{I})}
+|g'(0)|+\|g'(\cdot)\|_{BV((0,\infty))}\big),
$$
provided that $\varepsilon$ and $\tau$ are sufficiently small, where $C>0$ depends only on $(\underline{U}, a_{\infty})$.
This implies the uniform bound of $\big\|(\omega^{(\tau)}_{h,\nu}, p^{(\tau)}_{h,\nu})\big\|_{BV(\mathcal{Y}(\Sigma_x))}$
for all $x>0$.
\end{proof}

Based on Lemma \ref{lem-BV-velocity}, we have the following $L^1$-comparison estimate
on the flow slope function $\frac{v^{(\tau)}_{h,\nu}}{1+\tau^2 u^{(\tau)}_{h,\nu}}$.

\begin{proposition}\label{prop-flow-slop}
Suppose that $U^{(\tau)}_{h,\nu}$ are the approximate solutions
of problem \eqref{eq:1.11}--\eqref{eq:1.12} constructed by the wave-front tracking scheme in {\rm \S 2.4}.
Then, for $\varepsilon>0$ and $\tau>0$ sufficiently small,
\begin{align}\label{eq-flow-slop}
&\bigg\|\Big(\frac{v^{(\tau)}_{h,\nu}}{1+\tau^2 u^{(\tau)}_{h,\nu}}\Big)(s, \check{g}_{h}(s))
-\Big(\frac{v^{(\tau)}_{h,\nu}}{1+\tau^2 u^{(\tau)}_{h,\nu}}\Big)(s, g_{h}(s))\bigg\|_{L^1((x^\natural_0,x^\natural_1))}\nonumber\\[2mm]
&\le C_{2,6}\|\check{g}_{h}(s)-g_{h}(s)\|_{L^{\infty}((x^\natural_0,x^\natural_1))}
\end{align}
for any $x^\natural_1>x^\natural_0>0$, where $C_{2,6}>0$ depends only on $(\underline{U},a_{\infty})$.
\end{proposition}

\begin{proof}
Let $y=\mathcal{Y}(x)$ be a Lipschitz continuous curve that is a linear combination of $g_{h}(x)$ and $\check{g}_{h}(x)$
as described above.
Without loss of generality, we take $\mathcal{Y}(x)$ in the following form:
\begin{eqnarray*}
\mathcal{Y}_{\vartheta}(x)\doteq\vartheta g_{h}(x)+(1-\vartheta)\check{g}_{h}(x)
\qquad \mbox{for $x\in (0,\infty)$ and $\vartheta\in[0,1]$}.
\end{eqnarray*}

Choosing a positive $C^{0,1}$-function $\eta(x)$ with $\|\eta\|_{C^{0,1}(\mathbb{R}_+)}$ sufficiently small such that
there are only a finitely number of wave-fronts $y_{\alpha}(x)\in \mathcal{J}^{(\tau)}(U^{(\tau)}_{h, \nu})$
between $\mathcal{Y}_{\vartheta}(x)$ and $\mathcal{Y}_{\vartheta}(x)+\eta(x)$,
which divide region $(\mathcal{Y}_{\vartheta}(x), \mathcal{Y}_{\vartheta}(x)+\eta(x))$ into many
subregions
with constant states, and
all the fronts are not interacted with each other.

We first consider estimate \eqref{eq-flow-slop} for the special case $\check{g}_h=\mathcal{Y}_{\vartheta}(x)$
and $g_h=\mathcal{Y}_{\vartheta}(x)+\eta(x)$.
For notational convenience, we still denote $\omega^{(\tau)}_{h,\nu}=\frac{v^{(\tau)}_{h,\nu}}{1+\tau^2 u^{(\tau)}_{h,\nu}}$.
Since $\omega^{(\tau)}_{h,\nu}$ is unchanged when it crosses wave-fronts $\alpha\in \mathcal{C}^{(\tau)}_{2}\cup \mathcal{C}^{(\tau)}_{3}$,
it suffices
to consider the case:
$\alpha\in \mathcal{J}^{(\tau)}(U^{(\tau)}_{h, \nu})\setminus (\mathcal{C}^{(\tau)}_{2}\cup \mathcal{C}^{(\tau)}_{3})$.
To this end, let $y=y_{\alpha}(x)$ be a wave-front issue
from $(x^\natural_{0}, \mathcal{Y}_{\vartheta}(x^\natural_{0}))$
(or $(x^\natural_{0}, \mathcal{Y}_{\vartheta}(x^\natural_{0})+\eta(x^\natural_{0}))$) to
$(x^\natural_{1}, \mathcal{Y}_{\vartheta}(x^\natural_{1})+\eta(x^\natural_{1}))$
(or $(x^\natural_{1},\mathcal{Y}_{\vartheta}(x^\natural_{1}))$).
Moreover, for $\varepsilon>0$  and $\tau>0$ sufficiently small,
there exists a constant $d^\natural>0$ depending only on $(\underline{U}, a_{\infty})$
such that
\begin{eqnarray*}
|\mathcal{Y}'_{\vartheta}(x)-y'_{\alpha}(x)|>d^\natural, \qquad   |\mathcal{Y}'_{\vartheta}(x)+\eta'(x)-y'_{\alpha}(x)|>d^\natural.
\end{eqnarray*}
This implies that
\begin{eqnarray*}
\begin{split}
d^\natural\,(x^\natural_{1}-x^\natural_{0})
&<\int^{x^\natural_{1}}_{x^\natural_{0}}\big|\mathcal{Y}'_{\vartheta}(t)+\eta'(t)-y'_{\alpha}(t)\big|\, {\rm d}t \\[2pt]
&=\bigg|\int^{x^\natural_{1}}_{x^\natural_0}\big(\mathcal{Y}'_{\vartheta}(t)+\eta'(t)-y'_{\alpha}(t)\big)\, {\rm d}t\bigg|\\[2pt]
&=\Big|\mathcal{Y}_{\vartheta}(x^\natural_{1})+\eta(x^\natural_{1})-y_{\alpha}(x^\natural_{1})-\big(\mathcal{Y}_{\vartheta}(x^\natural_0)+\eta(x^\natural_0)-y_{\alpha}(x^\natural_0)\big)\Big|\\[2pt]
&\leq \|\eta\|_{L^{\infty}((x^\natural_{0}, x^\natural_{1}))}.
\end{split}
\end{eqnarray*}

Then it follows from the above estimate that
\begin{align}\label{eq-velocity-special}
&\int^{x^\natural_1}_{x^\natural_0}\big|\omega^{(\tau)}_{h,\nu}(t,\mathcal{Y}_{\vartheta}(t)+\eta(t))-\omega^{(\tau)}_{h,\nu}(t,\mathcal{Y}_{\vartheta}(t))\big|\, {\rm d}t\nonumber \\[2pt]
&\quad =\sum_{\alpha\in\mathcal{J}^{(\tau)}(U^{(\tau)}_{h, \nu})\setminus (\mathcal{C}^{(\tau)}_{2}\cup \mathcal{C}^{(\tau)}_{3})}\big|\omega^{(\tau)}_{h,\nu}(t, y_{\alpha}(t)-)-\omega^{(\tau)}_{h,\nu}(t, y_{\alpha}(t)+)\big|(x^\natural_{1}-x^\natural_{0})\nonumber\\[2pt]
&\quad \leq \frac{\|\eta\|_{L^{\infty}((x^\natural_0, x^\natural_1))}}{d^{\natural}}\sum_{\alpha\in\mathcal{J}^{(\tau)}(U^{(\tau)}_{h, \nu})\setminus (\mathcal{C}^{(\tau)}_{2}\cup \mathcal{C}^{(\tau)}_{3})}
\big|\omega^{(\tau)}_{h,\nu}(t, y_{\alpha}(t)-)-\omega^{(\tau)}_{h,\nu}(t, y_{\alpha}(t)+)\big|\nonumber \\[2pt]
&\quad \leq C\big\|(\omega^{(\tau)}_{h,\nu},p^{(\tau)}_{h,\nu})\big\|_{BV(\mathcal{Y}_{\vartheta}((x^\natural_0,x^\natural_1)))}\|\eta\|_{L^{\infty}((x^\natural_0, x^\natural_1))}.
\end{align}

Next, we turn to consider estimate \eqref{eq-flow-slop}.
Choose a sequence $\{\vartheta_j\}^{n}_{j=0}$ in $[0,1]$
with $0=\vartheta_0<\vartheta_1<\cdot\cdot\cdot <\vartheta_{n-1}<\vartheta_n=1$
satisfying that $\vartheta_{j+1}-\vartheta_{j}$, $j=0, \cdots, n-1$, are sufficiently small when $n$ is sufficiently large.
Define
\begin{eqnarray*}
\mathcal{Y}_{\vartheta_j}(x)\doteq\vartheta_j g_{h}(x)+(1-\vartheta_j)\check{g}_{h}(x) \quad\,\,\mbox{for $x\in (0,\infty)$, $\vartheta_j\in[0,1]$,\ $j=0,1,\cdots, n$}.
\end{eqnarray*}
By direct calculation, we see that, for $j=0,1,\cdots, n$,
\begin{eqnarray*}
\|\mathcal{Y}_{\vartheta_{j+1}}(x)-\mathcal{Y}_{\vartheta_j}(x)\|_{C^{0,1}}=(\vartheta_{j+1}-\vartheta_{j})\|g_{h}(x)-\check{g}_{h}(x)\|_{C^{0,1}}\ll 1.
\end{eqnarray*}

Then, applying estimate \eqref{eq-velocity-special}, we have
\begin{align}\label{eq-velocity-general}
&\int^{x^\natural_1}_{x^\natural_0}\big|\omega^{(\tau)}_{h,\nu}(t,\check{g}_h(t))-\omega^{(\tau)}_{h,\nu}(t,g_h(t))\big|\, {\rm d}t\nonumber \\[2pt]
&\quad \leq \sum^{n}_{j=0}\int^{x^\natural_1}_{x^\natural_0}\big|\omega^{(\tau)}_{h,\nu}(t,\mathcal{Y}_{\vartheta_{j+1}}(t))-\omega^{(\tau)}_{h,\nu}(t,\mathcal{Y}_{\vartheta_j}(t))\big|\, {\rm d}t\nonumber\\[2pt]
&\quad \leq C\sum^{n}_{j=0}(\vartheta_{j+1}-\vartheta_{j})\big\|(\omega^{(\tau)}_{h,\nu},p^{(\tau)}_{h,\nu})\big\|_{BV(\mathcal{Y}_{\vartheta_{j}}((x^\natural_0,x^\natural_1)))}
\|\check{g}_{h}(s)-g_{h}(s)\|_{L^{\infty}((x^\natural_0, x^\natural_1))},
\end{align}
where $C>0$ depends only on $(\underline{U}, a_{\infty})$.

Therefore, using estimate \eqref{eq-BV-velocity} in Lemma \ref{lem-BV-velocity} and estimate \eqref{eq-velocity-general},
we can finally derive inequality \eqref{eq-flow-slop} by choosing a constant $C_{2,6}>0$ that depends only on $(\underline{U},a_{\infty})$.
\end{proof}
}

\subsection{Proof of Theorem 1.1 for Problem I}
In this subsection, we first establish the $L^{1}$--stability estimate of the approximate solutions $U^{(\tau)}_{h, \nu}$
of problem \eqref{eq:1.11}--\eqref{eq:1.12} constructed by the wave-front tracking scheme in \S 2.4.
Then we further show the convergence of the approximate solutions $U^{(\tau)}_{h, \nu}$
in $L^1$ as $h\rightarrow 0$ and $\nu\rightarrow \infty$, whose limit is the entropy solution of Problem I.

To establish the $L^{1}$--stability estimate, we introduce a Lyapunov functional
$\mathscr{L}^{(\tau)}(U^{(\tau)}_{h, \nu}, V^{(\tau)}_{h', \nu})$, which is equivalent to the $L^{1}$--distance between
two approximate solutions:
$$
U^{(\tau)}_{h, \nu}=\big(\rho^{(\tau)}_{h, \nu}, u^{(\tau)}_{h, \nu}, v^{(\tau)}_{h, \nu}, p^{(\tau)}_{h, \nu}\big)^{\top},
\qquad
V^{(\tau)}_{h', \nu}=\big(\hat{\rho}^{(\tau)}_{h', \nu}, \hat{u}^{(\tau)}_{h', \nu}, \hat{v}^{(\tau)}_{h', \nu}, \hat{p}^{(\tau)}_{h', \nu}\big)^{\top},
$$
corresponding to the initial data and boundaries $(U^{\nu}_{0},g_{h})$ and $(V^{\nu}_{0},g_{h'})$, respectively.

First, we connect states $U^{(\tau)}_{h,\nu}$ and $V^{(\tau)}_{h',\nu}$ in the phase space along the Hugoniot
curves $\mathcal{S}^{(\tau)}_{1}$, $\mathcal{C}^{(\tau)}_{2}$, $\mathcal{C}^{(\tau)}_{3}$, and $\mathcal{S}^{(\tau)}_{4}$:
\begin{align}\label{eq:2.41}
V^{(\tau)}_{h',\nu}&=\mathcal{H}^{(\tau)}(\boldsymbol{q}^{(\tau)}; U^{(\tau)}_{h, \nu},\tau^2)\nonumber\\[4pt]
&\doteq\mathcal{H}^{(\tau)}_{4}(q^{(\tau)}_{4};\mathcal{H}^{(\tau)}_{3}(q^{(\tau)}_{3}; \mathcal{H}^{(\tau)}_{2}(q^{(\tau)}_{2}; \mathcal{H}^{(\tau)}_{1}(q^{(\tau)}_{1}; U^{(\tau)}_{h,\nu},\tau^2),\tau^2),\tau^2),\tau^2),
\end{align}
where $\boldsymbol{q}^{(\tau)}=(q^{(\tau)}_{1}, q^{(\tau)}_{2},q^{(\tau)}_{3}, q^{(\tau)}_{4})$ is the Hugoniot wave vector.
Then, due to the boundary effect, we modify $\boldsymbol{q}^{(\tau)}$ by imposing some weights:
\begin{eqnarray}\label{eq:2.42}
\hat{q}^{(\tau)}_{j}=
\begin{cases}
q^{(\tau)}_{j}
&\ \,\,\mbox{for $j=1$},\\[5pt]
\textsc{w}^{(\tau)}_{j}q^{(\tau)}_{j}
&\ \,\, \mbox{for $2\leq j\leq 4$},
\end{cases}
\end{eqnarray}
where $\textsc{w}^{(\tau)}_{j}>0, j=2,3,4$, will be determined later.
Set
\begin{eqnarray}\label{eq:2.43}
\hat{g}_{h}(x)=\max\big\{g_{h}(x), g_{h'}(x)\big\}, \qquad \check{g}_{h}(x)=\min\big\{g_{h}(x), g_{h'}(x)\big\}.
\end{eqnarray}
Then the Lyapunov functional is given as
\begin{align}\label{eq:2.44}
&\mathscr{L}^{(\tau)}\big(U^{(\tau)}_{h, \nu}, V^{(\tau)}_{h', \nu}\big)(x)\nonumber\\[2pt]
&=\sum^{4}_{j=1}\int^{\check{g}_{h}(x)}_{-\infty}|\hat{q}^{(\tau)}_{j}(y)|W^{(\tau)}_{j}(y)\, {\rm d}y,\nonumber\\[2pt]
&\ \ \  +\kappa^{(\tau)}_{g}\int^{\infty}_{x}\Big|\Big(\frac{v^{(\tau)}_{h,\nu}}{1+\tau^{2}u^{(\tau)}_{h,\nu}}\Big)(s,\check{g}_{h}(s))
-\Big(\frac{\hat{v}^{(\tau)}_{h,\nu}}{1+\tau^{2}\hat{u}^{(\tau)}_{h,\nu}}\Big)(s,\check{g}_{h}(s))\Big|\,{\rm d}s,
\end{align}
where
\begin{align}\label{eq:2.45}
W^{(\tau)}_{j}(y)&=1+\kappa_{1}A^{(\tau)}_{j}(y)+\kappa^{(\tau)}_{2}\big(\mathcal{Q}^{(\tau)}(U^{(\tau)}_{h, \nu})
+\mathcal{Q}^{(\tau)}(V^{(\tau)}_{h', \nu})\big)\nonumber\\[2pt]
&\ \ \ +\kappa^{(\tau)}_{3} \sum_{\alpha_{4}\in \mathcal{J}^{(\tau)}}|\sigma^{(\tau)}_{\alpha_4}|
+\kappa^{(\tau)}_{c}\sum_{k>[\frac{x}{h}]}\omega_{k}+\kappa^{(\tau)'}_{c}
\sum_{k'>[\frac{x}{h'}]}\omega'_{k'}
\end{align}
with constants  $\kappa^{(\tau)}_{g}$, $\kappa^{(\tau)}_{c}$, $\kappa^{(\tau)'}_{c}$, and $\kappa^{(\tau)}_{j}, j=1,2,3$, to be determined later.
The functional, $\mathcal{Q}^{(\tau)}$, represents the total interaction potential with respect to $U^{(\tau)}_{h, \nu}$ and $V^{(\tau)}_{h',\nu}$
as defined in \eqref{eq:2.33}, and $\omega_{k}$ and $\omega'_{k'}$ represent the turning angles at the corner points $(x_{k}, g_{h}(x_k))$ and $(x_{k'}, g_{h'}(x_{k'}))$ respectively.
Finally,
$A^{(\tau)}_{j}(y)$, which measures the total strength of the waves in $U^{(\tau)}_{h, \nu}$ and $V^{(\tau)}_{h,\nu}$ approaching to the $j^{\rm th}$ Hugoniot
wave $q^{(\tau)}_{j}$ at $y=y_{\alpha_i}$, is defined by
\begin{align}\label{eq:2.46}
A^{(\tau)}_j(y)&=\sum _{\alpha \in \mathcal{J}(U^{(\tau)}_{h,\nu})\cup\mathcal{J}(V^{(\tau)}_{h',\nu})
\atop y_{\alpha}<y, j<i_{\alpha}\leq 4}|\sigma^{(\tau)}_{\alpha}|+\sum _{\alpha \in \mathcal{J}(U^{(\tau)}_{h,\nu})\cup\mathcal{J}(V^{(\tau)}_{h',\nu})
\atop y_{\alpha}>y, 1\leq i_{\alpha}< j}|\sigma^{(\tau)}_{\alpha}|\nonumber\\[5pt]
&\ \ \  +
\begin{cases}
\bigg(\sum_{ \alpha\in \mathcal{J}(U^{(\tau)}_{h,\nu}) \atop y<y_{\alpha}, i_{\alpha}=j}
+ \sum_{ \alpha\in \mathcal{J}(V^{(\tau)}_{h',\nu}) \atop y>y_{\alpha}, i_{\alpha}=j} \bigg) |\sigma^{(\tau)}_{\alpha}| \qquad \text{if} \  q^{(\tau)}_{j}(y)<0,\\[15pt]
\bigg(\sum_{ \alpha\in \mathcal{J}(V^{(\tau)}_{h',\nu}) \atop y>y_{\alpha}, i_{\alpha}=j}
+ \sum_{ \alpha\in \mathcal{J}(U^{(\tau)}_{h,\nu}) \atop y<y_{\alpha}, i_{\alpha}=j} \bigg) |\sigma^{(\tau)}_{\alpha}|\qquad \text{if}  \ q^{(\tau)}_{j}(y)> 0,
\end{cases}
\end{align}
where $\sigma^{(\tau)}_{\alpha}$ is the strength of the weak wave $\alpha$ of the $i_{\alpha}^{\rm th}$ characteristic family.

By \eqref{eq:2.37} in Proposition \ref{prop:2.1}, we can take $\varepsilon$ sufficiently small such that
$1<W^{(\tau)}_{j}(y)<2$ for $1\leq j\leq 4$. Therefore, we can choose a constant $C_{2,7}>0$, independent of $(h, \nu, \tau)$, so that
\begin{align}\label{eq:2.47}
C_{2,7}^{-1}\big\|U^{(\tau)}_{h, \nu}(x,\cdot)-V^{(\tau)}_{h', \nu}(x,\cdot)\big\|_{L^{1}}
&\leq \sum^{4}_{j=1}\int^{\check{g}_{h}(x)}_{-\infty}|\hat{q}^{(\tau)}_{j}(y)|W^{(\tau)}_{j}(y)\,{\rm d}y\nonumber\\[2pt]
&\leq C_{2,7}\big\|U^{(\tau)}_{h, \nu}(x,\cdot)-V^{(\tau)}_{h', \nu}(x,\cdot)\big\|_{L^{1}}.
\end{align}

Next, we analyze the Lyapunov functional along the flow direction $x>0$.
To achieve this, we first give the following lemma.
\begin{lemma}\label{lem:2.8}
Let two constant states
$U^{(\tau)}(b)=\big(\rho^{(\tau)}, u^{(\tau)}, v^{(\tau)},p^{(\tau)}\big)^{\top}(b)$
and $V^{(\tau)}(b')=\big(\hat{\rho}^{(\tau)},\hat{u}^{(\tau)},\hat{v}^{(\tau)},\hat{p}^{(\tau)}\big)^{\top}(b')$
satisfy
\begin{eqnarray}\label{eq:2.48}
V^{(\tau)}(b')=\mathcal{H}^{(\tau)}(\boldsymbol{q}^{(\tau)}(b); U^{(\tau)}(b),\tau^2),
\end{eqnarray}
and
\begin{eqnarray}\label{eq:2.49}
(1+\tau^{2}u^{(\tau)}(b), v^{(\tau)}(b))\cdot \mathbf{n}_{k}=0,\ \ \
(1+\tau^{2}\hat{u}^{(\tau)}(b'), \hat{v}^{(\tau)}(b'))\cdot \mathbf{n}'_{k'}=0,
\end{eqnarray}
where
$\boldsymbol{q}^{(\tau)}(b)=(q^{(\tau)}_{1},q^{(\tau)}_{2}, q^{(\tau)}_{3}, q^{(\tau)}_{4})(b)$, $\mathbf{n}_{k}=(\sin \theta_{k}, -\cos \theta_{k})$,
and $\mathbf{n}'_{k'}=(\sin\theta'_{k'}, -\cos\theta'_{k'})$.
Then there exist constants $\varepsilon_{3}>0$ and $\tau_{2}>0$ depending only on $(\underline{U},a_{\infty})$ such that,
if $U^{(\tau)}_{b}, \hat{U}^{(\tau)}_{b}\in \mathcal{O}_{\varepsilon}(\underline{U})$,
and $|\theta_{k}|+|\theta'_{k'}|<\varepsilon$ for some $\varepsilon\in(0,\varepsilon_{3})$ and $\tau\in(0,\tau_{2})$, then
\begin{eqnarray}\label{eq:2.50}
q^{(\tau)}_{1}(b)=\mathbb{K}^{(\tau)}_{b}(q^{(\tau)}_{2}(b),q^{(\tau)}_{3}(b),q^{(\tau)}_{4}(b),\tau^{2})q^{(\tau)}_{4}(b)
+\mathbb{K}^{(\tau)}_{\theta}(\theta'_{k'}-\theta_{k}),
\end{eqnarray}
where the bounds of $\mathbb{K}^{(\tau)}_{b}(q^{(\tau)}_{2}(b),q^{(\tau)}_{3}(b),q^{(\tau)}_{4}(b),\tau^{2})$ and $\mathbb{K}^{(\tau)}_{\theta}$
depend only on $(\underline{U},  a_{\infty})$.
\end{lemma}

\begin{proof}
According to \eqref{eq:2.48}--\eqref{eq:2.49}, we denote
\begin{align*}
\mathbb{L}_{b}(\boldsymbol{q}^{(\tau)}(b), \theta'_{k'}, U^{(\tau)}(b),\tau^2)
&:=\big(1+\tau^{2}\mathcal{H}^{(\tau), (2)}(\boldsymbol{q}^{(\tau)}_{b};U^{(\tau)}(b),\tau^2)\big)\sin (\theta'_{k'})\\[2pt]
&\,\quad\ -\mathcal{H}^{(\tau), (3)}(\boldsymbol{q}^{(\tau)}(b);U^{(\tau)}(b),\tau^2)\cos(\theta'_{k'}).
\end{align*}
Notice that
$\mathbb{L}_{b}(0, 0, \underline{U},\tau^2)=0$ and
\begin{align*}
&\frac{\partial \mathbb{L}_{b}(\boldsymbol{q}^{(\tau)}(b), \theta'_{k'}, U^{(\tau)}(b),\tau^2)}{\partial q^{(\tau)}_{1}(b)}
\Bigg|_{\boldsymbol{q}^{(\tau)}(b)=\boldsymbol{0}, \theta'_{k'}=0, U^{(\tau)}(b)=\underline{U}}\\[2pt]
&=-\boldsymbol{r}^{(\tau),3}_{1}(\underline{U},\tau^{2})=-\frac{2(a^2_{\infty}-\tau^2)^2}{(\gamma+1)a^4_{\infty}}.
\end{align*}
Then we can choose $\tau'_{2}>0$ depending only on $a_{\infty}$ such that, for $\tau\in (0, \tau'_{2})$,
$$
\frac{\partial \mathbb{L}_{b}(\boldsymbol{q}^{(\tau)}(b), \theta'_{k'}, U^{(\tau)}(b),\tau^2)}{\partial q^{(\tau)}_{1}(b)}
\Bigg|_{\boldsymbol{q}^{(\tau)}(b)=\boldsymbol{0}, \theta'_{k'}=0, U^{(\tau)}(b)=\underline{U}}<0.
$$
Therefore, by the implicit function theorem, for $\varepsilon>0$ sufficiently small, the equation:
$$
\mathbb{L}_{b}(\boldsymbol{q}^{(\tau)}(b), \theta'_{k'}, U^{(\tau)}(b),\tau^2)=0
$$
admits a unique solution $q^{(\tau)}_{1}(b)=q^{(\tau)}_{1}(b)(q^{(\tau)}_{2}(b), q^{(\tau)}_{3}(b), q^{(\tau)}_{4}(b), \theta'_{k'},\tau^2 )$
that is a $C^2$--function of $(q^{(\tau)}_{2}(b), q^{(\tau)}_{3}(b), q^{(\tau)}_{4}(b), \theta'_{k'},\tau^2 )$.
Denote
$$
q^{(\tau)}_{1 *}(b)\doteq q^{(\tau)}_{1}(b)(q^{(\tau)}_{2}(b), q^{(\tau)}_{3}(b), q^{(\tau)}_{4}(b), \theta_{k},\tau^2 ).
$$
Then, by the Taylor formula, we have
\begin{eqnarray}\label{eq:2.51}
q^{(\tau)}_{1}(b)(q^{(\tau)}_{2}(b), q^{(\tau)}_{3}(b), q^{(\tau)}_{4}(b), \theta'_{k'},\tau^2 )
=q^{(\tau)}_{1 *}(b)+\mathbb{K}^{(\tau)}_{\theta}(\theta'_{k'}-\theta_{k}),
\end{eqnarray}
where coefficient $\mathbb{K}^{(\tau)}_{\theta}$ is a $C^2$--function of
$(q^{(\tau)}_{2}(b), q^{(\tau)}_{3}(b), q^{(\tau)}_{4}(b), \theta'_{k'},\tau^2)$.
Moreover, by direct computation, coefficient $\mathbb{K}^{(\tau)}_{\theta}$ satisfies
\begin{align*}
\mathbb{K}^{(\tau)}_{\theta}\Big|_{\boldsymbol{q}^{(\tau)}(b)=\boldsymbol{0}, \theta'_{k'}=0, U^{(\tau)}(b)=\underline{U}}
&=-\frac{\frac{\partial \mathbb{L}_{b}(\boldsymbol{q}^{(\tau)}(b), \theta'_{k'}, U^{(\tau)}(b),\tau^2)}{\partial \theta'_{k'}}\Big|_{\boldsymbol{q}^{(\tau)}(b)=\boldsymbol{0}, \theta'_{k'}=0, U^{(\tau)}(b)=\underline{U}}}{\frac{\partial \mathbb{L}_{b}(\boldsymbol{q}^{(\tau)}(b), \theta'_{k'}, U^{(\tau)}(b),\tau^2)}{\partial q^{(\tau)}_{1}(b)}\Big|_{\boldsymbol{q}^{(\tau)}(b)=\boldsymbol{0}, \theta'_{k'}=0, U^{(\tau)}(b)=\underline{U}}}\\[4pt]
&=\frac{(\gamma+1)a^4_{\infty}}{2(a^2_{\infty}-\tau^2)^2}.
\end{align*}
This implies that $\mathbb{K}^{(\tau)}_{\theta}$ is uniformly bounded if $\tau\in (0,\tau''_{2})$
with $\tau''_{2}>0$ small and depending only on $a_{\infty}$.
Thus, we choose $\varepsilon'_{3}>0$ such that, for $\varepsilon\in(0,\varepsilon'_{3})$,
if $U^{(\tau)}_{b}\in \mathcal{O}_{\varepsilon}(\underline{U})$ and $|\theta_{k}|+|\theta'_{k'}|<\varepsilon$,
then the bound of $\mathbb{K}^{(\tau)}_{\theta}$ is independent of
$(\underline{U},a_{\infty})$.

Next, we estimate $q^{(\tau)}_{1 *}(b)$. Let
$$
\boldsymbol{q}^{(\tau)}_{*}(b)=(q^{(\tau)}_{1 *}(b), q^{(\tau)}_{2}(b), q^{(\tau)}_{3 }(b),q^{(\tau)}_{4 }(b)),
\quad V^{(\tau)}(b)=\big(\hat{\rho}^{(\tau)},\hat{u}^{(\tau)},\hat{v}^{(\tau)},\hat{p}^{(\tau)}\big)^{\top}(b).
$$
Then
\begin{eqnarray}\label{eq:2.52}
V^{(\tau)}(b)=\mathcal{H}^{(\tau)}(\boldsymbol{q}^{(\tau)}_{*}(b); U^{(\tau)}(b),\tau^2),\quad
(1+\tau^{2}\hat{u}^{(\tau)}(b), \hat{v}^{(\tau)}(b))\cdot \mathbf{n}_{k}=0.\quad
\end{eqnarray}

Let $U^{(\tau)}_{i}(b)=(\rho^{(\tau)}_{i},u^{(\tau)}_{i},v^{(\tau)}_{i},p^{(\tau)}_{i})^{\top}(b)$, $i=1,2,3$,
be the middle states of the solution of problem \eqref{eq:2.52} from $U^{(\tau)}(b)$ to $V^{(\tau)}(b)$.
If $q^{(\tau)}_{4}(b)=0$, then $V^{(\tau)}(b)=U^{(\tau)}_{3}(b)$.
Thus, by the  boundary conditions \eqref{eq:2.49} and \eqref{eq:2.52},
we see that
\begin{eqnarray}\label{eq:2.53}
\frac{v^{(\tau)}_{3}(b)}{1+\tau^2 u^{(\tau)}_{3}(b)}=\frac{v^{(\tau)}(b)}{1+\tau^2 u^{(\tau)}(b)}.
\end{eqnarray}

On the other hand, along the $2^{\rm nd}$ wave curve $\mathcal{C}^{(\tau)}_{2}$, we have
\begin{eqnarray*}
u^{(\tau)}_{2}(b)=\big(1+\tau^{2}u^{(\tau)}_{1}(b)\big)e^{\tau^2q^{(\tau)}_{2}(b)}, \quad
v^{(\tau)}_{2}(b)=v^{(\tau)}_{1}(b)e^{\tau^2q^{(\tau)}_{2}(b)},
\end{eqnarray*}
and $\rho^{(\tau)}_{2}(b)=\rho^{(\tau)}_{1}(b)$ and $ p^{(\tau)}_{2}(b)=p^{(\tau)}_{1}(b)$.
Along the $3^{\rm rd}$ wave curve $\mathcal{C}^{(\tau)}_{3}$, we have
\begin{eqnarray*}
\rho^{(\tau)}_{3}(b)=\rho^{(\tau)}_{2}(b)+q^{(\tau)}_{3}(b),\,\, u^{(\tau)}_{3}(b)=u^{(\tau)}_{2}(b),
\,\, v^{(\tau)}_{3}(b)=v^{(\tau)}_{2}(b),
\,\, p^{(\tau)}_{3}(b)=p^{(\tau)}_{2}(b).
\end{eqnarray*}
This implies that
\begin{eqnarray}\label{eq:2.54}
\frac{v^{(\tau)}_{3}(b)}{1+\tau^2 u^{(\tau)}_{3}(b)}
=\frac{v^{(\tau)}_{1}(b)e^{\tau^2q^{(\tau)}_{2}(b)}}{\big(1+\tau^{2}u^{(\tau)}_{1}(b)\big)e^{\tau^2q^{(\tau)}_{2}(b)}}
=\frac{\mathcal{H}^{(\tau), 3}(q^{(\tau)}_{1 *}(b);U^{(\tau)}(b))}{1+\tau^2\mathcal{H}^{(\tau), 2}(q^{(\tau)}_{1 *}(b);U^{(\tau)}(b))}.\quad\,\,
\end{eqnarray}

Then it follows from \eqref{eq:2.53}--\eqref{eq:2.54} that
$$
q^{(\tau)}_{1 *}(b)\Big|_{q^{(\tau)}_{4}(b)=0}=0,
$$
so that applying the Taylor formula again yields
\begin{eqnarray*}
q^{(\tau)}_{1 *}(b)=\mathbb{K}^{(\tau)}_{b}(q^{(\tau)}_{2}(b),q^{(\tau)}_{3}(b),q^{(\tau)}_{4}(b),\theta_{k},\tau^{2})q^{(\tau)}_{4}(b),
\end{eqnarray*}
where
\begin{eqnarray*}
&&\mathbb{K}^{(\tau)}_{b}(q^{(\tau)}_{2}(b),q^{(\tau)}_{3}(b),q^{(\tau)}_{4}(b),\theta_{k}, \tau^{2})\\
&&=\int^{1}_{0}\partial_{\mu}q^{(\tau)}_{1}(b)(q^{(\tau)}_{2}(b), q^{(\tau)}_{3}(b), \mu q^{(\tau)}_{4}(b), \theta_{k},\tau^2)\,{\rm d}\mu,
\end{eqnarray*}
and
\begin{eqnarray*}
\mathbb{K}^{(\tau)}_{b}\Big|_{\boldsymbol{q}^{(\tau)}_{*}(b)=\boldsymbol{0}, \theta_{k}=0, U^{(\tau)}(b)=\underline{U}}
=-\frac{\frac{\partial \mathbb{L}_{b}(\boldsymbol{q}^{(\tau)}_{*}(b), \theta_{k}, U^{(\tau)}(b),\tau^2)}{\partial q^{(\tau)}_{4}(b)}\Big|_{\boldsymbol{q}^{(\tau)}_{*}(b)=\boldsymbol{0}, \theta_{k}=0, U^{(\tau)}(b)=\underline{U}}}{\frac{\partial \mathbb{L}_{b}(\boldsymbol{q}^{(\tau)}_{*}(b), \theta_{k}, U^{(\tau)}(b),\tau^2)}{\partial q^{(\tau)}_{1*}(b)}\Big|_{\boldsymbol{q}^{(\tau)}_{*}(b)=\boldsymbol{0}, \theta_{k}=0, U^{(\tau)}(b)=\underline{U}}}
=-1.
\end{eqnarray*}

Therefore, we can choose $\tau'''_{2}>0$ depending only on $a_{\infty}$ and $\varepsilon''_{3}>0$ such that,
for $\varepsilon\in(0,\varepsilon''_{3})$ and $\tau\in (0,\tau'''_{2})$,
if $U^{(\tau)}_{b}\in \mathcal{O}_{\varepsilon}(\underline{U})$ and $|\theta_{k}|+|\theta'_{k'}|<\varepsilon$,
then the bound of $\mathbb{K}_{b}$ depends only on $(\underline{U},a_{\infty})$.

Finally, let $\varepsilon_{3}=\min\{\varepsilon'_{3}, \varepsilon''_{3}\}$ and $\tilde{\tau}_{0}=\min\{\tau'_{2}, \tau''_{2},\tau'''_{2}\}$.
Then, for $\varepsilon\in(0,\varepsilon_{3})$ and $\tau\in (0,\tau_{2})$, we obtain the desired results.
\end{proof}

Now we can analyze the Lyapunov functional along the flow direction $x>0$.

\begin{lemma}\label{lem:2.9}
Suppose that $U^{(\tau)}_{h,\nu}$ and $V^{(\tau)}_{h',\nu}$ are the approximate solutions
corresponding to the initial data and boundaries $(U^{\nu}_{0},g_{h})$ and $(V^{\nu}_{0},g_{h'})$, respectively,
and satisfy the properties listed in {\rm Proposition \ref{prop:2.1}}.
Then there exist suitable positive constants  $\kappa^{(\tau)}_{c}$, $\kappa^{(\tau)'}_{c}$, $\kappa^{(\tau)}_{g}$,
$\kappa^{(\tau)}_{i}$ for $i=1,2,3$, and $\textsc{w}^{(\tau)}_{j}$ for $j=2,3,4$,
depending only on $\underline{U}$, such that functional $\mathscr{L}^{(\tau)}$ defined by \eqref{eq:2.41} satisfies
the estimate{\rm :}
\begin{eqnarray}\label{eq:2.55}
\mathscr{L}^{(\tau)}\big(U^{(\tau)}_{h, \nu}, V^{(\tau)}_{h', \nu}\big)(x'')
\leq \mathscr{L}^{(\tau)}\big(U^{(\tau)}_{h, \nu}, V^{(\tau)}_{h', \nu}\big)(x')+C\big(x''-x'\big)\nu^{-1}
\end{eqnarray}
for any $0<x'<x''$, where constant $C>0$ depends only on $(\underline{U},a_{\infty})$.
\end{lemma}

\begin{proof}
Denote by $U^{(\tau)}_0\doteq U^{(\tau)}_{h, \nu}(x,y), U^{(\tau)}_1,\ U^{(\tau)}_2,\ U^{(\tau)}_3$,
and $U^{(\tau)}_4\doteq V^{(\tau)}_{h',\nu}(x,y)$
the intermediate states of the Riemann solution of equation \eqref{eq:2.41}.
Let $\lambda^{(\tau)}_{i}\doteq\lambda_{i}(U^{(\tau)}_{i-1},U^{(\tau)}_{i-1} )$, $1\leq i\leq 4$,
be the corresponding speed of the $i^{\rm th}$ wave $q^{(\tau)}_{i}$ along the Hugoniot curve in the phase space.
Let $\check{g}^{-}_{h}(x)\doteq b>y_{1}>y_{2}>\cdot\cdot\cdot >y_{N}$ with $(x,y_l)$ as
the jump point of $U^{(\tau)}_{h,\nu}$ and $V^{(\tau)}_{h',\nu}$ for $1\leq l\leq N$,
and let $\mathcal{J}^{(\tau)}=\mathcal{J}^{(\tau)}(U^{(\tau)}_{h, \nu})\cup \mathcal{J}^{(\tau)}(V^{(\tau)}_{h', \nu})$
be the set of indices $\alpha\in \{1,2, \cdot\cdot\cdot, N \}$.
We define
\begin{eqnarray}\label{eq:2.56}
\hat{q}^{(\tau),\alpha\pm}_{i}=\hat{q}^{(\tau)}_{i}(y^{\pm}_{\alpha}),\quad\,
 W^{(\tau),\alpha\pm}_{i}=W^{(\tau)}_{i}(y^{\pm}_{\alpha}), \quad\,
 \lambda^{(\tau), \alpha\pm}_{i}=\lambda^{(\tau)}_{i}(y^{\pm}_{\alpha}).
\end{eqnarray}

When there is no wave interaction, no wave reflection, or no corner point on the boundary at $x$,
\begin{align}\label{eq:2.57}
&\frac{{\rm d} \mathscr{L}^{(\tau)}(U^{(\tau)}_{h, \nu}, V^{(\tau)}_{h', \nu})(x)}{{\rm d}x}\nonumber\\[2pt]
& =\sum_{\alpha \in \mathcal{J}^{(\tau)}}\sum^{4}_{j=1}\Big(|\hat{q}^{(\tau),\alpha-}_{j}|W^{(\tau),\alpha-}_{j}-|\hat{q}^{(\tau),\alpha+}_{j}|W^{(\tau),\alpha+}_{j}\Big)\dot{y}_{\alpha}
+\sum^{4}_{j=1}|\hat{q}^{(\tau)}_{j}(b)|W^{(\tau)}_{j}(b)\dot{y}_{b}\quad\nonumber\\[5pt]
&\quad\, -\kappa^{(\tau)}_{g}\bigg|\bigg(\frac{v^{(\tau)}_{h,\nu}}{1+\tau^{2}u^{(\tau)}_{h,\nu}}
-\frac{\hat{v}^{(\tau)}_{h,\nu}}{1+\tau^{2}\hat{u}^{(\tau)}_{h,\nu}}\bigg)(x,\check{g}_{h}(x))\bigg|,
\end{align}
where $\dot{y}_{\alpha}$ is the speed of the weak wave $\alpha\in \mathcal{J}^{(\tau)}$, $b=\check{g}^{-}_{h}(x)$,
and $\dot{y}_{b}=\check{g}_{h}'(x)$.
Notice that
\begin{eqnarray*}
&&|\hat{q}^{(\tau),1+}_{j}|W^{(\tau),1+}_{j}\lambda^{(\tau),1+}_{j}=|\hat{q}^{(\tau)}_{j}(b)|W^{(\tau)}_{j}(b)\lambda^{(\tau)}_{j}(b),\\[2pt]
&&|\hat{q}^{(\tau),(\alpha-1)-}_{j}|W^{(\tau),(\alpha-1)-}_{j}\lambda^{(\tau),(\alpha-1)-}_{j}
=|\hat{q}^{(\tau),\alpha+}_{j}|W^{(\tau),\alpha+}_{j}\lambda^{(\tau),\alpha+}_{j}.
\end{eqnarray*}
Then
\begin{align}\label{eq:2.58}
&\sum_{\alpha\in \mathcal{J}^{(\tau)}}\sum^{4}_{j=1}
|\hat{q}^{(\tau),\alpha-}_{j}|W^{(\tau),\alpha-}_{j}\lambda^{(\tau),\alpha-}_{j}\nonumber\\[2pt]
&\ =\sum_{\alpha\in \mathcal{J}^{(\tau)}}\sum^{4}_{j=1}|\hat{q}^{(\tau),\alpha+}_{j}|W^{(\tau),\alpha+}_{j}\lambda^{(\tau),\alpha+}_{j}
 -\sum^{4}_{j=1}|\hat{q}^{(\tau)}_{j}(b)|W^{(\tau)}_{j}(b)\lambda^{(\tau)}_{j}(b).
\end{align}

Substituting \eqref{eq:2.58} into \eqref{eq:2.57}, we have
\begin{align}\label{eq:2.59}
&\frac{{\rm d} \mathscr{L}^{(\tau)}(U^{(\tau)}_{h, \nu}, V^{(\tau)}_{h', \nu})(x)}{{\rm d}x}\nonumber\\[2pt]
&=\sum_{\alpha \in \mathcal{J}^{(\tau)}}\sum^{4}_{j=1}\Big(|\hat{q}^{(\tau),\alpha-}_{j}|W^{(\tau),\alpha-}_{j}\big(\dot{y}_{\alpha}-\lambda^{(\tau),\alpha-}_{i}\big)
-|\hat{q}^{(\tau),\alpha+}_{j}|W^{(\tau),\alpha+}_{j}\big(\dot{y}_{\alpha}-\lambda^{(\tau),\alpha+}_{i}\big)\Big)\nonumber\\[2pt]
&\quad \ +\sum^{4}_{j=1}|\hat{q}^{(\tau)}_{j}(b)|W^{(\tau)}_{j}(b)\big(\dot{y}_{b}-\lambda^{(\tau)}_{j}(b)\big)\nonumber\\
&\quad \, -\kappa^{(\tau)}_{g}\bigg|\bigg(\frac{v^{(\tau)}_{h,\nu}}{1+\tau^{2}u^{(\tau)}_{h,\nu}}
-\frac{\hat{v}^{(\tau)}_{h',\nu}}{1+\tau^{2}\hat{u}^{(\tau)}_{h',\nu}}\bigg)(x,\check{g}_{h}(x))\bigg|\nonumber\\[2pt]
&\doteq\sum_{\alpha \in \mathcal{J}^{(\tau)}}\sum^{4}_{j=1}E^{(\tau)}_{\alpha, j}
+\sum^{4}_{j=1}E^{(\tau)}_{b, j}-\kappa^{(\tau)}_{g}\bigg|\bigg(\frac{v^{(\tau)}_{h,\nu}}{1+\tau^{2}u^{(\tau)}_{h,\nu}}
-\frac{\hat{v}^{(\tau)}_{h',\nu}}{1+\tau^{2}\hat{u}^{(\tau)}_{h',\nu}}\bigg)(x,\check{g}_{h}(x))\bigg|.
\end{align}

Following the standard arguments
in \cite{bressan-liu-yang},
if $\alpha\in \mathcal{S}^{(\tau)}\cup \mathcal{C}^{(\tau)}\cup\mathcal{R}^{(\tau)}\cup\mathcal{NP}^{(\tau)}$,
then we choose $\kappa_{1}$ sufficiently large to obtain
\begin{eqnarray}\label{eq:2.60}
\sum^{4}_{j=1} E^{(\tau)}_{\alpha, j} \leq C |\sigma^{(\tau)}_{\alpha}|\nu^{-1} \qquad \mbox{or}\qquad  \sum^{4}_{j=1} E^{(\tau)}_{\alpha, j} \leq C |\sigma^{(\tau)}_{\alpha}|,
\end{eqnarray}
where constant {$C>0$} depends only on $(\underline{U},a_{\infty})$.
Therefore, the remaining task is to estimate $\sum^{4}_{i=1}E^{(\tau)}_{b,i}$ involving the approximate boundary $\check{\Gamma}_{h}$.

\smallskip
$\bullet$ If $\check{g}_{h}(x)=g_{h}(x)$, then, by Lemma \ref{lem:2.8}, we have
\begin{align*}
\sum^{4}_{j=1}E^{(\tau)}_{b,j}&=|q^{(\tau)}_{1}(b)|W^{(\tau)}_{1}(b)\big(\dot{y}_{b}-\lambda^{(\tau)}_{1}(b)\big)
+\textsc{w}^{(\tau)}_{4}|q^{(\tau)}_{4}(b)|W^{(\tau)}_{4}(b)\big(\dot{y}_{b}-\lambda^{(\tau)}_{4}(b)\big)\\[-2pt]
&\quad\ +\sum^{4}_{j=3}\textsc{w}^{(\tau)}_{j}|q^{(\tau)}_{j}(b)|W^{(\tau)}_{j}(b)\big(\dot{y}_{b}-\lambda^{(\tau)}_{j}(b)\big)\\[2pt]
&\leq \Big(|\mathbb{K}^{(\tau)}_{b}|W^{(\tau)}_{1}(b)(\dot{y}_{b}-\lambda^{(\tau)}_{1}(b))+|\text{w}^{(\tau)}_{4}|W^{(\tau)}_{4}(b)(\dot{y}_{b}-\lambda^{(\tau)}_{4}(b))\Big)
|q^{(\tau)}_{4}(b)|\\[2pt]
&\quad\ +|\mathbb{K}^{(\tau)}_{\theta}|W^{(\tau)}_{1}(b)\big(\dot{y}_{b}-\lambda^{(\tau)}_{1}(b)\big)
 \\
&\quad\quad\,\,\, \times\Big|\Big(\arctan(\frac{v^{(\tau)}_{h,\nu}}{1+\tau^{2}u^{(\tau)}_{h,\nu}})
-\arctan(\frac{\hat{v}^{(\tau)}_{h',\nu}}{1+\tau^{2}\hat{u}^{(\tau)}_{h',\nu}})\Big)(x,g_{h}(x))\Big|\\[-2pt]
&\quad\ +\sum^{4}_{j=3}\textsc{w}^{(\tau)}_{j}|q^{(\tau)}_{j}(b)|W^{(\tau)}_{j}(b)\big(\dot{y}_{b}-\lambda^{(\tau)}_{j}(b)\big)\\[2pt]
&\doteq J_{1, b}+J_{2,b}+J_{3,b}.
\end{align*}

Since
\begin{eqnarray*}
\lambda^{(\tau)}_{1}(b)\Big|_{U^{(\tau)}_{h, \nu}=V^{(\tau)}_{h', \nu}=\underline{U}}
=-\lambda^{(\tau)}_{4}(b)\Big|_{U^{(\tau)}_{h, \nu}=V^{(\tau)}_{h', \nu}=\underline{U}}
\end{eqnarray*}
at the background states when $U^{(\tau)}_{h, \nu}(b)=V^{(\tau)}_{h', \nu}(b)=\underline{U}$ and $\dot{y}_{b}=0$,
we find that, for the coefficient function in $J_{1, b}$,
if $\textsc{w}^{(\tau)}_{4}>0$ sufficiently large and $\tau>0$ sufficiently small,
\begin{align*}
&\Big(|\mathbb{K}^{(\tau)}_{b}|W^{(\tau)}_{1}(b)\big(\dot{y}_{b}-\lambda^{(\tau)}_{1}(b)\big)
+|\textsc{w}^{(\tau)}_{4}|W^{(\tau)}_{4}(b)\big(\dot{y}_{b}
-\lambda^{(\tau)}_{4}(b)\big)\Big)\Big|_{U^{(\tau)}_{h, \nu}(b)=V^{(\tau)}_{h', \nu}(b)=\underline{U}, \dot{y}_{b}=0}\\[2pt]
&= \Big(|\mathbb{K}^{(\tau)}_{b}|W^{(\tau)}_{1}(b)-|\textsc{w}^{(\tau)}_{4}|W^{(\tau)}_{4}(b)\Big)\lambda^{(\tau)}_{4}(b)
\Big|_{U^{(\tau)}_{h, \nu}(b)=V^{(\tau)}_{h', \nu}(b)=\underline{U}, \dot{y}_{b}=0}\\[2pt]
&\leq -C_{2,8},
\end{align*}
where constant $C_{2,8}>0$ depends only on $(\underline{U},a_{\infty})$.

Therefore, if $\textsc{w}^{(\tau)}_{4}>0$ is large enough, depending only on $(\underline{U},a_{\infty})$,
and if $\varepsilon\in(0, \min\{\varepsilon_{2},\varepsilon_{3} \})$ and $\tau\in(0,\min\{\tau_{1},\tau_{2}\})$,
there exists a constant $C_{2,9}>0$ depending only on $(\underline{U},a_{\infty})$ such that
\begin{eqnarray*}
J_{1, b}\leq -C_{2,9}|q^{(\tau)}_{4}(b)|.
\end{eqnarray*}

For $J_{2,b}$, if $\varepsilon>0$ and $\tau>0$ are sufficiently small, we obtain
\begin{eqnarray*}
J_{2, b}\leq C\bigg|\Big(\frac{v^{(\tau)}_{h,\nu}}{1+\tau^{2}u^{(\tau)}_{h,\nu}}
-\frac{\hat{v}^{(\tau)}_{h',\nu}}{1+\tau^{2}\hat{u}^{(\tau)}_{h',\nu}}\Big)(x,g_{h}(x))\bigg|.
\end{eqnarray*}

For $J_{3,b}$, note that, when $j=2,3$, it follows from Lemma \ref{lem:2.8} that
\begin{align*}
\dot{y}_{b}-\lambda^{(\tau)}_{j}(b)&\leq C\Big(|q^{(\tau)}_{1}(b)|+|q^{(\tau)}_{4}(b)|\Big)\\[2pt]
&\leq C|q^{(\tau)}_{4}(b)|+O(1)\bigg|\Big(\frac{v^{(\tau)}_{h,\nu}}{1+\tau^{2}u^{(\tau)}_{h,\nu}}
-\frac{\hat{v}^{(\tau)}_{h',\nu}}{1+\tau^{2}\hat{u}^{(\tau)}_{h',\nu}}\Big)(x,g_{h}(x))\bigg|.
\end{align*}
Inserting this estimate into $J_{3, b}$ to obtain
\begin{eqnarray*}
J_{3,b}
\leq C\big(\textsc{w}^{(\tau)}_{2}+\textsc{w}^{(\tau)}_3\big)\bigg(|q^{(\tau)}_{4}(b)|
+\bigg|\Big(\frac{v^{(\tau)}_{h,\nu}}{1+\tau^{2}u^{(\tau)}_{h,\nu}}
-\frac{\hat{v}^{(\tau)}_{h',\nu}}{1+\tau^{2}\hat{u}^{(\tau)}_{h',\nu}}\Big)(x,g_{h}(x))\bigg|\bigg).
\end{eqnarray*}

Finally, by letting $\textsc{w}^{(\tau)}_{4}>0$ and $\kappa^{(\tau)}_{g}$ sufficiently large
and choosing $\textsc{w}^{(\tau)}_{2}$ and $\textsc{w}^{(\tau)}_{3}$ appropriately,
depending only on $(\underline{U},a_{\infty})$,
it follows from the estimates of $J_{1, b}$, $J_{2, b}$, and $J_{3, b}$ that
\begin{eqnarray}\label{eq:2.61}
\sum^{4}_{j=1}E^{(\tau)}_{b,j}-\kappa^{(\tau)}_{g}\bigg|\Big(\frac{v^{(\tau)}_{h,\nu}}{1+\tau^{2}u^{(\tau)}_{h,\nu}}
-\frac{\hat{v}^{(\tau)}_{h',\nu}}{1+\tau^{2}\hat{u}^{(\tau)}_{h',\nu}}\Big)(x,g_{h}(x))\bigg|
< 0.
\end{eqnarray}

$\bullet$ If $\check{g}_{h}(x)=g_{h'}(x)$, this can done in the same way as done for the first case.

$\bullet$ If $\check{g}_{h}(x)=g_{h}(x)=g_{h'}(x)$, then it follows from Lemma \ref{lem:2.8} that
$\mathbb{K}^{(\tau)}_{\theta}=0$. Then $J_{2,b}\equiv 0$.
Now, choosing $\textsc{w}^{(\tau)}_{4}>0$ sufficiently large and choosing $\textsc{w}^{(\tau)}_{2}$
and $\textsc{w}^{(\tau)}_{3}$ appropriately, we obtain that $\sum^{4}_{j=1}E^{(\tau)}_{b,j}<0$.

Thus, inequality \eqref{eq:2.61} holds for all the three subcases above.
Then, by estimates \eqref{eq:2.60}--\eqref{eq:2.61}, we see from equality \eqref{eq:2.59} that
\begin{eqnarray}\label{eq:2.62}
\frac{\dd \mathscr{L}^{(\tau)}(U^{(\tau)}_{h, \nu}, V^{(\tau)}_{h', \nu})(x)}{\dd x}<C\nu^{-1}.
\end{eqnarray}

To complete the proof of this lemma, let $\check{x}_{1}<\check{x}_{2}<\cdot\cdot\cdot<\check{x}_{N}$ be the points in $(x',x'')$
such that, when $x=\check{x}_j$ for $1\leq j\leq N$, the fronts in $U^{(\tau)}_{h, \nu}$ or $V^{(\tau)}_{h',\nu}$ interact,
reflect on the boundary, or are generated by the corner points on the boundary.
We integrate inequality \eqref{eq:2.60} along the flow direction
on each of intervals $(x',\check{x}_{1})$, $(\check{x}_{1},\check{x}_{2})$, $\cdots$,
$(\check{x}_{N}, x'')$ to obtain
\begin{equation}\label{eq:2.63}
\begin{split}
&\mathscr{L}^{(\tau)}(U^{(\tau)}_{h, \nu}, V^{(\tau)}_{h', \nu})(\check{x}_{1}-)
  <\mathscr{L}^{(\tau)}(U^{(\tau)}_{h, \nu}, V^{(\tau)}_{h', \nu})(x')+C(\check{x}_{1}-x')\nu^{-1},\\[2pt]
&\mathscr{L}^{(\tau)}(U^{(\tau)}_{h, \nu}, V^{(\tau)}_{h', \nu})(\check{x}_{2}-)
 <\mathscr{L}^{(\tau)}(U^{(\tau)}_{h, \nu}, V^{(\tau)}_{h', \nu})(\check{x}_{1}+)+C(\check{x}_{2}-\check{x}_{1})\nu^{-1},\\[2pt]
& \qquad\qquad\qquad\qquad\qquad \cdot\cdot\cdot\cdot \\[2pt]
&\mathscr{L}^{(\tau)}(U^{(\tau)}_{h, \nu}, V^{(\tau)}_{h', \nu})(x)<\mathscr{L}^{(\tau)}(U^{(\tau)}_{h, \nu}, V^{(\tau)}_{h', \nu})(\check{x}_{N}+)
 +C(x''-\check{x}_{N})\nu^{-1},\qquad
\end{split}
\end{equation}
where $C>0$ depends only on $(\underline{U},a_{\infty})$.
When $x=\check{x}_{k}$, for $k=1,\cdot\cdot\cdot, N$,
choosing $\kappa^{(\tau)}_{1}$, $\kappa^{(\tau)}_{2}$, $\kappa^{(\tau)}_{3}$, $\kappa^{(\tau)}_{c}$, and $\kappa^{(\tau)'}_{c}$ sufficiently large and
using Lemma \ref{lem:2.7}, we have
\begin{equation}\label{eq:2.64}
\mathscr{L}^{(\tau)}(U^{(\tau)}_{h, \nu}, V^{(\tau)}_{h', \nu})(\check{x}_{k}+)<\mathscr{L}^{(\tau)}(U^{(\tau)}_{h, \nu}, V^{(\tau)}_{h', \nu})(\check{x}_{k}-)
\quad\,\,\mbox{for $k=1,\cdot\cdot\cdot, N$}.\,\,\,
\end{equation}

Combining estimates \eqref{eq:2.63}--\eqref{eq:2.64}, we obtain inequality \eqref{eq:2.55}.
\end{proof}

Now we are ready to show the following proposition.
\begin{proposition}\label{prop:2.2}
Under assumptions $\mathbf{(U_0)}$ and $\mathbf{(g)}$,
for a given fixed hypersonic similarity parameter $a_{\infty}$,
there exist constants $\varepsilon_{0}<\min\{\varepsilon_{2},\varepsilon_{3} \})$ and $\tau_{0}<\min\{\tau_{1},\tau_{2}\}$
depending only on $(\underline{U},a_{\infty})$ such that, for $\varepsilon\in(0, \varepsilon_{0})$ and $\tau\in(0,\tau_{0})$,
if $U_{0}$ and $g(x)$ satisfy \eqref{eq:2.36}, then
the approximate solution $U^{(\tau)}_{h,\nu}$ converges to $U^{(\tau)}_{h}$ in $L^{1}(-\infty, g_{h}(x))$ as $\nu\rightarrow\infty$.
Moreover, there exist a constant $L>0$, a domain $\mathcal{D}_{h,x}$, and a map $\mathcal{P}^{(\tau)}_{h}(x,x'_{0})${\rm :}
\begin{eqnarray}\label{eq:2.65}
\mathcal{P}^{(\tau)}_{h}(x,x'_{0}):\mathcal{D}_{h,x'_{0}}\mapsto \mathcal{D}_{h, x} \qquad \mbox{for all $x\geq x'_{0}\geq 0$},
\end{eqnarray}
such that
\begin{enumerate}
\item [\rm (i)]\
For any $x\geq 0$ and $U^{(\tau)}_{h}(x,\cdot)\in \mathcal{D}_{h, x}$,  $\mathcal{P}^{(\tau)}_{h}(x,x)(U^{(\tau)}_{h}(x,\cdot))=U^{(\tau)}_{h}(x,\cdot)$, and,
for all $x\geq x'\geq x'_{0}\geq 0$ and $U^{(\tau)}_{h}(x'_{0},\cdot)\in \mathcal{D}_{h, x'_{0}}$,
\begin{eqnarray}\label{eq:2.66}
\mathcal{P}^{(\tau)}_{h}(x,x'_{0})(U^{(\tau)}_{h}(x'_{0}))=\mathcal{P}^{(\tau)}_{h}(x,x')\circ \mathcal{P}^{(\tau)}_{h}(x',x'_{0})(U^{(\tau)}_{h}(x'_{0}));
\end{eqnarray}

\item[\rm (ii)]\ If $\mathcal{P}^{(\tau)}_{h'}$ and $\mathcal{D}_{h', x}$ are the map and domain corresponding
to the boundary function $g_{h'}(x)$ and initial data $V^{\nu}_{0}$, then, for any $x''\geq x'>x'_{0}>0$,
\begin{eqnarray}\label{eq:2.67}
&&\big\|\mathcal{P}^{(\tau)}_{h}(x',x'_0)(U^{(\tau)}_{h}(x'_0,\cdot))-\mathcal{P}^{(\tau)}_{h'}(x'',x'_0)(V^{(\tau)}_{h'}(x'_0,\cdot))
\big\|_{L^{1}((-\infty, \hat{g}_{h}(x'')))}\nonumber\\[4pt]
&&\leq
L\Big(|x'-x''|+\|U^{(\tau)}_{h}(x'_0,\cdot)-V^{(\tau)}_{h'}(x'_0,\cdot)\|_{L^{1}((-\infty, \hat{g}_{h}(x'_0)))}\nonumber\\
&&\qquad\,\,\, +\|g'_{h}(\cdot)-g'_{h'}(\cdot)\|_{L^{1}((x_0',\infty))}\Big),
\end{eqnarray}
where constant $L$ depends only on $U_{\infty}$ and $a_{\infty}$, and
we have used that
\begin{align*}
&\mathcal{P}_h^{(\tau)}(x',x_{0})(U_h^{(\tau)}(x_{0},\cdot))=U_h^{(\tau)}(x',\cdot)=\underline{U} \qquad\,\,\mbox{on $(g_h(x'),\infty)$},\\ &\mathcal{P}_{h'}^{(\tau)}({x}'',{x}_{0})(V_{h'}^{(\tau)}({x}_{0},\cdot))=V_{h'}^{(\tau)}(x'',\cdot)=\underline{U}  \qquad\mbox{on $(g_{h'}(x''),\infty)$}.
\end{align*}
\end{enumerate}
\end{proposition}

\begin{proof}
For two approximate solutions $U^{(\tau)}_{h,\nu}=(\rho^{(\tau)}_{h,\nu},u^{(\tau)}_{h,\nu},
v^{(\tau)}_{h,\nu}, p^{(\tau)}_{h,\nu})^{\top}$
and $V^{(\tau)}_{h',\nu}=(\hat{\rho}^{(\tau)}_{h',\nu},\hat{u}^{(\tau)}_{h',\nu},
\hat{v}^{(\tau)}_{h',\nu}, \hat{p}^{(\tau)}_{h',\nu})^{\top}$ corresponding
to $(U^{\nu}_{0}, g_{h})$ and $(V^{\nu}_{0}, g_{h'})$, respectively,
it follows from estimates \eqref{eq:2.47} and \eqref{eq:2.55} that
\begin{eqnarray*}
&&\big\|U^{(\tau)}_{h,\nu}(x,\cdot)-V^{(\tau)}_{h',\nu}(x,\cdot)\big\|_{L^1((-\infty,\check{g}_{h}(x)))}\\[2pt]
&&\leq C\big\|U^{(\tau)}_{h,\nu}(x'_{0},\cdot)-V^{(\tau)}_{h',\nu}(x'_0,\cdot)\big\|_{L^1((-\infty,\check{g}_{h}(x'_0)))}\\[2pt]
&&\quad +C\bigg\|\Big(\frac{v^{(\tau)}_{h,\nu}}{1+\tau^{2}u^{(\tau)}_{h,\nu}}
-\frac{\hat{v}^{(\tau)}_{h',\nu}}{1+\tau^{2}\hat{u}^{(\tau)}_{h',\nu}}\Big)(s,\check{g}_{h}(s))\bigg\|_{L^{1}((x'_0,x))}
+C(x-x'_{0})\nu^{-1}.
\end{eqnarray*}

Notice that, by the triangle inequality,
\begin{eqnarray}\label{eq:2.68}
&&\bigg\|\Big(\frac{v^{(\tau)}_{h,\nu}}{1+\tau^{2}u^{(\tau)}_{h,\nu}}
-\frac{\hat{v}^{(\tau)}_{h',\nu}}{1+\tau^{2}\hat{u}^{(\tau)}_{h',\nu}}\Big)(s,\check{g}_{h}(s))\bigg\|_{L^{1}((x'_0,x))}\nonumber\\[2pt]
&&\leq \bigg\|\Big(\frac{v^{(\tau)}_{h,\nu}}{1+\tau^{2}u^{(\tau)}_{h,\nu}}\Big)(s,\check{g}_{h}(s))
-g'_{h}(s)\bigg\|_{L^{1}((x'_0,x))}+\big\|g'_{h}(s)-g'_{h'}(s)\big\|_{L^{1}((x'_0,x))}\nonumber\\[2pt]
&&\quad +\bigg\|\Big(\frac{\hat{v}^{(\tau)}_{h',\nu}}{1+\tau^{2}\hat{u}^{(\tau)}_{h',\nu}}\Big)(s,\check{g}_{h}(s))
-g'_{h'}(s)\Bigg\|_{L^{1}((x'_0,x))}.
\end{eqnarray}

For the first term on the right side of \eqref{eq:2.68}, {by Proposition \ref{prop-flow-slop}},
we can obtain
\begin{align*}
& \bigg\|\Big(\frac{v^{(\tau)}_{h,\nu}}{1+\tau^{2}u^{(\tau)}_{h,\nu}}\Big)(s,\check{g}_{h}(s))
-g'_{h}(s)\bigg\|_{L^{1}((x'_0,x))}\\[2pt]
&=\bigg\|\Big(\frac{v^{(\tau)}_{h,\nu}}{1+\tau^{2}u^{(\tau)}_{h,\nu}}\Big)(s,\check{g}_{h}(s))
-\Big(\frac{v^{(\tau)}_{h,\nu}}{1+\tau^{2}u^{(\tau)}_{h,\nu}}\Big)(s,g_{h}(s))\bigg\|_{L^{1}((x'_0,x))}\\[2pt]
&\leq C\big\|g_{h}(s)-g_{h'}(s)\big\|_{L^{\infty}((x'_0,x))} \\[2pt]
&\leq C\Big(\|g'_{h}(\cdot)-g'_{h'}(\cdot)\|_{L^{1}((x'_0,\infty))}+|g'_{h}(x_0')-g'_{h'}(x_0')|\Big),
\end{align*}
where constant $C>0$ depends only on $(\underline{U},a_{\infty})$, provided that $\tau$ is sufficiently small.

In the same way, we can also apply Proposition \ref{prop-flow-slop}
for $\frac{\hat{v}^{(\tau)}_{h',\nu}}{1+\tau^{2}\hat{u}^{(\tau)}_{h',\nu}}$ to obtain
\begin{eqnarray*}
&&\bigg\|\Big(\frac{\hat{v}^{(\tau)}_{h',\nu}}{1+\tau^{2}\hat{u}^{(\tau)}_{h',\nu}}\Big)(s,\check{g}_{h}(s))
-g'_{h'}(s)\bigg\|_{L^{1}((x_0',x))}\\[2pt]
&&\leq C\big\|g_{h}(s)-g_{h'}(s)\big\|_{L^{\infty}((x_0',x))}\\[2pt]
&&\leq C\Big(\|g'_{h}(\cdot)-g'_{h'}(\cdot)\|_{L^{1}((x'_0,\infty))}+|g'_{h}(x_0')-g'_{h'}(x_0')|\Big),
\end{eqnarray*}
provided  that $\tau$ is sufficiently small.

Inserting the above two estimates into \eqref{eq:2.68}, we have
\begin{eqnarray*}
&&\bigg\|\Big(\frac{v^{(\tau)}_{h,\nu}}{1+\tau^{2}u^{(\tau)}_{h,\nu}}
-\frac{\hat{v}^{(\tau)}_{h',\nu}}{1+\tau^{2}\hat{u}^{(\tau)}_{h',\nu}}\Big)(s,\check{g}_{h}(s))\bigg\|_{L^{1}((x_0',x))}\\[1pt]
&&\leq C\big(\|g'_{h}(\cdot)-g'_{h'}(\cdot)\|_{L^{1}((x'_0,\infty))}+|g_{h}(x_0')-g_{h'}(x_0')|\big).
\end{eqnarray*}

Therefore, there exists a constant $C_{2,10}>0$ depending only on $(\underline{U},a_{\infty})$ such that
\begin{align*}
&\big\|U^{(\tau)}_{h,\nu}(x,\cdot)-V^{(\tau)}_{h',\nu}(x,\cdot)\big\|_{L^1((-\infty,\check{g}_{h}(x)))}\\[4pt]
&\leq C_{2,10}\Big(\big\|U^{(\tau)}_{h,\nu}(x_{0},\cdot)-V^{(\tau)}_{h',\nu}(x_0,\cdot)\big\|_{L^1((-\infty,\check{g}_{h}(x'_0)))}\\[2pt]
&\qquad\qquad\, +\big\|g'_{h}(\cdot)-g'_{h'}(\cdot)\big\|_{L^{1}((x'_0,\infty))}+\big|g_{h}(x_0')-g_{h'}(x_0')\big|\Big)+ C(x-x'_0)\nu^{-1}.\quad
\end{align*}

We notice from Proposition \ref{prop:2.1} that
\begin{align*}
\big|g_{h}(x_0')-g_{h'}(x_0')\big|
&\leq C\int^{\hat{g}_{h}(x_0')}_{\check{g}_{h}(x_0')}\big|U^{(\tau)}_{h,\nu}(x_0',s)-V^{(\tau)}_{h',\nu}(x_0',s)\big|\,{\rm d}s\\[2pt]
&\leq C\big\|U^{(\tau)}_{h,\nu}(x'_{0},\cdot)-V^{(\tau)}_{h',\nu}(x'_0,\cdot)\big\|_{L^1((-\infty,\hat{g}_{h}(x'_0)))}.
\end{align*}
Then
\begin{align}\label{eq:2.69}
&\big\|U^{(\tau)}_{h,\nu}(x,\cdot)-V^{(\tau)}_{h',\nu}(x,\cdot)\big\|_{L^1((-\infty,\hat{g}_{h}(x)))}\nonumber\\[2pt]
&\leq \big\|U^{(\tau)}_{h,\nu}(x,\cdot)-V^{(\tau)}_{h',\nu}(x,\cdot)\big\|_{L^1((-\infty,\check{g}_{h}(x)))}
+\int^{\hat{g}_{h}(x)}_{\check{g}_{h}(x)}\big|U^{(\tau)}_{h,\nu}(x,s)-V^{(\tau)}_{h',\nu}(x,s)\big|\,{\rm d}s\nonumber\\[3pt]
&\leq C\Big(\big\|U^{(\tau)}_{h,\nu}(x'_{0},\cdot)-V^{(\tau)}_{h',\nu}(x'_0,\cdot)\big\|_{L^1((-\infty,\hat{g}_{h}(x'_0)))}
+\|g'_{h}(\cdot)-g'_{h'}(\cdot)\|_{L^{1}((x'_0,\infty))}\Big)\nonumber\\[2pt]
&\quad + C\big|g_{h}(x_0')-g_{h'}(x_0')\big|+C(x-x'_0)\nu^{-1}\nonumber\\[2pt]
&\leq C_{2,11}\Big(\big\|U^{(\tau)}_{h,\nu}(x'_{0},\cdot)-V^{(\tau)}_{h',\nu}(x'_0,\cdot)\big\|_{L^1((-\infty,\hat{g}_{h}(x'_0)))}
+\big\|g'_{h}(\cdot)-g'_{h'}(\cdot)\big\|_{L^{1}((x'_0,\infty))}\Big)\nonumber\\[2pt]
&\quad + C(x-x'_0)\nu^{-1},
\end{align}
where constant $C_{2,11}>0$ depends only on $(\underline{U},a_{\infty})$.

Now, let $x_{0}'=0$, $h=h'$, and $V^{(\tau)}_{h', \nu}=U^{(\tau)}_{h, \nu'}$. Then
\begin{eqnarray}\label{eq:2.70}
&&\big\|U^{(\tau)}_{h,\nu}-U^{(\tau)}_{h,\nu'}\big\|_{L^1((-\infty,g_{h}(x)))}\nonumber\\
&&\leq C_{2,9}\big\|U^{\nu}_{0}-U^{\nu'}_{0}\big\|_{L^{1}(\mathcal{I})}+ C\,x\max\big\{\nu^{-1}, \nu'^{-1}\big\}.
\end{eqnarray}
For given $U_{0}\in \mathcal{D}_{h, 0}$, we approximate $U_{0}$ by $U^{\nu}_{0}$ with property \eqref{eq:2.28}
and construct the approximate solution $U^{(\tau)}_{h,\nu}$ with boundary $y=g_{h}(x)$.
Thus, $\big\{U^{(\tau)}_{h,\nu}\big\}_{\nu>0}$ is a Cauchy sequence and
converges in $L^{1}(-\infty, g_{h}(x))$ to a unique limit $U^{(\tau)}_{h}$ as $\nu\rightarrow\infty$.
Then define
$$
\mathcal{P}^{(\tau)}_{h}(x,0)(U_{0}(y))\doteq U^{(\tau)}_{h}(x,y)
$$
and
\begin{align}\label{eq:2.71}
\mathcal{D}_{h, x}\doteq \mathbf{cl}\left\{U^{(\tau)}_{h}-\underline{U}\in (L^{1}\cap BV)(\mathbb{R};\mathbb{R}^{4})\, :
\begin{array}{l}
U^{(\tau)}_{h}=\underline{U} \ \mbox{for $y>g_{h}(x)$},\\
\mathbf{G}^{(\tau)}(U^{(\tau)}_{h})<\varepsilon
\end{array}
\right \},\nonumber\\[2pt]
\end{align}
where $\mathbf{cl}$ denotes the closure in the $L^{1}$--topology, and $\mathbf{G}^{(\tau)}$ is defined by
\begin{eqnarray*}
&&\mathbf{G}^{(\tau)}(U^{(\tau)}_{h})=\mathbf{V}^{(\tau)}(U^{(\tau)}_{h})+\mathcal{K}^{(\tau)}\mathbf{Q}^{(\tau)}(U^{(\tau)}_{h}),\\[2pt]
&&\mathbf{V}^{(\tau)}(U^{(\tau)}_{h})=\mathbf{V}^{(\tau)}_{1}(U^{(\tau)}_{h})+\sum^{4}_{ i=2}\mathcal{K}^{(\tau)}_{i}\mathbf{V}^{(\tau)}_{i}(U^{(\tau)}_{h})
+\mathcal{K}^{(\tau)}_{c}\mathbf{V}^{(\tau)}_{c}(U^{(\tau)}_{h}),\\[2pt]
&&
\mathbf{V}^{(\tau)}_{i}(U^{(\tau)}_{h})=\sum_{\alpha_i\in \mathcal{J}^{(\tau)}}|\sigma^{(\tau)}_{\alpha_i}|
\qquad\, \mbox{for $1\leq i\leq 4$}, \\[2pt]
&&\mathbf{V}^{(\tau)}_{c}(U^{(\tau)}_{h})=\sum_{k>[\frac{x}{h}]}|\omega_{k}|, \\[2pt]
&&\mathbf{Q}^{(\tau)}(U^{(\tau)}_{h})=\sum_{(\alpha_i,\beta_j)\in \mathcal{A}^{(\tau)}\setminus \mathcal{NP}^{(\tau)}}|\sigma^{(\tau)}_{\alpha_i}||\sigma^{(\tau)}_{\beta_j}|.
\end{eqnarray*}

By assumption \eqref{eq:2.36}, we notice that
\begin{eqnarray*}
\mathcal{G}^{(\tau)}(0)< \mathbf{G}^{(\tau)}(U^{\nu}_{0})
<\mathbf{G}^{(\tau)}(U_{0})<\varepsilon.
\end{eqnarray*}
Then
\begin{eqnarray*}
\mathbf{G}^{(\tau)}(U^{(\tau)}_{h,\nu})<\mathcal{G}^{(\tau)}(x)< \mathcal{G}^{(\tau)}(0)
<\varepsilon.
\end{eqnarray*}
By the lower semicontinuity of $\mathbf{G}^{(\tau)}(U^{(\tau)}_{h, \nu})$
and the convergence of $U^{(\tau)}_{h,\nu}$ as $\nu\rightarrow\infty$,
we see that $\mathcal{P}^{(\tau)}_{h}(x,0)(U_{0}(y))\in \mathcal{D}_{h, x}$.

From the construction, we know that map $\mathcal{P}^{(\tau)}_{h}$ can start at any ``time'' $x\geq 0$.
Clearly,
for $x\geq 0$ and $U^{(\tau)}_{h}(x,y)\in \mathcal{D}_{h, x}$,
$\mathcal{P}^{(\tau)}_{h}(x,x)(U^{(\tau)}_{h}(x,y))=U^{(\tau)}_{h}(x,y)$
and, for any $x\geq x'_{0}\geq 0$ and $U^{(\tau)}_{h}(x'_{0},\cdot)\in \mathcal{D}_{h, x_0'}$,
$\mathcal{P}^{(\tau)}_{h}(x,x'_{0})(U^{(\tau)}_{h}(x'_0,\cdot))\in \mathcal{D}_{h, x}$.

To show equality \eqref{eq:2.66} for map $\mathcal{P}^{(\tau)}_{h}$, for a given solution $U^{(\tau)}_{h}$,
consider its approximate solutions $U^{(\tau)}_{h,\nu}(x,\cdot)$ and $U^{(\tau)}_{h,\nu'}(x,\cdot)$ corresponding
to the initial data $U^{(\tau)}_{h,\nu}(x'_0,\cdot)$ and $U^{(\tau)}_{h,\nu}(x',\cdot)$ respectively for $\nu'>\nu$.
By \eqref{eq:2.70}, we see that, for $x>x'$,
\begin{align*}
&\big\|U^{(\tau)}_{h,\nu}(x,\cdot)-\mathcal{P}^{(\tau)}_{h}(x,x')(U^{(\tau)}_{h,\nu}(x',\cdot))\big\|_{L^1((-\infty, g_{h}(x)))}\\[2pt]
&=\big\|U^{(\tau)}_{h,\nu}(x,\cdot)-U^{(\tau)}_{h,\nu'}(x,\cdot)\big\|_{L^1((-\infty, g_{h}(x)))}\\[2pt]
&\leq C(x-x')\nu'^{-1},
\end{align*}
where constant $C$ depends only on $(\underline{U},a_{\infty})$.
Taking the limit: $\nu$, $\nu'\rightarrow \infty$ above, we obtain
\begin{eqnarray*}
\big\|U^{(\tau)}_{h}(x,\cdot)-\mathcal{P}^{(\tau)}_{h}(x,x')(U^{(\tau)}_{h}(x',\cdot))\big\|_{L^1((-\infty, g_{h}(x)))}=0,
\end{eqnarray*}
which implies \eqref{eq:2.66}.

Finally, we show \eqref{eq:2.67}.
For the Lipschtiz continuity of map $\mathcal{P}^{(\tau)}_{h}(x,x'_{0})$
with respect to $x$, it is a direct consequence of Proposition \ref{prop:2.1}.
For the Lipschitz continuity of $\mathcal{P}^{(\tau)}_{h}(x,x'_{0})$ with respect to the initial data,
by
\eqref{eq:2.69},
\begin{align*}
&\big\|U^{(\tau)}_{h}(x,\cdot)-V^{(\tau)}_{h'}(x,\cdot)\big\|_{L^1((-\infty,\hat{g}_{h}(x)))}\\[4pt]
&\leq C_{2,11}\Big(\|U^{(\tau)}_{h}(x'_{0},\cdot)-V^{(\tau)}_{h'}(x'_0,\cdot)\|_{L^1((-\infty,\check{g}_{h}(x_0)))}
+\|g'_{h}(\cdot)-g'_{h'}(\cdot)\|_{L^{1}((x'_0,\infty))}\\
&\qquad\qquad +|g_{h}(x_0')-g_{h'}(x_0')|\Big).
\end{align*}
Then estimate \eqref{eq:2.67} can be obtained by the triangle inequality,
and the Lipschitz continuity of $\mathcal{P}^{(\tau)}_{h}(x,x'_{0})$ with respect to both $x$ and the initial data.
This completes the proof.
\end{proof}

\smallskip
Now we are ready to prove Theorem \ref{thm:1.1}.

\begin{proof}[Proof of {\rm Theorem \ref{thm:1.1}}]
Consider two approximate solutions $U^{(\tau)}_{h}$ and $U^{(\tau)}_{h'}$, which are the limits of $U^{(\tau)}_{h,\nu}$ and $U^{(\tau)}_{h',\nu}$
as $\nu\rightarrow \infty$ respectively, to the same initial data $U_{0}$ and the boundary functions $g_{h}$ and $g_{h'}$ respectively.
Then, by \eqref{eq:2.67},
\begin{eqnarray*}
\big\|\mathcal{P}^{(\tau)}_{h}(x,0)(U_{0}(\cdot))-\mathcal{P}^{(\tau)}_{h'}(x,0)(U_{0}(\cdot))
\big\|_{L^{1}((-\infty, \hat{g}_{h}(x)))}
\leq L\big\|g'_{h}(\cdot)-g'_{h'}(\cdot)\big\|_{L^{1}(\mathbb{R}_{+})}.
\end{eqnarray*}
Note that the right-hand side of the above inequality tends to zero as $h$, $h'\rightarrow 0$.
Thus,
$\{\mathcal{P}^{(\tau)}_{h}(x,0)(U_{0}(\cdot))\}_{h}$ is a Cauchy sequence and
converges to a unique limit $\mathcal{P}^{(\tau)}(x,0)(U_{0})$, \emph{i.e.},
$\mathcal{P}^{(\tau)}_{h}(x,0)(U_{0})\rightarrow \mathcal{P}^{(\tau)}(x,0)(U_{0})$ in $L^{1}((-\infty, g(x)))$ as $h \rightarrow 0$.
Define
\begin{equation}\label{eq-D}
\mathcal{D}_{x}=\bigcap_{h>0}\mathcal{D}_{h, x}\subseteq\big\{U^{(\tau)}-\underline{U}\in (L^{1}\cap BV)(\mathbb{R};\mathbb{R}^{4}) \, :
\, U^{(\tau)}=\underline{U} \  \mbox{for}\   y>g(x)\big\}
\end{equation}
and $U^{(\tau)}(x,y)\doteq\mathcal{P}^{(\tau)}(x,0)(U_{0}(y))$.
Then, by Proposition \ref{prop:2.2} and the uniqueness of $\mathcal{P}^{(\tau)}$,
we know that $\mathcal{P}^{(\tau)}(x,x'_{0}):[0, \infty)\times \mathcal{D}_{x_0}\mapsto \mathcal{D}_{x}$, and property \rm(i) in Theorem \ref{thm:1.1} holds.
Furthermore, by the same arguments as done in \cite{chen-kuang-zhang}, using Proposition \ref{prop:2.1} and the Lebesgue dominate convergence theorem,
we can show that $U^{(\tau)}(x,y)$ is a global entropy solution of Problem I.
Finally, consider the approximate solutions $\mathcal{P}^{(\tau)}_{h}(x,x'_{0})(U^{(\tau)}_{h}(x'_{0},y))$
and $\mathcal{P}^{(\tau)}_{h'}(x,x'_{0})(V^{(\tau)}_{h'}(x'_{0},y))$
corresponding to the initial states and the boundary functions $(U^{(\tau)}_{h}(x'_{0},y),g_h)$ and $(V^{(\tau)}_{h'}(x_{0},y),g_{h'})$, respectively.
Then, by letting $h\rightarrow0$ and $h'\rightarrow0$, it follows from estimate \eqref{eq:2.67} that estimate \eqref{eq:1.24} holds.
\end{proof}

Finally, we remark that, by a similar  argument, the well-posedness theorems for the initial-boundary value problem \eqref{eq:1.12}--\eqref{eq:1.15}
for $\tau=0$ can also be obtained. We provide brief proofs of the well-posedness theorems in Appendix A.

\section{Comparison between the Riemann Solvers for the Two Systems}\setcounter{equation}{0}
In this section, we compare the Riemann solvers for system \eqref{eq:1.11} and system \eqref{eq:1.14} with appropriate initial and boundary conditions.

\subsection{Comparison of the Riemann solvers away from the boundary}
For the approximate solutions $U_{h,\nu}$ constructed in Appendix A for system \eqref{eq:1.14},
let $(\xi, y_I)$ be a discontinuity point of $U_{h,\nu}$. Define
\begin{equation}\label{eq:4.1}
\begin{split}
&U_{a}\doteq(\rho_{a}, u_{a}, v_{a}, p_{a})^{\top}=\lim_{y\rightarrow y_{I}+}U_{h,\nu}(\xi,y),\\
&U_{b}\doteq(\rho_{b}, u_{b}, v_{b}, p_{b})^{\top}=\lim_{y\rightarrow y_{I}-}U_{h,\nu}(\xi,y).
\end{split}
\end{equation}
Let  $\mathcal{P}^{(\tau)}_{h}$ be the Lipschtiz continuous map as given in Proposition \ref{prop:2.2} for system \eqref{eq:1.11}.
We first establish the following lemma.

\begin{lemma}\label{lem:4.1}
For $\varepsilon\in (0,\min\{\varepsilon_0, \tilde{\varepsilon}_0\})$,
let  $U_{b}\in \mathcal{O}_{\varepsilon}(\underline{U})$ be a given constant state, and let $|\sigma_j|<\varepsilon$ be parameters for $1\leq j\leq 4$.
Then,
for $\tau\in (0,\min\{\tau_0, \tilde{\tau}_0\})$, and for maps
$\Phi^{(\tau)}_j$ and $\Phi_j$ given by {\rm Lemma \ref{lem:2.2}} and
{\rm Lemma A.1}
respectively,
\begin{eqnarray}
&&\Phi^{(\tau)}_j(\sigma_j; U_{b},\tau^2)\big|_{\tau=0}=\Phi_j(\sigma_j;U_{b}) \qquad \mbox{for $1\leq j\leq 4$},\label{eq:4.2}\\
&&\Phi^{(\tau)}(\boldsymbol{\sigma}; U_{b},\tau^2)\big|_{\tau=0}=\Phi(\boldsymbol{\sigma};U_{b}),
\quad \boldsymbol{\sigma}=(\sigma_1, \sigma_2,\sigma_3,\sigma_4),\quad\label{eq:4.3}
\end{eqnarray}
where $\Phi(\boldsymbol{\sigma};U_{b})=\Phi_4(\sigma_4;\Phi_3(\sigma_3;\Phi_2(\sigma_2; \Phi_{1}(\sigma_1;U_{b}))))$.
\end{lemma}

\begin{proof}
By definition,
\begin{eqnarray*}
\Phi^{(\tau)}_{j}(\sigma_{j}; U_{b},\tau^2)\Big|_{\sigma_{j}=0}
=\Phi_{j}(\sigma_{j}; U_{b})\Big|_{\sigma_{j}=0}=U_{b} \qquad\mbox{for $1\leq j\leq 4$}.
\end{eqnarray*}
Using \eqref{eq:3.5a},
we have
\begin{eqnarray*}
\frac{\partial\Phi^{(\tau)}_{j}(\sigma_{j}; U_{b},\tau^2)}{\partial \sigma_{j}}\bigg|_{\sigma_{j}=0, \tau=0}
=\frac{\partial\Phi_{j}(\sigma_{j}; U_{b})}{\partial \sigma_{j}}\bigg|_{\sigma_{j}=0}=\boldsymbol{r}_j(U_{b})
\qquad\,\, \mbox{for $1\leq j\leq 4$}.
\end{eqnarray*}
Thus, by the uniqueness of solutions of the ordinary differential equations,
we can obtain \eqref{eq:4.2} for $|\sigma_{j}|<\min\{\varepsilon_0,\tilde{\varepsilon}_0\}$
and $U_{b}\in \mathcal{O}_{\min\{\varepsilon_0,\tilde{\varepsilon}_0\}}(\underline{U})$.
Moreover, by \eqref{eq:2.10} and \eqref{eq:4.2}, we conclude \eqref{eq:4.3}.
\end{proof}

We now consider the comparison of the Riemann solvers of problem \eqref{eq:2.11} and  problem \eqref{eq:3.6}
with the same \emph{below} and \emph{above} states away from the boundary.

\begin{lemma}\label{lem:4.2}
For $\varepsilon\in (0, \min\{\varepsilon_{0}, \tilde{\varepsilon}_{0}\})$,
let $U_{a}$, $U_{b}\in \mathcal{O}_{\varepsilon}(\underline{U})$ be two constant states defined by \eqref{eq:4.1}.
Let $\alpha_{k}\in \mathcal{S}\cup \mathcal{C}\cup\mathcal{R}$ for some $k\in \{1,2,3,4\}$ and $\alpha_{\mathcal{NP}}\in \mathcal{NP}$
be fronts with strength $\sigma_{\alpha_k}$ and $\sigma_{\alpha_{\mathcal{NP}}}$, respectively, satisfying
\begin{eqnarray}\label{eq:4.4}
U_{a}=\Phi_{k}(\sigma_{\alpha_k}; U_{b})\qquad \mbox{or}\qquad  \sigma_{\alpha_{\mathcal{NP}}}=|U_{a}-U_{b}|.
\end{eqnarray}
Let $\boldsymbol{\beta}=(\beta_1, \beta_2,\beta_3, \beta_4)$ be the physical wave-fronts with strengths
$\boldsymbol{\sigma}^{(\tau)}_{\boldsymbol{\beta}}=(\sigma^{(\tau)}_{\beta_1}, \sigma^{(\tau)}_{\beta_2}, \sigma^{(\tau)}_{\beta_3},\sigma^{(\tau)}_{\beta_4})$  and satisfy
\begin{eqnarray}\label{eq:4.5x}
U_{a}=\Phi^{(\tau)}(\boldsymbol{\sigma}^{(\tau)}_{\boldsymbol{\beta}}; U_{b},\tau^2).
\end{eqnarray}
Then, for the given fixed constant $a_{\infty}$
and for $\tau<\min\{\tau_{0}, \tilde{\tau}_{0}\}$,
\begin{eqnarray}\label{eq:4.6}
\sigma^{(\tau)}_{\beta_{j}}=\delta_{jk}\sigma_{\alpha_{k}}+O(1)|\sigma_{\alpha_{k}}|\tau^{2} \qquad\  \mbox{for $1\leq j, k\leq 4$},
\end{eqnarray}
or
\begin{eqnarray}\label{eq:4.7}
\sigma^{(\tau)}_{\beta_{j}}=O(1)\sigma_{\alpha_{\mathcal{NP}}} \qquad \mbox{for $1\leq j\leq 4$},
\end{eqnarray}
where $\delta_{jk}$ is the Kronecker symbol and the bound of $O(1)$ depends only on $(\underline{U},a_{\infty})$,
but independent of $(\tau, h, \nu)$.
\end{lemma}

\begin{proof}
By Lemma \ref{lem:2.2}, we know that, for $\tau\in (0,\min\{\tau_{0}, \tilde{\tau}_{0}\})$,
\begin{eqnarray*}
\det\bigg(\frac{\partial \Phi^{(\tau)}(\boldsymbol{\sigma}^{(\tau)}_{\boldsymbol{\beta}}; U_{b},\tau^2)}{\partial \boldsymbol{\sigma}^{(\tau)}_{\boldsymbol{\beta}}}\bigg)\bigg|_{\boldsymbol{\sigma}^{(\tau)}_{\boldsymbol{\beta}}=\boldsymbol{0},\ U_{b}=\underline{U}}\neq 0.
\end{eqnarray*}

For $\alpha_{k}\in \mathcal{S}\cup \mathcal{C}\cup\mathcal{R}$, using conditions \eqref{eq:4.4}--\eqref{eq:4.5x}, we have
\begin{eqnarray*}
\Phi^{(\tau)}(\boldsymbol{\sigma}^{(\tau)}_{\boldsymbol{\beta}}; U_{b},\tau^2)=\Phi_{k}(\sigma_{\alpha_k};U_{b}), \qquad
\boldsymbol{\sigma}^{(\tau)}_{\boldsymbol{\beta}}=(\sigma^{(\tau)}_{\beta_{1}}, \sigma^{(\tau)}_{\beta_{2}},\sigma^{(\tau)}_{\beta_{3}},\sigma^{(\tau)}_{\beta_{4}}).
\end{eqnarray*}
Thus, it follows from the implicit function theorem that the above equations admit a unique solution
$\boldsymbol{\sigma}^{(\tau)}_{\boldsymbol{\beta}}(\sigma_{\alpha_k}, \tau^2)$
for $U_{b}\in \mathcal{O}_{\min\{\varepsilon_{0}, \tilde{\varepsilon}_{0}\}}(\underline{U})$ and $|\sigma_{\alpha_{k}}|<\min\{\varepsilon_{0}, \tilde{\varepsilon}_{0}\}$.

To estimate $\sigma^{(\tau)}_{\beta_j}(\sigma_{\alpha_k}, \tau^2)$ for $1\leq j\leq 4$, by Lemma \ref{lem:4.1},
\begin{eqnarray*}
\Phi^{(\tau)}(\boldsymbol{\sigma}^{(\tau)}_{\boldsymbol{\beta}}\big|_{\tau=0}; U_{b},\tau^2)\big|_{\tau=0}
=\Phi(\boldsymbol{\sigma}^{(\tau)}_{\boldsymbol{\beta}}|_{\tau=0}; U_{b})=\Phi_{k}(\sigma_{\alpha_k}; U_{b}).
\end{eqnarray*}
Thus, it follows from $\sigma^{(\tau)}_{\beta_j}(0, \tau^2)=0$ and $\sigma^{(\tau)}_{\beta_j}(0, \tau^2)\big|_{\tau=0}=0$ that
\begin{eqnarray*}
\sigma^{(\tau)}_{\beta_j}(\sigma_{\alpha_k}, \tau^2)\big|_{\tau=0}=\delta_{jk}\sigma_{\alpha_k}.
\end{eqnarray*}

Then, by the Taylor formula, we have
\begin{align*}
\sigma^{(\tau)}_{\beta_j}(\sigma_{\alpha_k}, \tau^2)&=\sigma^{(\tau)}_{\beta_j}(\sigma_{\alpha_k}, \tau^2)\big|_{\tau=0}
+\sigma^{(\tau)}_{\beta_j}(0, \tau^2)-\sigma^{(\tau)}_{\beta_j}(0, \tau^2)\big|_{\tau=0}+O(1)|\sigma_{\alpha_k}|\tau^2\\[2pt]
&=\delta_{jk}\sigma_{\alpha_k}+O(1)|\sigma_{\alpha_k}|\tau^2,
\end{align*}
which is exactly
estimate \eqref{eq:4.6}.

Next,
for $\alpha_{\mathcal{NP}}\in\mathcal{NP}$, using estimate \eqref{eq:2.15} in Lemma \ref{lem:2.2} and \eqref{eq:4.4} for $\sigma_{\alpha_{\mathcal{NP}}}$,
we can follow the same argument above to obtain estimate \eqref{eq:4.7}.
\end{proof}

Based on Lemma \ref{lem:4.2}, we now consider the error estimates between the trajectory of $\mathcal{P}^{(\tau)}_{h}(\xi+s,\xi)(U_{h, \nu}(\xi, y))$
and $U_{h,\nu}(\xi+s, y)$ away from the boundary for $s>0$ sufficiently small.

\begin{proposition}\label{prop:4.1}
Let $\mathcal{P}^{(\tau)}_{h}$ be the Lipschtiz continuous map as given in {\rm Proposition \ref{prop:2.2}} for \emph{Problem I}.
Let $U_{h,\nu}(\xi,y)$ be the approximate solution constructed in {\rm Appendix A} for system \eqref{eq:1.14}
with $U_a$, $U_{b}\in \mathcal{O}_{\min\{\varepsilon_0,\tilde{\varepsilon}_0\}}(\underline{U})$
defined by \eqref{eq:4.1}.
Assume that, near point $(\xi,y_I)$, $U_{h,\nu}(\xi,y)$ is of the following form{\rm :}
\begin{equation}\label{eq:4.9}
U_{h,\nu}(\xi+s,y)=
\begin{cases}
U_{a} \quad  &\mbox{for $y>y_{I}+\dot{y}_{\alpha_k}s$},\\[2pt]
U_{b} \quad  &\mbox{for $y<y_{I}+\dot{y}_{\alpha_k}s$},
\end{cases}
\end{equation}
where $|\dot{y}_{\alpha_k}|<\hat{\lambda}$ if $1\leq k\leq 4$ and $\dot{y}_{\alpha_{k}}=\hat{\lambda}$ if $k=\mathcal{NP}$,
with the fixed constant $\hat{\lambda}$ satisfying
$\hat{\lambda}>\max_{1\leq j\leq 4}\big\{\lambda^{(\tau)}_{j}(U^{(\tau)},\tau^{2}),\lambda_{j}(U)\big\}$ for $U^{(\tau)}, U\in \mathcal{O}_{\min\{\varepsilon_0,\tilde{\varepsilon}_0\}}(\underline{U})$, and $\tau \in(0,\min\{\tau_{0},\tilde{\tau}_0\})$.
For the fixed given hypersonic similarity parameter $a_{\infty}$, if $s>0$ is sufficiently small,
the following estimates hold{\rm :}

\begin{enumerate}
\item[\rm (i)] if $U_{a}=\Phi_{k}(\sigma_{\alpha_k}; U_{b})$ for $\alpha_k\in \mathcal{S}_k$, $k=1,4$,
and $|\dot{y}_{\alpha_k}-\dot{\mathcal{S}}_{k}(\sigma_{\alpha_k})|<2^{-\nu}$, then
\begin{equation}\label{eq:4.10}
\int^{y_I+\hat{\lambda}s}_{y_I-\hat{\lambda}s}|\mathcal{P}^{(\tau)}_{h}(\xi+s,\xi)(U_{h,\nu}(\xi,y))
-U_{h,\nu}(\xi+s,y)|\,{\rm d}y
\leq {C_{3,1}}\big(\tau^{2}+2^{-\nu}\big)|\sigma_{\alpha_k}|s,
\end{equation}
where $\dot{\mathcal{S}}_{k}(\sigma_{\alpha_k})$ is the speed of the $k^{\rm th}$ shock-front $\alpha_k${\rm ;}

\smallskip
\item[\rm (ii)] if $U_{a}=\Phi_{k}(\sigma_{\alpha_k}; U_{b})$ for $\alpha_k\in \mathcal{R}_k$, $k=1,4$,
and $|\dot{y}_{\alpha_k}-\lambda_{k}(U_a)|<2^{-\nu}$, then
\begin{align}
&\int^{y_I+\hat{\lambda}s}_{y_I-\hat{\lambda}s}|\mathcal{P}^{(\tau)}_{h}(\xi+s,\xi)(U_{h,\nu}(\xi,y))
-U_{h,\nu}(\xi+s,y)|\,{\rm d}y\nonumber\\[2pt]
&\leq {C_{3,2}}\big(\tau^{2}+\nu^{-1}+2^{-\nu}\big)|\sigma_{\alpha_k}|s, \label{eq:4.11}
\end{align}
where $\lambda_{k}(U_a)$ is the speed of the $k^{\rm th}$ rarefaction-front $\alpha_k${\rm ;}

\smallskip
\item[\rm (iii)] if $U_{a}=\Phi_{k}(\sigma_{\alpha_k}; U_{b})$ for $\alpha_k\in \mathcal{C}_k$, $k=2,3$,
and $|\dot{y}_{\alpha_k}-\lambda_{k}(U_b)|<2^{-\nu}$, then
\begin{equation}\label{eq:4.12}
\int^{y_I+\hat{\lambda}s}_{y_I-\hat{\lambda}s}|\mathcal{P}^{(\tau)}_{h}(\xi+s,\xi)(U_{h,\nu}(\xi,y))
-U_{h,\nu}(\xi+s,y)|\,{\rm d}y
\leq {C_{3,3}}\big(\tau^{2}+2^{-\nu}\big)|\sigma_{\alpha_k}|s,
\end{equation}
where $\lambda_{k}(U_b)$ is the speed of the $k^{\rm th}$ contact discontinuity $\alpha_k${\rm ;}

\smallskip
\item[\rm (iv)] if $\sigma_{\alpha_{\mathcal{NP}}}=|U_{a}-U_{b}|$ with $\alpha_{\mathcal{NP}}\in \mathcal{NP}$, then
\begin{equation}\label{eq:4.13}
\int^{y_I+\hat{\lambda}s}_{y_I-\hat{\lambda}s}|\mathcal{P}^{(\tau)}_{h}(\xi+s,\xi)(U_{h,\nu}(\xi,y))
-U_{h,\nu}(\xi+s,y)|\,{\rm d}y
\leq {C_{3,4}}|\sigma_{\alpha_{\mathcal{NP}}}|s,
\end{equation}
\end{enumerate}
where {all the constants $C_{3,k}>0$}, $1\leq k\leq4$, depend only on $(\underline{U},a_{\infty})$.
\end{proposition}

\begin{proof}  We divide the proof into four steps.

\smallskip
1. We only consider the case: $k=1$, since the case: $k=4$ can be dealt with in the same way.
From Proposition \ref{prop:2.2}, we know that
\begin{equation}\label{eq:4.14}
\Phi^{(\tau)}(\boldsymbol{\sigma}^{(\tau)}_{\boldsymbol{\beta}}; U_{b},\tau^2)=\Phi_{1}(\sigma_{\alpha_1}; U_{b}),\qquad
\boldsymbol{\sigma}^{(\tau)}_{\boldsymbol{\beta}}=(\sigma^{(\tau)}_{\beta_1}, \sigma^{(\tau)}_{\beta_2}, \sigma^{(\tau)}_{\beta_3}, \sigma^{(\tau)}_{\beta_4}).
\end{equation}
Then, by Lemma \ref{lem:4.2}, we have estimate \eqref{eq:4.6} for $\sigma^{(\tau)}_{\beta_j}$ with $1\leq j\leq 4$ and $k=1$.
Then it follows from $\sigma_{\alpha_1}<0$ that $\sigma^{(\tau)}_{\beta_1}<0$,
if $\tau$ is sufficiently small. This implies that $\beta_1\in \mathcal{S}^{(\tau)}_{1}$.

\begin{figure}[ht]
\begin{center}
\begin{tikzpicture}[scale=0.6]
\draw [line width=0.05cm](-4,4.5)--(-4,-4.5);
\draw [line width=0.05cm](0.5,4.5)--(0.5,-4.5);

\draw [thin](-4,0)--(0.5, 3.8);
\draw [thick](-4,0)--(0.5, 2.5);
\draw [thick][dashed](-4,0)--(0.5, 0.6);
\draw [thick](-4,0)--(0.5, -2.5);
\draw [thick][red](-4,0)--(0.5, -1.2);
\draw [thin](-4,0)--(0.5, -3.8);

\draw [thin][<->](-3.9,4.2)--(0.4,4.2);

\node at (2.6, 2) {$$};
\node at (-1.8, 3.9){$s$};
\node at (-4.9, 0) {$(\xi, y_I)$};
\node at (-4, -4.9) {$x=\xi$};
\node at (0, -4.9) {$x=\xi+s$};

\node at (2.4, 3.9) {$y=y_I+\hat{\lambda}s$};
\node at (1.2, 2.4) {$\beta_{4}$};
\node at (1.4, 0.5) {$\beta_{2(3)}$};
\node at (1.2, -1.2) {$\alpha_1$};
\node at (1.2, -2.6) {$\beta_{1}$};
\node at (2.4, -3.8) {$y=y_I-\hat{\lambda}s$};

\node at (-1.4, 2.9){$U_{a}$};
\node at (-1.0, 0.9){$U^{(\tau)}_{m_{2(3)}}$};
\node at (-1.0, -0.2){$U^{(\tau)}_{m_{1}}$};
\node at (-1.2, -3.2){$U_{b}$};

\end{tikzpicture}
\end{center}
\caption{Comparison of the Riemann solutions for the case: $\alpha_1\in \mathcal{S}_1$}\label{fig4.1}
\end{figure}
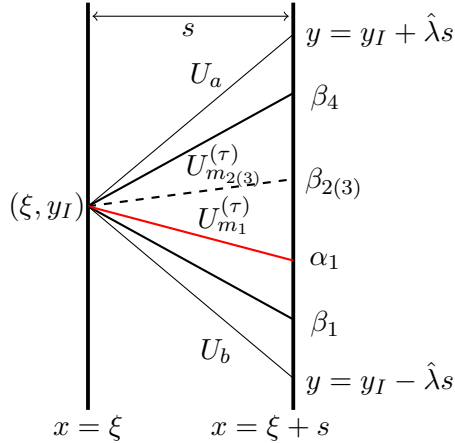

We denote $U^{(\tau)}_{m_1}$, $U^{(\tau)}_{m_2}$, and $U^{(\tau)}_{m_3}$ as the intermediate states of the Riemann
solutions $\mathcal{P}^{(\tau)}_{h}(\xi+s,\xi)(U_{h,\nu}(\xi,y))$ with wave strength $\boldsymbol{\sigma}^{(\tau)}_{\boldsymbol{\beta}}$.
As shown in Fig. \ref{fig4.1}, let $\dot{\mathcal{S}}_{1}(\sigma_{\alpha_1})$ and $\dot{\mathcal{S}}^{(\tau)}_{1}(\sigma^{(\tau)}_{\beta_{1}}, \tau^{2})$
be the speeds of the shock-fronts $\alpha_1$ and $\beta_{1}$, respectively.
From Lemma \ref{lem:2.2} and the results in
\cite[Chapter 17]{smoller}, we have
\begin{align*}
&\dot{\mathcal{S}}_{1}(\sigma_{\alpha_1})\Big|_{\sigma_{\alpha_1}=0}=\lambda_{1}(U_{b}), \quad\,
\frac{\partial\dot{\mathcal{S}}_{1}(\sigma_{\alpha_1}) }{\partial \sigma_{\alpha_1}}\bigg|_{\sigma_{\alpha_1}=0}
=\frac{1}{2}\nabla_{U}\lambda_{1}(U_{b})\cdot \boldsymbol{r}_{1}(U_{b})=1,\\
&\dot{\mathcal{S}}^{(\tau)}_{1}(\sigma^{(\tau)}_{\beta_1},\tau^{2})\Big|_{\sigma^{(\tau)}_{\beta_1}=0}=\lambda^{(\tau)}_{1}(U_{b},\tau^{2}),\\
&\frac{\partial\dot{\mathcal{S}}^{(\tau)}_{1}(\sigma^{(\tau)}_{\beta_1}) }{\partial \sigma^{(\tau)}_{\beta_1}}\bigg|_{\sigma^{(\tau)}_{\beta_1}=0}
=\frac{1}{2}\nabla_{U^{(\tau)}}\lambda^{(\tau)}_{1}(U_{b},\tau^{2})\cdot \boldsymbol{r}^{(\tau)}_{1}(U_{b},\tau^{2})=1.
\end{align*}

Thus,  using \eqref{eq:3.5a} and Lemma \ref{lem:4.2}, we obtain
\begin{align}\label{eq:4.5}
&\dot{\mathcal{S}}^{(\tau)}_{1}(\sigma^{(\tau)}_{\beta_1},\tau^{2})-\dot{\mathcal{S}}_{1}(\sigma_{\alpha_1})\nonumber\\[2pt]
&=\dot{\mathcal{S}}^{(\tau)}_{1}(\sigma^{(\tau)}_{\beta_1},\tau^{2})\Big|_{\sigma^{(\tau)}_{\beta_1}=0}
+\frac{\partial\dot{\mathcal{S}}^{(\tau)}_{1}(\sigma^{(\tau)}_{\beta_1}) }{\partial \sigma^{(\tau)}_{\beta_1}}\bigg|_{\sigma^{(\tau)}_{\beta_1}=0}
\sigma^{(\tau)}_{\beta_1}+O(1)(\sigma^{(\tau)}_{\beta_1})^{2}\nonumber\\[2pt]
&\quad\ -\bigg(\dot{\mathcal{S}}_{1}(\sigma_{\alpha_1})\Big|_{\sigma_{\alpha_1}=0}+\frac{\partial\dot{\mathcal{S}}_{1}(\sigma_{\alpha_1}) }{\partial \sigma_{\alpha_1}}\bigg|_{\sigma_{\alpha_1}=0}\sigma_{\alpha_1}+O(1)\sigma_{\alpha_1}^{2}\bigg)\nonumber\\[2pt]
&=\lambda^{(\tau)}_{1}(U_{b},\tau^{2})-\lambda_{1}(U_{b})+\frac{1}{2}\big(\sigma^{(\tau)}_{\beta_1}-\sigma_{\alpha_1}\big)
+O(1)\big(\sigma^{(\tau)}_{\beta_1}-\sigma_{\alpha_1}\big)^{2}+O(1)\tau^2(\sigma^{(\tau)}_{\beta_1})^{2}\nonumber\\[2pt]
&=O(1)\big(1+|\sigma_{\alpha_1}|\big)\tau^{2}.
\end{align}

Now we begin to derive estimate \eqref{eq:4.10}. A direct computation shows that
\begin{eqnarray}\label{3.15x}
&&\int^{y_I+\hat{\lambda}s}_{y_I-\hat{\lambda}s}\big|\mathcal{P}^{(\tau)}_{h}(\xi+s,\xi)(U_{h,\nu}(\xi,y))
-U_{h,\nu}(\xi+s,y)|\,{\rm d}y\nonumber\\[2pt]
&& =\int^{\min\{y_I+\dot{\mathcal{S}}^{(\tau)}_{1}(\sigma^{(\tau)}_{\beta_1},\tau^{2})s,y_I+\dot{y}_{\alpha_1}s\}}
_{y_I-\hat{\lambda}s}\big|\mathcal{P}^{(\tau)}_{h}(\xi+s,\tilde{x})(U_{h,\nu}(\xi,y))
-U_{h,\nu}(\xi+s,y)\big|\,{\rm d}y\nonumber\\[2pt]
&&\quad\ +\int^{\max\{y_I+\dot{\mathcal{S}}^{(\tau)}_{1}(\sigma^{(\tau)}_{\beta_1},\tau^{2})s,y_I+\dot{y}_{\alpha_1}s\}}
_{\min\{y_I+\dot{\mathcal{S}}^{(\tau)}_{1}(\sigma^{(\tau)}_{\beta_1},\tau^{2})s,y_I+\dot{y}_{\alpha_1}s\}}
\big|\mathcal{P}^{(\tau)}_{h}(\xi+s,\xi)(U_{h,\nu}(\xi,y))-U_{h,\nu}(\xi+s,y)\big|\,{\rm d}y\nonumber\\[2pt]
&&\quad\ +\int^{y_I+\hat{\lambda}s}
_{\max\{y_I+\dot{\mathcal{S}}^{(\tau)}_{1}(\sigma^{(\tau)}_{\beta_1},\tau^{2})s,y_I+\dot{y}_{\alpha_1}s\}}
\big|\mathcal{P}^{(\tau)}_{h}(\xi+s,\xi)(U_{h,\nu}(\xi,y))-U_{h,\nu}(\xi+s,y)\big|\,{\rm d}y\nonumber\\[2pt]
&&\doteq I_{\mathcal{S}_1,1}+I_{\mathcal{S}_1,2}+I_{\mathcal{S}_1,3}.
\end{eqnarray}

It is clear that $I_{\mathcal{S}_1, 1}=0$.
For $I_{\mathcal{S}_1, 2}$, by Lemmas \ref{lem:2.2} and \ref{lem:4.2}, we have
\begin{align*}
I_{\mathcal{S}_1, 2}
&\leq C \big|\dot{\mathcal{S}}^{(\tau)}(\sigma^{(\tau)}_{\beta_1},\tau^{2})-\dot{y}_{\alpha_1}\big|\big|U^{(\tau)}_{m_{1}}-U_{b}\big|s\\[2pt]
&\leq C\big(|\dot{\mathcal{S}}^{(\tau)}(\sigma^{(\tau)}_{\beta_1},\tau^{2})-\dot{\mathcal{S}}(\sigma_{\alpha})|+2^{-\nu}\big)
\big|\sigma^{(\tau)}_{\beta_1}\big|s\\[2pt]
&\leq C\big((1+|\sigma_{\alpha_1}|)\tau^{2}+2^{-\nu}\big)\big(1+ C\,\tau^{2}\big)|\sigma_{\alpha_1}|
\\[2pt]
&\leq C\big(\tau^2+2^{-\nu}\big)|\sigma_{\alpha_1}|\tau^{2}s,
\end{align*}
or
\begin{align*}
I_{\mathcal{S}_1, 2}
&\leq C\big|\dot{\mathcal{S}}^{(\tau)}(\sigma^{(\tau)}_{\beta_1},\tau^{2})-\dot{y}_{\alpha_1}\big|\big|U^{(\tau)}_{a}-U_{b}\big|s\\[2pt]
&\leq C\big((1+|\sigma_{\alpha_1}|)\tau^{2}+2^{-\nu}\big)|\sigma_{\alpha_1}|\\[2pt]
&\leq C\big(\tau^2+2^{-\nu}\big)|\sigma_{\alpha_1}|\tau^{2}s.
\end{align*}

For $I_{\mathcal{S}_1,3}$, by Lemma \ref{lem:4.1} again, for both cases: $\beta_4\in \mathcal{S}^{(\tau)}_{4}$ and $\beta_4\in \mathcal{R}^{(\tau)}_{4}$,
we have
\begin{align*}
I_{\mathcal{S}_1, 3}&\leq C\bigg(\sum_{j=1,2}\big|U^{(\tau)}_{m_{j}}-U^{(\tau)}_{m_{j+1}}\big|+\big|U_{a}-U^{(\tau)}_{m_{3}}\big|\bigg)s\\[2pt]
&\leq C\bigg(\sum^{3}_{j=2}|\sigma^{(\tau)}_{\beta_{j}}|\bigg)s
\\[2pt]
&
\leq C|\sigma_{\alpha_1}|\tau^{2}s.
\end{align*}

Finally, combining the estimates of $I_{\mathcal{S}_1, 1}$, $I_{\mathcal{S}_1, 2}$, and $I_{\mathcal{S}_1, 3}$ together,
we can choose a constant  {$C_{3,1}>0$} depending only on $(\underline{U}, a_{\infty})$ such that, for $s>0$ sufficiently small,
estimate \eqref{eq:4.11} holds.

\smallskip
2.
Similarly as done for Step 1,
without loss of generality, we only consider the case: $k=1$.
The solutions of the Riemann solver, $\mathcal{P}^{(\tau)}_{h}(\xi+s,\xi)(U_{h,\nu}(\xi,y))$, satisfy equation \eqref{eq:4.14}.
Following Lemma \ref{lem:4.2}, we know that equation \eqref{eq:4.14} admits a unique solution $\boldsymbol{\sigma}^{(\tau)}_{\boldsymbol{\beta}}=(\sigma^{(\tau)}_{\beta_1},\sigma^{(\tau)}_{\beta_2},\sigma^{(\tau)}_{\beta_3},
\sigma^{(\tau)}_{\beta_4})$, which satisfies estimates \eqref{eq:4.6} for $k=1$ and $1\leq j\leq 4$.
Since $\sigma_{\alpha_1}>0$ and $\sigma^{(\tau)}_{\beta_1}>0$, which implies that $\beta_1 \in \mathcal{R}^{(\tau)}_{1}$.
Denoted by $U^{(\tau)}_{m_j}, j=1,2,3$, the intermediate states of $\mathcal{P}^{(\tau)}_{h}(\xi+s,\xi)(U_{h,\nu}(\xi,y))$
that satisfy relations \eqref{eq:2.12} of the following forms:
{\small
\begin{equation}\label{eq:4.16a}
\mathcal{P}^{(\tau)}_{h}(\xi+s,\xi)(U_{h,\nu}(\xi,y))=\left\{
\begin{array}{ll}
\hspace{-4pt}U_{a}, \  &\zeta\in \big[\dot{\mathcal{S}}^{(\tau)}_{4}(\sigma^{(\tau)}_{\beta_4},\tau^{2}),\hat{\lambda}\big),\\[5pt]
\hspace{-4pt}U^{(\tau)}_{m_3}, \  &\zeta\in \big[\lambda^{(\tau)}_{3}(U^{(\tau)}_{m_3},\tau^{2}),\dot{\mathcal{S}}^{(\tau)}_{4}(\sigma^{(\tau)}_{\beta_4},\tau^{2})\big),\\[5pt]
\hspace{-4pt}U^{(\tau)}_{m_1}, \  &\zeta\in \big[\lambda^{(\tau)}_{1}(U^{(\tau)}_{m_1},\tau^{2}),\lambda^{(\tau)}_{3}(U^{(\tau)}_{m_3},\tau^{2})\big),\\[5pt]
\hspace{-4pt}\Phi^{(\tau)}_{1}(\sigma^{(\tau)}_{\beta_1}(\zeta); U_{b}, \tau^2),
\hspace{-5pt}&\zeta\in \big[\lambda^{(\tau)}_{1}(U_{b},\tau^{2}), \lambda^{(\tau)}_{1}(U^{(\tau)}_{m_1},\tau^{2}) \big),\\[5pt]
\hspace{-4pt}U_{b}, \  &\zeta\in(-\hat{\lambda},\lambda^{(\tau)}_{1}(U_{b},\tau^{2})),
\end{array}
\right.
\end{equation}
}
or
{\small
\begin{equation}\label{eq:4.16b}
\mathcal{P}^{(\tau)}_{h}(\xi+s,\xi)(U_{h,\nu}(\xi,y))=\left\{
\begin{array}{ll}
\hspace{-4pt}U_{a}, \  &\zeta\in \big[\lambda^{(\tau)}_{4}(U_{a},\tau^{2}),\hat{\lambda}\big),\\[5pt]
\hspace{-4pt}\Phi^{(\tau)}_{1}(\sigma^{(\tau)}_{\beta_4}(\zeta); U_{b}, \tau^2),
 \hspace{-5pt}&\zeta\in \big[\lambda^{(\tau)}_{4}(U^{(\tau)}_{m_3},\tau^{2}),\,\lambda^{(\tau)}_{4}(U_{a},\tau^{2})\big),\\[5pt]
\hspace{-4pt}U^{(\tau)}_{m_3}, \  &\zeta\in \big[\lambda^{(\tau)}_{3}(U^{(\tau)}_{m_3},\tau^{2}),\lambda^{(\tau)}_{4}(U^{(\tau)}_{m_3},\tau^{2}) \big),\\[5pt]
\hspace{-4pt}U^{(\tau)}_{m_1}, \  &\zeta\in \big[\lambda^{(\tau)}_{1}(U^{(\tau)}_{m_1},\tau^{2}),\lambda^{(\tau)}_{3}(U^{(\tau)}_{m_3},\tau^{2})\big),\\[5pt]
\Phi^{(\tau)}_{1}(\sigma^{(\tau)}_{\beta_1}(\zeta); U_{b}, \tau^2),
 \hspace{-5pt}&\zeta\in \big[\lambda^{(\tau)}_{1}(U_{b},\tau^{2}),\lambda^{(\tau)}_{1}(U^{(\tau)}_{m_1},\tau^{2}) \big),\\[5pt]
\hspace{-4pt}U_{b}, \  &\zeta\in(-\hat{\lambda},\lambda^{(\tau)}_{1}(U_{b},\,\tau^{2})),
\end{array}
\right.
\end{equation}
}
where $\zeta=\frac{y-y_{I}}{s}$.
Moreover, $\sigma^{(\tau)}_{\beta_1}(\zeta)$ and $\sigma^{(\tau)}_{\beta_4}(\zeta)$
satisfy
\begin{align}\label{eq:4.16c}
\sigma^{(\tau)}_{\beta_1}(\zeta_{0})=0,\quad
\sigma^{(\tau)}_{\beta_1}(\zeta_1)=\sigma^{(\tau)}_{\beta_1},\quad \sigma^{(\tau)}_{\beta_4}(\zeta_{2})=0,\quad
\sigma^{(\tau)}_{\beta_4}(\zeta_3)=\sigma^{(\tau)}_{\beta_4},
\end{align}
where $\zeta_0=\lambda^{(\tau)}_{1}(U_{b},\tau^{2})$, $\zeta_1=\lambda^{(\tau)}_{1}(U^{(\tau)}_{m_1},\tau^{2})$,
$\zeta_2=\lambda^{(\tau)}_{4}(U^{(\tau)}_{m_3},\tau^{2})$, and $\zeta_3=\lambda^{(\tau)}_{4}(U_{a},\tau^{2})$.
By the results in \cite[Chapter 17]{smoller}, we know that the speed, $\lambda_{1}(U_{a})$, of front $\alpha_1$ satisfies
\begin{eqnarray*}
\lambda_{1}(U_{a})=\lambda_{1}(U_{b})+\sigma_{\alpha_1}+O(1)|\sigma_{\alpha_1}|^{2},
\end{eqnarray*}
and the speed, $\lambda^{(\tau)}_{1}(U^{(\tau)}_{m_1}, \tau^{2})$, of front $\beta_1$ satisfies
\begin{eqnarray*}
\lambda^{(\tau)}_{1}(U^{(\tau)}_{m_1}, \tau^{2})=\lambda^{(\tau)}_{1}(U_{b},\tau^{2})+\sigma^{(\tau)}_{\beta_1}+O(1)|\sigma^{(\tau)}_{\beta_1}|^{2},
\end{eqnarray*}
where we have used the relation: $U^{(\tau)}_{m_1}=\Phi^{(\tau)}_{1}(\sigma^{(\tau)}_{\beta_{1}}; U_{b},\tau^{2})$.

Then,  using \eqref{eq:3.5a} and Lemma \ref{lem:4.2}, we obtain
\begin{align}
\lambda^{(\tau)}_{1}(U^{(\tau)}_{m_1}, \tau^{2})-\lambda_{1}(U_{a})
&=\lambda^{(\tau)}_{1}(U_{b},\tau^{2})-\lambda_{1}(U_{b})+O(1)|\sigma_{\alpha_1}|\tau^{2}+O(1)|\sigma^{(\tau)}_{\beta_1}|^{2}\tau^{2}
\nonumber\\[2pt]
&=O(1)\big(1+|\sigma_{\alpha_1}|\big)\tau^{2},\label{eq:4.16}\\[4pt]
\lambda_{1}(U_{a})-\lambda^{(\tau)}_{1}(U_{b}, \tau^{2})
&=\lambda_{1}(U_{b})-\lambda^{(\tau)}_{1}(U_{b},\tau^{2})+\sigma_{\alpha_1}+O(1)|\sigma_{\alpha_1}|^{2}\nonumber\\[2pt]
&=O(1)\big(|\sigma_{\alpha_1}|+\tau^{2}\big).\label{eq:4.17}
\end{align}

Now we can follow the argument for \eqref{3.15x} in Step 1 to rewrite
$$
\int^{y_I+\hat{\lambda}s}_{y_I-\hat{\lambda}s}\big|\mathcal{P}^{(\tau)}_{h}(\xi+s, \xi)(U_{h,\nu}(\xi,y))-U_{h,\nu}(\xi+s,y)\big|\,{\rm d}y
=I_{\mathcal{R}_1,1}+I_{\mathcal{R}_1,2}+I_{\mathcal{R}_1,3}.
$$

For the first term, it follows from \eqref{eq:4.16a}--\eqref{eq:4.16b} that
\begin{align*}
I_{\mathcal{R}_1,1}&\doteq\int^{\min\{y_I+\lambda^{(\tau)}_{1}(U_{b},\tau^{2})s, y_I+\dot{y}_{\alpha_1}s\}}_{y_I-\hat{\lambda}s}
\big|\mathcal{P}^{(\tau)}_{h}(\xi+s,\xi)(U_{h,\nu}(\xi,y))-U_{h,\nu}(\xi+s,y)\big|\,{\rm d}y\\
&=0.
\end{align*}
For the third term, for either $\beta_4 \in \mathcal{S}^{(\tau)}_{4}$ or $\beta_4 \in \mathcal{R}^{(\tau)}_{4}$, we have
\begin{align*}
I_{\mathcal{R}_1,3}&\doteq\int^{y_I+\hat{\lambda}s}_{\max\{y_I+\lambda^{(\tau)}_{1}(U^{(\tau)}_{m_1},\tau^{2})s, y_I+\dot{y}_{\alpha_1}s\}}
\big|\mathcal{P}^{(\tau)}_{h}(\xi+s,\xi)(U_{h,\nu}(\xi,y))-U_{h,\nu}(\xi+s,y)\big|\,{\rm d}y\\[2pt]
&\leq C\bigg(\sum_{j=1,2}|U^{(\tau)}_{m_{j}}-U^{(\tau)}_{m_{j+1}}|+|U^{(\tau)}_{m_3}-U_{a}|\bigg)s\\[2pt]
&\leq C\bigg(\sum^{4}_{j=2}|\sigma^{(\tau)}_{\beta_j}|\bigg)s\\[2pt]
&\leq C|\sigma_{\alpha_1}|\tau^{2}s.
\end{align*}

Finally, for the second term, we have
\begin{eqnarray*}
I_{\mathcal{R}_1,2}\doteq\int^{\max\{y_I+\lambda^{(\tau)}_{1}(U^{(\tau)}_{m_1},\tau^{2})s, y_I+\dot{y}_{\alpha_1}s\}}_{\min\{y_I+\lambda^{(\tau)}_{1}(U_{b},\tau^{2})s, y_I+\dot{y}_{\alpha_1}s\}}
\big|\mathcal{P}^{(\tau)}_{h}(\xi+s, \xi)(U_{h,\nu}(\xi,y))-U_{h, \nu}(\xi+s, y)\big|\,{\rm d}y.
\end{eqnarray*}

Then we estimate $I_{\mathcal{R}_1,2}$ by three subcases based on the different location of $\alpha_1$.

\begin{figure}[ht]
\begin{center}
\begin{tikzpicture}[scale=0.6]
\draw [line width=0.05cm](-4,4.5)--(-4,-4.5);
\draw [line width=0.05cm](0.5,4.5)--(0.5,-4.5);

\draw [thin](-4,0)--(0.5, 3.8);
\draw [thick](-4,0)--(0.5, 2.5);
\draw [thick][dashed](-4,0)--(0.5, 0.6);
\draw [thick][blue](-4,0)--(0.5, -2.0);
\draw [thick][blue](-4,0)--(0.5, -2.3);
\draw [thick][blue](-4,0)--(0.5, -2.6);
\draw [thick][blue](-4,0)--(0.5, -2.9);
\draw [thick][red](-4,0)--(0.5, -1.2);
\draw [thin](-4,0)--(0.5, -3.8);

\draw [thin][<->](-3.9,4.2)--(0.4,4.2);

\node at (2.6, 2) {$$};
\node at (-1.8, 3.9){$s$};
\node at (-4.7, 0) {$(\xi, y_I)$};
\node at (-4, -4.9) {$x=\xi$};
\node at (0, -4.9) {$x=\xi+s$};

\node at (2.4, 3.9) {$y=y_I+\hat{\lambda}s$};
\node at (1.2, 2.4) {$\beta_{4}$};
\node at (1.4, 0.5) {$\beta_{2(3)}$};
\node at (1.1, -1.2) {$\alpha_1$};
\node at (1.2, -2.6) {$\beta_{1}$};
\node at (2.4, -3.8) {$y=y_I-\hat{\lambda}s$};

\node at (-1.4, 2.9){$U_{a}$};
\node at (-1.2, 0.9){$U^{(\tau)}_{m_{2(3)}}$};
\node at (-1.2, -0.2){$U^{(\tau)}_{m_{1}}$};
\node at (-1.2, -3.2){$U_{b}$};

\end{tikzpicture}
\end{center}
\caption{Subcase 1: $\alpha\in \mathcal{R}_1$ and $\dot{y}_{\alpha_1}\geq \lambda^{(\tau)}_{1}(U^{(\tau)}_{m_1},\tau^{2})$}\label{fig4.3}
\end{figure}
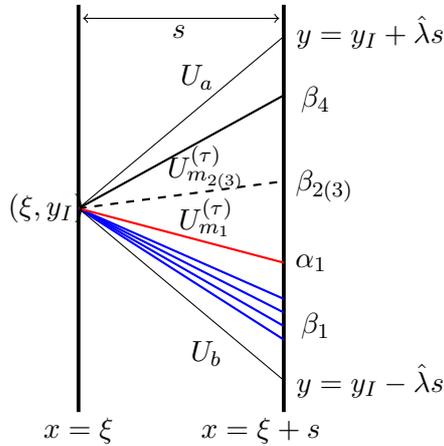

\emph{Subcase 1}:\ $\dot{y}_{\alpha_1}\geq \lambda^{(\tau)}_{1}(U^{(\tau)}_{m_1},\tau^{2})$.
As shown in Fig. \ref{fig4.3}, we can decompose $I_{\mathcal{R}_1,2}$ into two terms:
\begin{align*}
I_{\mathcal{R}_1,2}&=\int^{y_I+\dot{y}_{\alpha_1}s}_{y_I+\lambda^{(\tau)}_{1}(U_{b},\tau^{2})s}
\big|\mathcal{P}^{(\tau)}_{h}(\xi+s,\xi)(U_{h,\nu}(\xi,y))-U_{h, \nu}(\xi+s, y)\big|\,{\rm d}y\\[2pt]
&=\int^{y_I+\lambda^{(\tau)}_{1}(U^{(\tau)}_{m_1},\tau^2)s}_{y_I+\lambda^{(\tau)}_{1}(U_{b},\tau^{2})s}
\big|\mathcal{P}^{(\tau)}_{h}(\xi+s,\xi)(U_{h,\nu}(\xi,y))-U_{h, \nu}(\xi+s, y)\big|\,{\rm d}y\\[2pt]
&\quad\, +\int^{y_I+\dot{y}_{\alpha_1}s}_{y_I+\lambda^{(\tau)}_{1}(U^{(\tau)}_{m_1},\tau^{2})s}
\big|\mathcal{P}^{(\tau)}_{h}(\xi+s,\xi)(U_{h,\nu}(\xi,y))-U_{h, \nu}(\xi+s, y)\big|\,{\rm d}y.
\end{align*}
For the first term, by \eqref{eq:4.16a}--\eqref{eq:4.16c}
together with Proposition A.1
and Lemma \ref{lem:4.2}, we obtain
\begin{eqnarray*}
&&\int^{y_I+\lambda^{(\tau)}_{1}(U^{(\tau)}_{m_1},\tau^2)s}_{y_I+\lambda^{(\tau)}_{1}(U_{b},\tau^{2})s}
\big|\mathcal{P}^{(\tau)}_{h}(\xi+s,\xi)(U_{h,\nu}(\xi,y))-U_{h, \nu}(\xi+s, y)\big|\,{\rm d}y\\[2pt]
&&\, \leq \big|\lambda^{(\tau)}_{1}(U^{(\tau)}_{m_1},\tau^2)-\lambda^{(\tau)}_{1}(U^{(\tau)}_{b},\tau^2)\big|\big|\Phi^{(\tau)}_{1}(\sigma^{(\tau)}_{\beta_1}(\zeta); U_{b}, \tau^2)-U_{b}\big|s\\[2pt]
&&\, \leq C\big|\lambda^{(\tau)}_{1}(U^{(\tau)}_{m_1},\tau^2)-\lambda^{(\tau)}_{1}(U^{(\tau)}_{b},\tau^2)\big||\zeta_1-\zeta_0|s\\[2pt]
&&\, \leq C|\sigma^{(\tau)}_{\beta_1}|^{2}\\[2pt]
&&\, \leq C\big(\tau^{2}+\nu^{-1}\big)\big|\sigma_{\alpha_1}\big|s.
\end{eqnarray*}
For the second term, using Lemma \ref{lem:4.2} and \eqref{eq:4.16}, we have
\begin{eqnarray*}
&&\int^{y_I+\dot{y}_{\alpha_1}s}_{y_I+\lambda^{(\tau)}_{1}(U^{(\tau)}_{m_1},\tau^{2})s}
\big|\mathcal{P}^{(\tau)}_{h}(\xi+s,\xi)(U_{h,\nu}(\xi,y))-U_{h, \nu}(\xi+s, y)\big|\,{\rm d}y\\[2pt]
&&\, \leq C\Big(\big|\lambda^{(\tau)}_{1}(U^{(\tau)}_{m_1},\tau^2)-\lambda_{1}(U_a)\big|+2^{-\nu}\Big)|\sigma^{(\tau)}_{\beta_1}|s\\[2pt]
&&\, \leq C(\tau^2+2^{-\nu})|\sigma_{\alpha_1}|s.
\end{eqnarray*}
Combining the two estimates above, we conclude
\begin{eqnarray*}
I_{\mathcal{R}_1,2}\leq C\big(\tau^2+\nu^{-1}+2^{-\nu}\big)|\sigma_{\alpha_1}|s.
\end{eqnarray*}

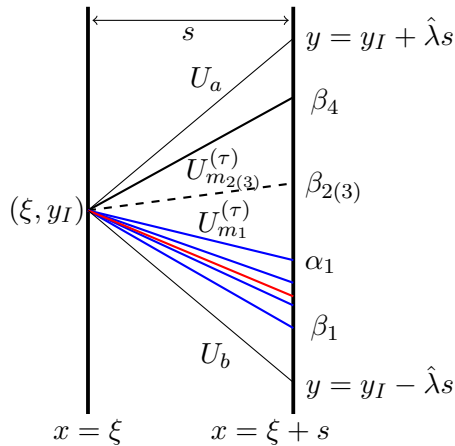
\begin{figure}[ht]
\begin{center}
\begin{tikzpicture}[scale=0.6]
\draw [line width=0.05cm](-4,4.5)--(-4,-4.5);
\draw [line width=0.05cm](0.5,4.5)--(0.5,-4.5);

\draw [thin](-4,0)--(0.5, 3.8);
\draw [thick](-4,0)--(0.5, 2.5);
\draw [thick][dashed](-4,0)--(0.5, 0.6);
\draw [thick][blue](-4,0)--(0.5, -1.1);
\draw [thick][blue](-4,0)--(0.5, -1.6);
\draw [thick][blue](-4,0)--(0.5, -2.1);
\draw [thick][blue](-4,0)--(0.5, -2.6);
\draw [thick][red](-4,0)--(0.5, -1.9);
\draw [thin](-4,0)--(0.5, -3.8);

\draw [thin][<->](-3.9,4.2)--(0.4,4.2);

\node at (2.6, 2) {$$};
\node at (-1.8, 3.9){$s$};
\node at (-4.9, 0) {$(\xi, y_I)$};
\node at (-4, -4.9) {$x=\xi$};
\node at (0, -4.9) {$x=\xi+s$};

\node at (2.4, 3.9) {$y=y_I+\hat{\lambda}s$};
\node at (1.2, 2.4) {$\beta_{4}$};
\node at (1.4, 0.5) {$\beta_{2(3)}$};
\node at (1.1, -1.2) {$\alpha_1$};
\node at (1.2, -2.6) {$\beta_{1}$};
\node at (2.4, -3.8) {$y=y_I-\hat{\lambda}s$};

\node at (-1.4, 2.9){$U_{a}$};
\node at (-1.0, 0.9){$U^{(\tau)}_{m_{2(3)}}$};
\node at (-1.0, -0.2){$U^{(\tau)}_{m_{1}}$};
\node at (-1.2, -3.2){$U_{b}$};

\end{tikzpicture}
\end{center}
\caption{Subcase 2: $\alpha_1\in \mathcal{R}_1$ and $\lambda^{(\tau)}_{1}(U_{b}, \tau^{2})<\dot{y}_{\alpha_1}< \lambda^{(\tau)}_{1}(U^{(\tau)}_{m_1},\tau^{2})$}\label{fig4.4}
\end{figure}

\emph{Subcase 2}:\ $\lambda^{(\tau)}_{1}(U_{b}, \tau^{2})<\dot{y}_{\alpha_1}< \lambda^{(\tau)}_{1}(U^{(\tau)}_{m_1},\tau^{2})$.
As shown in Fig. \ref{fig4.4}, we can decompose $I_{\mathcal{R}_1,2}$ into two terms:
\begin{align*}
I_{\mathcal{R}_1,2}&=\int^{y_I+\lambda^{(\tau)}_{1}(U^{(\tau)}_{m_1})s}_{y_I+\lambda^{(\tau)}_{1}(U_{b},\tau^{2})s}
\big|\mathcal{P}^{(\tau)}_{h}(\xi+s,\xi)(U_{h,\nu}(\xi,y))-U_{h, \nu}(\xi+s, y)\big|\,{\rm d}y\\[2pt]
&=\int^{y_I+\dot{y}_{\alpha_1}s}_{y_I+\lambda^{(\tau)}_{1}(U_{b},\tau^{2})s}
\big|\mathcal{P}^{(\tau)}_{h}(\xi+s,\xi)(U_{h,\nu}(\xi,y))-U_{h, \nu}(\xi+s, y)\big|\,{\rm d}y\\[2pt]
&\quad +\int^{y_I+\lambda^{(\tau)}_{1}(U^{(\tau)}_{m_1},\tau^2)s}_{y_I+\dot{y}_{\alpha_1}s}
\big|\mathcal{P}^{(\tau)}_{h}(\xi+s,\xi)(U_{h,\nu}(\xi,y))-U_{h, \nu}(\xi+s, y)\big|\,{\rm d}y.
\end{align*}
For the first term, it follows from \eqref{eq:4.6}, \eqref{eq:4.16c}, \eqref{eq:4.17}, and
Proposition A.1
that
\begin{eqnarray*}
&&\int^{y_I+\dot{y}_{\alpha_1}s}_{y_I+\lambda^{(\tau)}_{1}(U_{b},\tau^{2})s}
\big|\mathcal{P}^{(\tau)}_{h}(\xi+s,\xi)(U_{h,\nu}(\xi,y))-U_{h, \nu}(\xi+s, y)\big|\,{\rm d}y\\[2pt]
&&\, \leq \big(\dot{y}_{\alpha_1}-\lambda^{(\tau)}_{1}(U_{b},\tau^{2})\big)\big|\Phi^{(\tau)}_{1}(\sigma^{(\tau)}_{\beta_1}(\zeta); U_{b}, \tau^2)-U_{b}\big|s\\[2pt]
&&\, \leq C\Big(\big|\lambda_{1}(U_{a})-\lambda^{(\tau)}_{1}(U_{b},\tau^{2})\big|+2^{-\nu}\Big)|\sigma^{(\tau)}_{\beta_1}|s\\[4pt]
&&\, \leq C\big(\tau^{2}+\nu^{-1}+2^{-\nu}\big)|\sigma_{\alpha_1}|s.
\end{eqnarray*}
For the second term, further together with \eqref{eq:4.16}, we have
\begin{eqnarray*}
&&\int^{y_I+\lambda^{(\tau)}_{1}(U^{(\tau)}_{m_1},\tau^2)s}_{y_I+\dot{y}_{\alpha_1}s}
\big|\mathcal{P}^{(\tau)}_{h}(\xi+s,\xi)(U_{h,\nu}(\xi,y))-U_{h, \nu}(\xi+s, y)\big|\,{\rm d}y\\[2pt]
&&\, \leq \big(\lambda^{(\tau)}_{1}(U^{(\tau)}_{m_1},\tau^{2})-\dot{y}_{\alpha_1}\big)\big|\Phi^{(\tau)}_{1}(\sigma^{(\tau)}_{\beta_1}(\zeta); U_{b}, \tau^2)-U_{a}\big|s\\[2pt]
&&\, \leq C\Big(\big|\lambda_{1}(U_{a})-\lambda^{(\tau)}_{1}(U^{(\tau)}_{m_1},\tau^{2})\big|+2^{-\nu}\Big)\\
&&\qquad\quad\,\,\,\times \Big(\big|\Phi^{(\tau)}_{1}(\sigma^{(\tau)}_{\beta_1}(\zeta); U_{b}, \tau^2)-U_{b}\big|+|U_{a}-U_{b}|\Big)s\\[2pt]
&&\, \leq C\big((1+|\sigma_{\alpha_1}|)\tau^{2}+2^{-\nu}\big)\big(|\sigma^{(\tau)}_{\beta_1}|+|\sigma_{\alpha_1}|\big)s\\[2pt]
&&\, \leq C\big(\tau^{2}+2^{-\nu}\big)|\sigma_{\alpha_1}|s.
\end{eqnarray*}
Therefore, we obtain
\begin{eqnarray*}
I_{\mathcal{R}_1,2}\leq C\big(\tau^2+\nu^{-1}+2^{-\nu}\big)|\sigma_{\alpha_1}|s.
\end{eqnarray*}

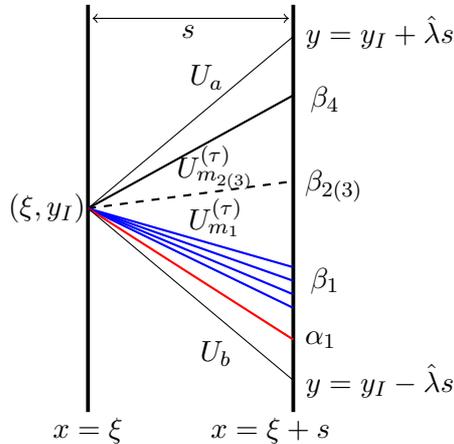
\begin{figure}[ht]
\begin{center}
\begin{tikzpicture}[scale=0.6]
\draw [line width=0.05cm](-4,4.5)--(-4,-4.5);
\draw [line width=0.05cm](0.5,4.5)--(0.5,-4.5);

\draw [thin](-4,0)--(0.5, 3.8);
\draw [thick](-4,0)--(0.5, 2.5);
\draw [thick][dashed](-4,0)--(0.5, 0.6);
\draw [thick][blue](-4,0)--(0.5, -1.3);
\draw [thick][blue](-4,0)--(0.5, -1.6);
\draw [thick][blue](-4,0)--(0.5, -1.9);
\draw [thick][blue](-4,0)--(0.5, -2.2);
\draw [thick][red](-4,0)--(0.5, -2.9);
\draw [thin](-4,0)--(0.5, -3.8);

\draw [thin][<->](-3.9,4.2)--(0.4,4.2);

\node at (2.6, 2) {$$};
\node at (-1.8, 3.9){$s$};
\node at (-4.9, 0) {$(\xi, y_I)$};
\node at (-4, -4.9) {$x=\xi$};
\node at (0, -4.9) {$x=\xi+s$};

\node at (2.4, 3.9) {$y=y_I+\hat{\lambda}s$};
\node at (1.2, 2.4) {$\beta_{4}$};
\node at (1.4, 0.5) {$\beta_{2(3)}$};
\node at (1.1, -2.9) {$\alpha_1$};
\node at (1.2, -1.6) {$\beta_{1}$};
\node at (2.4, -3.8) {$y=y_I-\hat{\lambda}s$};

\node at (-1.4, 2.9){$U_{a}$};
\node at (-1.2, 0.9){$U^{(\tau)}_{m_{2(3)}}$};
\node at (-1.2, -0.2){$U^{(\tau)}_{m_{1}}$};
\node at (-1.2, -3.2){$U_{b}$};

\end{tikzpicture}
\end{center}
\caption{Subcase 3: $\alpha\in \mathcal{R}_1$ and $\dot{y}_{\alpha_1}<\lambda^{(\tau)}_{1}(U_{b}, \tau^{2})$}\label{fig4.5}
\end{figure}

\emph{Subcase 3}:  $\dot{y}_{\alpha_1}<\lambda^{(\tau)}_{1}(U_{b}, \tau^{2})$. As shown in Fig. \ref{fig4.5},
\begin{align*}
I_{\mathcal{R}_1,2}&=\int^{y_I+\lambda^{(\tau)}_{1}(U^{(\tau)}_{m_1})s}_{y_I+\dot{y}_{\alpha_1}s}
\big|\mathcal{P}^{(\tau)}_{h}(\xi+s, \xi)(U_{h,\nu}(\xi,y))-U_{h, \nu}(\xi+s, y)\big|\,{\rm d}y\\[2pt]
&=\int^{y_I+\lambda^{(\tau)}_{1}(U_{b})s}_{y_I+\dot{y}_{\alpha_1}s}
\big|\mathcal{P}^{(\tau)}_{h}(\xi+s, \xi)(U_{h,\nu}(\xi,y))-U_{h, \nu}(\xi+s, y)\big|\,{\rm d}y\\[2pt]
&\quad +\int^{y_I+\lambda^{(\tau)}_{1}(U^{(\tau)}_{m_1})s}_{y_I+\lambda^{(\tau)}_{1}(U_{b})s}
\big|\mathcal{P}^{(\tau)}_{h}(\xi+s, \xi)(U_{h,\nu}(\xi,y))-U_{h, \nu}(\xi+s, y)\big|\,{\rm d}y.
\end{align*}

Then it follows from Lemma \ref{lem:4.2},
Proposition A.1,
and \eqref{eq:4.17} that
\begin{eqnarray*}
&&\int^{y_I+\lambda^{(\tau)}_{1}(U_{b})s}_{y_I+\dot{y}_{\alpha_1}s}
\big|\mathcal{P}^{(\tau)}_{h}(\xi+s, \xi)(U_{h,\nu}(\xi,y))-U_{h, \nu}(\xi+s, y)\big|\,{\rm d}y\\[2pt]
&&\, \leq C\big|\lambda^{(\tau)}_{1}(U_{b},\tau^{2})-\dot{y}_{\alpha_1}\big||U_{a}-U_{b}|s\\[2pt]
&&\, \leq C\Big(|\lambda^{(\tau)}_{1}(U_{b},\tau^{2})-\lambda_{1}(U_a)|+2^{-\nu}\Big)|\sigma_{\alpha_1}|s\\[2pt]
&&\, \leq C\big(\tau^{2}+\nu^{-1}+2^{-\nu}\big)|\sigma_{\alpha_1}|s,
\end{eqnarray*}
and
\begin{eqnarray*}
&&\int^{y_I+\lambda^{(\tau)}_{1}(U^{(\tau)}_{m_1})s}_{y_I+\lambda^{(\tau)}_{1}(U_{b})s}
\big|\mathcal{P}^{(\tau)}_{h}(\xi+s, \xi)(U_{h,\nu}(\xi,y))-U_{h, \nu}(\xi+s, y)\big|\,{\rm d}y\\[2pt]
&&\, \leq C\big|\lambda^{(\tau)}_{1}(U^{(\tau)}_{m_1},\tau^{2})-\lambda^{(\tau)}_{1}(U_b, \tau^2)\big|
\big|\Phi^{(\tau)}_{1}(\sigma^{(\tau)}_{\beta_1};U_{b},\tau^2)-U_{a}\big|s\\[2pt]
&&\, \leq C|\sigma^{(\tau)}_{\beta_1}|\big(|\sigma^{(\tau)}_{\beta_1}|+|\sigma_{\alpha_1}|\big)s\\[2pt]
&&\, \leq C\big(\tau^{2}+\nu^{-1}\big)|\sigma_{\alpha_1}|s.
\end{eqnarray*}
{
Thus, we also obtain in this subcase that
\begin{eqnarray*}
I_{\mathcal{R}_1,2}\leq C\big(\tau^2+\nu^{-1}+2^{-\nu}\big)|\sigma_{\alpha_1}|s.
\end{eqnarray*}
}

Now, for all the three subcases, we know that
$$
I_{\mathcal{R}_1,2}\leq C\big(\tau^{2}+\nu^{-1}+2^{-\nu}\big)|\sigma_{\alpha_1}|s.
$$
Then, combining the estimates for $I_{\mathcal{R}_1, 1}$, $I_{\mathcal{R}_1,2}$, and $I_{\mathcal{R}_1, 3}$ together,
we can choose a constant $C_{4,2}$ depending only on $(\underline{U}, a_{\infty})$ so that \eqref{eq:4.11} holds.

\medskip
3. We know that $\alpha_{k}\in \mathcal{C}_{k}$ for $k=2,3$. Without loss of generality,
we focus only on the case: $k=2$. In this case, we know that the solutions, $\mathcal{P}^{(\tau)}_{h}(\xi+s,\xi)(U_{h,\nu}(\xi,y))$, satisfy
\begin{eqnarray*}
\Phi^{(\tau)}(\boldsymbol{\sigma}^{(\tau)}_{\boldsymbol{\beta}}; U_{b},\tau^2)=\Phi_2(\sigma_{\alpha_2};U_{b})
\qquad \mbox{for $\boldsymbol{\sigma}^{(\tau)}_{\boldsymbol{\beta}}=(\sigma^{(\tau)}_{\beta_1},\sigma^{(\tau)}_{\beta_2},\sigma^{(\tau)}_{\beta_3},
\sigma^{(\tau)}_{\beta_4})$}.
\end{eqnarray*}

By Lemma \ref{lem:4.2}, there exists a unique solution $\boldsymbol{\sigma}^{(\tau)}_{\boldsymbol{\beta}}(\sigma^{(\tau)}_{\beta_1},\sigma^{(\tau)}_{\beta_2},\sigma^{(\tau)}_{\beta_3},
\sigma^{(\tau)}_{\beta_4})$ of the above equation, which satisfies the estimates:
\begin{eqnarray}\label{eq:4.18}
\sigma^{(\tau)}_{\beta_2}=\sigma_{\alpha_2}+O(1)|\sigma_{\alpha_2}|\tau^2, \qquad\,\, \sigma^{(\tau)}_{\beta_j}=O(1)|\sigma_{\alpha_2}|\tau^2
\quad \mbox{for $j\neq2$}.
\end{eqnarray}

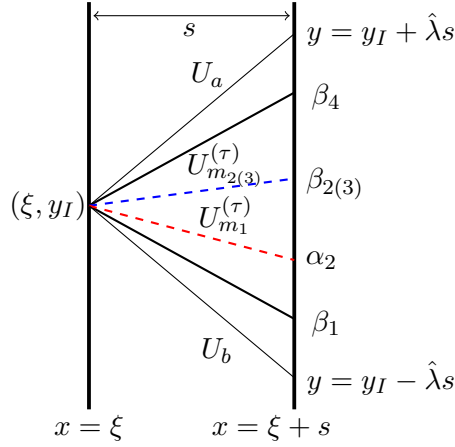
\begin{figure}[ht]
\begin{center}
\begin{tikzpicture}[scale=0.6]
\draw [line width=0.05cm](-4,4.5)--(-4,-4.5);
\draw [line width=0.05cm](0.5,4.5)--(0.5,-4.5);

\draw [thin](-4,0)--(0.5, 3.8);
\draw [thick](-4,0)--(0.5, 2.5);
\draw [thick][dashed][blue](-4,0)--(0.5, 0.6);
\draw [thick](-4,0)--(0.5, -2.5);
\draw [thick][dashed][red](-4,0)--(0.5, -1.2);
\draw [thin](-4,0)--(0.5, -3.8);

\draw [thin][<->](-3.9,4.2)--(0.4,4.2);

\node at (2.6, 2) {$$};
\node at (-1.8, 3.9){$s$};
\node at (-4.9, 0) {$(\xi, y_I)$};
\node at (-4, -4.9) {$x=\xi$};
\node at (0, -4.9) {$x=\xi+s$};

\node at (2.4, 3.9) {$y=y_I+\hat{\lambda}s$};
\node at (1.2, 2.4) {$\beta_{4}$};
\node at (1.4, 0.5) {$\beta_{2(3)}$};
\node at (1.1, -1.2) {$\alpha_2$};
\node at (1.2, -2.6) {$\beta_{1}$};
\node at (2.4, -3.8) {$y=y_I-\hat{\lambda}s$};

\node at (-1.4, 2.9){$U_{a}$};
\node at (-1.0, 0.9){$U^{(\tau)}_{m_{2(3)}}$};
\node at (-1.0, -0.2){$U^{(\tau)}_{m_{1}}$};
\node at (-1.2, -3.2){$U_{b}$};

\end{tikzpicture}
\end{center}
\caption{Comparison of the Riemann solutions for the case: $\alpha_2\in \mathcal{C}_2$}\label{fig4.2}
\end{figure}

Denoted by $U^{(\tau)}_{j}, j=1,2,3$, the corresponding intermediate states of the Riemann solutions to \eqref{eq:2.12} (see Fig. \ref{fig4.2}).
Note that speeds $\dot{\mathcal{S}}_{2}(\sigma_{\alpha_2})$ and $\dot{\mathcal{S}}^{(\tau)}_{2}(\sigma^{(\tau)}_{\beta_{2}})$ of fronts $\alpha_2$ and $\beta_{2}$ satisfy that $\dot{\mathcal{S}}_{2}(\sigma_{\alpha_2})=\lambda_{2}(U_{b})$ and
\begin{align*}
\dot{\mathcal{S}}^{(\tau)}_{2}(\sigma^{(\tau)}_{\beta_{2}},\tau^{2})&=\lambda^{(\tau)}_{2}(U^{(\tau)}_{m_{1}},\tau^{2})\\[2pt]
&=\lambda^{(\tau)}_{2}(U_{b},\tau^{2})+O(1)|U^{(\tau)}_{m_1}-U_{b}|\\[2pt]
&=\lambda^{(\tau)}_{2}(U_{b},\tau^{2})+O(1)|\sigma^{(\tau)}_{\beta_{1}}|\\[2pt]
&=\lambda^{(\tau)}_{2}(U_{b},\tau^{2})+O(1)|\sigma_{\alpha_1}|\tau^{2}.
\end{align*}

Thus, using \eqref{eq:3.5a},
we have
\begin{align*}
\dot{\mathcal{S}}^{(\tau)}_{2}(\sigma^{(\tau)}_{\beta_{2}},\tau^{2})-\dot{\mathcal{S}}_{2}(\sigma_{\alpha_2})
&=\lambda^{(\tau)}_{2}(U_{b},\tau^{2})+O(1)|\sigma_{\alpha_2}|\tau^{2}-\lambda_{2}(U_{b})\\
&=O(1)\big(1+|\sigma_{\alpha_2}|\big)\tau^{2}.
\end{align*}

Now, notice that
\begin{eqnarray*}
&&\int^{y_I+\hat{\lambda}s}_{y_I-\hat{\lambda}s}\big|\mathcal{P}^{(\tau)}_{h}(\xi+s,\tilde{x})(U_{h,\nu}(\tilde{x},y))
-U_{h,\nu}(\xi+s,y)\big|\,{\rm d}y\\[2pt]
&&=\int^{\min\{y_I+\dot{\mathcal{S}}^{(\tau)}_{2}(\sigma^{(\tau)}_{\beta_2}, \,\tau^{2})s,y_I+\dot{y}_{\alpha_2}s\}}
_{y_I-\hat{\lambda}s}\big|\mathcal{P}^{(\tau)}_{h}(\xi+s,\xi)(U_{h,\nu}(\xi,\cdot))-U_{h,\nu}(\xi+s,y)\big|\,{\rm d}y\\[2pt]
&&\quad\ +\int^{\max\{y_{I}+\dot{\mathcal{S}}^{(\tau)}_{2}(\sigma^{(\tau)}_{\beta_2}, \,\tau^{2})s,\tilde{y}+\dot{y}_{\alpha_2}s\}}
_{\min\{y_I+\dot{\mathcal{S}}^{(\tau)}_{2}(\sigma^{(\tau)}_{\beta_2}, \,\tau^{2})s, \,y_I+\dot{y}_{\alpha_2}s\}}
\big|\mathcal{P}^{(\tau)}_{h}(\xi+s,\xi)(U_{h,\nu}(\xi,y))-U_{h,\nu}(\xi+s,y)\big|\,{\rm d}y\\[2pt]
&&\quad \ +\int^{y_I+\hat{\lambda}s}
_{\max\{y_I+\dot{\mathcal{S}}^{(\tau)}_{2}(\sigma^{(\tau)}_{\beta_2}, \,\tau^{2})s,y_I+\dot{y}_{\alpha_2}s\}}
\big|\mathcal{P}^{(\tau)}_{h}(\xi+s,\xi)(U_{h,\nu}(\xi,y))-U_{h,\nu}(\xi+s,y)\big|\,{\rm d}y\\[4pt]
&&\doteq I_{\mathcal{C}_2,1}+I_{\mathcal{C}_2,2}+I_{\mathcal{C}_2,3}.
\end{eqnarray*}

It is clear that $I_{\mathcal{C}_2,1}=0$.

For $I_{\mathcal{C}_2, 2}$, if $\dot{y}_{\alpha_2}\leq \dot{\mathcal{S}}^{(\tau)}_{2}(\sigma^{(\tau)}_{\beta_{2}},\tau^{2})$, then
\begin{align*}
I_{\mathcal{C}_2, 2}
&\leq C\big|\dot{\mathcal{S}}^{(\tau)}_{2}(\sigma^{(\tau)}_{\beta_{2}},\tau^{2})-\dot{y}_{\alpha_2}\big||U^{(\tau)}_{m_{2}}-U_{b}|s\\[2pt]
&\leq C\Big(\big|\dot{\mathcal{S}}^{(\tau)}_{2}(\sigma^{(\tau)}_{\beta_{2}},\tau^{2})-\dot{\mathcal{S}}_{2}(\sigma_{\alpha_2})\big|+2^{-\nu}\Big)
\Big(\sum_{j=1,2}|\sigma^{(\tau)}_{\beta_j}|\Big)s\\[2pt]
&\leq C\Big(\big(1+|\sigma_{\alpha_2}|\big)\tau^{2}+2^{-\nu}\Big)\big(1+ C\,\tau^{2}\big)|\sigma_{\alpha_2}|s\\[2pt]
&\leq C\big(\tau^2+2^{-\nu}\big)|\sigma_{\alpha_2}|s;
\end{align*}
otherwise, if $\dot{y}_{\alpha_2}\geq \dot{\mathcal{S}}^{(\tau)}_{2}(\sigma^{(\tau)}_{\beta_{2}},\tau^{2})$, then
\begin{align*}
I_{\mathcal{C}_2, 2}
&\leq C\big|\dot{y}_{\alpha_2}-\dot{\mathcal{S}}^{(\tau)}_{2}(\sigma^{(\tau)}_{\beta_{2}},\tau^{2})\big||U^{(\tau)}_{m_{2}}-U_{a}|s\\[2pt]
&\leq C\Big(\big|\dot{\mathcal{S}}^{(\tau)}_{2}(\sigma^{(\tau)}_{\beta_{2}},\tau^{2})-\dot{\mathcal{S}}_{2}(\sigma_{\alpha_2})\big|+2^{-\nu}\Big)
\Big(\sum_{k=3,4}|\sigma^{(\tau)}_{\beta_k}|\Big)s\\[2pt]
&\leq C\Big(\big(1+|\sigma_{\alpha_2}|\big)\tau^{2}+2^{-\nu}\Big)|\sigma_{\alpha_2}|\tau^{2}s\\[2pt]
&\leq C|\sigma_{\alpha_2}|\tau^2s.
\end{align*}
This implies that
$$
I_{\mathcal{C}_2, 2}\leq C\big(\tau^2+2^{-\nu}\big)|\sigma_{\alpha_2}|s.
$$

For $I_{\mathcal{C}_2,3}$, it follows from estimates \eqref{eq:4.18} that
\begin{align*}
I_{\mathcal{C}_2,3}
&\leq C\big(|U^{(\tau)}_{m_{2}}-U_{m_3}|+|U^{(\tau)}_{m_{3}}-U_{a}|\big)s\\[2pt]
&\leq C\big(|\sigma^{(\tau)}_{\beta_3}|+|\sigma^{(\tau)}_{\beta_4}|\big)s\\[2pt]
&\leq C|\sigma_{\alpha_2}|\tau^{2}s.
\end{align*}

Therefore, we collect the estimates for $I_{\mathcal{C}_2,1}$, $I_{\mathcal{C}_2,2}$, and $I_{\mathcal{C}_2,3}$ altogether
to conclude that there exists a constant {$C_{3,3}>0$} depending only on $(\underline{U},a_{\infty})$
such that estimate \eqref{eq:4.12} holds.

\medskip
4. In this case, solutions $\mathcal{P}^{(\tau)}_{h}(\xi+s,\xi)(U_{h,\nu}(\xi,y))$ satisfy
\begin{eqnarray*}
U_{a}=\Phi^{(\tau)}(\boldsymbol{\sigma}^{(\tau)}_{\boldsymbol{\beta}}; U_{b},\tau^2) \qquad\,\,
\mbox{for $\boldsymbol{\sigma}^{(\tau)}_{\boldsymbol{\beta}}=(\sigma^{(\tau)}_{\beta_1},\sigma^{(\tau)}_{\beta_2},\sigma^{(\tau)}_{\beta_3},
\sigma^{(\tau)}_{\beta_4})$}.
\end{eqnarray*}

Applying Lemma \ref{lem:4.2}, we have
\begin{eqnarray}
|\sigma^{(\tau)}_{\beta_{j}}|=O(1)\sigma_{\alpha_{\mathcal{NP}}}  \qquad \mbox{for $1\leq j\leq 4$}.
\end{eqnarray}

\begin{figure}[ht]
\begin{center}
\begin{tikzpicture}[scale=0.6]
\draw [line width=0.05cm](-4,4.5)--(-4,-4.5);
\draw [line width=0.05cm](0.5,4.5)--(0.5,-4.5);

\draw [thick][red](-4,0)--(0.5, 3.8);
\draw [thick](-4,0)--(0.5, 2.5);
\draw [thick][dashed](-4,0)--(0.5, 0.6);
\draw [thick](-4,0)--(0.5, -1.3);
\draw [thin](-4,0)--(0.5, -2.9);

\draw [thin][<->](-3.9,4.2)--(0.4,4.2);

\node at (2.6, 2) {$$};
\node at (-1.8, 3.9){$s$};
\node at (-4.9, 0) {$(\xi, y_I)$};
\node at (-4, -4.9) {$x=\xi$};
\node at (0, -4.9) {$x=\xi+s$};

\node at (1.4, 3.9) {$\alpha_{\mathcal{NP}}$};
\node at (1.2, 2.4) {$\beta_{4}$};
\node at (1.4, 0.5) {$\beta_{2(3)}$};
\node at (1.2, -1.6) {$\beta_{1}$};
\node at (2.4, -2.9) {$y=y_I-\hat{\lambda}s$};

\node at (-1.4, 2.9){$U_{a}$};
\node at (-1.2, 0.9){$U^{(\tau)}_{m_{2}}$};
\node at (-1.2, -0.2){$U^{(\tau)}_{m_{1}}$};
\node at (-1.2, -3.2){$U_{b}$};

\end{tikzpicture}
\end{center}
\caption{Comparison of the Riemann solutions for the case: $\alpha_{\mathcal{NP}}\in \mathcal{NP}$}\label{fig4.6}
\end{figure}
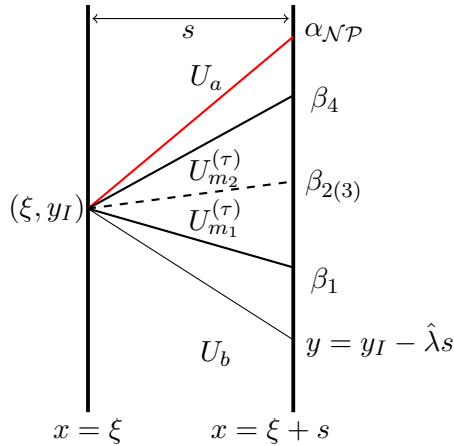

As shown in Fig. \ref{fig4.6}, let $U^{(\tau)}_{m_j}, 1\leq j\leq 3$, be the intermediate states of solutions
$\mathcal{P}^{(\tau)}_{h}(\xi+s,\xi)(U_{h,\nu}(\xi,y))$
and satisfy \eqref{eq:2.12}.
Then
\begin{eqnarray*}
&&\int^{y_I+\hat{\lambda}s}_{y_I-\hat{\lambda}s}\big|\mathcal{P}^{(\tau)}(\xi+s, \xi)(U_{h,\nu}(\xi,y))-U_{h,\nu}(\xi+s,y)\big|\,{\rm d}y\\[2pt]
&&\leq C\Big(|U_{a}-U_{b}|+\sum^{3}_{j=1}|U^{(\tau)}_{m_{j}}-U_{b}|\Big)s\\[2pt]
&&\leq C\Big(|\sigma_{\alpha_{\mathcal{NP}}}|+\sum^{4}_{j=1}|\sigma^{(\tau)}_{\beta_{j}}|\Big)s\\[2pt]
&&\leq C_{4,4}|\sigma_{\alpha_{\mathcal{NP}}}|s,
\end{eqnarray*}
where $C_{4,4}>0$ depends only on $(\underline{U}, a_{\infty})$. This completes the proof.
\end{proof}

Next, based on Proposition \ref{prop:4.1}, we consider the more general case when the approximate solution $U_{h,\nu}(\xi+s,y)$
contains more than one wave.

\begin{proposition}\label{prop:4.2}
Let $\mathcal{P}^{(\tau)}_{h}$ be the Lipschitz continuous map as given in {\rm Proposition \ref{prop:2.2}} for \emph{Problem I}.
Let $U_{h,\nu}(\xi,y)$ be the approximate solution constructed in {\rm Appendix A} for system \eqref{eq:1.14}
with $U_a$, $U_{b}\in \mathcal{O}_{\min\{\varepsilon_0,\tilde{\varepsilon}_0\}}(\underline{U})$ defined by \eqref{eq:4.1}.
Assume that, near point $(\xi,y_I)$, $U_{h,\nu}(\xi,y)$ is of the form{\rm :}
\begin{equation}\label{eq:4.20}
U_{h,\nu}(\xi+s,y)=\left\{
\begin{array}{llll}
U_{a}, \quad  &y>y_{I}+\hat{\lambda}s,\\[2pt]
\hat{U}_{a}, \quad  &y_{I}+\dot{y}_{\alpha_\ell}s< y<y_{I}+\hat{\lambda}s,\\[2pt]
U_{m}, \quad   &y_{I}+\dot{y}_{\alpha_k}s<y< y_{I}+\dot{y}_{\alpha_\ell}s,\\[2pt]
U_{b}, \quad  &y<y_{I}+\dot{y}_{\alpha_k}s,
\end{array}
\right.
\end{equation}
where $\hat{\lambda}>\max_{1\leq j\leq 4}\big\{\lambda^{(\tau)}_{j}(U^{(\tau)},\tau^{2}), \lambda_{j}(U)\big\}$ for $U^{(\tau)}$,
$U\in \mathcal{O}_{\min\{\varepsilon_0, \tilde{\varepsilon}_0\}}(\underline{U})$,
$\tau \in(0,\min\{\tau_{0},\tilde{\tau}_0\})$,
and $|\dot{y}_{\alpha_k}|$, $|\dot{y}_{\alpha_\ell}|<\hat{\lambda}$ for $k<\ell$ with $1\leq k,\ell \leq 4$.
That is, $U_b$ and $U_{m}$ are connected by a $k^{\rm th}$ wave that can be a shock $\alpha_k\in \mathcal{S}_k$
with $|\dot{y}_{\alpha_k}-\dot{\mathcal{S}}_{k}(\sigma_{\alpha_k})|<2^{-\nu}$,
rarefaction-front $\alpha_k\in \mathcal{R}_k$ with $|\dot{y}_{\alpha_k}-\lambda_{k}(U_m)|<2^{-\nu}$,
or vortex sheet/entropy wave $\alpha_k\in \mathcal{C}_k$ with $|\dot{y}_{\alpha_k}-\lambda_{k}(U_b)|<2^{-\nu}${\rm ;}
$U_m$ and $\hat{U}_a$ are connected by an $l^{\rm th}$ wave that could be shock $\alpha_\ell\in \mathcal{S}_\ell$
with $|\dot{y}_{\alpha_\ell}-\dot{\mathcal{S}}_{\ell}(\sigma_{\alpha_\ell})|<2^{-\nu}$,
rarefaction-front $\alpha_\ell\in \mathcal{R}_\ell$ with $|\dot{y}_{\alpha_\ell}-\lambda_{\ell}(\hat{U}_a)|<2^{-\nu}$,
or vortex sheet/entropy wave $\alpha_\ell\in \mathcal{C}_\ell$ with $|\dot{y}_{\alpha_\ell}-\lambda_{\ell}(U_m)|<2^{-\nu}${\rm ;}
and $\hat{U}_a$ and $U_{a}$ are connected by a non-physical wave $\alpha_{\mathcal{NP}}\in \mathcal{NP}$,
where $\dot{\mathcal{S}}_{k}(\sigma_{\alpha_k})$, $\lambda_k(U_m)$, and $\lambda_k(U_b)$ are the speeds of $k^{\rm th}$ shock-front,
rarefaction-front, and vortex sheet/entropy wave-front, respectively, for $1\leq k\leq 4$.
Then, for $s>0$ sufficiently small,
\begin{align}\label{eq:4.21}
 &\int^{y_I+\hat{\lambda}s}_{y_I-\hat{\lambda}s}\big|\mathcal{P}^{(\tau)}_{h}(\xi+s, \xi)(U_{h,\nu}(\xi,y))-U_{h,\nu}(\xi+s,y)\big|\,{\rm d}y\nonumber\\[2pt]
&\,\, \leq {C_{3,5}}\Big(\big(\tau^2+\nu^{-1}+2^{-\nu}\big)\big(|\sigma_{\alpha_{k}}|+|\sigma_{\alpha_\ell}|\big)+\sigma_{\alpha_{\mathcal{NP}}}\Big)s,
\end{align}
where constant {$C_{3,5}>0$} depends only on $(\underline{U},a_{\infty})$.
\end{proposition}

\begin{proof}
By Proposition \ref{prop:2.2}, solution $\mathcal{P}^{(\tau)}_{h}(\xi+s,\xi)(U_{h,\nu}(\xi,y))$ satisfies
\begin{equation}\label{eq:4.22}
U_{a}=\Phi^{(\tau)}(\boldsymbol{\sigma}^{(\tau)}_{\boldsymbol{\gamma}}; U_{b}, \tau^{2}),\,
\hat{U}_{a}=\Phi_{\ell}(\sigma_{\alpha_\ell}; U_{m}),\, U_{m}=\Phi_{k}(\sigma_{\alpha_k}; U_{b}),
\, \sigma_{\alpha_{\mathcal{NP}}}=|U_{a}-\hat{U}_a|,
\end{equation}
where $\boldsymbol{\sigma}^{(\tau)}_{\boldsymbol{\gamma}}=(\sigma^{(\tau)}_{\gamma_{1}}, \sigma^{(\tau)}_{\gamma_{2}},\sigma^{(\tau)}_{\gamma_{3}}, \sigma^{(\tau)}_{\gamma_{4}})$.
Define
\begin{equation}\label{eq:4.23}
\Phi^{(\tau)}(\hat{\boldsymbol{\sigma}}^{(\tau)}_{\hat{\boldsymbol{\gamma}}}; U_{b}, \tau^{2})
=\Phi_{\ell}(\sigma_{\alpha_\ell}; \Phi_{k}(\sigma_{\alpha_k}; U_{b}))
\quad\,\,\mbox{for $\hat{\boldsymbol{\sigma}}^{(\tau)}_{\hat{\boldsymbol{\gamma}}}=(\hat{\sigma}^{(\tau)}_{\hat{\gamma}_{1}},
 \hat{\sigma}^{(\tau)}_{\hat{\gamma}_{2}},\hat{\sigma}^{(\tau)}_{\hat{\gamma}_{3}}, \hat{\sigma}^{(\tau)}_{\hat{\gamma}_{4}})$}.
\end{equation}

By Lemma \ref{lem:2.3}, the implicit function theorem, and \eqref{eq:4.22}--\eqref{eq:4.23},
for $U_b\in \mathcal{O}_{\min\{\varepsilon_0,\tilde{\varepsilon}_0\}}(\underline{U})$ and $\tau\in (0,\min\{\tau_0,\tilde{\tau}_0\})$,
\begin{eqnarray}\label{eq:4.24}
\sigma^{(\tau)}_{\gamma_j}=\hat{\sigma}^{(\tau)}_{\hat{\gamma}_j}+O(1)\sigma_{\alpha_{\mathcal{NP}}} \qquad \mbox{for $1\leq j\leq 4$}.
\end{eqnarray}
Thus, to estimate $\sigma^{(\tau)}_{\gamma_j}$ for $1\leq j\leq 4$ for the completion of the proof of estimate \eqref{eq:4.21},
one needs to estimate $\hat{\sigma}^{(\tau)}_{\hat{\gamma}_j}$ for $1\leq j\leq 4$.

Let $\boldsymbol{\sigma}^{(\tau)}_{\hat{\boldsymbol{\gamma}}'}
=(\hat{\sigma}^{(\tau)}_{\hat{\gamma}'_{1}}, \hat{\sigma}^{(\tau)}_{\hat{\gamma}'_{2}},\hat{\sigma}^{(\tau)}_{\hat{\gamma}'_{3}}, \hat{\sigma}^{(\tau)}_{\hat{\gamma}'_{4}})$
and $\boldsymbol{\sigma}^{(\tau)}_{\hat{\boldsymbol{\gamma}}''}
=(\hat{\sigma}^{(\tau)}_{\hat{\gamma}''_{1}}, \hat{\sigma}^{(\tau)}_{\hat{\gamma}''_{2}},\hat{\sigma}^{(\tau)}_{\hat{\gamma}''_{3}}, \hat{\sigma}^{(\tau)}_{\hat{\gamma}''_{4}})$
satisfy
\begin{eqnarray}\label{eq:4.25}
\hat{U}_a=\Phi^{(\tau)}(\hat{\boldsymbol{\sigma}}^{(\tau)}_{\hat{\boldsymbol{\gamma}}''}; U_{m}, \tau^{2}),\qquad
\hat{U}_m=\Phi^{(\tau)}(\hat{\boldsymbol{\sigma}}^{(\tau)}_{\hat{\boldsymbol{\gamma}}'}; U_{b}, \tau^{2})
\end{eqnarray}

Then, by \eqref{eq:4.22} and Lemma \ref{lem:4.2}, we have
\begin{equation}\label{eq:4.26}
\hat{\sigma}^{(\tau)}_{\hat{\gamma}'_{j}}
=\delta_{jk}\sigma_{\alpha_k}+O(1)|\sigma_{\alpha_k}|\tau^{2},\quad
\hat{\sigma}^{(\tau)}_{\hat{\gamma}''_{j}}
=\delta_{j\ell}\sigma_{\alpha_\ell}+O(1)|\sigma_{\alpha_\ell}|\tau^{2}
\qquad\,\, \mbox{for $1\leq j\leq 4$}.
\end{equation}

By \eqref{eq:4.23}, \eqref{eq:4.25}, and Lemma \ref{lem:2.4}, we have
\begin{eqnarray}\label{eq:4.27}
\hat{\sigma}^{(\tau)}_{\hat{\gamma}_{j}}=\hat{\sigma}^{(\tau)}_{\hat{\gamma}'_{j}}+\hat{\sigma}^{(\tau)}_{\hat{\gamma}''_{j}}
+O(1)\Delta\big(\boldsymbol{\sigma}^{(\tau)}_{\hat{\boldsymbol{\gamma}}'}, \boldsymbol{\sigma}^{(\tau)}_{\hat{\boldsymbol{\gamma}}''}\big),
\end{eqnarray}
where
$$\Delta\big(\boldsymbol{\sigma}^{(\tau)}_{\hat{\boldsymbol{\gamma}}'}, \boldsymbol{\sigma}^{(\tau)}_{\hat{\boldsymbol{\gamma}}''}\big)=\Big(\sum^{4}_{j=2}|\hat{\sigma}^{(\tau)}_{\hat{\gamma}'_{j}}|\Big)
|\hat{\sigma}^{(\tau)}_{\hat{\gamma}''_{1}}|+\Big(\sum^{3}_{j=2}|\hat{\sigma}^{(\tau)}_{\hat{\gamma}'_{j}}|\Big)
|\hat{\sigma}^{(\tau)}_{\hat{\gamma}''_{4}}|+\sum^{4}_{j=1}\Delta_{j}\big(\boldsymbol{\sigma}^{(\tau)}_{\hat{\boldsymbol{\gamma}}'}, \boldsymbol{\sigma}^{(\tau)}_{\hat{\boldsymbol{\gamma}}''}\big),$$
and
\begin{equation*}
\left.\Delta_{j}\right(\boldsymbol{\sigma}^{(\tau)}_{\hat{\boldsymbol{\gamma}}'}, \boldsymbol{\sigma}^{(\tau)}_{\hat{\boldsymbol{\gamma}}''})=\left\{
\begin{array}{llll}
 0,\qquad  & \hat{\sigma}^{(\tau)}_{\hat{\gamma}'_{j}}>0\ \ \mbox{and}\ \ \hat{\sigma}^{(\tau)}_{\hat{\gamma}''_{j}}>0,\\[5pt]
|\hat{\sigma}^{(\tau)}_{\hat{\gamma}'_{j}}||\hat{\sigma}^{(\tau)}_{\hat{\gamma}''_{j}}|, \qquad  &\hat{\sigma}^{(\tau)}_{\hat{\gamma}'_{j}}<0\ \ \mbox{or}\ \ \hat{\sigma}^{(\tau)}_{\hat{\gamma}''_{j}}<0.
\end{array}
\right.
\end{equation*}

Therefore, it follows from \eqref{eq:4.24} and \eqref{eq:4.26}--\eqref{eq:4.27} that, for $ 1\leq j\leq 4 $ and $1\leq k<\ell\leq 4$,
\begin{eqnarray}\label{eq:4.28}
\sigma^{(\tau)}_{\gamma_{j}}
=\delta_{jk}\sigma_{\alpha_k}+\delta_{j\ell}\sigma_{\alpha_\ell}+O(1)\big(|\sigma_{\alpha_k}|+|\sigma_{\alpha_\ell}|\big)\tau^{2}
+O(1)\sigma_{\alpha_{\mathcal{NP}}}.
\end{eqnarray}

With estimate \eqref{eq:4.28} in hand, we can now follow the method exactly as done in the proof of Proposition \ref{prop:4.1}
to deduce estimates \eqref{eq:4.21}. Since the proof is long but the same, we omit it.
\end{proof}

\subsection{Comparison of the Riemann solvers with the boundary}
We first study the case when all the Riemann solutions are generated by corner points on the approximate boundary $\Gamma_{h}$
(see also Fig. \ref{fig4.7a}).
Let $U^{(\tau)}_{h,\nu}$ and $U_{h,\nu}$ be approximate solutions of Problem I (obtained in Proposition \ref{prop:2.1})
and the initial-boundary value problem \eqref{eq:1.12}--\eqref{eq:1.15} (obtained in Proposition A.1),
respectively.
Let $\textsc{C}_{k}(x_k,g_{k})$ be the corner points on the approximate boundary $\Gamma_h$ with turning angle $\omega_k=\theta_{k+1}-\theta_k$.
Define
\begin{align}
&U_{B}=(\rho_{B}, u_{B}, v_{B},p_{B})^{\top}\doteq U_{h,\nu}(x_k,g_{k}-),\,\,\,
U_{\Gamma}=(\rho_{\Gamma}, u_{\Gamma}, v_{\Gamma},p_{\Gamma})^{\top}\doteq U_{h,\nu}(x_k,g_k),\label{eq:4.29}\\
&U^{(\tau)}_{\Gamma}=(\rho^{(\tau)}_{\Gamma}, u^{(\tau)}_{\Gamma}, v^{(\tau)}_{\Gamma},p^{(\tau)}_{\Gamma})^{\top}
\doteq U^{(\tau)}_{h,\nu}(x_k,g_k).\label{eq:4.30}
\end{align}

\begin{figure}[ht]
\begin{center}
\begin{tikzpicture}[scale=0.9]
\draw [thick](-5.0,1.5)--(-3,1)--(-1,1.8);

\draw [line width=0.05cm](-5.0,1.5)--(-5.0,-1.5);
\draw [line width=0.05cm](-3,1)--(-3,-1.5);
\draw [line width=0.05cm](-1.0,1.8)--(-1.0,-1.5);

\draw [thin](-4.8,1.45)--(-4.5, 1.75);
\draw [thin](-4.5,1.38)--(-4.2, 1.68);
\draw [thin](-4.2,1.30)--(-3.9, 1.60);
\draw [thin](-3.9, 1.23)--(-3.6,1.53);
\draw [thin](-3.6, 1.16)--(-3.3,1.46);
\draw [thin](-3.3, 1.08)--(-3.0,1.38);
\draw [thin](-3.0, 1.0)--(-2.8,1.4);
\draw [thin](-2.7,1.12)--(-2.5, 1.52);
\draw [thin](-2.4,1.23)--(-2.2, 1.63);
\draw [thin](-2.1,1.35)--(-1.9, 1.75);
\draw [thin](-1.8,1.47)--(-1.6, 1.87);
\draw [thin](-1.5, 1.59)--(-1.3,1.99);
\draw [thin](-1.2,1.71)--(-1.0, 2.11);

\draw [thick][blue](-3,1)--(-1.7,0.7);
\draw [thick][blue](-3,1)--(-1.7,0.4);
\draw [thick](-3,1)--(-1.9,-0.4);

\node at (-5.5, 1.7) {$\textsc{C}_{k-1}$};
\node at (-3.0, 1.5) {$\textsc{C}_{k}$};
\node at (-0.4, 1.8) {$\textsc{C}_{k+1}$};
\node at (-1.3, 0.5) {$\beta_{1}$};
\node at (-1.4, -0.5) {$\alpha_{1}$};
\node at (-1.8, 1.1) {$U^{(\tau)}_{\Gamma}$};
\node at (-2.0, 0.2) {$U_{\Gamma}$};
\node at (-2.2, -0.8) {$U_{B}$};
\node at (1, 2) {$$};

\node at (-5.0, -1.9) {$x_{k-1}$};
\node at (-3.0, -1.9) {$x_{k}$};
\node at (-0.9, -1.9) {$x_{k+1}$};
\end{tikzpicture}
\caption{Comparison of the Riemann solutions near boundary $\Gamma_h$}\label{fig4.7a}
\end{center}
\end{figure}
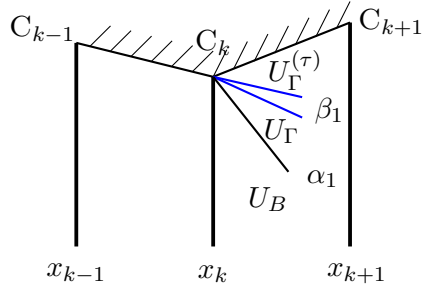

Then we establish the following lemma.

\begin{lemma}\label{lem:4.3}
Let $U_{B},\, U_{\Gamma},\, U^{(\tau)}_{\Gamma}\in \mathcal{O}_{\min\{\varepsilon_0,\tilde{\varepsilon}_0\}}(\underline{U})$
be constant states defined by \eqref{eq:4.29}--\eqref{eq:4.30} and satisfy
\begin{eqnarray}
&&U_{\Gamma}=\Phi_1(\sigma_{\alpha_1};U_{B}),\qquad U^{(\tau)}_{\Gamma}=\Phi^{(\tau)}_1(\sigma^{(\tau)}_{\beta_1};U_{B},\tau^2),\label{eq:4.31}\\
&&v_{B}=\tan(\theta_{k}),\quad  v_{\Gamma}=\tan(\theta_{k+1}), \quad (1+\tau^2u^{(\tau)}_{\Gamma}, v^{(\tau)}_{\Gamma})\cdot \mathbf{n}_{k+1}=0,\qquad\quad \label{eq:4.32}
\end{eqnarray}
where
$\mathbf{n}_{k+1}=\big(\sin(\theta_{k+1}),-\cos(\theta_{k+1})\big)$.
For a given hypersonic similarity parameter $a_{\infty}$, $\varepsilon\in(0,\min\{\varepsilon_0,\tilde{\varepsilon}_0\})$, and $\tau\in(0,\min\{\tau_0,\tilde{\tau}_0\})$,
when $|\theta_{k}|+|\omega_{k}|<\varepsilon$, then
\begin{align}
\sigma^{(\tau)}_{\beta_{1}}& =\sigma_{\alpha_1}+O(1)\big(1+|\sigma_{\alpha_1}|\big)\tau^{2}\label{eq:4.33}\\
&=O(1)(1+\tau^{2})|\omega_{k}|,\label{eq:4.34}
\end{align}
where $\omega_{k}=\theta_{k+1}-\theta_{k}$.
\end{lemma}

\begin{proof}
By \eqref{eq:4.31}--\eqref{eq:4.32}, we know that $\sigma_{\alpha_1}$ and $\sigma^{(\tau)}_{\beta_1}$
satisfy the following equation:
\begin{align}\label{eq:4.35}
&\mathcal{L}_{B}(\sigma^{(\tau)}_{\beta_{1}}, \sigma_{\alpha_{1}},\tau^2;U_{B})\nonumber\\
&\doteq\Phi^{(3)}_{1}(\sigma^{(\tau)}_{\beta_{1}}; U_{B},\tau^2)
-\big(1+\tau^2\Phi^{(\tau),(2)}_{1}(\sigma^{(\tau)}_{\beta_{1}}; U_{B},\tau^2)\big)\Phi^{(3)}_{1}(\sigma_{\alpha_{1}}; U_{B})
\nonumber\\
&=0.
\end{align}

By direct computation and using Remark \ref{rem:2.1}, we have
\begin{align*}
\frac{\partial\mathcal{L}_{B}(\sigma^{(\tau)}_{\beta_{1}}, \sigma_{\alpha_{1}},\tau^2;U_{B})}{\partial \sigma^{(\tau)}_{\beta_1}}
\Bigg|_{\sigma^{(\tau)}_{\beta_{1}}=\sigma_{\alpha_{1}}=0,\ U_{B}=\underline{U}}
&=\frac{\partial\Phi^{(3)}_{1}(\sigma^{(\tau)}_{\beta_{1}}; U_{B},\tau^2)}{\partial \sigma^{(\tau)}_{\beta_1}}
\Bigg|_{\sigma^{(\tau)}_{\beta_{1}}=0,\ U_{B}=\underline{U}}\\[4pt]
&=\mathbf{r}^{(\tau),(3)}_1(\underline{U},\tau^2)=e_{1}(\underline{U},\tau^2)\neq0
\end{align*}
for $\tau\in (0,\min\{\tau_0,\tilde{\tau}_0\})$.
Then, by the implicit function theorem, $\sigma^{(\tau)}_{\beta_1}$ can be solved uniquely from equation \eqref{eq:4.35}
as a $C^{2}$--function of $(\sigma_{\alpha_{1}},\tau^{2})$.

Notice that, when $\tau=0$, by Lemma \ref{lem:4.1} and \eqref{eq:4.35},
\begin{eqnarray*}
\Phi^{(3)}_{1}\big(\sigma^{(\tau)}_{\beta_{1}}\big|_{\tau=0}; U_{B},\tau^2\big)\Big|_{\tau=0}=\Phi^{(3)}_{1}\big(\sigma^{(\tau)}_{\beta_{1}}\big|_{\tau=0}; U_{B}\big)
=\Phi^{(3)}_{1}(\sigma_{\alpha_{1}}; U_{B}),
\end{eqnarray*}
so that
\begin{eqnarray*}
\sigma^{(\tau)}_{\beta_{1}}\big|_{\tau=0}=\sigma_{\alpha_{1}}, \quad\,\, \sigma^{(\tau)}_{\beta_{1}}\big|_{\sigma_{\alpha_{1}}=\tau=0}=0,
\quad\,\, \sigma^{(\tau)}_{\beta_{1}}\big|_{\sigma_{\alpha_{1}}=0}=O(1)\tau^2.
\end{eqnarray*}

Therefore, by applying the Taylor formula, it follows that
\begin{align*}
\sigma^{(\tau)}_{\beta_1}&=\sigma^{(\tau)}_{\beta_1}\big|_{\tau=0}+\sigma^{(\tau)}_{\beta_1}\big|_{\sigma_{\alpha_1}=0}
-\sigma^{(\tau)}_{\beta_{1}}\big|_{\sigma_{\alpha_{1}}=\tau=0}
+O(1)|\sigma_{\alpha_1}|\tau^{2}\\
&=\sigma_{\alpha_{1}}+O(1)\big(1+|\sigma_{\alpha_{1}}|\big)\tau^2,
\end{align*}
which is estimate \eqref{eq:4.33}.

Next, by Lemma A.3,
for $\varepsilon>0$ sufficiently small, we know that
\begin{eqnarray*}
\sigma_{\alpha_{1}}=K_{B}\omega_{k}
\end{eqnarray*}
with $K_{B}>0$. Then, substituting it into \eqref{eq:4.33},
we can derive estimate \eqref{eq:4.34}.
\end{proof}

Based on Lemma \ref{lem:4.3}, we have
\begin{proposition}\label{prop:4.3}
Let $\textrm{C}_k(x_k,g_k)$ with $g_k=g(x_k)$ be a corner point on the approximate boundary $\Gamma_h$.
Let $U^{(\tau)}_{\Gamma}$, $U_{\Gamma}$, $U_{B}\in\mathcal{O}_{\min\{\varepsilon_0,\tilde{\varepsilon}_0\}}(\underline{U})$.
For $s>0$,
{\rm Lemma A.3}
implies that
\begin{equation}\label{eq:4.37}
U^{\Gamma}_{h,\nu}(x_{k}+s,y)=
\begin{cases}
U_{\Gamma}, \quad  &g_{k}+\dot{y}_{\alpha_1}s<y\leq g_{k}+\tan(\theta_{k+1})s,\\[5pt]
U_{B}, \quad  &y<g_{k}+\dot{y}_{\alpha_1}s,
\end{cases}
\end{equation}
where $U_{\Gamma}\in\mathcal{O}_{\min\{\varepsilon_0,\tilde{\varepsilon}_0\}}(\underline{U})$,
and $v_{B}$ and $v_{\Gamma}$ satisfy \eqref{eq:4.32} and $|\dot{y}_{\alpha_1}|<\hat{\lambda}$ for a fixed constant $\hat{\lambda}$
satisfying
$$
\hat{\lambda}>\max_{1\leq j\leq4}\big\{\lambda^{(\tau)}_{j}(U^{(\tau)},\tau^{2}), \lambda_{j}(U)\big\}
$$
for $U^{(\tau)},U\in\mathcal{O}_{\min\{\varepsilon_0,\tilde{\varepsilon}_0\}}(\underline{U})$.
Let $\mathcal{P}^{(\tau)}_{h}$ be the Lipschtiz map generated by \emph{Problem I} obtained in {\rm Proposition \ref{prop:2.2}}.
Then, for a fixed given hypersonic similarity parameter $a_{\infty}$ and for $s>0$ sufficiently small, the following
statements hold{\rm :}

\begin{enumerate}
\item[\rm (i)] if $\omega_{k}=\theta_{k+1}-\theta_{k}>0$, then $U_{B}$ and $U_{\Gamma}$ are connected by the $1^{\rm st}$
rarefaction-front $\alpha_1\in\mathcal{R}_{1}$ with strength $\sigma_{\alpha_1}>0$ and $|\dot{y}_{\alpha_1}-\lambda_{1}(U_{\Gamma})|<2^{-\nu}$
such that
\begin{align}\label{eq:4.38}
&\int^{g_{k}+\tan(\theta_{k+1})s}_{g_{k}-\hat{\lambda}s}\big|\mathcal{P}^{(\tau)}_{h}(x_k+s, x_{k})(U^{(\tau),\Gamma}_{h,\nu}(x_{k},y))-U^{\Gamma}_{h,\nu}
(x_{k}+s,y)\big|\,{\rm d}y\nonumber\\[2pt]
&\leq {C_{3,6}}\Big(\big(1+|\omega_{k}|\big)\tau^{2}+\big(|\omega_{k}|+\tau^2\big)(\nu^{-1}+2^{-\nu})\Big)s,
\end{align}
where constant {$C_{3,6}>0$} depends only on $(\underline{U},a_{\infty})$.

\item[\rm (ii)]
if $\omega_{k}=\theta_{k+1}-\theta_{k}<0$, then $U_{B}$ and $U_{\Gamma}$ are connected by the $1^{\rm st}$ shock-front $\alpha_1\in\mathcal{S}_{1}$
with strength $\sigma_{\alpha_1}<0$ and $|\dot{y}_{\alpha_1}-\dot{\mathcal{S}}_1(\sigma_{\alpha_1})|<2^{-\nu}$
such that
\begin{align}
&\int^{g_{k}+\tan(\theta_{k+1})s}_{g_{k}-\hat{\lambda}s}\big|\mathcal{P}^{(\tau)}_{h}(x_k+s, x_{k})(U^{(\tau),\Gamma}_{h,\nu}(x_{k},y))-U^{\Gamma}_{h,\nu}
(x_{k}+s,y)\big|\,{\rm d}y\nonumber\\[2pt]
&\leq {C_{3,7}}\Big(\big(1+|\omega_{k}|\big)\tau^{2}+\big(|\omega_{k}|+\tau^2\big)2^{-\nu}\Big)s,\label{eq:4.39}
\end{align}
where
constant {$C_{3,7}>0$} depends only on $(\underline{U},a_{\infty})$.
\end{enumerate}
\end{proposition}

\begin{proof}  We divide the proof into two steps.

\smallskip
1. Case \rm{(i)}: $\omega_{k}>0$.
It follows from
Lemma A.3
that $\sigma_{\alpha_1}>0$.
Then $U_{B}$ and $U_{\Gamma}$ are connected by the $1^{\rm st}$ rarefaction-front.
Meanwhile, it follows from Proposition \ref{prop:2.2} that  equation \eqref{eq:4.35} admits
solution $P^{(\tau)}_{h}(x_{k}+s,x_{k})(U^{(\tau)}_{\Gamma}(x_{k},y))$.
Moreover, for $s>0$ sufficiently small, based on the construction of the Lipschitz continuous
map $\mathcal{P}^{(\tau)}_{h}$, we know
\begin{align}\label{eq:4.39a}
&\mathcal{P}^{(\tau)}_{h}(x_k+s,x_{k})(U^{(\Gamma)}_{h,\nu}(x_{k},y))\nonumber\\
&=
\begin{cases}
U^{(\tau)}_{\Gamma}, \  &\zeta\in \big [\lambda^{(\tau)}_{1}(U^{(\tau)}_{\Gamma},\tau^2), \,\tan(\theta_{k+1})\big),\\[2pt]
\Phi^{(\tau)}_{1}(\sigma^{(\tau)}_{\beta_1}(\zeta); U_{B}, \tau^2),
\  &\zeta\in \big[\lambda^{(\tau)}_{1}(U_{B},\tau^2),\,\lambda^{(\tau)}_{1}(U^{(\tau)}_{\Gamma},\tau^2) \big),\\[2pt]
U_{B}, \  &\zeta\in(-\hat{\lambda},\,\lambda^{(\tau)}_{1}(U_{B},\tau^{2})),
\end{cases}
\end{align}
where $\zeta=\frac{y-g_{k}}{s}$,
\begin{eqnarray}\label{eq:4.39b}
\sigma^{(\tau)}_{\beta_1}(\lambda^{(\tau)}_{1}(U_{B},\tau^2))=0,  \qquad
\sigma^{(\tau)}_{\beta_1}(\lambda^{(\tau)}_{1}(U^{(\tau)}_{\Gamma},\tau^2))=\sigma^{(\tau)}_{\beta_1}.
\end{eqnarray}
Moreover, by Lemma \ref{lem:4.3}, $\sigma^{(\tau)}_{\beta_1}$ satisfies estimate \eqref{eq:4.33} and $\sigma^{(\tau)}_{\beta_1}>0$ when $\tau$ is sufficiently small.
Thus, it is also a rarefaction-front of the $1^{\rm st}$ family.

Now, to derive estimate \eqref{eq:4.38},
we rewrite its left-hand side as
{\small
\begin{align*}
&\int^{g_{k}+\tan(\theta_{k+1})s}_{g_{k}-\hat{\lambda}s}\big|\mathcal{P}^{(\tau)}_{h}(x_k+s, x_{k})(U^{(\tau),\Gamma}_{h,\nu}(x_{k},y))-U^{\Gamma}_{h,\nu}
 (x_{k}+s,y)\big|\,{\rm d}y\\[2pt]
&=\int^{\min\{g_{k}+\dot{y}_{\alpha_1}s,\,g_{k}+\lambda^{(\tau)}_{1}(U_{B},\tau^{2})s\}}_{g_{k}-\hat{\lambda}s}
\big|\mathcal{P}^{(\tau)}_{h}(x_k+s, x_{k})(U^{(\tau),\Gamma}_{h,\nu}(x_{k},y))-U^{\Gamma}_{h,\nu}
 (x_{k}+s,y)\big|\,{\rm d}y\\[2pt]
&\quad+\int^{\max\{g_{k}+\dot{y}_{\alpha_1}s,\,g_{k}+\lambda^{(\tau)}_{1}(U^{(\tau)}_{\Gamma},\tau^2)s\}}
_{\min\{g_{k}+\dot{y}_{\alpha_1}s,\,g_{k}+\lambda^{(\tau)}_{1}(U_{B},\tau^{2})s\}}
\big|\mathcal{P}^{(\tau)}_{h}(x_k+s, x_{k})(U^{(\tau),\Gamma}_{h,\nu}(x_{k},y))-U^{\Gamma}_{h,\nu}(x_{k}+s,y)\big|\,{\rm d}y\\[2pt]
&\quad+\int^{g_{k}+\tan(\theta_{k+1})s}_{\max\{g_{k}+\dot{y}_{\alpha_1}s,\,g_{k}+\lambda^{(\tau)}_{1}(U^{(\tau)}_{\Gamma},\tau^2)s\}}
\big|\mathcal{P}^{(\tau)}_{h}(x_k+s, x_{k})(U^{(\tau),\Gamma}_{h,\nu}(x_{k},y))-U^{\Gamma}_{h,\nu}
 (x_{k}+s,y)\big|\,{\rm d}y\\[2pt]
&\doteq C_{\mathcal{R}_1,1}+C_{\mathcal{R}_1,2}+C_{\mathcal{R}_1,3}.
\end{align*}
}

For $C_{\mathcal{R}_1,1}$, it is clear that
$$
C_{\mathcal{R}_1,1}=0.
$$
For $C_{\mathcal{R}_1,3}$, by direct computation, we have
\begin{eqnarray*}
C_{\mathcal{R}_1,3}
\leq C|U^{(\tau)}_{\Gamma}-U_{\Gamma}|s
=C\big|\Phi_{1}(\sigma^{(\tau)}_{\beta_1}; U_{B},\tau^{2})-\Phi(\sigma_{\alpha_1}; U_{B})\big|s.
\end{eqnarray*}

Define
\begin{eqnarray*}
\mathcal{E}_{\Gamma}(\sigma_{\alpha_1},\tau^{2})\doteq
\Phi_{1}(\sigma^{(\tau)}_{\beta_1}; U_{B},\tau^{2})-\Phi(\sigma_{\alpha_1}; U_{B}).
\end{eqnarray*}
By Lemma \ref{lem:4.3}, we know that $\sigma^{(\tau)}_{\beta_1}$ is a function
of $\sigma_{\alpha_1}$. Note that $\mathcal{E}_{\Gamma}(\sigma_{\alpha_1},0)=\mathcal{E}_{\Gamma}(0,0)=0$ and $\mathcal{E}_{\Gamma}(0,\tau^{2})=O(1)\tau^2$.
Then, by the Taylor formula,
\begin{eqnarray}\label{eq:4.39c}
\mathcal{E}_{\Gamma}(\sigma_{\alpha_1},\tau^2)=O(1)\big(1+|\sigma_{\alpha_1}|\big)\tau^{2}=O(1)\big(1+|\omega_{k}|\big)\tau^{2},
\end{eqnarray}
so that
\begin{eqnarray*}
C_{\mathcal{R}_1,3}\leq C\big(1+|\omega_{k}|\big)\tau^2.
\end{eqnarray*}

For $C_{\mathcal{R}_1,2}$, by Lemma \ref{lem:4.3}, we have
\begin{align}\label{eq:4.40}
\lambda^{(\tau)}_{1}(U^{(\tau)}_{\Gamma}, \tau^{2})-\lambda_{1}(U_{\Gamma})
&=\lambda^{(\tau)}_{1}(\Phi^{(\tau)}_{1}(\sigma^{(\tau)}_{\beta_1}; U_{B},\tau^2),\tau^2)-\lambda_{1}(\Phi_{1}(\sigma_{\alpha_1}; U_{B}))\nonumber\\[2pt]
&=\lambda^{(\tau)}_{1}(U_{B},\tau^2)-\lambda_{1}(U_{B})+O(1)\big(\sigma^{(\tau)}_{\beta_1}-\sigma_{\alpha_1}\big)\nonumber\\[2pt]
&=O(1)(1+|\sigma_{\alpha_1}|)\tau^2\nonumber\\[2pt]
&=O(1)(1+|\omega_{k}|)\tau^2,\
\end{align}
and
\begin{align}\label{eq:4.41}
\lambda_{1}(U_{\Gamma})-\lambda^{(\tau)}_{1}(U_{B},\tau^2)
&=\lambda_{1}(\Phi_{1}(\sigma_{\alpha_1}; U_{B}))-\lambda_{1}(U_{B})+\lambda_{1}(U_{B})-\lambda^{(\tau)}_{1}(U_{B},\tau^2)\nonumber\\[2pt]
&=O(1)(|\sigma_{\alpha_1}|+\tau^2)\nonumber\\[2pt]
&=O(1)\big(|\omega_{k}|+\tau^{2}\big),
\end{align}
where we have used \eqref{eq:3.5a}
for $U=U_{B}$.

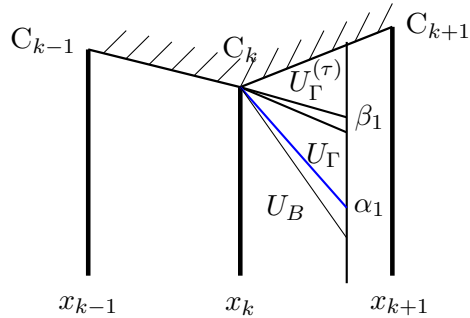
\begin{figure}[ht]
\begin{center}
\begin{tikzpicture}[scale=1.0]
\draw [line width=0.03cm](-5.0,1.5)--(-3,1)--(-1,1.8);

\draw [line width=0.06cm](-5.0,1.5)--(-5.0,-1.5);
\draw [line width=0.06cm](-3,1)--(-3,-1.5);
\draw [line width=0.06cm](-1.0,1.8)--(-1.0,-1.5);
\draw [thick](-1.6,1.6)--(-1.6,-1.6);

\draw [thin](-4.8,1.45)--(-4.5, 1.75);
\draw [thin](-4.5,1.38)--(-4.2, 1.68);
\draw [thin](-4.2,1.30)--(-3.9, 1.60);
\draw [thin](-3.9, 1.23)--(-3.6,1.53);
\draw [thin](-3.6, 1.16)--(-3.3,1.46);
\draw [thin](-3.3, 1.08)--(-3.0,1.38);
\draw [thin](-3.0, 1.0)--(-2.8,1.4);
\draw [thin](-2.7,1.12)--(-2.5, 1.52);
\draw [thin](-2.4,1.23)--(-2.2, 1.63);
\draw [thin](-2.1,1.35)--(-1.9, 1.75);
\draw [thin](-1.8,1.47)--(-1.6, 1.87);
\draw [thin](-1.5, 1.59)--(-1.3,1.99);
\draw [thin](-1.2,1.71)--(-1.0, 2.11);

\draw [thick](-3,1)--(-1.6,0.6);
\draw [thick](-3,1)--(-1.6,0.4);
\draw [thick][blue](-3,1)--(-1.6,-0.6);
\draw [thin](-3,1)--(-1.6,-1.0);

\node at (-5.6, 1.6) {$\textsc{C}_{k-1}$};
\node at (-3.0, 1.5) {$\textsc{C}_{k}$};
\node at (-0.4, 1.8) {$\textsc{C}_{k+1}$};
\node at (-1.3, 0.6) {$\beta_{1}$};
\node at (-1.3, -0.6) {$\alpha_{1}$};
\node at (-2.0, 1.1) {$U^{(\tau)}_{\Gamma}$};
\node at (-1.9, 0.1) {$U_{\Gamma}$};
\node at (-2.4, -0.6) {$U_{B}$};
\node at (1, 2) {$$};

\node at (-5.0, -1.9) {$x_{k-1}$};
\node at (-3.0, -1.9) {$x_{k}$};
\node at (-0.9, -1.9) {$x_{k+1}$};
\end{tikzpicture}
\caption{Case: $\omega_{k}>0$ and $\dot{y}_{\alpha_1}<\lambda^{{\tau}}_{1}(U_{B},\tau^2)$}\label{fig5.7}
\end{center}
\end{figure}

If $\dot{y}_{\alpha_1}<\lambda^{(\tau)}_{1}(U_{B},\tau^2)$ (see Fig. \ref{fig5.7}),
then $C_{\mathcal{R}_1, 2} $ can be rewritten as
\begin{align*}
C_{\mathcal{R}_1,2}&=\int^{g_{k}+\lambda^{(\tau)}_{1}(U^{(\tau)}_{\Gamma},\tau^2)s}_{g_{k}+\dot{y}_{\alpha_1}s}
\big|\mathcal{P}^{(\tau)}_{h}(x_k+s, x_{k})(U^{(\tau),\Gamma}_{h,\nu}(x_{k},y))-U^{\Gamma}_{h,\nu}(x_{k}+s,y)\big|\,{\rm d}y\\[2pt]
&=\int^{g_{k}+\lambda^{(\tau)}_{1}(U_{B},\tau^2)s}_{g_{k}+\dot{y}_{\alpha_1}s}
\big|\mathcal{P}^{(\tau)}_{h}(x_k+s, x_{k})(U^{(\tau),\Gamma}_{h,\nu}(x_{k},y))-U^{\Gamma}_{h,\nu}(x_{k}+s,y)\big|\,{\rm d}y\\[2pt]
&\quad +\int^{g_{k}+\lambda^{(\tau)}_{1}(U^{(\tau)}_{\Gamma},\tau^2)s}_{g_{k}+\lambda^{(\tau)}_{1}(U_{B},\tau^2)s}
\big|\mathcal{P}^{(\tau)}_{h}(x_k+s, x_{k})(U^{(\tau),\Gamma}_{h,\nu}(x_{k},y))-U^{\Gamma}_{h,\nu}(x_{k}+s,y)\big|\,{\rm d}y.
\end{align*}
First of all, using \eqref{eq:4.41} and Lemma \ref{lem:4.3}, we have
\begin{eqnarray*}
&&\int^{g_{k}+\lambda^{(\tau)}_{1}(U_{B},\tau^2)s}_{g_{k}+\dot{y}_{\alpha_1}s}
\big|\mathcal{P}^{(\tau)}_{h}(x_k+s, x_{k})(U^{(\tau),\Gamma}_{h,\nu}(x_{k},y))-U^{\Gamma}_{h,\nu}(x_{k}+s,y)\big|\,{\rm d}y\\[2pt]
&&\quad \leq \big|\lambda^{(\tau)}_{1}(U_{B},\tau^2)-\dot{y}_{\alpha_1}\big||U_{\Gamma}-U_{B}|s\\[5pt]
&&\quad \leq C\Big(\big|\lambda^{(\tau)}_{1}(U_{B},\tau^2)-\lambda_{1}(U_{\Gamma})\big|+2^{-\nu}\Big)|\sigma_{\alpha_1}|s\\[2pt]
&&\quad \leq C\big(\tau^2+2^{-\nu}\big)|\omega_{k}|s.
\end{eqnarray*}
Then, using the properties of the rarefaction wave $\beta_1$ and \eqref{eq:4.39c}--\eqref{eq:4.41}, we can also obtain
{\small
\begin{align*}
&\int^{g_{k}+\lambda^{(\tau)}_{1}(U^{(\tau)}_{\Gamma},\tau^2)s}_{g_{k}+\lambda^{(\tau)}_{1}(U_{B},\tau^2)s}
\big|\mathcal{P}^{(\tau)}_{h}(x_k+s, x_{k})(U^{(\tau),\Gamma}_{h,\nu}(x_{k},y))-U^{\Gamma}_{h,\nu}(x_{k}+s,y)\big|\,{\rm d}y\\[2pt]
&\leq \big|\lambda^{(\tau)}_{1}(U^{(\tau)}_{\Gamma},\tau^2)-\lambda^{(\tau)}_{1}(U_{B},\tau^2)
\big|\big|\Phi^{(\tau)}_{1}(\sigma^{(\tau)}_{\beta_1}(\zeta); U_{B}, \tau^2)-\Phi_{1}(\sigma_{\alpha_1}; U_{B})\big|s\\[2pt]
&\leq C\,\Big|\Phi^{(\tau)}_{1}(\sigma^{(\tau)}_{\beta_1}; U_{B}, \tau^2)-U_{B}\Big|\\[2pt]
&\quad\,\times\Big(|\Phi^{(\tau)}_{1}(\sigma^{(\tau)}_{\beta_1}(\zeta); U_{B}, \tau^2)-\Phi^{(\tau)}_{1}(\sigma^{(\tau)}_{\beta_1}; U_{B},\tau^2)|
+|\Phi^{(\tau)}_{1}(\sigma^{(\tau)}_{\beta_1}; U_{B},\tau^2)-\Phi_{1}(\sigma_{\alpha_1}; U_{B})|\Big)\\[2pt]
&\leq C\big(|\omega_{k}|\tau^2+(|\omega_{k}|+\tau^2)\nu^{-1}\big)s.
\end{align*}
}

Therefore, it follows the three estimates above that
\begin{eqnarray*}
C_{\mathcal{R}_1,2}\leq C\big(|\omega_{k}|\tau^2+(|\omega_{k}|+\tau^2)(\nu^{-1}+2^{-\nu})\big)s.
\end{eqnarray*}

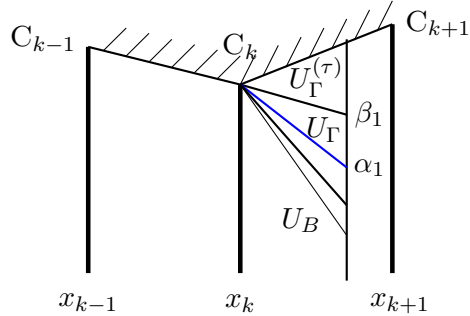
\begin{figure}[ht]
\begin{center}
\begin{tikzpicture}[scale=1.0]
\draw [thick](-5.0,1.5)--(-3,1)--(-1,1.8);

\draw [line width=0.06cm](-5.0,1.5)--(-5.0,-1.5);
\draw [line width=0.06cm](-3,1)--(-3,-1.5);
\draw [line width=0.06cm](-1.0,1.8)--(-1.0,-1.5);
\draw [thick](-1.6,1.6)--(-1.6,-1.6);

\draw [thin](-4.8,1.45)--(-4.5, 1.75);
\draw [thin](-4.5,1.38)--(-4.2, 1.68);
\draw [thin](-4.2,1.30)--(-3.9, 1.60);
\draw [thin](-3.9, 1.23)--(-3.6,1.53);
\draw [thin](-3.6, 1.16)--(-3.3,1.46);
\draw [thin](-3.3, 1.08)--(-3.0,1.38);
\draw [thin](-3.0, 1.0)--(-2.8,1.4);
\draw [thin](-2.7,1.12)--(-2.5, 1.52);
\draw [thin](-2.4,1.23)--(-2.2, 1.63);
\draw [thin](-2.1,1.35)--(-1.9, 1.75);
\draw [thin](-1.8,1.47)--(-1.6, 1.87);
\draw [thin](-1.5, 1.59)--(-1.3,1.99);
\draw [thin](-1.2,1.71)--(-1.0, 2.11);

\draw [thick](-3,1)--(-1.6,0.6);
\draw [thick][blue](-3,1)--(-1.6,-0.1);
\draw [thick](-3,1)--(-1.6,-0.6);
\draw [thin](-3,1)--(-1.6,-1.0);

\node at (-5.6, 1.6) {$\textsc{C}_{k-1}$};
\node at (-3.0, 1.5) {$\textsc{C}_{k}$};
\node at (-0.4, 1.8) {$\textsc{C}_{k+1}$};
\node at (-1.3, 0.6) {$\beta_{1}$};
\node at (-1.3, -0.1) {$\alpha_{1}$};
\node at (-2.0, 1.1) {$U^{(\tau)}_{\Gamma}$};
\node at (-1.9, 0.4) {$U_{\Gamma}$};
\node at (-2.2, -0.8) {$U_{B}$};
\node at (1, 2) {$$};

\node at (-5.0, -1.9) {$x_{k-1}$};
\node at (-3.0, -1.9) {$x_{k}$};
\node at (-0.9, -1.9) {$x_{k+1}$};
\end{tikzpicture}
\caption{Case: $\omega_{k}>0$ and $\lambda^{(\tau)}_{1}(U_{B},\tau^2)<\dot{y}_{\alpha_1}\leq\lambda^{(\tau)}_{1}(U^{(\tau)}_{\Gamma},\tau^2)$}\label{fig5.8}
\end{center}
\end{figure}

If $\lambda^{(\tau)}_{1}(U_{B},\tau^2)<\dot{y}_{\alpha_1}\leq\lambda^{(\tau)}_{1}(U^{(\tau)}_{\Gamma},\tau^2)$ (see Fig. \ref{fig5.8}), then
\begin{align*}
C_{\mathcal{R}_1,2}&=\int^{g_{k}+\lambda^{(\tau)}_{1}(U^{(\tau)}_{\Gamma},\tau^2)s}_{g_{k}+\lambda^{(\tau)}_{1}(U_{B},\tau^2)s}
\big|\mathcal{P}^{(\tau)}_{h}(x_k+s, x_{k})(U^{(\tau),\Gamma}_{h,\nu}(x_{k},y))-U^{\Gamma}_{h,\nu}(x_{k}+s,y)\big|\,{\rm d}y\\[2pt]
&=\int^{g_{k}+\dot{y}_{\alpha_1}s}_{g_{k}+\lambda^{(\tau)}_{1}(U_{B},\tau^2)s}\big|\mathcal{P}^{(\tau)}_{h}(x_k+s, x_{k})(U^{(\tau),\Gamma}_{h,\nu}(x_{k},y))
  -U^{\Gamma}_{h,\nu}(x_{k}+s,y)\big|\,{\rm d}y\\[2pt]
&\quad +\int^{g_{k}+\lambda^{(\tau)}_{1}(U^{(\tau)}_{\Gamma},\tau^2)s}_{g_{k}+\dot{y}_{\alpha_1}s}
\big|\mathcal{P}^{(\tau)}_{h}(x_k+s, x_{k})(U^{(\tau),\Gamma}_{h,\nu}(x_{k},y))-U^{\Gamma}_{h,\nu}(x_{k}+s,y)\big|\,{\rm d}y.
\end{align*}

Then, by Lemma \ref{eq:4.2},
Proposition A.1,
and estimates \eqref{eq:4.39c}--\eqref{eq:4.41},
\begin{eqnarray*}
&&\int^{g_{k}+\dot{y}_{\alpha_1}s}_{g_{k}+\lambda^{(\tau)}_{1}(U_{B},\tau^2)s}
 \big|\mathcal{P}^{(\tau)}_{h}(x_k+s, x_{k})(U^{(\tau),\Gamma}_{h,\nu}(x_{k},y))-U^{\Gamma}_{h,\nu}(x_{k}+s,y)\big|{\rm d}y\\[2pt]
&&=\big|\dot{y}_{\alpha_1}-\lambda^{(\tau)}_{1}(U_{B},\tau^2)\big|\big|\Phi^{(\tau)}_{1}(\sigma_{\beta_1}(\zeta);U_{B},\tau^2)-U_{B}\big|s\\[2pt]
&& \leq C\Big(\big|\lambda^{(\tau)}_{1}(U_{\Gamma})-\lambda^{(\tau)}_{1}(U_{B},\tau^2)\big|+2^{-\nu}\Big)|\sigma^{(\tau)}_{\beta_1}|s\\[2pt]
&& \leq C\big( |\omega_k|\tau^2+(|\omega_k|+\tau^2)(\nu^{-1}+2^{-\nu})\big)s,
\end{eqnarray*}
and
\begin{eqnarray*}
&&\int^{g_{k}+\lambda^{(\tau)}_{1}(U^{(\tau)}_{\Gamma},\tau^2)s}_{g_{k}+\dot{y}_{\alpha_1}s}
\big|\mathcal{P}^{(\tau)}_{h}(x_k+s, x_{k})(U^{(\tau),\Gamma}_{h,\nu}(x_{k},y))-U^{\Gamma}_{h,\nu}(x_{k}+s,y)\big|\,{\rm d}y\\[2pt]
&&= \big(\lambda^{(\tau)}_{1}(U^{(\tau)}_{\Gamma},\tau^2)-\dot{y}_{\alpha_1}\big)
\big|\Phi^{(\tau)}_{1}(\sigma_{\beta_1}(\zeta);U_{B},\tau^2)-U_{\Gamma}\big|s\\[2pt]
&&\leq C\Big(\big|\lambda^{(\tau)}_{1}(U^{(\tau)}_{\Gamma},\tau^2)-\lambda_{1}(U_{\Gamma})\big|
+2^{-\nu}\Big)\big|\Phi^{(\tau)}_{1}(\sigma^{(\tau)}_{\beta_1}(\zeta);U_{B},\tau^2)-\Phi_{1}(\sigma_{\alpha_1};U_{B})\big|s\\[2pt]
&&\leq C\big((1+|\omega_k|)\tau^2+2^{-\nu}\big)\big(|\sigma^{(\tau)}_{\beta_1}|+ (1+|\omega_k|)\tau^2\big)s\\[2pt]
&&\leq C\big( |\omega_k|\tau^2+(|\omega_k|+\tau^2)(\nu^{-1}+2^{-\nu})\big)s.
\end{eqnarray*}
Then
\begin{equation*}
C_{\mathcal{R}_1,2}\leq C\big(|\omega_{k}|\tau^2+(|\omega_{k}|+\tau^2)(\nu^{-1}+2^{-\nu})\big)s.
\end{equation*}

\begin{figure}[ht]
\begin{center}
\begin{tikzpicture}[scale=1.0]
\draw [line width=0.03cm](-5.0,1.5)--(-3,1)--(-1,1.8);

\draw [line width=0.06cm](-5.0,1.5)--(-5.0,-1.5);
\draw [line width=0.06cm](-3,1)--(-3,-1.5);
\draw [line width=0.06cm](-1.0,1.8)--(-1.0,-1.5);
\draw [thick](-1.6,1.6)--(-1.6,-1.6);

\draw [thin](-4.8,1.45)--(-4.5, 1.75);
\draw [thin](-4.5,1.38)--(-4.2, 1.68);
\draw [thin](-4.2,1.30)--(-3.9, 1.60);
\draw [thin](-3.9, 1.23)--(-3.6,1.53);
\draw [thin](-3.6, 1.16)--(-3.3,1.46);
\draw [thin](-3.3, 1.08)--(-3.0,1.38);
\draw [thin](-3.0, 1.0)--(-2.8,1.4);
\draw [thin](-2.7,1.12)--(-2.5, 1.52);
\draw [thin](-2.4,1.23)--(-2.2, 1.63);
\draw [thin](-2.1,1.35)--(-1.9, 1.75);
\draw [thin](-1.8,1.47)--(-1.6, 1.87);
\draw [thin](-1.5, 1.59)--(-1.3,1.99);
\draw [thin](-1.2,1.71)--(-1.0, 2.11);

\draw [thick][blue](-3,1)--(-1.6,0.6);
\draw [thick](-3,1)--(-1.6,-0.1);
\draw [thick](-3,1)--(-1.6,-0.6);
\draw [thin](-3,1)--(-1.6,-1.0);

\node at (-5.6, 1.6) {$\textsc{C}_{k-1}$};
\node at (-3.0, 1.5) {$\textsc{C}_{k}$};
\node at (-0.4, 1.8) {$\textsc{C}_{k+1}$};
\node at (-1.3, 0.5) {$\alpha_{1}$};
\node at (-1.3, -0.3) {$\beta_{1}$};
\node at (-2.0, 1.1) {$U^{(\tau)}_{\Gamma}$};
\node at (-1.8, 0.4) {$U_{\Gamma}$};
\node at (-2.0, -1.0) {$U_{B}$};
\node at (1, 2) {$$};

\node at (-5.0, -1.9) {$x_{k-1}$};
\node at (-3.0, -1.9) {$x_{k}$};
\node at (-0.9, -1.9) {$x_{k+1}$};
\end{tikzpicture}
\caption{Case: $\omega_{k}>0$ and $\dot{y}_{\alpha_1}>\lambda^{(\tau)}_{1}(U^{(\tau)}_{\Gamma},\tau^2)$}\label{fig5.9}
\end{center}
\end{figure}
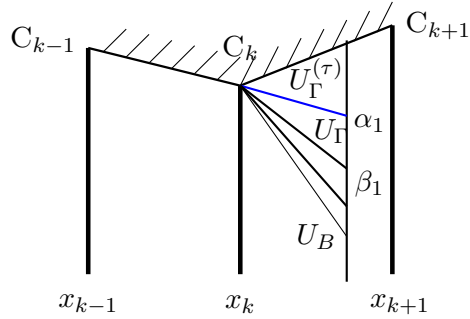

If $\dot{y}_{\alpha_1}>\lambda^{(\tau)}_{1}(U^{(\tau)}_{\Gamma},\tau^2)$ (see Fig. \ref{fig5.9}), then
\begin{align*}
C_{\mathcal{R}_2,3}&=\int^{g_{k}+\dot{y}_{\alpha_1}s}_{g_{k}+\lambda^{(\tau)}_{1}(U_{B},\tau^2)s}
\big|\mathcal{P}^{(\tau)}_{h}(x_k+s, x_{k})(U^{(\tau),\Gamma}_{h,\nu}(x_{k},y))-U^{\Gamma}_{h,\nu}(x_{k}+s,y)\big|\,{\rm d}y\\[2pt]
&=\int^{g_{k}+\lambda^{(\tau)}_1(U^{(\tau)}_{\Gamma}, \tau^2)s}_{g_{k}+\lambda^{(\tau)}_{1}(U_{B},\tau^2)s}
\big|\mathcal{P}^{(\tau)}_{h}(x_k+s, x_{k})(U^{(\tau),\Gamma}_{h,\nu}(x_{k},y))-U^{\Gamma}_{h,\nu}(x_{k}+s,y)\big|\,{\rm d}y\\[2pt]
&\quad+\int^{g_{k}+\dot{y}_{\alpha_1}s}_{g_{k}+\lambda^{(\tau)}_{1}(U^{(\tau)}_{\Gamma},\tau^2)s}
\big|\mathcal{P}^{(\tau)}_{h}(x_k+s, x_{k})(U^{(\tau),\Gamma}_{h,\nu}(x_{k},y))-U^{\Gamma}_{h,\nu}(x_{k}+s,y)\big|\,{\rm d}y.
\end{align*}

Using estimate \eqref{eq:4.41}, Lemma \ref{lem:4.2}, and
Proposition A.1,
we have
\begin{eqnarray*}
&&\int^{g_{k}+\lambda^{(\tau)}_1(U^{(\tau)}_{\Gamma}, \tau^2)s}_{g_{k}+\lambda^{(\tau)}_{1}(U_{B},\tau^2)s}
\big|\mathcal{P}^{(\tau)}_{h}(x_k+s, x_{k})(U^{(\tau),\Gamma}_{h,\nu}(x_{k},y))-U^{\Gamma}_{h,\nu}(x_{k}+s,y)\big|\,{\rm d}y\\[2pt]
&&\quad \leq \big|\lambda^{(\tau)}_1(U^{(\tau)}_{\Gamma}, \tau^2)-\lambda^{(\tau)}_{1}(U_{B},\tau^2)\big|\big|\Phi^{(\tau)}_{1}(\sigma^{(\tau)}_{\beta_1}; U_{B},\tau^2)-U_{B}\big|s\\[2pt]
&&\quad \leq C\,|\sigma^{(\tau)}_{\beta_1}|^2s\\[2pt]
&&\quad\leq C\big( \tau^2+|\omega_k|\big)\nu^{-1}s,
\end{eqnarray*}
and
\begin{eqnarray*}
&&\int^{g_{k}+\dot{y}_{\alpha_1}s}_{g_{k}+\lambda^{(\tau)}_{1}(U^{(\tau)}_{\Gamma},\tau^2)s}
\big|\mathcal{P}^{(\tau)}_{h}(x_k+s, x_{k})(U^{(\tau),\Gamma}_{h,\nu}(x_{k},y))-U^{\Gamma}_{h,\nu}(x_{k}+s,y)\big|\,{\rm d}y\\[2pt]
&&\quad \leq \big|\dot{y}_{\alpha_1}-\lambda^{(\tau)}_{1}(U^{(\tau)}_{\Gamma},\tau^2)\big|
\big|\Phi^{(\tau)}_1(\sigma^{(\tau)}_{\beta_1};U_{B},\tau^2)-U_{B}\big|s\\[2pt]
&&\quad \leq C\big(|\lambda^{(\tau)}_{1}(U^{(\tau)}_{\Gamma},\tau^2)-\lambda_{1}(U_{\Gamma})|+2^{-\nu}\big)|\sigma^{(\tau)}_{\beta_1}|s\\[2pt]
&&\quad\leq C\big((1+|\omega_k|)\tau^2+2^{-\nu}\big)(1+\tau^2)|\omega_k|s\\[2pt]
&&\quad \leq C\big( |\omega_k|\tau^2+(|\omega_k|+\tau^2)(\nu^{-1}+2^{-\nu})\big)s.
\end{eqnarray*}

Therefore, we obtain
\begin{equation*}
C_{\mathcal{R}_1,2}\leq C\big(|\omega_k|\tau^2+(|\omega_k|+\tau^2)(\nu^{-1}+2^{-\nu})\big)s.
\end{equation*}
Then
we can choose a constant {$C_{3,6}>0$} depending only on $(\underline{U},a_{\infty})$ so that estimate \eqref{eq:4.38} holds.

\smallskip
2. Case \rm{(ii)}: $\omega_{k}<0$.
By Lemma A.3,
we know that $\alpha_1\in \mathcal{S}_{1}$ with
strength $\sigma_{\alpha_1}<0$.
It follows from Lemma \ref{lem:4.2} that equation \eqref{eq:4.35} admits a unique
solution $\sigma^{(\tau)}_{\beta_1}$ corresponding to the Riemann solution
$\mathcal{P}^{(\tau)}_{h}(x_{k}+s,x_{k})(U^{(\tau)}_{\Gamma}(x_{k},y))$ obtained by Proposition \ref{prop:2.2} satisfying
$\sigma^{(\tau)}_{\beta_1}<0$ and estimates \eqref{eq:4.33}--\eqref{eq:4.34}; see also Fig. \ref{fig4.8} below.
Denoted by $\dot{\mathcal{S}}^{(\tau)}_{1}(\sigma^{(\tau)}_{\beta_1},\tau^2)$ the speed of the $1^{\rm st}$ shock-front $\beta_{1}$.
\begin{figure}[ht]
\begin{center}
\begin{tikzpicture}[scale=1.0]
\draw [thick](-5.0,1.5)--(-3,1);
\draw [thick](-7.0,1)--(-5,1.5);

\draw [line width=0.06cm](-7.0,1)--(-7.0,-1.5);
\draw [line width=0.06cm](-5.0,1.5)--(-5.0,-1.5);
\draw [line width=0.06cm](-3,1)--(-3,-1.5);
\draw[thick](-3.5,1.16)--(-3.5,-1.5);

\draw [thin](-7.0,1.00)--(-6.7, 1.35);
\draw [thin](-6.7,1.08)--(-6.4, 1.43);
\draw [thin](-6.4,1.16)--(-6.1, 1.50);
\draw [thin](-6.1,1.23)--(-5.8, 1.57);
\draw [thin](-5.8,1.30)--(-5.5, 1.61);
\draw [thin](-5.5,1.38)--(-5.2, 1.69);
\draw [thin](-5.2,1.45)--(-4.9, 1.75);
\draw [thin](-4.8,1.45)--(-4.5, 1.75);
\draw [thin](-4.5,1.38)--(-4.2, 1.68);
\draw [thin](-4.2,1.30)--(-3.9, 1.60);
\draw [thin](-3.9, 1.23)--(-3.6,1.53);
\draw [thin](-3.6, 1.16)--(-3.3,1.46);
\draw [thin](-3.3, 1.08)--(-3.0,1.38);

\draw [thick](-5.0,1.5)--(-3.5, 0.2);
\draw [thick][red](-5.0,1.5)--(-3.5,-0.8);
\draw [thin](-5,1.5)--(-3.5,-1.3);

\node at (-7.3, 1.3) {$\textsc{C}_{k-1}$};
\node at (-5.0, 1.8) {$\textsc{C}_{k}$};
\node at (-2.5, 1.2) {$\textsc{C}_{k+1}$};
\node at (-3.2, 0.2) {$\beta_{1}$};
\node at (-3.2, -0.8) {$\alpha_{1}$};
\node at (-3.3, -1.2) {$\hat{\lambda}$};
\node at (-3.9, 0.9) {$U^{(\tau)}_{\Gamma}$};
\node at (-3.7, 0) {$U_{\Gamma}$};
\node at (-4.5, -0.6) {$U_{B}$};

\node at (-7.0, -1.9) {$x_{k-1}$};
\node at (-5.0, -1.9) {$x_{k}$};
\node at (-3.0, -1.9) {$x_{k+1}$};
\end{tikzpicture}
\caption{Comparison of the Riemann solvers for the case: $\omega_{k}<0$}\label{fig4.8}
\end{center}
\end{figure}
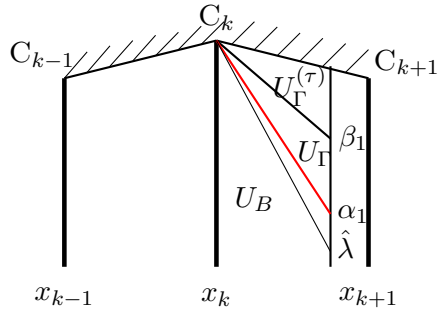

First, we have
{\small
\begin{align*}
&\int^{g_{k}+\tan(\theta_{k+1})s}_{g_{k}-\hat{\lambda}h}\big|\mathcal{P}^{(\tau)}_{h}(x_k+s, x_{k})(U^{(\tau),\Gamma}_{h,\nu}(x_{k},y))-U^{\Gamma}_{h,\nu}(x_{k}+s,y)\big|\,{\rm d}y\\[2pt]
&=\int^{\min\{g_{k}+\dot{y}_{\alpha_1}s,\,g_{k}+\dot{\mathcal{S}}^{(\tau)}_{1}(\sigma^{(\tau)}_{\beta_1},\tau^2)s\}}_{g_{k}-\hat{\lambda}s}
\big|\mathcal{P}^{(\tau)}_{h}(x_k+s, x_{k})(U^{(\tau),\Gamma}_{h,\nu}(x_{k},y))-U^{\Gamma}_{h,\nu}(x_{k}+s,y)\big|\,{\rm d}y\\[2pt]
&\quad +\int^{\max\{g_{k}+\dot{y}_{\alpha_1}s,\,g_{k}+\dot{\mathcal{S}}^{(\tau)}_{1}(\sigma^{(\tau)}_{\beta_1},\tau^2)s\}}_{\min\{g_{k}+\dot{y}_{\alpha_1}s, g_{k}+\dot{\mathcal{S}}^{(\tau)}_{1}(\sigma^{(\tau)}_{\beta_1},\tau^2)s\}}
\big|\mathcal{P}^{(\tau)}_{h}(x_k+s, x_{k})(U^{(\tau),\Gamma}_{h,\nu}(x_{k},y))-U^{\Gamma}_{h,\nu}(x_{k}+s,y)\big|\,{\rm d}y\\[2pt]
&\quad +\int^{g_{k}+\tan(\theta_{k+1})s}_{\max\{g_{k}+\dot{y}_{\alpha_1}s,\,g_{k}+\dot{\mathcal{S}}^{(\tau)}_{1}(\sigma^{(\tau)}_{\beta_1},\tau^2)s\}}
\big|\mathcal{P}^{(\tau)}_{h}(x_k+s, x_{k})(U^{(\tau),\Gamma}_{h,\nu}(x_{k},y))-U^{\Gamma}_{h,\nu}
(x_{k}+s,y)\big|\,{\rm d}y\\[2pt]
&\doteq C_{\mathcal{S}_1,1}+C_{\mathcal{S}_1,2}+C_{\mathcal{S}_1,3}.
\end{align*}
}

It is clear that
$$
C_{\mathcal{S}_1,1}=0.
$$
For $C_{\mathcal{S}_1,2}$, by Lemma \ref{lem:4.2}, we have
\begin{eqnarray}\label{eq:4.42}
\dot{\mathcal{S}}^{(\tau)}_{1}(\sigma^{(\tau)}_{\beta_1},\tau^2)-\dot{\mathcal{S}}_{1}(\sigma_{\alpha_1})
=O(1)\big(1+|\sigma_{\alpha_1}|\big)\tau^{2}
=O(1)\big(1+|\omega_{k}|\big)\tau^{2}.
\end{eqnarray}

If $\dot{y}_{\alpha_1}>\dot{\mathcal{S}}^{(\tau)}_{1}(\sigma^{(\tau)}_{\beta_1},\tau^2)$, then, by \eqref{eq:4.42} and Lemma \ref{lem:4.2},
\begin{align*}
C_{\mathcal{S}_1,2}&=\int^{g_{k}+\dot{y}_{\alpha_1}s}_{g_{k}+\dot{\mathcal{S}}^{(\tau)}_{1}(\sigma^{(\tau)}_{\beta_1},\tau^2)s}
\big|\mathcal{P}^{(\tau)}_{h}(x_k+s, x_{k})(U^{(\tau),\Gamma}_{h,\nu}(x_{k},y))-U^{\Gamma}_{h,\nu}(x_{k}+s,y)\big|\,{\rm d}y\\[2pt]
&\leq C\,\big|\dot{y}_{\alpha_1}-\dot{\mathcal{S}}^{(\tau)}_{1}(\sigma^{(\tau)}_{\beta_1},\tau^2)\big|\big|U^{(\tau)}_{\Gamma}-U_{B}\big|s\\[2pt]
&\leq C\Big(\big|\dot{\mathcal{S}}_{1}(\sigma_{\alpha_1})-\dot{\mathcal{S}}^{(\tau)}_{1}(\sigma^{(\tau)}_{\beta_1},\tau^2)\big|
+2^{-\nu}\Big)\big|\sigma^{(\tau)}_{\beta_1}\big|s\\[2pt]
&\leq C\Big((1+|\omega_{k}|)\tau^2+2^{-\nu}\Big)\big(|\omega_{k}|+\tau^2\big)s.
\end{align*}
If $\dot{y}_{\alpha_1}<\dot{\mathcal{S}}^{(\tau)}_{1}(\sigma^{(\tau)}_{\beta_1},\tau^2)$,
then it follows from \eqref{eq:4.34} and \eqref{eq:4.42} that
\begin{align*}
C_{\mathcal{S}_1,2}&=\int^{g_{k}+\dot{\mathcal{S}}^{(\tau)}_{1}(\sigma^{(\tau)}_{\beta_1},\tau^2)s}_{g_{k}+\dot{y}_{\alpha_1}s}
\big|\mathcal{P}^{(\tau)}_{h}(x_k+s, x_{k})(U^{(\tau),\Gamma}_{h,\nu}(x_{k},y))-U^{\Gamma}_{h,\nu}(x_{k}+s,y)\big|\,{\rm d}y\\[2pt]
&\leq C\,\big|\dot{\mathcal{S}}^{(\tau)}_{1}(\sigma^{(\tau)}_{\beta_1},\tau^2)-\dot{y}_{\alpha_1}\big|\big|U_{\Gamma}-U_{B}\big|s\\[2pt]
&\leq C\Big(\big|\dot{\mathcal{S}}_{1}(\sigma_{\alpha_1})-\dot{\mathcal{S}}^{(\tau)}_{1}(\sigma^{(\tau)}_{\beta_1},\tau^2)\big|
+2^{-\nu}\Big)\big|\sigma_{\alpha_1}\big|s\\[2pt]
&\leq C\Big((1+|\omega_{k}|)\tau^2+2^{-\nu}\Big)|\omega_{k}|s.
\end{align*}

Therefore, we obtain
$$
C_{\mathcal{S}_1,2}\leq C\Big((1+|\omega_{k}|)\tau^2+2^{-\nu}(|\omega_{k}|+\tau^2)\Big)s.
$$

For $C_{\mathcal{S}_1,3}$, it is direct to see
\begin{align*}
C_{\mathcal{S}_1,3}
&\leq C\,|U^{(\tau)}_{\Gamma}-U_{\Gamma}|s\\[2pt]
&\leq C\big|\Phi^{(\tau)}_{1}(\sigma_{\beta_1}; U_{B},\tau^{2})-\Phi(\sigma_{\alpha_1}; U_{B})\big|s\\[2pt]
&\leq C\big(1+|\omega_{k}|\big)\tau^2s.
\end{align*}
Then
we can obtain estimate \eqref{eq:4.37} by choosing
{$C_{3,7}>0$} depending only on $(\underline{U},a_{\infty})$.
\end{proof}

\begin{figure}[ht]
\begin{center}
\begin{tikzpicture}[scale=0.6]
\draw [line width=0.05cm] (-3.5,-4.0) --(-3.5,1.5);
\draw [line width=0.05cm] (2.5,-4.0) --(2.5,0.5);
\draw [thick] (-3.5,1.5)--(2.5,0.5);

\draw [thin][red] (-3.5,1.5)--(0.8,-1.5);
\draw [line width=0.03cm] (-1.5,-4.0) --(-1.5,1.2);
\draw [line width=0.03cm] (0.8,-4.0) --(0.8,0.8);

\draw [thin](-1.5, 0.1) --(0.8,-2.8);
\draw [thick](-1.5, 0.1) --(0.8,-2.0);
\draw [dashed](-1.5, 0.1) --(0.8,-0.9);
\draw [thick](-1.5, 0.1) --(0.8,-0.3);
\draw [thin](-1.5, 0.1) --(0.8,0.3);

\draw [thin] (-3, 1.4) --(-2.6, 1.8);
\draw [thin] (-2.6, 1.35) --(-2.2, 1.75);
\draw [thin] (-2.2, 1.30) --(-1.8, 1.70);
\draw [thin] (-1.8, 1.23) --(-1.4, 1.63);
\draw [thin] (-1.4, 1.16) --(-1.0, 1.56);
\draw [thin] (-1.0, 1.10) --(-0.6, 1.50);
\draw [thin] (-0.6, 1.03) --(-0.2, 1.43);
\draw [thin] (-0.2, 0.97) --(0.2, 1.37);
\draw [thin] (0.2, 0.9) --(0.6, 1.30);
\draw [thin] (0.6, 0.83) --(1, 1.23);
\draw [thin] (1, 0.76) --(1.4, 1.16);
\draw [thin] (1.4, 0.67) --(1.8, 1.07);
\draw [thin] (1.8, 0.60) --(2.2, 1.0);
\draw [thin] (2.2, 0.55) --(2.6, 0.95);

\node at (1.3, -0.3){$\beta_{4}$};
\node at (1.6, -1.1){$\beta_{2(3)}$};
\node at (1.3, -1.7){$\alpha_{1}$};
\node at (1.3, -2.2){$\beta_{1}$};

\node at (-0.4, 0.6){$U_{\Gamma}$};
\node at (-0.4, -2.5){$U_{B}$};

\node at (-2.4, -0.2){$(\xi,y_{\Gamma})$};
\node at (-4.3, 1.5){$\textsc{C}_{k}$};
\node at (3.5, 0.5){$\textsc{C}_{k+1}$};
\node at (-0.5, 1.8){$\Gamma_h$};

\node at (-1.5, -4.5){$\xi$};
\node at (0.8, -4.5){$\xi+s$};
\node at (-3.5, -4.5){$x_{k}$};
\node at (2.5, -4.5){$x_{k+1}$};
\end{tikzpicture}
\end{center}
\caption{Comparison of the Riemann solvers away from the corner points with boundary}\label{fig21}
\end{figure}
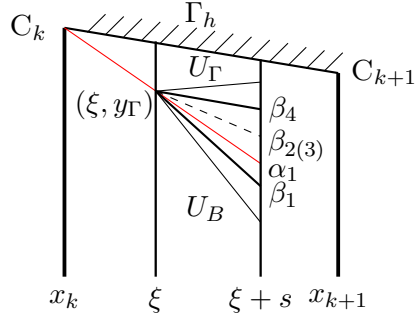

Finally, in this section, we consider the comparison of the Riemann solvers near the approximate boundary $\Gamma_{h}$ but away from the corner points, as shown in Fig. \ref{fig21}.
Let $U_{h,\nu}$ be the approximate solution of  the initial-boundary value problem \eqref{eq:1.12}--\eqref{eq:1.15}.
Let $\alpha_1$ be the front issuing from the corner point $\textsc{C}_{k}(x_k,g_{k})$ for $k\geq1$.
Let $y_{\alpha_1}(x)$ be the location of front $\alpha_1$ with a fixed speed $\dot{y}_{\alpha_1}$, which satisfies $y_{\alpha_1}(x_k)=g_{k}$.
We first establish the following lemma.

\begin{lemma}\label{lem:4.4}
Let $U_{B}=(\rho_B,u_{B}, v_{B},p_{B})^{\top}$ and $U_{\Gamma}=(\rho_{\Gamma}, u_{\Gamma},v_{\Gamma},p_{\Gamma})^{\top}$
be the two constant states as defined in \eqref{eq:4.29} and satisfy
\begin{equation}\label{eq:4.43}
U_{\Gamma}=\Phi_1(\sigma_{\alpha_1}; U_{B}), \,\,\, U_{\Gamma}=\Phi^{(\tau)}(\boldsymbol{\sigma}^{(\tau)}_{\boldsymbol{\beta}}; U_{B},\tau^2),
\,\,\, \boldsymbol{\sigma}^{(\tau)}_{\boldsymbol{\beta}}=\big(\sigma^{(\tau)}_{\beta_1}, \sigma^{(\tau)}_{\beta_2},
\sigma^{(\tau)}_{\beta_3},\sigma^{(\tau)}_{\beta_4}\big),
\end{equation}
and
\begin{eqnarray}\label{eq:4.44}
v_{B}=\tan(\theta_k), \qquad v_{\Gamma}=\tan(\theta_{k+1}).
\end{eqnarray}
Then, for $U_{B},\ U_{\Gamma}\in \mathcal{O}_{\min\{\varepsilon_0,\tilde{\varepsilon}_0\}}(\underline{U})$
and for $\tau\in(0,\min\{\tau_0,\tilde{\tau}_0\})$,
\begin{eqnarray}\label{eq:4.45}
\sigma^{(\tau)}_{\beta_{1}}=\sigma_{\alpha_1}+O(1)|\sigma_{\alpha_1}|\tau^{2},\qquad\, \sigma^{(\tau)}_{\beta_{j}}=O(1)|\sigma_{\alpha_1}|\tau^{2}
\quad \mbox{for $j\neq 1$}.
\end{eqnarray}
Furthermore, for $\varepsilon \in (0,\min\{\varepsilon_0,\tilde{\varepsilon}_0\})$, if $|\theta_{k}|+|\omega_{k}|<\varepsilon$,
\begin{eqnarray}\label{eq:4.46}
\sigma^{(\tau)}_{\beta_{1}}=O(1)\big(1+|\omega_{k}|\big)\tau^{2}, \qquad\, \sigma^{(\tau)}_{\beta_{j}}=O(1)|\omega_{k}|\tau^{2}
\quad \mbox{for $j\neq 1$},
\end{eqnarray}
where $\omega_{k}=\theta_{k+1}-\theta_{k}$, and the bound of $O(1)$ depends only on $(\underline{U},a_{\infty})$.
\end{lemma}

\begin{proof}
By \eqref{eq:4.43}, we know that $\boldsymbol{\sigma}^{(\tau)}_{\boldsymbol{\beta}}=\big(\sigma^{(\tau)}_{\beta_1}, \sigma^{(\tau)}_{\beta_2},
\sigma^{(\tau)}_{\beta_3},\sigma^{(\tau)}_{\beta_4}\big)$ satisfies
\begin{eqnarray}\label{eq:4.47}
\Phi^{(\tau)}(\boldsymbol{\sigma}^{(\tau)}_{\boldsymbol{\beta}}; U_{B},\tau^2)=\Phi_1(\sigma_{\alpha_{1}}; U_{B}).
\end{eqnarray}
Then, by Lemma \ref{lem:2.3} and the implicit function theorem, for $U_{B}\in \mathcal{O}_{\min\{\varepsilon_0,\tilde{\varepsilon}_0\}}(\underline{U})$
and $\tau\in(0,\min\{\tau_0,\tilde{\tau}_0\})$, equation \eqref{eq:4.47} admits a unique solution
$\sigma^{(\tau)}_{\beta_{j}}=\sigma^{(\tau)}_{\beta_{j}}(\sigma_{\alpha_{1}}, \tau^{2})\in C^{2}$ with $1\leq j\leq 4$.
Moreover, from equation \eqref{eq:4.47} and Lemma \ref{lem:4.1}, we see that
$$
\sigma^{(\tau)}_{\beta_{j}}(\sigma_{\alpha_{1}}, 0)=\delta_{j1}\sigma_{\alpha_1},\quad \sigma^{(\tau)}_{\beta_{j}}(0, \tau^2)=0.
$$
Thus, estimates \eqref{eq:4.45} follows from the Taylor formula.

Then, using the boundary conditions \eqref{eq:4.44} and
Lemma A.3,
we further obtain that $\sigma_{\alpha_{1}}=O(1)\omega_k$,
which leads to estimate \eqref{eq:4.46}. This completes the proof.
\end{proof}

Let $(\xi, y_{\Gamma})$ be a point on front $y_{\alpha_1}(x)$ that lies on the lower side of the approximate boundary $\Gamma_{h}$
(see Fig. \ref{fig21}).
Then we have the following proposition:

\begin{proposition}\label{prop:4.4}
For a given hypersonic similarity parameter $a_{\infty}$ and for $\tau\in(0,\min\{\tau_0,\tilde{\tau}_0\})$,
let $\mathcal{P}^{(\tau)}_{h}$ be a uniformly Lipschtiz continuous map obtained in {\rm Proposition \ref{prop:2.2}}.
Let $U^{B}_{h,\nu}(\xi,y)$ be the approximate solution of the initial-boundary value problem \eqref{eq:1.12}--\eqref{eq:1.15}
with $U_{B},\ U_{\Gamma}\in \mathcal{O}_{\min\{\varepsilon_0, \tilde{\varepsilon}_0\}}(\underline{U})$ defined in \eqref{eq:4.29}
and satisfying \eqref{eq:4.44}. By
{\rm Lemma A.3},
for $s>0$ sufficiently small,
\begin{eqnarray}\label{eq:4.49}
U^{B}_{h,\nu}(\xi+s, y)=
\left\{
\begin{array}{llll}
U_{\Gamma}, \quad & y_{\Gamma}+\dot{y}_{\alpha_1}s<y<g_{h}(\xi+s),\\[2pt]
U_{B}, \quad & y<y_{\Gamma}+\dot{y}_{\alpha_1}s,
\end{array}
\right.
\end{eqnarray}
where $|\dot{y}_{\alpha_1}|\leq \hat{\lambda}$ and $\hat{\lambda}$ is a fixed constant
satisfying
$$
\hat{\lambda}>\max_{1\leq j\leq 4}\{\lambda^{(\tau)}_{j}(U^{(\tau)},\tau^{2}), \lambda_{j}(U)\}
\qquad\,\,\mbox{{\rm for} $U^{(\tau)}, U\in \mathcal{O}_{\min\{\varepsilon_0, \tilde{\varepsilon}_0\}}(\underline{U})$}.
$$
Then
\begin{enumerate}
\item[\rm (i)]
if $\omega_{k}=\theta_{k+1}-\theta_{k}>0$, $U_{B}$ and $U_{\Gamma}$ are connected by the $1^{\rm st}$ rarefaction-front $\alpha_1\in \mathcal{R}_1$
satisfying $U_{\Gamma}=\Phi_{1}(\sigma_{\alpha_1}; U_{B})$ with $\sigma_{\alpha_1}>0$ and $|\dot{y}_{\alpha_1}-\lambda_{1}(U_{\Gamma})|<2^{-\nu}$,
and
\begin{align}\label{eq:4.50}
&\int^{y_{\Gamma}+\hat{\lambda}s}_{y_{\Gamma}-\hat{\lambda}s}\big|\mathcal{P}^{(\tau)}_{h}(\xi+s,\xi)(U^{B}_{h,\nu}(\xi,y))
-U^{B}_{h,\nu}(\xi+s,y)\big|\,{\rm d}y\nonumber\\[2pt]
&\leq {C_{3,8}}\big(\tau^{2}+\nu^{-1}+2^{-\nu}\big)|\omega_{k}|s.
\end{align}

\item[\rm (ii)]
if $\omega_{k}=\theta_{k+1}-\theta_{k}<0$, $U_{B}$ and $U_{\Gamma}$ are connected by the $1^{\rm st}$ shock-front $\alpha_1\in \mathcal{S}_1$
satisfying $U_{\Gamma}=\Phi_{1}(\sigma_{\alpha_1}; U_{B})$ with $\sigma_{\alpha_1}<0$ and $|\dot{y}_{\alpha_1}-\dot{\mathcal{S}}_{1}(\sigma_{\alpha_1})|<2^{-\nu}$,
and
\begin{align}\label{eq:4.51}
&\int^{y_{\Gamma}+\hat{\lambda}s}_{y_{\Gamma}-\hat{\lambda}s}\big|\mathcal{P}^{(\tau)}_{h}(\xi+s,\xi)(U^{B}_{h,\nu}(\xi,y))
-U^{B}_{h,\nu}(\xi+s,y)\big|\,{\rm d}y\nonumber\\[2pt]
&\leq {C_{3,9}}\big(\tau^{2}+2^{-\nu}\big)|\omega_{k}|,
\end{align}
where $\dot{\mathcal{S}}_{1}(\sigma_{\alpha_1})$ is the exact speed of the $1^{\rm st}$ shock-front $\alpha_1$,
and constants {$C_{3,8}>0$ and $C_{3,9}>0$ }depend only on $(\underline{U},a_{\infty})$.
\end{enumerate}
\end{proposition}

\begin{proof}
For case \rm (i), \emph{i.e.}, $\omega_{k}=\theta_{k+1}-\theta_{k}>0$, by
Lemma A.3
and conditions \eqref{eq:4.44},
we know that $\sigma_{\alpha_1}=K_{B}\omega_k>0$ for a constant $K_{B}>0$ and satisfies $U_{\Gamma}=\Phi_{1}(\sigma_{\alpha_1}; U_{B})$.
Moreover, note that $\boldsymbol{\sigma}^{(\tau)}_{\boldsymbol{\beta}}=\big(\sigma^{(\tau)}_{\beta_1}, \sigma^{(\tau)}_{\beta_2},
\sigma^{(\tau)}_{\beta_3},\sigma^{(\tau)}_{\beta_4}\big)$ can be solved from equation \eqref{eq:4.47},
and its corresponding solution $\mathcal{P}^{(\tau)}_{h}(\xi+s,\xi)(U^{B}_{h,\nu}(\xi,y))$ satisfies \eqref{eq:4.43}.
Thus, we can apply Lemma \ref{lem:4.4} to show that $\sigma_{\beta_j}$, $1\leq j\leq 4$, satisfy estimate \eqref{eq:4.45}.
It follows from Proposition \ref{prop:4.1} that
\begin{align*}
&\int^{y_{\Gamma}+\hat{\lambda}s}_{y_{\Gamma}-\hat{\lambda}s}\big|\mathcal{P}^{(\tau)}_{h}(\xi+s,\xi)(U^{B}_{h,\nu}(\xi,y))
-U^{B}_{h,\nu}(\xi+s,y)\big|\,{\rm d}y\\[2pt]
&\leq C\big(\tau^{2}+\nu^{-1}+2^{-\nu}\big)|\sigma_{\alpha_{1}}|s.
\end{align*}
Therefore, estimate \eqref{eq:4.50} follows from the estimate on $\sigma_{\alpha_1}$ as stated above.
The proof of estimate \eqref{eq:4.51} is similar as the one for \eqref{eq:4.50},
so we omit it.
\end{proof}

\section{Proof of Theorem \ref{thm:1.2} for Problem II}\setcounter{equation}{0}
In this section, we prove Theorem \ref{thm:1.2} by establishing the optimal convergence rate with respect to the $L^1$--norm
between the entropy solution $U^{(\tau)}$ of {Problem I} and the entropy solution $U$ of the initial-boundary value problem
\eqref{eq:1.12}--\eqref{eq:1.15}.
We first show a lemma to the $L^1$--error formula for $\mathcal{P}^{(\tau)}_{h}$ with boundary.

\begin{lemma}\label{lem:5.1}
Let $U(x): [0,X]\mapsto \mathcal{D}_{h,x}$ be a piecewise Lipschitz continuous function for some $X>0$ with a finite number of wave-fronts,
where domain $\mathcal{D}_{h,x}$ is given by \eqref{eq:2.71}.
Let $\mathcal{P}^{(\tau)}_{h}(x,x'_{0}): \mathcal{D}_{h,x'_{0}}\mapsto \mathcal{D}_{h,x}$ be the Lipschitz continuous map $\mathcal{P}^{(\tau)}_{h}$
generated by \emph{Problem I} corresponding to an approximate boundary $\Gamma_{h}$ as given in {\rm Proposition \ref{prop:2.2}}. Then
\begin{align}
&\big\|\mathcal{P}^{(\tau)}_{h}(x, 0)(U(0))-U(x)\big\|_{L^{1}((-\infty, g_{h}(x)))}\nonumber\\[4pt]
& \leq \bar{L}\int^{x}_{0}\overline{\lim}_{s\rightarrow0^{+}}\frac{\big\|\mathcal{P}^{(\tau)}_{h}(\xi+s, \xi)(U(\xi))-U(\xi+s)\big\|_{L^{1}((-\infty, g_{h}(\xi+s)))}}{{s}}{\rm d}\xi,\label{eq:5.1}
\end{align}
where constant $\bar{L}$ depends only on the Lipschitz constant $L$ given in {\rm Proposition \ref{prop:2.2}}.
\end{lemma}

\begin{proof}
Define the set
\begin{align*}
\Lambda&\doteq
\bigg\{x\in [0,X]\, :\,\begin{array}{ll}
     \mbox{ there are at least two interacting wave-fronts in $U$,}\\[2pt]
     \mbox{ or $\,$ a wave-front hitting boundary $y=g_{h}(x)$, $\,$ at $x$}\end{array}\bigg\}\\
     &\subset [0,X].
\end{align*}

Let $|\Lambda|=\ell$, and let $a_{k}\in \Lambda $, $k=1,2,\cdots, \ell$,
satisfy $a_{1}<a_{2}<\cdot\cdot\cdot<a_{\ell}$. For any $\tilde{\epsilon}>0$, define
\begin{eqnarray*}
&&\tilde{x}_{0}=0, \ \tilde{x}_{1}=a_{1}-\frac{\tilde{\epsilon}}{2\ell}, \ \tilde{x}_{2}=a_{1}+\frac{\tilde{\epsilon}}{2\ell},
\,\cdots,\,\tilde{x}_{2\ell-1}=a_{\ell}-\frac{\tilde{\epsilon}}{2\ell},\\
&& \tilde{x}_{2\ell}=a_{\ell}+\frac{\tilde{\epsilon}}{2\ell},\
\tilde{x}_{2\ell+1}=X.
\end{eqnarray*}
Then we have
\begin{align*}
&\big\|\mathcal{P}^{(\tau)}_{h}(X,0)(U(0))-U(X)\big\|_{L^1((-\infty, g_{h}(X)))}\\[2pt]
&\leq \sum^{2\ell+1}_{k=1}\big\|\mathcal{P}^{(\tau)}_{h}(X,\tilde{x}_{k-1})(U(\tilde{x}_{k-1}))
-\mathcal{P}^{(\tau)}_{h}(X,\tilde{x}_{k})(U(\tilde{x}_{k}))\big\|_{L^1((-\infty, g_{h}(X)))}\\[2pt]
&=\sum^{\ell}_{k=0}\big\|\mathcal{P}^{(\tau)}_{h}(X,\tilde{x}_{2k})(U(\tilde{x}_{2k}))
-\mathcal{P}^{(\tau)}_{h}(X,\tilde{x}_{2k+1})(U(\tilde{x}_{2k+1}))\big\|_{L^1((-\infty, g_{h}(X)))}\\[2pt]
&\ \ \ + \sum^{\ell}_{k=1}\big\|\mathcal{P}^{(\tau)}_{h}(X,\tilde{x}_{2k-1})(U(\tilde{x}_{2k-1}))
-\mathcal{P}^{(\tau)}_{h}(X,\tilde{x}_{2k})(U(\tilde{x}_{2k}))\big\|_{L^1((-\infty, g_{h}(X)))}
\\[2pt]
&\doteq J_{1}+J_{2}.
\end{align*}

For $J_{1}$, we have
\begin{align*}
J_{1}&\leq \sum^{\ell}_{k=0}\big\|\mathcal{P}^{(\tau)}_{h}(X,\tilde{x}_{2k+1})\circ\mathcal{P}^{(\tau)}_{h}(\tilde{x}_{2k+1},\tilde{x}_{2k})(U(\tilde{x}_{2k}))\\
&\qquad\quad\,\,\,
-\mathcal{P}^{(\tau)}_{h}(X,\tilde{x}_{2k+1})(U(\tilde{x}_{2k+1}))\big\|_{L^1((-\infty, g_{h}(X)))}\\[2pt]
&\leq L\sum^{\ell}_{k=0}\big\|\mathcal{P}^{(\tau)}_{h}(\tilde{x}_{2k+1},\tilde{x}_{2k})(U(\tilde{x}_{2k}))
-U(\tilde{x}_{2k+1})\big\|_{L^1((-\infty, g_{h}(\tilde{x}_{2k+1})))}\\[2pt]
&=L\sum^{\ell}_{k=0}\big\|\mathcal{P}^{(\tau)}_{h}(\tilde{x}_{2k+1},\tilde{x}_{2k})(U(\tilde{x}_{2k}))\\
&\qquad\qquad\,\,-\mathcal{P}^{(\tau)}_{h}(\tilde{x}_{2k+1},\tilde{x}_{2k+1})(U(\tilde{x}_{2k+1}))\big\|_{L^1((-\infty, g_{h}(\tilde{x}_{2k+1})))}.
\end{align*}

We claim:
\begin{align}
&\big\|\mathcal{P}^{(\tau)}_{h}(\tilde{x}_{2k+1},\tilde{x}_{2k})(U(\tilde{x}_{2k}))
-\mathcal{P}^{(\tau)}_{h}(\tilde{x}_{2k+1},\tilde{x}_{2k+1})(U(\tilde{x}_{2k+1}))\big\|_{L^1((-\infty, g_{h}(\tilde{x}_{2k+1})))}\nonumber\\[2pt]
&\leq L\int^{\tilde{x}_{2k+1}}_{\tilde{x}_{2k}}\overline{\lim}_{s\rightarrow0^{+}}\frac{\big\|\mathcal{P}^{(\tau)}_{h}(\xi+s,\xi)(U(\xi))
-U(\xi+s)\big\|_{L^1((-\infty, g_{h}(\xi+s)))}}{{s}}\,{\rm d}\xi.\label{eq:5.2}
\end{align}

In fact,
consider a partition $\{\tilde{x}^j_{2k}\}^{k}_{j=1}$ of interval $[\tilde{x}_{2k},\tilde{x}_{2k+1}]$:
\begin{eqnarray*}
\tilde{x}_{2k}\doteq \tilde{x}^1_{2k}<\tilde{x}^2_{2k}<\cdot\cdot\cdot<\tilde{x}^k_{2k}\doteq \tilde{x}_{2k+1}.
\end{eqnarray*}
Let $\lambda=\max_{1\leq j\leq k-1}\{\tilde{x}^j_{2k}-\tilde{x}^{j+1}_{2k}\}$. Then
{\small
\begin{align*}
&\big\|\mathcal{P}^{(\tau)}_{h}(\tilde{x}_{2k+1}, \tilde{x}_{2k})(U(\tilde{x}_{2k}))-\mathcal{P}^{(\tau)}_{h}(\tilde{x}_{2k+1}, \tilde{x}_{2k+1})(U(\tilde{x}_{2k+1}))\big\|_{L^{1}((-\infty, g_{h}(\tilde{x}_{2k+1})))}\\[2pt]
&\leq\sum^{k-1}_{j=1}\big\|\mathcal{P}^{(\tau)}_{h}(\tilde{x}_{2k+1}, \tilde{x}^j_{2k})(U(\tilde{x}^j_{2k}))-\mathcal{P}^{(\tau)}_{h}(\tilde{x}_{2k+1}, \tilde{x}^{j+1}_{2k})(U(\tilde{x}^{j+1}_{2k}))\big\|_{L^{1}((-\infty, g_{h}(\tilde{x}_{2k+1})))}\\[2pt]
&\leq\sum^{k-1}_{j=1}\big\|\mathcal{P}^{(\tau)}_{h}(\tilde{x}_{2k+1}, \tilde{x}^{j+1}_{2k})\circ\mathcal{P}^{(\tau)}_{h}( \tilde{x}^{j+1}_{2k}, \tilde{x}^{j}_{2k})(U(\tilde{x}^{j}_{2k}))\\
&\qquad\quad\,\,\, -P^{(\tau)}_{h}(\tilde{x}_{2k+1}, \tilde{x}^{j+1}_{2k})(U(\tilde{x}^{j+1}_{2k}))\big\|_{L^{1}((-\infty, g_{h}(\tilde{x}_{2k+1})))}\\[2pt]
&\leq L\sum^{k-1}_{j=1}\big\|\mathcal{P}^{(\tau)}_{h}(\tilde{x}^{j+1}_{2k}, \tilde{x}^{j}_{2k})(U(\tilde{x}^{j}_{2k}))-U(\tilde{x}^{j+1}_{2k})\big\|_{L^{1}((-\infty, g_{h}(\tilde{x}^{j+1}_{2k})))}\\[2pt]
&\leq L\sum^{k}_{j=1}\frac{\big\|\mathcal{P}^{(\tau)}_{h}(\tilde{x}^{j+1}_{2k}, \tilde{x}^{j}_{2k})(U(\tilde{x}^{j}_{2k}))-U(\tilde{x}^{j+1}_{2k})\big\|_{L^{1}((-\infty, g_{h}(\tilde{x}^{j+1}_{2k})))}}{ \tilde{x}^{j}_{2k}- \tilde{x}^{j+1}_{2k}}\big(\tilde{x}^{j}_{2k}- \tilde{x}^{j+1}_{2k}\big).
\end{align*}
}
Taking $\lambda\rightarrow 0$, we can obtain claim \eqref{eq:5.2} so that
\begin{eqnarray*}
J_{1}\leq L^2\sum^{\ell}_{k=0}\int^{\tilde{x}_{2k+1}}_{\tilde{x}_{2k}}\frac{\big\|\mathcal{P}^{(\tau)}_{h}(\xi+s,\xi)(U(\xi))
-U(\xi+s)\big\|_{L^1((-\infty, g_{h}(\xi+s)))}}{{s}}\,{\rm d}\xi.
\end{eqnarray*}

For $J_{2}$, we have
\begin{align*}
J_{2}&= \sum^{\ell}_{k=0}\big\|\mathcal{P}^{(\tau)}_{h}(X,\tilde{x}_{2k})\circ\mathcal{P}^{(\tau)}_{h}(\tilde{x}_{2k},\tilde{x}_{2k-1})(U(\tilde{x}_{2k-1}))\\
&\qquad\quad\,\, -\mathcal{P}^{(\tau)}_{h}(X,\tilde{x}_{2k})(U(\tilde{x}_{2k}))\big\|_{L^1((-\infty, g_{h}(X)))}\\[2pt]
&\leq L\sum^{\ell}_{k=0}\big\|\mathcal{P}^{(\tau)}_{h}(\tilde{x}_{2k},\tilde{x}_{2k-1})(U(\tilde{x}_{2k-1}))
-U(\tilde{x}_{2k})\big\|_{L^1((-\infty, g_{h}(\tilde{x}_{2k})))}\\[2pt]
&\leq L\sum^{\ell}_{k=0}\big\|\mathcal{P}^{(\tau)}_{h}(\tilde{x}_{2k},\tilde{x}_{2k-1})(U(\tilde{x}_{2k-1}))
-U(\tilde{x}_{2k-1})\big\|_{L^1((-\infty, g_{h}(\tilde{x}_{2k})))}\\[2pt]
&\ \ \ +L\sum^{\ell}_{k=0}\big\|U(\tilde{x}_{2k-1})
-U(\tilde{x}_{2k})\big\|_{L^1((-\infty, g_{h}(\tilde{x}_{2k})))}\\[2pt]
&\doteq J_{21}+J_{22}.
\end{align*}

First, for $J_{21}$, we have
\begin{align*}
J_{21}&=L\sum^{\ell}_{k=0}\big\|\mathcal{P}^{(\tau)}_{h}(\tilde{x}_{2k},\tilde{x}_{2k-1})(U(\tilde{x}_{2k-1}))\\
&\qquad\qquad\,\,-\mathcal{P}^{(\tau)}_{h}(\tilde{x}_{2k-1},\tilde{x}_{2k-1})(U(\tilde{x}_{2k-1}))\big\|_{L^1((-\infty, g_{h}(\tilde{x}_{2k})))}
\\[2pt]
& \leq C\sum^{\ell}_{k=0}(\tilde{x}_{2k}-\tilde{x}_{2k-1})
\leq C\ell \frac{\tilde{\epsilon}}{2\ell}=C\tilde{\epsilon},
\end{align*}
while, for $J_{22}$, it follows from the Lipschitz continuity of $U(x)$ that
\begin{eqnarray*}
J_{22}\leq C\sum^{\ell}_{k=0}(\tilde{x}_{2k}-\tilde{x}_{2k-1})\leq C\tilde{\epsilon}.
\end{eqnarray*}
Thus, with the estimates for $J_{21}$ and $J_{22}$, we obtain that $J_{2}\leq C\tilde{\epsilon}$.
Finally, combining the estimates for $J_{1}$ and $J_{2}$ and letting $\tilde{\epsilon}\rightarrow 0$,
we conclude \eqref{eq:5.1}.
\end{proof}

\smallskip
We are now ready to prove Theorem \ref{thm:1.2}, which gives a positive answer to {Problem II}.

\begin{proof}[Proof of {\rm Theorem \ref{thm:1.2}}] We first show estimate \eqref{eq:1.25}.
This is achieved by three steps:

\smallskip
1. \emph{Estimate on $\big\|\mathcal{P}^{(\tau)}_{h}(\xi+s,\xi)(U_{h, \nu}(\xi,y))-U_{h,\nu}(\xi+s, y)\big\|_{L^{1}((-\infty, g_{h}(\xi+s)))}$ for sufficiently small $s$}.
Let $U_{h,\nu}(x,y)$ be the approximate solution of the initial-boundary value problem \eqref{eq:1.12}--\eqref{eq:1.15}
corresponding to the approximate boundary $y=g_h(x)$ and the approximate initial data $U^{\nu}_{0}(y)$ with \eqref{eq:2.27} and \eqref{eq:2.28}.
Let the jumps of $U_{h,\nu}(\xi,y)$ at $x=\xi>0$ be $y_{1}>y_{2}>\cdots > y_{N}$ with $y_{1}\leq g_{h}(\xi)$.
Let $\mathcal{S}$ be the set of indices $\alpha\in\{1,2, \cdot\cdot\cdot, N\}$ such that $U_{h,\nu}(\xi,y_{\alpha}+)$ and $U_{h,\nu}(\xi,y_{\alpha}-)$
are connected by a shock-front with strength $\sigma_{\alpha}$.
Let $\mathcal{R}$ (resp., $\mathcal{C}$ and $\mathcal{NP}$) be the set of indices $\alpha\in\{1,2, \cdot\cdot\cdot, N\}$
such that $U_{h,\nu}(\xi,y_{\alpha}+)$ and $U_{h,\nu}(\xi,y_{\alpha}-)$ are connected
by a rarefaction-front (resp., a vortex sheet/entropy wave and a non-physical wave) with strength $\sigma_{\alpha}$ (or $\sigma_{\alpha_{\mathcal{NP}}}$ for the non-physical waves), respectively.

If $\xi=[\frac{\xi}{h}]h$, then $\xi=x_{k}$ for $k=[\frac{\xi}{h}]$.
According to the construction of the approximate solution $U_{h,\nu}$ in Appendix A, $U_{h,\nu}(\xi,y)$ has no interaction on line $\xi=x_{k}$
and no reflection at the corner point $\textsc{C}_k(x_k,g_{k})$.
Thus,  we choose $s>0$ sufficiently small such that there is no interaction or reflection of the fronts for $U_{h,\nu}(\xi,y)$ between
lines $\xi=x_k$ and $\xi=x_{k}+s$.
Let $\eta=\frac{1}{2}\min_{1\leq j\leq N-1}\big\{b_{h}(\varsigma)-y_1, y_j-y_{j+1}\big\}$. Then
\begin{align}\label{eq:5.3}
&\big\|\mathcal{P}^{(\tau)}_{h}(x_{k}+s,x_{k})(U_{h,\nu}(x_{k},\cdot))-U_{h,\nu}(x_{k}+s,\cdot)\big\|_{L^{1}((-\infty, g_{h}(x_{k}+s)))}\nonumber\\[2pt]
&  =\sum_{\alpha\in \mathcal{S}\cup\mathcal{C}\cup \mathcal{R}\cup \mathcal{NP}}\int^{y_{\alpha}+\eta}_{y_{\alpha}-\eta}
\big|\mathcal{P}^{(\tau)}_{h}(x_{k}+s,x_{k})(U_{h,\nu}(x_{k},\cdot))-U_{h,\nu}(x_{k}+s,\cdot)\big|\,{\rm d}y \nonumber\\[2pt]
&\ \ \  +\int^{g_{h}(x_{k}+s)}_{g_{k}-\eta}\big|\mathcal{P}^{(\tau)}_{h}(x_{k}+s,x_k)(U_{h,\nu}(x_{k},\cdot))
-U_{h,\nu}(x_{k}+s,\cdot)\big|\,{\rm d}y.
\end{align}

For $U_{h,\nu}\in \mathcal{O}_{\min\{\varepsilon_0,\tilde{\varepsilon}_0\}}(\underline{U})$ and $\tau\in(0,\min\{\tau_0,\tilde{\tau}_0\})$,
by Propositions \ref{prop:4.1}--\ref{prop:4.2} and
Proposition A.1,
we have
{\small
\begin{align*}
&\sum_{\alpha\in \mathcal{S}\cup\mathcal{C}\cup \mathcal{R}\cup \mathcal{NP}}\int^{y_{\alpha}+\eta}_{y_{\alpha}-\eta}
\big|\mathcal{P}^{(\tau)}_{h}(x_{k}+s,x_{k})(U_{h,\nu}(x_{k},\cdot))-U_{h,\nu}(x_{k}+s,\cdot)\big|\,{\rm d}y \\[2pt]
&\leq C\big(\tau^2+\nu^{-1}+2^{-\nu}\big)\Big(\sum _{\alpha\in \mathcal{S}\cup\mathcal{C}\cup \mathcal{R}}|\sigma_{\alpha}|\Big)s
+C\Big(\sum _{\alpha\in \mathcal{NP}}\sigma_{\alpha_{\mathcal{NP}}}\Big)s\\[2pt]
&\leq C\Big(\big(\|U_{h,\nu}(x_k, \cdot)\|_{BV((-\infty,g_{h}(x_k)))} +|g'(0)|+\|g'_{h}(\cdot)\|_{BV(\mathbb{R}_{+})}\big)\big(\tau^2+\nu^{-1}+2^{-\nu}\big)+2^{-\nu}
\Big)s\\[2pt]
&\leq C\Big(\big(\|U_{0}(\cdot)\|_{BV(\mathcal{I})}+|g'(0)|+\|g'(\cdot)\|_{BV(\mathbb{R}_{+})}\big)\big(\tau^2+\nu^{-1}+2^{-\nu}\big)+2^{-\nu}
\Big)s.
\end{align*}
}
By Proposition \ref{prop:4.3}, we have
\begin{eqnarray*}
&&\int^{g_{h}(x_{k}+s)}_{g_{k}-\eta}\big|\mathcal{P}^{(\tau)}_{h}(x_{k}+s,x_k)(U_{h,\nu}(x_{k},\cdot))
-U_{h,\nu}(x_{k}+s,\cdot)\big|\,{\rm d}y\\[2pt]
&&\leq C\Big( \big(1+|\omega_{k}|\big)\tau^2+\big(|\omega_{k}|+\tau^{2}\big)(\nu^{-1}+2^{-\nu})\Big)\Big(\sum_{k\geq 0}|\omega_{k}|\Big)s\\[2pt]
&&\leq C\Big(1+ \sum_{k\geq 0}|\omega_{k}|\Big)\big(\tau^{2}+\nu^{-1}+2^{-\nu}\big)s\\[2pt]
&&\leq C\Big(\big(|g'(0)|+\|g'(\cdot)\|_{BV(\mathbb{R}_{+})}+1\big)\big(\tau^2+\nu^{-1}+2^{-\nu}\big)+2^{-\nu}
\Big)s.
\end{eqnarray*}

Then, combining the above two estimates altogether, we finally obtain
\begin{align}
&\big\|\mathcal{P}^{(\tau)}_{h}(x_{k}+s,x_{k})(U_{h,\nu}(x_{k},\cdot))-U_{h,\nu}(x_{k}+s,\cdot)\big\|_{L^{1}((-\infty, g_{h}(x_{k}+s)))}\nonumber\\[2pt]
&\leq C\Big(\big(\|U_{0}(\cdot)\|_{BV(\mathcal{I})}+|g'(0)|+\|g'(\cdot)\|_{BV(\mathbb{R}_{+})}+1\big)\big(\tau^2+\nu^{-1}+2^{-\nu}\big)+2^{-\nu}
\Big)s,\label{eq:5.4}
\end{align}
where $C>0$ depends only on $(\underline{U},a_{\infty})$.

If $\xi>[\frac{\xi}{h}]h$, then $(\xi, g_{h}(\xi))$ is not a corner point.
For $s>0$ sufficiently small, and for $U_{h,\nu}\in \mathcal{O}_{\min\{\varepsilon_0,\tilde{\varepsilon}_0\}}(\underline{U})$ and $\tau\in(0,\min\{\tau_0,\tilde{\tau}_0\})$,
by Propositions \ref{prop:4.1}--\ref{prop:4.3} and
Proposition A.1,
we have
{\small
\begin{align}\label{eq:5.5}
& \big\|\mathcal{P}^{(\tau)}_{h}(\xi+s,\xi)(U_{h,\nu}(\xi,\cdot))-U_{h,\nu}(\xi+s,\cdot)\big\|_{L^{1}((-\infty, g_{h}(\xi+s)))}\nonumber\\[2pt]
& =\sum_{\alpha\in \mathcal{S}\cup \mathcal{R}\cup \mathcal{NP}}\int^{y_{\alpha}+\hat{\eta}}_{y_{\alpha}-\hat{\eta}}
\big|\mathcal{P}^{(\tau)}_{h}(\xi+s,\xi)(U_{h,\nu}(\xi,\cdot))-U_{h,\nu}(\xi+s,\cdot)\big|\,{\rm d}y\nonumber \\[2pt]
&\quad\, +\int^{g_{h}(\xi+s)}_{y_{1}-\hat{\eta}}\big|\mathcal{P}^{(\tau)}_{h}(\xi+s,\xi)(U_{h,\nu}(\xi,\cdot))
-U_{h,\nu}(\xi+s,\cdot)\big|\,{\rm d}y\nonumber\\[2pt]
&\leq C\big(\tau^2+\nu^{-1}+2^{-\nu}\big)\Big(\sum _{\alpha\in \mathcal{S}\cup\mathcal{C}\cup \mathcal{R}}|\sigma_{\alpha}|\Big)s+O(1)\Big(\sum _{\alpha_{\mathcal{NP}}\in \mathcal{NP}}\sigma_{\alpha_{\mathcal{NP}}}\Big)s\nonumber\\[2pt]
&\quad\, +C\big(\tau^2+\nu^{-1}+2^{-\nu}\big)\Big(\sum_{k\geq 0}|\omega_{k}|\Big)s\nonumber\\[2pt]
&\leq C\Big(\big(\|U_{h,\nu}(\xi, \cdot)\|_{ ((-\infty,g_{h}(\xi)))}+|g'(0)|+\|g'_{h}(\cdot)\|_{BV(\mathbb{R}_{+})}\big)\big(\tau^2+\nu^{-1}+2^{-\nu}\big)+2^{-\nu}
\Big)s\nonumber\\[2pt]
&\leq C\Big(\big(\|U_{0}(\cdot)\|_ {\mathcal{I}}+|g'(0)|+\|g'(\cdot)\|_{BV(\mathbb{R}_{+})}\big)
\big(\tau^2+\nu^{-1}+2^{-\nu}\big)+2^{-\nu}
\Big)s,
\end{align}
}
where $k=[\frac{\xi}{h}]$ and $\hat{\eta}=\frac{1}{2}\min_{1\leq j\leq N-1}\big\{y_j-y_{j+1}\big\}$.

\smallskip
2. \emph{Estimate on $\big\|\mathcal{P}^{(\tau)}_{h}(x,0)(U^{\nu}_{0}(y))-U_{h,\nu}(x,y)\big\|_{L^{1}((-\infty, g_{h}(x)))}$.}
From the construction of the approximate solution $U_{h,\nu}$ in Appendix A,
we know that $U_{h,\nu}$ satisfies the assumptions in Lemma \ref{lem:5.1}.
Then it follows from
Lemma \ref{lem:5.1} and estimates \eqref{eq:5.4}--\eqref{eq:5.5} that
{\small
\begin{align}
&\big\|\mathcal{P}^{(\tau)}_{h}(x,0)(U^{\nu}_{0}(y))-U_{h,\nu}(x,y)\big\|_{L^{1}((-\infty, g_{h}(x)))}\nonumber\\[2pt]
&=\big\|\mathcal{P}^{(\tau)}_{h}(x,0)(U^{\nu}_{0}(y))-U_{h,\nu}(x,y)\big\|_{L^{1}((-\infty, g_{h}(x)))}\nonumber\\[2pt]
& \leq
\overline{L}\int^{x}_{0}\overline{\lim}_{s\rightarrow0^{+}}\frac{\big\|\mathcal{P}^{(\tau)}_{h}(\xi+s,\xi)(U_{h,\nu}(\xi,y))
-U_{h,\nu}(\xi+s,y)\big\|_{L^{1}((-\infty, g_{h}(\xi+s)))}}{s}\,{\rm d}\xi\nonumber\\[2pt]
& \leq C\overline{L}\int^{x}_{0} \Big(\big(\|U_{0}(\cdot)\|_{BV(\mathcal{I})}+|g'(0)|+\|g'(\cdot)\|_{BV(\mathbb{R}_{+})}+1\big)
\big(\tau^{2}+\nu^{-1}+2^{-\nu}\big)+2^{-\nu}\Big)\,{\rm d}\xi \nonumber\\[2pt]
&\leq C\overline{L}\Big(\big(\|U_{0}(\cdot)\|_{BV(\mathcal{I})}+|g'(0)|+\|g'(\cdot)\|_{BV(\mathbb{R}_{+})}+h\big)
\big(\tau^{2}+\nu^{-1}+2^{-\nu}\big)+2^{-\nu}\Big)x.\label{eq:5.6}
\end{align}
}

\smallskip
3. \emph{Completion of the proof of estimate \eqref{eq:1.25}}.
As shown in Proposition \ref{prop:2.2} and
Proposition A.2,
for any given $\nu\in \mathbb{N}_{+}$ and $h>0$, we can construct a global approximate solution $U^{(\tau)}_{h,\nu}$
of {Problem I} and a global approximate solution $U_{h,\nu}$ of the initial boundary value
problem \eqref{eq:1.12}--\eqref{eq:1.15}, respectively,
corresponding to the initial data $U^{\nu}_{0}$ and the approximate boundary $y=g_{h}(x)$.
Let $U^{(\tau)}$ and $U$ be the entropy solutions of {Problem I} and the initial boundary value problem \eqref{eq:1.12}--\eqref{eq:1.15}, respectively,
corresponding to the initial data  $U_{0}$ for boundary $y=g(x)$.
By Theorem \ref{thm:1.1}, we have
\begin{align*}
\mathcal{P}^{(\tau)}_{h}(x,0)(U_{0}(\cdot))\rightarrow \mathcal{P}^{(\tau)}(x,0)(U_{0}(\cdot))=U^{(\tau)}(x,\cdot)
\qquad \mbox{in $L^{1}((-\infty,g(x)))$}
\end{align*}
as $h\rightarrow0$. By
Proposition A.2,
\begin{align*}
U_{h,\nu}(x,\cdot)\rightarrow U(x,\cdot)\qquad \mbox{in $L^{1}_{\rm loc}(\Omega)$ as $ h\rightarrow0$ and $\nu\rightarrow \infty$}.
\end{align*}
Then,
by the triangle inequality, we have
\begin{align}\label{eq:5.7}
\begin{split}
&\big\|U^{(\tau)}(x,\cdot)-U(x,\cdot)\big\|_{L^{1}((-\infty, b(x)))}\\[1pt]
&\leq \big\|U^{(\tau)}(x,\cdot)-\mathcal{P}^{(\tau)}_{h}(x,0)(U_{0})\big\|_{L^{1}((-\infty, g(x)))}\\[1pt]
&\quad +\big\|\mathcal{P}^{(\tau)}_{h}(x,0)(U_{0})-\mathcal{P}^{(\tau)}_{h}(x,0)(U^{\nu}_{0}(\cdot))\big\|_{L^{1}((-\infty, g(x)))}\\[1pt]
&\quad +\big\|\mathcal{P}^{(\tau)}_{h}(x,0)(U^{\nu}_{0}(\cdot))-U_{h,\nu}(x,\cdot)\big\|_{L^{1}((-\infty, g(x)))}\\[1pt]
&\quad +\big\|U_{h,\nu}(x,\cdot)-U(x,\cdot)\big\|_{L^{1}((-\infty, g(x)))}.
\end{split}
\end{align}

For the second term on the right-hand side of \eqref{eq:5.7}, by Propositions \ref{prop:2.1}--\ref{prop:2.2} and estimate \eqref{eq:2.28},
we can choose constants $\varepsilon^{*}_0$ and $\tau^{*}_0$ sufficiently small, depending only on $(\underline{U},a_{\infty})$, such that,
for $\varepsilon\in (0,\varepsilon^{*})$, if $\tau\in(0,\tau^{*}_0)$ and $\|U_{0}(\cdot)\|_{BV(\mathcal{I})}+|g'(0)|+\|g'(\cdot)\|_{BV(\mathbb{R}_{+})}<\varepsilon$, then
{\small
\begin{align*}
&\big\|\mathcal{P}^{(\tau)}_{h}(x,0)(U^{\nu}_{0}(\cdot))-\mathcal{P}^{(\tau)}_{h}(x,0)(U_{0}(\cdot))\big\|_{L^{1}((-\infty, g(x)))}\\[2pt]
&\leq \big\|\mathcal{P}^{(\tau)}_{h}(x,0)(U^{\nu}_{0}(\cdot))-\mathcal{P}^{(\tau)}_{h}(x,0)(U_{0}(\cdot))\big\|_{L^{1}((-\infty, g_h(x)))}\\[2pt]
&\quad +\big\|\mathcal{P}^{(\tau)}_{h}(x,0)(U^{\nu}_{0}(\cdot))-\mathcal{P}^{(\tau)}_{h}(x,0)(U_{0}(\cdot))\big\|_{L^{1}(( \min\{g_{h}(x), g(x)\},\,\max\{g_{h}(x), g(x)\}))}\\[2pt]
&\leq L\|U^{\nu}_0(\cdot)-U_0(\cdot)\|_{L^1(\mathcal{I})}
+\big\|\mathcal{P}^{(\tau)}_{h}(x,0)(U^{\nu}_{0}(\cdot))-\underline{U}\big\|_{L^{\infty}((-\infty,g_h(x)))}
\big\|g_{h}(\cdot)-g(\cdot)\big\|_{L^{\infty}(\mathbb{R}_{+})}\\[2pt]
&\quad +\big\|\mathcal{P}^{(\tau)}_{h}(x,0)(U_{0}(\cdot))-\underline{U}\big\|_{L^{\infty}((-\infty,g_h(x)))}
\big\|g_{h}(\cdot)-g(\cdot)\big\|_{L^{\infty}(\mathbb{R}_{+})}\\[2pt]
&\leq \big\|(\mathcal{P}^{(\tau)}_{h}(x,0)(U^{\nu}_{0}(\cdot))-\underline{U},\, \mathcal{P}^{(\tau)}_{h}(x,0)(U_{0}(\cdot))-\underline{U})\big\|_{L^{\infty}((-\infty,g_h(x)))}
  \big\|g'_{h}(\cdot)-g'(\cdot)\big\|_{L^{1}(\mathbb{R}_{+})}
\\[2pt]
&\quad +L\|U^{\nu}_0(\cdot)-U_0(\cdot)\|_{L^1(\mathcal{I})}.
\end{align*}
}
It follows from \eqref{eq:2.27} and \eqref{eq:2.28} that, as $\nu\rightarrow \infty$ and $h \rightarrow 0$,
\begin{align*}
\big\|\mathcal{P}^{(\tau)}_{h}(x,0)(U^{\nu}_{0}(\cdot))-\mathcal{P}^{(\tau)}_{h}(x,0)(U_{0}(\cdot))\big\|_{L^{1}((-\infty, g(x)))}\rightarrow 0.
\end{align*}

For the third term on the right-hand side of \eqref{eq:5.7}, it follows from \eqref{eq:5.6} that
\begin{align*}
&\big\|\mathcal{P}^{(\tau)}_{h}(x,0)(U^{\nu}_{0}(\cdot))-U_{h,\nu}(x,\cdot)\big\|_{L^{1}((-\infty, g(x)))}\\[2pt]
&\leq \big\|\mathcal{P}^{(\tau)}_{h}(x,0)(U^{\nu}_{0}(\cdot))-U_{h,\nu}(x,\cdot)\big\|_{L^{1}((-\infty, g_h(x)))}\\[2pt]
&\ \ \ +\big\|\mathcal{P}^{(\tau)}_{h}(x,0)(U^{\nu}_{0}(\cdot))-U_{h,\nu}(x,\cdot)\big\|_{L^{1}(( \min\{g_{h}(x), g(x)\},\,\max\{g_{h}(x), g(x)\}))}\\[2pt]
&\leq C\overline{L}\Big(\big(\|U_{0}(\cdot)\|_{BV(\mathcal{I})}+|g'(0)|+\|g'(\cdot)\|_{BV(\mathbb{R}_{+})}+1\big)
\big(\tau^{2}+\nu^{-1}+2^{-\nu}\big)+2^{-\nu}\Big)x\\[2pt]
&\ \ \ +\big\|(U_{h,\nu}(x,\cdot)-\underline{U},\,
  \mathcal{P}^{(\tau)}_{h}(x,0)(U^{\nu}_{0}(\cdot))-\underline{U})\big\|_{L^{\infty}((-\infty,g_{h}(x)))}
\big\|g'_{h}(\cdot)-g'(\cdot)\big\|_{L^{1}(\mathbb{R}_{+})}\\[2pt]
&\longrightarrow C\overline{L}\big(\|U_{0}(\cdot)\|_{BV(\mathcal{I})}+|g'(0)|+\|g'(\cdot)\|_{BV(\mathbb{R}_{+})}+1\big)x\tau^{2},
\end{align*}
as $\nu\rightarrow \infty$ and $h \rightarrow 0$.

Now choose $C_2=C\overline{L}\big(\|U_{0}(\cdot)\|_{BV(\mathcal{I})}+|g'(0)|+\|g'(\cdot)\|_{BV(\mathbb{R}_{+})}+1\big)\tau^{2}$,
$\varepsilon^{*}_0=\min\{\varepsilon_0,\tilde{\varepsilon}_0\}>0$, and $\tau^*_0=\min\{\tau_0,\tilde{\tau}_0\}>0$.
Then, for $\varepsilon\in(0,\varepsilon^{*}_0)$ and $\tau\in(0,\tau^*_0)$,
if $U_{0}$ and $g(x)$ satisfy \eqref{eq:1.21}, estimate \eqref{eq:1.25} follows from the estimates
on the right-hand side of \eqref{eq:5.7} altogether.

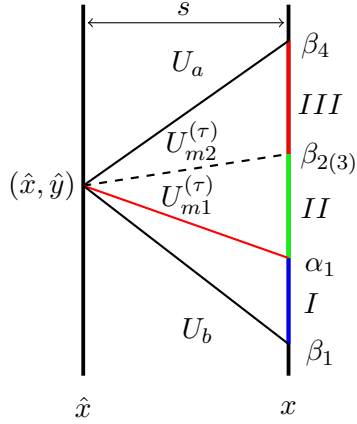
\begin{figure}[ht]
\begin{center}
\begin{tikzpicture}[scale=0.6]
\draw [line width=0.05cm](-4,4.0)--(-4,-4.2);
\draw [line width=0.05cm](0.5,4.0)--(0.5,-4.2);

\draw [line width=0.05cm][red](0.5,3.2)--(0.5,0.7);
\draw [line width=0.05cm][green](0.5,0.7)--(0.5,-1.6);
\draw [line width=0.05cm][blue](0.5,-1.6)--(0.5,-3.5);

\draw [thick](-4,0)--(0.5,3.2);
\draw [thick][dashed](-4,0)--(0.5, 0.7);
\draw [thick][red](-4,0)--(0.5, -1.6);
\draw [thick](-4,0)--(0.5, -3.5);

\draw [thin][<->](-3.9,3.6)--(0.4, 3.6);

\node at (2.6, 2) {$$};
\node at (-1.8, 3.9){$s$};
\node at (-4.9, 0) {$(\hat{x}, \hat{y})$};
\node at (-4, -4.9) {$\hat{x}$};
\node at (0.5, -4.9) {$x$};

\node at (1.1, 3.1) {$\beta_{4}$};
\node at (1.4, 0.6) {$\beta_{2(3)}$};
\node at (1.2, -1.8) {$\alpha_1$};
\node at (1.2, -3.7) {$\beta_{1}$};

\node at (-1.7, 2.7){$U_{a}$};
\node at (-1.6, 1.0){$U^{(\tau)}_{m2}$};
\node at (-1.7, -0.2){$U^{(\tau)}_{m1}$};
\node at (-1.5, -3.2){$U_{b}$};

\node at (1.2, 1.8){$III$};
\node at (1.1, -0.5){$II$};
\node at (1.0, -2.6){$I$};
\end{tikzpicture}
\end{center}
\caption{The optimal convergence rate}\label{fig22}
\end{figure}

\medskip
Second, we show that the convergence rate in \eqref{eq:1.25} is optimal.
To make this, { it suffices to} consider a special case.
As shown in Fig. \ref{fig22}, suppose that the solution of the Riemann problem \eqref{eq:3.6}
consists of only one $1^{\rm st}$ shock-front $\alpha_1$ with
strength $\sigma_{\alpha_1}$, which issues from point $(\hat{x}, \hat{y})$,
with $U_{b}=(\rho_b, u_{b},v_{b}, p_{b})^{\top}=(1, 0, \epsilon, \frac{1}{\gamma a^2_\infty})^{\top}$
and $U_{a}=(\rho_a, u_{a},v_{a}, p_{a})^{\top}=(1+a_{\infty}\epsilon, \frac{\epsilon}{a_{\infty}}, \epsilon, \frac{1}{\gamma a^2_\infty}+\frac{\epsilon}{a_{\infty}})^{\top}$
as its \emph{below} and \emph{above} states for a constant $\epsilon>0$.

Note that, when $\epsilon=0$, $U_{b}=U_{a}=\underline{U}$.
Then, by
Lemma A.1
and the implicit function theorem, we know that the equation:
\begin{eqnarray}\label{eq:5.8}
\Phi_{1}(\sigma_{\alpha_1}; U_{b})=U_{a}
\end{eqnarray}
admits a unique solution $\sigma_{\alpha_1}(\epsilon)$ for sufficiently small $\epsilon$ with $\sigma_{\alpha_1}(0)=0$
and $\sigma'_{\alpha_1}(0)=-\frac{\gamma+1}{2}$.
By the Taylor expansion theorem, we have
\begin{eqnarray}\label{eq:5.9}
\sigma_{\alpha_1}(\epsilon)=-\frac{\gamma+1}{2}\epsilon+O(1)\epsilon^2,
\end{eqnarray}
where the bound of function $O(1)$ depends only on $\underline{U}$,
so that $\alpha_1$ is a shock-front.
Let the speed of $\alpha_1$ be $\mathcal{\dot{S}}_{1}(\sigma_{\alpha_1})$. Then
\begin{eqnarray}\label{eq:5.10}
\mathcal{\dot{S}}_{1}(\sigma_{\alpha_1})=\lambda_{1}(\underline{U})+\frac{1}{2}\sigma_{\alpha_1}+O(1)\sigma^2_{\alpha_1}
=-\frac{1}{a_{\infty}}+\frac{1}{2}\sigma_{\alpha_1}+O(1)\sigma^2_{\alpha_1}.
\end{eqnarray}

We now consider the Riemann problem \eqref{eq:2.11} with \emph{below} state $U_{b}$ and \emph{above} state $U_{a}$.
Let $\beta_j$ be the elementary wave-fronts of the solution with strength $\sigma^{(\tau)}_{\beta_{j}}$
for $1\leq j\leq 4$, that is,
\begin{eqnarray}\label{eq:5.11}
\Phi^{(\tau)}(\boldsymbol{\sigma}^{(\tau)}_{\boldsymbol{\beta}}; U_{b},\tau^2)=\Phi_{1}(\sigma_{\alpha_1}; U_{b}),
\end{eqnarray}
where $\boldsymbol{\sigma}^{(\tau)}_{\boldsymbol{\beta}}=\big(\sigma^{(\tau)}_{\beta_1}, \sigma^{(\tau)}_{\beta_2}, \sigma^{(\tau)}_{\beta_3}, \sigma^{(\tau)}_{\beta_4}\big)$,
and $\Phi^{(\tau)}$ is defined by \eqref{eq:2.10}.

By Lemma \ref{lem:2.2}, for sufficiently small $\tau>0$, we have
\begin{align*}
&\det\Big(\frac{\partial \Phi^{(\tau)}(\boldsymbol{\sigma}^{(\tau)}_{\boldsymbol{\beta}}; U_{b},\tau^2)}{\partial \boldsymbol{\sigma}^{(\tau)}_{\boldsymbol{\beta}}} \Big)\bigg|_{\boldsymbol{\sigma}^{(\tau)}_{_{\boldsymbol{\beta}}}=0}\\
&=\det\Big(\mathbf{r}^{(\tau)}_{1}(\underline{U},\tau^{2}),\mathbf{r}^{(\tau)}_{2}(\underline{U},\tau^{2}),
\mathbf{r}^{(\tau)}_{3}(\underline{U},\tau^{2}), \mathbf{r}^{(\tau)}_{4}(\underline{U},\tau^{2})\Big)\\
&=-\frac{8(a^{2}_{\infty}-\tau^{2})^{\frac{7}{2}}}{(\gamma+1)^{2}a^{8}_{\infty}}<0.
\end{align*}
Then it follows from the implicit function theorem and Lemma \ref{lem:2.3} that equation \eqref{eq:5.11} admits
a unique solution  $\boldsymbol{\sigma}^{(\tau)}_{\boldsymbol{\beta}}$, which is a function of $\sigma_{\alpha_1}$ and $\tau^2$,
such that
\begin{eqnarray}
&&\sigma^{(\tau)}_{\beta_{1}}(\sigma_{\alpha_{1}},0)=\sigma_{\alpha_{1}}, \qquad\,\,
\sigma^{(\tau)}_{\beta_{j}}(\sigma_{\alpha_{1}},0)=0 \quad \mbox{for $j=2,3,4$},\qquad \label{eq:5.12}\\[1pt]
&&\sigma^{(\tau)}_{\beta_{1}}(0,0)=\sigma^{(\tau)}_{\beta_{j}}(0,\tau^2)=0 \qquad\,\, \mbox{for $1\leq j\leq 4$}.\label{eq:5.13}
\end{eqnarray}

We now compute $\frac{\partial^2 \sigma^{(\tau)}_{\beta_j}}{\partial \sigma_{\alpha_1}\partial \tau^2}$ for $1\leq j\leq 4$.
Taking the derivative on \eqref{eq:5.11} with respect to $\sigma_{\alpha_1}$, we have
\begin{eqnarray}\label{eq:5.14}
\sum^{4}_{j=1}\frac{\partial\Phi^{(\tau)}(\boldsymbol{\sigma}^{(\tau)}_{\boldsymbol{\beta}};U_{b}, \tau^2)}
{\partial \sigma^{(\tau)}_{\beta_j}}\frac{\partial \sigma^{(\tau)}_{\beta_j}}{\partial \sigma_{\alpha_1}}
=\frac{\partial \Phi(\sigma_{\alpha_1}; U_{b})}{\partial \sigma_{\alpha_1}}.
\end{eqnarray}

Letting $\alpha_{1}=\tau=0$ in \eqref{eq:5.14}, we obtain
\begin{eqnarray}\label{eq:5.15}
\sum^{4}_{j=1}\mathbf{r}_{j}(\underline{U}, 0)\frac{\partial \sigma^{(\tau)}_{\beta_j}}{\partial \sigma_{\alpha_1}}\bigg|_{\alpha_1=\tau=0}
=\mathbf{r}_{1}(\underline{U}),
\end{eqnarray}
so that
\begin{eqnarray}\label{eq:5.16}
\frac{\partial \sigma^{(\tau)}_{\beta_1}}{\partial \sigma_{\alpha_1}}\bigg|_{\alpha_1=\tau=0}=1,
\qquad\,\,  \frac{\partial \sigma^{(\tau)}_{\beta_j}}{\partial \sigma_{\alpha_1}}\bigg|_{\alpha_1=\tau=0}=0 \quad \mbox{for $j=2,3,4$}.
\end{eqnarray}

Next, we take the derivative on \eqref{eq:5.14} with respect to $\tau^2$ and let $\alpha_1=\tau=0$ to obtain
\begin{eqnarray}\label{eq:5.17}
\sum^{4}_{j=1}\mathbf{r}_{j}(\underline{U}, 0)
\frac{\partial^2 \sigma^{(\tau)}_{\beta_j}}{\partial \sigma_{\alpha_1}\partial \tau^2}\bigg|_{\alpha_1=\tau=0}
=-\sum^{4}_{j=1}\frac{\partial \mathbf{r}_{j}(\underline{U}, \tau^2)}{\partial \tau^2}\bigg|_{\tau=0}\frac{\partial \sigma^{(\tau)}_{\beta_j}}{\partial \sigma_{\alpha_1}}\bigg|_{\alpha_1=\tau=0}.
\end{eqnarray}
Substituting \eqref{eq:5.16} into \eqref{eq:5.17}, we obtain
\begin{eqnarray}
&&\frac{\partial^2 \sigma^{(\tau)}_{\beta_1}}{\partial \sigma_{\alpha_1}\partial \tau^2}\bigg|_{\alpha_1=\tau=0}=\frac{7}{4a^2_{\infty}},
\qquad \frac{\partial^2 \sigma^{(\tau)}_{\beta_4}}{\partial \sigma_{\alpha_1}\partial \tau^2}\bigg|_{\alpha_1=\tau=0}=\frac{1}{4a^2_{\infty}},\label{eq:5.18}\\[2pt]
&&\frac{\partial^2 \sigma^{(\tau)}_{\beta_2}}{\partial \sigma_{\alpha_1}\partial \tau^2}\bigg|_{\alpha_1=\tau=0}=0,
\qquad \frac{\partial^2 \sigma^{(\tau)}_{\beta_3}}{\partial \sigma_{\alpha_1}\partial \tau^2}\bigg|_{\alpha_1=\tau=0}=-\frac{4(a_{\infty}+1)}{(\gamma+1)a_{\infty}}.\qquad\label{eq:5.19}
\end{eqnarray}

Therefore, by the Taylor formula, \eqref{eq:5.12}--\eqref{eq:5.13}, and \eqref{eq:5.18}--\eqref{eq:5.19}, we have
\begin{align}\label{eq:5.20}
\sigma^{(\tau)}_{\beta_1}(\sigma_{\alpha_1},\tau^2)&=\sigma^{(\tau)}_{\beta_1}(\sigma_{\alpha_1},0)
+\sigma^{(\tau)}_{\beta_1}(0,\tau^2)-\sigma^{(\tau)}_{\beta_1}(0,0)\nonumber\\[1mm]
 &\quad\  +\sigma_{\alpha_1}\tau^2\int^{1}_{0}\int^{1}_{0}
\frac{\partial^2 \sigma^{(\tau)}_{\beta_2}}{\partial \sigma_{\alpha_1}\partial \tau^2}(\mu_1\sigma_{\alpha_1}, \mu_2\tau^2)\,{\rm d}\mu_1 {\rm d}\mu_2\nonumber\\[2pt]
&=\sigma_{\alpha_1}+\frac{7}{4a^2_{\infty}}\sigma_{\alpha_1}\tau^2+O(1)\big(|\sigma_{\alpha_1}|+\tau^2\big)|\sigma_{\alpha_1}|\tau^2,
\end{align}
where the bound of function $O(1)$ depends only on $\underline{U}$.

In the same way, we can also deduce that
\begin{align}\label{eq:5.22}
\begin{split}
&\sigma^{(\tau)}_{\beta_2}(\sigma_{\alpha_1},\tau^2)=O(1)\big(|\sigma_{\alpha_1}|+\tau^2\big)|\sigma_{\alpha_1}|\tau^2,\\[2pt]
&\sigma^{(\tau)}_{\beta_3}(\sigma_{\alpha_1},\tau^2)=-\frac{4(a_{\infty}+1)}{(\gamma+1)a_{\infty}}\sigma_{\alpha_1}\tau^2
+O(1)\big(|\sigma_{\alpha_1}|+\tau^2\big)|\sigma_{\alpha_1}|\tau^2,\\[2pt]
&\sigma^{(\tau)}_{\beta_4}(\sigma_{\alpha_1},\tau^2)=\frac{1}{4a^2_{\infty}}\sigma_{\alpha_1}\tau^2+O(1)\big(|\sigma_{\alpha_1}|+\tau^2\big)|\sigma_{\alpha_1}|\tau^2,
\end{split}
\end{align}
where the bound of functions $O(1)$ depends only on $\underline{U}$.

\medskip
Now, we consider the estimate of $\|U^{(\tau)}-U\|_{L^{1}}$.
When $\epsilon$ and $\tau$ are sufficiently small, it follows from \eqref{eq:5.20}--\eqref{eq:5.22} that
$\beta_1$ and $\beta_4$ are shock waves, and $\beta_2$ and $\beta_3$ are vortex sheets/entropy waves.
As shown in Fig. \ref{fig22}, we denote $U^{(\tau)}_{m1}$
and $U^{(\tau)}_{m2}$ as the middle states, between waves $\beta_1$ and $\beta_{2}$, and between waves $\beta_3$ and $\beta_{4}$, respectively.
Then
\begin{align}
&U^{(\tau)}_{m1}=\Phi^{(\tau)}_{1}(\sigma^{(\tau)}_{\beta_1}; U_{b},\tau^2), \,\,\,
U^{(\tau)}_{m2}=\Phi^{(\tau)}_{3}(\sigma^{(\tau)}_{\beta_3}; \Phi^{(\tau)}_{2}(\sigma^{(\tau)}_{\beta_2}; U^{(\tau)}_{m1}, \tau^2),\tau^2),\nonumber\\[2pt]
&U_{a}=\Phi^{(\tau)}_{4}(\sigma^{(\tau)}_{\beta_4}; U^{(\tau)}_{m2},\tau^2).\label{eq:5.24}
\end{align}

Let $\dot{\mathcal{S}}^{(\tau)}_{k}(\sigma^{(\tau)}_{\beta_k},\tau^2)$ be the speed of wave $\beta_k$ for $1\leq k\leq 4$.
Note that, for $k=2,3$, waves $\beta_2$ and $\beta_3$ are vortex sheets/entropy waves so that their speeds satisfy
\begin{align*}
\dot{\mathcal{S}}^{(\tau)}_{2}(\sigma^{(\tau)}_{\beta_2},\tau^2)=\dot{\mathcal{S}}^{(\tau)}_{3}(\sigma^{(\tau)}_{\beta_3},\tau^2)
=\lambda^{(\tau)}_{2}(U^{(\tau)}_{m1},\tau^2)=\lambda^{(\tau)}_{3}(U^{(\tau)}_{m2},\tau^2).
\end{align*}
In addition, it is direct to see that
\begin{align*}
\dot{\mathcal{S}}^{(\tau)}_{1}(\sigma^{(\tau)}_{\beta_1},\tau^2)<\dot{\mathcal{S}}_{1}(\sigma_{\alpha_1},\tau^2)
<\dot{\mathcal{S}}^{(\tau)}_{2}(\sigma^{(\tau)}_{\beta_2},\tau^2)=\dot{\mathcal{S}}^{(\tau)}_{3}(\sigma^{(\tau)}_{\beta_3},\tau^2)
<\dot{\mathcal{S}}^{(\tau)}_{4}(\sigma^{(\tau)}_{\beta_4},\tau^2).
\end{align*}
For any $x>\hat{x}$, let $I_x:=\{(x,y)\,:\, y\in\mathbb{R}\}$.
As shown in Fig. \ref{fig22}, let $I$, $II$ and $III$ be the intervals on $I_x$ between waves $\beta_1$ and $\alpha_1$, $\alpha_1$ and $\beta_2(3)$,
and $\beta_2(3)$ and $\beta_4$, respectively, issuing from point $(\hat{x}, \hat{y})$.
From the construction, we see that
$$
\|U^{(\tau)}-U\|_{L^{1}(I_x\backslash(I\cup II\cup III))}\equiv0.
$$
Thus, it suffices to consider the comparison of the Riemann solutions on intervals $I$, $II$, and $III$ in the following three cases:

\medskip
\emph{Case 1}:\ \emph{Estimate of $\|U^{(\tau)}-U\|_{L^1(I)}$.} By direct computation, we have
\begin{align}\label{eq:5.26}
\dot{\mathcal{S}}^{(\tau)}_{1}(\sigma^{(\tau)}_{\beta_1},\tau^2)&=\dot{\mathcal{S}}^{(\tau)}_{1}(0,\tau^2)
+\frac{1}{2}\sigma^{(\tau)}_{\beta_1}+O(1)(\sigma^{(\tau)}_{\beta_1})^2\nonumber\\
&=-\big(a^2_{\infty}-\tau^2 \big)^{-\frac{1}{2}}
+\frac{1}{2}\sigma^{(\tau)}_{\beta_1}+O(1)(\sigma^{(\tau)}_{\beta_1})^2.
\end{align}
Then, if $\tau$ and $\epsilon$ are sufficiently small, it follows from \eqref{eq:5.10} and \eqref{eq:5.26} that
\begin{align}\label{eq:5.27}
&\dot{\mathcal{S}}_{1}(\sigma_{\alpha_1})-\dot{\mathcal{S}}^{(\tau)}_{1}(\sigma^{(\tau)}_{\beta_1},\tau^2)\nonumber\\[2pt]
&=\big(a^2_{\infty}-\tau^2 \big)^{-\frac{1}{2}}-\frac{1}{a_{\infty}}+\frac{1}{2}\big(\sigma_{\alpha_1}-\sigma^{(\tau)}_{\beta_1}\big)
+O(1)\tau^2(\sigma^{(\tau)}_{\beta_1})^{2}
+O(1)\big(\sigma_{\alpha_1}-\sigma^{(\tau)}_{\beta_1}\big)^{2}\nonumber\\[2pt]
&=\frac{1}{2a^3_{\infty}}\tau^2+O(1)\big(|\sigma_{\alpha_1}|+\tau^2\big)\tau^2>0.
\end{align}
Therefore, $I=\big\{(x,y)\,:\,\hat{y}+\dot{\mathcal{S}}^{(\tau)}_{1}(\sigma^{(\tau)}_{\beta_1},\tau^2)(x-\hat{x})<y<\hat{y}+\dot{\mathcal{S}}_{1}(\sigma_{\alpha_1})(x-\hat{x})\big\}$,
so that the length of $I$ is
\begin{align}\label{eq:5.28}
|I|&=\big(\dot{\mathcal{S}}_{1}(\sigma_{\alpha_1})-\dot{\mathcal{S}}^{(\tau)}_{1}(\sigma^{(\tau)}_{\beta_1},\tau^2)\big)(x-\hat{x})\nonumber\\
&=\frac{1}{2a^3_{\infty}}\tau^2(x-\hat{x})+O(1)\big(|\sigma_{\alpha_1}|+\tau^2\big)\tau^2(x-\hat{x}).
\end{align}

Note that $U^{(\tau)}=U^{(\tau)}_{m1}$ and $U=U_{b}$ on $I$.
Then $|U^{(\tau)}-U|=|U^{(\tau)}_{m1}-U_{b}|$ on $I$.
We now estimate the term: $|U^{(\tau)}_{m1}-U_{b}|$.
By \eqref{eq:5.22}
and direct computation, we have
\begin{align}\label{eq:5.29}
|\rho^{(\tau)}_{m1}-\rho_{b}|&=\big|\Phi^{(\tau),(1)}_1(\sigma^{(\tau)}_{\beta_1}; U_{b}, \tau^2)-\rho_{b}\big|\nonumber\\[2pt]
&=\frac{2\big((a^2_{\infty}-\tau^2)^{\frac{1}{2}}-\tau^2\big)(a^2_{\infty}-\tau^2)^2}{(\gamma+1)a^4_{\infty}}|\sigma^{(\tau)}_{\beta_1}|
+O(1)(\sigma^{(\tau)}_{\beta_1})^2\nonumber\\[2pt]
&=\frac{2a_{\infty}}{\gamma+1}\Big(1-\frac{2a_{\infty}+5}{2a^2_{\infty}}\tau^2\Big)\Big(1+\frac{7}{4a^2_{\infty}}\tau^2\Big)|\sigma_{\alpha_1}|\nonumber\\
&\quad\, +O(1)\big(|\sigma_{\alpha_1}|+|\sigma_{\alpha_1}|\tau^2+\tau^4\big)|\sigma_{\alpha_1}|\nonumber\\[2pt]
&=\frac{2a_{\infty}}{\gamma+1}\Big(1-\frac{4a_{\infty}+3}{4a^2_{\infty}}\tau^2\Big)|\sigma_{\alpha_1}|
+O(1)\big(|\sigma_{\alpha_1}|+|\sigma_{\alpha_1}|\tau^2+\tau^4\big)|\sigma_{\alpha_1}|,
\end{align}
\begin{align}
|u^{(\tau)}_{m1}-u_{b}|&=\big|\Phi^{(\tau),(2)}_1(\sigma^{(\tau)}_{\beta_1}; U_{b}, \tau^2)-u_{b}\big|\nonumber\\[2pt]
&=\frac{2(a^2_{\infty}-\tau^2)^{\frac{3}{2}}}{(\gamma+1)a^4_{\infty}}|\sigma^{(\tau)}_{\beta_1}|
+O(1)(\sigma^{(\tau)}_{\beta_1})^2\nonumber\\[2pt]
&=\frac{2}{(\gamma+1)a_{\infty}}\Big(1+\frac{1}{4a^2_{\infty}}\tau^2\Big)|\sigma_{\alpha_1}|
+O(1)\big(|\sigma_{\alpha_1}|+|\sigma_{\alpha_1}|\tau^2+\tau^4\big)|\sigma_{\alpha_1}|,
\end{align}
\begin{align}
|v^{(\tau)}_{m1}-v_{b}|&=\big|\Phi^{(\tau),(3)}_1(\sigma^{(\tau)}_{\beta_1}; U_{b}, \tau^2)-v_{b}\big|\nonumber\\[2pt]
&=\frac{2(a^2_{\infty}-\tau^2)^2}{(\gamma+1)a^4_{\infty}}|\sigma^{(\tau)}_{\beta_1}|
+O(1)(\sigma^{(\tau)}_{\beta_1})^2\nonumber\\[2pt]
&=\frac{2}{\gamma+1}\Big(1-\frac{1}{4a^2_{\infty}}\tau^2\Big)|\sigma_{\alpha_1}|
+O(1)\big(|\sigma_{\alpha_1}|+|\sigma_{\alpha_1}|\tau^2+\tau^4\big)|\sigma_{\alpha_1}|,
\end{align}
\begin{align}\label{eq:5.30}
|p^{(\tau)}_{m1}-p_{b}|&=\big|\Phi^{(\tau),(4)}_1(\sigma^{(\tau)}_{\beta_1}; U_{b}, \tau^2)-p_{b}\big|\nonumber\\[2pt]
&=\frac{2(a^2_{\infty}-\tau^2)^{\frac{3}{2}}}{(\gamma+1)a^4_{\infty}}|\sigma^{(\tau)}_{\beta_1}|
+O(1)(\sigma^{(\tau)}_{\beta_1})^2\nonumber\\[2pt]
&=\frac{2}{(\gamma+1)a_{\infty}}\Big(1+\frac{1}{4a^2_{\infty}}\tau^2\Big)|\sigma_{\alpha_1}|
+O(1)\big(|\sigma_{\alpha_1}|+|\sigma_{\alpha_1}|\tau^2+\tau^4\big)|\sigma_{\alpha_1}|.
\end{align}

Therefore, it follows from \eqref{eq:5.29}--\eqref{eq:5.30} that
\begin{align*}
|U^{(\tau)}_{m1}-U_{b}|&=|\rho^{(\tau)}_{m1}-\rho_{b}|+|u^{(\tau)}_{m1}-u_{b}|+|v^{(\tau)}_{m1}-v_{b}|
+|p^{(\tau)}_{m1}-p_{b}|\\[2pt]
&=\bigg(\frac{2(a^2_{\infty}+a_{\infty}+2)}{(\gamma+1)a_{\infty}}
-\frac{4a^3_{\infty}+3a^2_{\infty}+a_{\infty}-2}{2(\gamma+1)a^3_{\infty}}\tau^2\bigg)|\sigma_{\alpha_1}|\\[3pt]
&\quad \ +O(1)\big(|\sigma_{\alpha_1}|+|\sigma_{\alpha_1}|\tau^2+\tau^4\big)|\sigma_{\alpha_1}|.
\end{align*}
Then we have
\begin{align}
&\int_{I}|U^{(\tau)}-U|\,{\rm d}y\nonumber\\[2pt]
&=|U^{(\tau)}_{m1}-U_{b}||I|\nonumber\\[2pt]
&=\frac{1}{2a^3_{\infty}}\Big(\frac{2(a^2_{\infty}+a_{\infty}+2)}{(\gamma+1)a_{\infty}}
-\frac{4a^3_{\infty}+3a^2_{\infty}+a_{\infty}-2}{2(\gamma+1)a^3_{\infty}}\tau^2\Big)|\sigma_{\alpha_1}|\tau^2(x-\hat{x})\nonumber\\[2pt]
&\quad  +O(1)\big(|\sigma_{\alpha_1}|+\tau^2\big)|\sigma_{\alpha_1}|\tau^2(x-\hat{x})\nonumber\\[2pt]
&=\frac{a^2_{\infty}+a_{\infty}+2}{(\gamma+1)a^{4}_{\infty}}|\sigma_{\alpha_1}|\tau^2(x-\hat{x})
+O(1)\big(|\sigma_{\alpha_1}|+\tau^2\big)|\sigma_{\alpha_1}|\tau^2(x-\hat{x}).\label{eq:5.31}
\end{align}

\smallskip
\emph{Case 2}:\ \emph{Estimate of $\|U^{(\tau)}-U\|_{L^1(II)}$}.
By \eqref{eq:5.20}, for sufficiently small $\tau$, we can obtain
\begin{align}
\dot{\mathcal{S}}^{(\tau)}_{2}(\sigma^{(\tau)}_{\beta_2},\tau^2)&=\lambda^{(\tau)}_{2}(\Phi^{(\tau)}_{1}(\sigma^{(\tau)}_{\beta_1}; U_b, \tau^2),\tau^2)\nonumber\\[2pt]
&=\lambda^{(\tau)}_{2}(\underline{U},\tau^2)+\frac{\partial \lambda^{(\tau)}_{2}(\Phi^{(\tau)}_{1}(\sigma^{(\tau)}_{\beta_1}; U_b, \tau^2),\tau^2)}{\partial \sigma^{(\tau)}_{\beta_1}}\bigg|_{\sigma^{(\tau)}_{\beta_1}=0}\sigma^{(\tau)}_{\beta_1}+O(1)(\sigma^{(\tau)}_{\beta_1})^2\nonumber\\[2pt]
&=\frac{2(a^2_{\infty}-\tau^2)^2}{(\gamma+1)a^4_{\infty}}\sigma^{(\tau)}_{\beta_1}+O(1)(\sigma^{(\tau)}_{\beta_1})^2\nonumber\\[2pt]
&=\frac{2}{\gamma+1}\Big(1-\frac{2}{a^2_{\infty}}\tau^2\Big)\Big(1+\frac{7}{4a^2_{\infty}}\tau^2\Big)\sigma_{\alpha_1}\nonumber\\[2pt]
&\quad\, +O(1)\big(|\sigma_{\alpha_1}|+|\sigma_{\alpha_1}|\tau^2+\tau^4\big)|\sigma_{\alpha_1}|\nonumber\\[2pt]
&=\frac{2}{\gamma+1}\Big(1-\frac{\tau^2}{4a^2_{\infty}}\Big)\sigma_{\alpha_1}
+O(1)\big(|\sigma_{\alpha_1}|+|\sigma_{\alpha_1}|\tau^2+\tau^4\big)|\sigma_{\alpha_1}|.\label{eq:5.32}
\end{align}

Then, by \eqref{eq:5.10} and \eqref{eq:5.32}, we have
\begin{align}\label{eq:5.33}
&\dot{\mathcal{S}}^{(\tau)}_{2}(\sigma^{(\tau)}_{\beta_2},\tau^2)-\dot{\mathcal{S}}_{1}(\sigma_{\alpha_1})\nonumber\\[5pt]
&=\frac{1}{a_{\infty}}+\frac{1}{2(\gamma+1)}\Big(\gamma+5-\frac{\tau^2}{a^2_{\infty}}\Big)|\sigma_{\alpha_1}|
+O(1)\big(|\sigma_{\alpha_1}|+|\sigma_{\alpha_1}|\tau^2+\tau^4\big)|\sigma_{\alpha_1}|>0,
\end{align}
so that $II=\big\{(x,y)\,:\,\hat{y}+\dot{\mathcal{S}}_{1}(\sigma_{\alpha_1}(x-\hat{x})<y<\hat{y}+\dot{\mathcal{S}}^{(\tau)}_{2}(\sigma^{(\tau)}_{\beta_2},\tau^2)(\sigma_{\alpha_1}(x-\hat{x})\big\}$,
and its length is
\begin{align}
|II|&=\big(\dot{\mathcal{S}}^{(\tau)}_{2}(\sigma^{(\tau)}_{\beta_2},\tau^2)-\dot{\mathcal{S}}_{1}(\sigma_{\alpha_1})\big)(x-\hat{x})\nonumber\\[2pt]
&=\Big(\frac{1}{a_{\infty}}+\frac{1}{2(\gamma+1)}\Big(\gamma+5-\frac{\tau^2}{a^2_{\infty}}\Big)|\sigma_{\alpha_1}|\Big)(x-\hat{x})\nonumber\\[2pt]
&\quad\,+O(1)\big(|\sigma_{\alpha_1}|+|\sigma_{\alpha_1}|\tau^2+\tau^4\big)|\sigma_{\alpha_1}|(x-\hat{x}).\label{eq:5.34}
\end{align}

We now compute $|U^{(\tau)}-U|$ on $II$. Set
\begin{align*}
\tilde{\Phi}^{(\tau)}(\tilde{\boldsymbol{\sigma}}^{(\tau)}_{\tilde{\boldsymbol{\beta}}}; U^{(\tau)}_{m1},\tau^2)
=\Phi^{(\tau)}_{4}(\sigma^{(\tau)}_{\beta_{4}};\Phi^{(\tau)}_{3}\big(\sigma^{(\tau)}_{\beta_{3}}; \Phi^{(\tau)}_{2}(\sigma^{(\tau)}_{\beta_{2}}; U^{(\tau)}_{m1},\tau^2 ),\tau^2),\tau^2\big),\tau^2),
\end{align*}
where $\tilde{\boldsymbol{\sigma}}^{(\tau)}_{\tilde{\boldsymbol{\beta}}}=(\sigma^{(\tau)}_{\beta_{2}}, \sigma^{(\tau)}_{\beta_{3}}, \sigma^{(\tau)}_{\beta_{4}})$.

Denote $\tilde{\Phi}^{(\tau),(k)}(\tilde{\boldsymbol{\sigma}}^{(\tau)}_{\tilde{\boldsymbol{\beta}}}; U^{(\tau)}_{m1},\tau^2)$ as the $k^{\rm th}$-component of $\tilde{\Phi}^{(\tau)}(\tilde{\boldsymbol{\sigma}}^{(\tau)}_{\tilde{\boldsymbol{\beta}}}; U^{(\tau)}_{m1},\tau^2)$ for $1\leq k\leq 4$.
Then it follows from
\eqref{eq:5.22} that
\begin{align}\label{eq:5.35}
|\rho_{a}-\rho^{(\tau)}_{m1}|
&=\big|\tilde{\Phi}^{(\tau),(1)}(\tilde{\boldsymbol\sigma}^{(\tau)}_{\tilde{\boldsymbol\beta}}; U^{(\tau)}_{m1}, \tau^2)-\rho^{(\tau)}_{m1}\big|\nonumber\\[2pt]
&=\Big|\sigma^{(\tau)}_{\beta_{3}}+\frac{2}{(\gamma+1)a^4_{\infty}}\big((a^2_{\infty}-\tau^2)^{\frac{1}{2}}+\tau^2\big)
\big(a^2_{\infty}-\tau^2\big)^2\sigma^{(\tau)}_{\beta_{4}}\Big|\nonumber\\[2pt]
&\quad\,+O(1)\sum^{4}_{k=2}|\sigma^{(\tau)}_{\beta_{k}}|^{2}\nonumber \\[2pt]
&=\frac{8a_{\infty}+7}{2(\gamma+1)a_{\infty}}|\sigma_{\alpha_1}|\tau^2
+O(1)\big(|\sigma_{\alpha_1}|+\tau^2\big)|\sigma_{\alpha_1}|\tau^2.
\end{align}

Similarly, we have
\begin{align}\label{eq:5.36}
|u_{a}-u^{(\tau)}_{m1}|&=\big|\tilde{\Phi}^{(\tau),(2)}(\tilde{\boldsymbol\sigma}^{(\tau)}_{\tilde{\boldsymbol\beta}}; U^{(\tau)}_{m1}, \tau^2)-u^{(\tau)}_{m1}\big|\nonumber \\[2pt]
&=\bigg|\sigma^{(\tau)}_{\beta_{2}}+\frac{2(a^2_{\infty}-\tau^2)^{\frac{3}{2}}}{(\gamma+1)a^4_{\infty}}\sigma^{(\tau)}_{\sigma_{\beta_4}}\bigg|
+O(1)\sum^{4}_{k=2}|\sigma^{(\tau)}_{\beta_{k}}|^{2}\nonumber \\[2pt]
&=\frac{1}{2(\gamma+1)a^3_{\infty}}|\sigma_{\alpha_1}|\tau^2
+O(1)\big(|\sigma_{\alpha_1}|+\tau^2\big)|\sigma_{\alpha_1}|\tau^2,
\end{align}
\begin{align}
|v_{a}-v^{(\tau)}_{m1}|&=\big|\tilde{\Phi}^{(\tau),(3)}(\tilde{\boldsymbol\sigma}^{(\tau)}_{\tilde{\boldsymbol\beta}}; U^{(\tau)}_{m1}, \tau^2)-v^{(\tau)}_{m1}\big|\nonumber \\[2pt]
&=\frac{2(a^2_{\infty}-\tau^2)^2}{(\gamma+1)a^4_{\infty}}|\sigma^{(\tau)}_{\sigma_{\beta_4}}|
+O(1)\sum^{4}_{k=2}|\sigma^{(\tau)}_{\beta_{k}}|^{2}\nonumber \\[2pt]
&=\frac{1}{2(\gamma+1)a^2_{\infty}}|\sigma_{\alpha_1}|\tau^2
+O(1)\big(|\sigma_{\alpha_1}|+\tau^2\big)|\sigma_{\alpha_1}|\tau^2,
\end{align}
\begin{align}\label{eq:5.37}
|p_{a}-p^{(\tau)}_{m1}|&=\big|\tilde{\Phi}^{(\tau),(4)}(\tilde{\boldsymbol\sigma}^{(\tau)}_{\tilde{\boldsymbol\beta}}; U^{(\tau)}_{m1}, \tau^2)-p^{(\tau)}_{m1}\big|\nonumber \\[2pt]
&=\frac{2(a^2_{\infty}-\tau^2)^{\frac{3}{2}}}{(\gamma+1)a^4_{\infty}}|\sigma^{(\tau)}_{\sigma_{\beta_4}}|
+O(1)\sum^{4}_{k=2}|\sigma^{(\tau)}_{\beta_{k}}|^{2}\nonumber \\[2pt]
&=\frac{1}{2(\gamma+1)a^3_{\infty}}|\sigma_{\alpha_1}|\tau^2
+O(1)\big(|\sigma_{\alpha_1}|+\tau^2\big)|\sigma_{\alpha_1}|\tau^2.
\end{align}

Therefore, we obtain
\begin{align}
|U_{a}-U^{(\tau)}_{m1}|&=|\rho_{a}-\rho^{(\tau)}_{m1}|+|u_{a}-u^{(\tau)}_{m1}|+|v_{a}-v^{(\tau)}_{m1}|+|p_{a}-p^{(\tau)}_{m1}|\nonumber\\[2pt]
&=\frac{8a^3_{\infty}+7a^2_{\infty}+a_{\infty}+2}{2(\gamma+1)a^3_{\infty}}|\sigma_{\alpha_1}|\tau^2
+O(1)\big(|\sigma_{\alpha_1}|+\tau^2\big)|\sigma_{\alpha_1}|\tau^2.\label{eq:5.38}
\end{align}

Then, by \eqref{eq:5.34} and \eqref{eq:5.38}, we obtain
\begin{align}\label{eq:5.39}
&\int_{II}|U^{(\tau)}-U|\,{\rm d}y\nonumber\\[2pt]
&=|U^{(\tau)}-U||II|\nonumber\\[2pt]
&=\frac{8a^3_{\infty}+7a^2_{\infty}+a_{\infty}+2}{2(\gamma+1)a^4_{\infty}}|\sigma_{\alpha_1}|\tau^2(x-\hat{x})\nonumber\\
&\quad\, +O(1)\big(|\sigma_{\alpha_1}|+\tau^2\big)|\sigma_{\alpha_1}|\tau^2(x-\hat{x}).
\end{align}

\emph{Case 3}:\ \emph{Estimate of $\|U^{(\tau)}-U\|_{L^1(III)}$}. For sufficiently small $\tau$, by \eqref{eq:5.22}, we have
\begin{align*}
\dot{\mathcal{S}}^{(\tau)}_{4}(\sigma^{(\tau)}_{\beta_4},\tau^2)&=\dot{\mathcal{S}}^{(\tau)}_{4}(0,\tau^2)
+\frac{1}{2}\sigma^{(\tau)}_{\beta_4}+O(1)(\sigma^{(\tau)}_{\beta_4})^2\\[2pt]
&=\big(a_{\infty}-\tau^2 \big)^{-\frac{1}{2}}+\frac{1}{2}\sigma^{(\tau)}_{\beta_4}+O(1)(\sigma^{(\tau)}_{\beta_4})^2\\[2pt]
&=\frac{1}{a_{\infty}}+\frac{1}{2a^{2}_{\infty}}\Big(1+\frac{a_{\infty}}{4}\sigma_{\alpha_1}\Big)\tau^2
+O(1)\big(|\sigma_{\alpha_1}|^2+|\sigma_{\alpha_1}|\tau^2+\tau^2\big)\tau^2.
\end{align*}

Then, by \eqref{eq:5.32},
\begin{align}\label{eq:5.40}
&\dot{\mathcal{S}}^{(\tau)}_{4}(\sigma^{(\tau)}_{\beta_4},\tau^2)-\dot{\mathcal{S}}^{(\tau)}_{3}(\sigma^{(\tau)}_{\beta_3},\tau^2)\nonumber\\[2pt]
&=\frac{1}{a_{\infty}}+\frac{1}{2a^2_{\infty}}\tau^2-\frac{2}{\gamma+1}\sigma_{\alpha_1}
+\frac{2a_{\infty}+\gamma+1}{(\gamma+1)a^3_{\infty}}\sigma_{\alpha_1}\tau^2\nonumber\\[2pt]
&\quad \ +O(1)\big(|\sigma_{\alpha_1}|^2+|\sigma_{\alpha_1}|^2\tau^2+\tau^4\big)>0,
\end{align}
so that $III=\{(x,y)\,:\,\hat{y}+\dot{\mathcal{S}}^{(\tau)}_{3}(\sigma^{(\tau)}_{\beta_3},\tau^2)(x-\hat{x})<y<\hat{y}+\dot{\mathcal{S}}^{(\tau)}_{4}(\sigma^{(\tau)}_{\beta_4},\tau^2)(x-\hat{x})\}$.
Then, by \eqref{eq:5.40}, its length is
\begin{align}\label{eq:5.41}
|III|&=\big(\dot{\mathcal{S}}^{(\tau)}_{4}(\sigma^{(\tau)}_{\beta_4},\tau^2)-\dot{\mathcal{S}}^{(\tau)}_{3}(\sigma^{(\tau)}_{\beta_3},\tau^2)\big)
(x-\hat{x})\nonumber\\[2pt]
&=\Big(\frac{1}{a_{\infty}}+\frac{1}{2a^2_{\infty}}\tau^2-\frac{2}{\gamma+1}\sigma_{\alpha_1}
+\frac{2a_{\infty}+\gamma+1}{(\gamma+1)a^3_{\infty}}\sigma_{\alpha_1}\tau^2\Big)(x-\hat{x})\nonumber\\[2pt]
&\quad \ +O(1)\big(|\sigma_{\alpha_1}|^2+|\sigma_{\alpha_1}|^2\tau^2+\tau^4\big)(x-\hat{x}).
\end{align}

It follows from \eqref{eq:5.40} that $|U^{(\tau)}-U|=|U_{a}-U^{(\tau)}_{m2}|$ on $III$. Hence, by \eqref{eq:5.24} and direct calculation,
\begin{align}\label{eq:5.42}
\begin{split}
|\rho_{a}-\rho^{(\tau)}_{m2}|&=\big|\Phi^{(\tau),(1)}_{4}(\sigma^{(\tau)}_{\beta_4}; U^{(\tau)}_{m2}, \tau^2)-\rho^{(\tau)}_{m2}\big|\\[1pt]
&=\frac{2}{2(\gamma+1)a_{\infty}}|\sigma_{\alpha_1}|\tau^2
+O(1)\big(|\sigma_{\alpha_1}|+\tau^2\big)|\sigma_{\alpha_1}|\tau^2,\\[2mm]
|u_{a}-u^{(\tau)}_{m1}|&=\big|\Phi^{(\tau),(2)}_{4}(\sigma^{(\tau)}_{\beta_4}; U^{(\tau)}_{m2}, \tau^2)-u^{(\tau)}_{m2}\big|\\[1pt]
&=\frac{3}{2(\gamma+1)a^3_{\infty}}|\sigma_{\alpha_1}|\tau^2
+O(1)\big(|\sigma_{\alpha_1}|+\tau^2\big)|\sigma_{\alpha_1}|\tau^2,\\[2mm]
|v_{a}-v^{(\tau)}_{m1}|&=\Big|\Phi^{(\tau),(3)}_{4}(\sigma^{(\tau)}_{\beta_4}; U^{(\tau)}_{m2}, \tau^2)-v^{(\tau)}_{m2}\Big|\\[1pt]
&=\frac{1}{2(\gamma+1)a^2_{\infty}}|\sigma_{\alpha_1}|\tau^2
+O(1)\big(|\sigma_{\alpha_1}|+\tau^2\big)|\sigma_{\alpha_1}|\tau^2,\\[2mm]
|p_{a}-p^{(\tau)}_{m2}|&=\Big|\Phi^{(\tau),(4)}_{4}(\sigma^{(\tau)}_{\beta_4}; U^{(\tau)}_{m2}, \tau^2)-p^{(\tau)}_{m2}\Big|\\[1pt]
&=\frac{3}{2(\gamma+1)a^3_{\infty}}|\sigma_{\alpha_1}|\tau^2
+O(1)\big(|\sigma_{\alpha_1}|+\tau^2\big)|\sigma_{\alpha_1}|\tau^2.
\end{split}
\end{align}

Therefore, we have
\begin{align}\label{eq:5.44}
|U_{a}-U^{(\tau)}_{m2}|&=|\rho_{a}-\rho^{(\tau)}_{m2}|+|u_{a}-u^{(\tau)}_{m2}|+|v_{a}-v^{(\tau)}_{m2}|+|p_{a}-p^{(\tau)}_{m2}|\nonumber\\[2pt]
&=\frac{6+a_{\infty}+a^2_{\infty}}{2(\gamma+1)a^3_{\infty}}|\sigma_{\alpha_1}|\tau^2
+O(1)\big(|\sigma_{\alpha_1}|+\tau^2\big)|\sigma_{\alpha_1}|\tau^2.
\end{align}

By \eqref{eq:5.41} and \eqref{eq:5.44}, we obtain
\begin{align}\label{eq:5.45}
&\int_{III}|U_{a}-U^{(\tau)}_{m2}|\,{\rm d}y\nonumber\\[2pt]
&=|U_{a}-U^{(\tau)}_{m2}||III|\nonumber\\[2pt]
&=\frac{a^2_{\infty}+a_{\infty}+6}{2(\gamma+1)a^3_{\infty}}|\sigma_{\alpha_1}|\tau^2(x-\hat{x})+O(1)\big(|\sigma_{\alpha_1}|+\tau^2\big)|\sigma_{\alpha_1}|\tau^2(x-\hat{x}).
\end{align}

Finally, combining \eqref{eq:5.31}, \eqref{eq:5.39}, and \eqref{eq:5.45} altogether, we have
\begin{align}\label{eq:5.46}
&\|U^{(\tau)}-U\|_{L^{1}(I_x)}\nonumber\\[3pt]
&=\bigg(\int_{I}+\int_{II}+\int_{III}\bigg)|U^{(\tau)}-U|\,{\rm d}x\nonumber\\[3pt]
&=\bigg(\frac{a^2_{\infty}+a_{\infty}+2}{(\gamma+1)a^{4}_{\infty}}
+\frac{8a^3_{\infty}+7a^2_{\infty}+a_{\infty}+2}{2(\gamma+1)a^4_{\infty}}
+\frac{a^2_{\infty}+a_{\infty}+6}{2(\gamma+1)a^3_{\infty}}\bigg)|\sigma_{\alpha_1}|\tau^2(x-\hat{x})\nonumber\\[2pt]
&\quad\ +O(1)\big(|\sigma_{\alpha_1}|+\tau^2\big)|\sigma_{\alpha_1}|\tau^2(x-\hat{x})\nonumber\\[2pt]
&=\frac{9a^3_{\infty}+10a^2_{\infty}+9a_{\infty}+6}{2(\gamma+1)a^{4}_{\infty}}|\sigma_{\alpha_1}|\tau^2(x-\hat{x})
+O(1)\big(|\sigma_{\alpha_1}|+\tau^2\big)|\sigma_{\alpha_1}|\tau^2(x-\hat{x}).
\end{align}

Substitute \eqref{eq:5.9} into \eqref{eq:5.46} and let $\hat{x}=0$. Then
\begin{align*}
\|U^{(\tau)}-U\|_{L^{1}(I_x)}=\frac{9a^3_{\infty}+10a^2_{\infty}+9a_{\infty}+6}{4a^{4}_{\infty}}\epsilon x\tau^2
+O(1)\big(\epsilon+\tau^2\big)\epsilon x\tau^2.
\end{align*}
Therefore, based on the special solution, we know the convergence rate in \eqref{eq:1.25} is optimal.
This completes the proof of Theorem \ref{thm:1.2}.
\end{proof}

\appendix
\section{Existence of Entropy Solutions of Problem \eqref{eq:1.12}--\eqref{eq:1.15}}
\setcounter{equation}{0}

In this appendix, we are concerned with the global existence of the entropy solutions of the initial-boundary value
problem \eqref{eq:1.12}--\eqref{eq:1.15} around the \emph{background solution} $\underline{U}\doteq(1,0,0,\frac{1}{\gamma a^{2}_{\infty}})^{\top}$
and $g\equiv0$.
That is, we consider problem \eqref{eq:1.12}--\eqref{eq:1.15}
for the initial-boundary data $U_{0}$ and $g(x)$, which are small perturbations of $U=\underline{U}$ and $g\equiv 0$.
Because the argument is similar and shorter than the one for {Problem I} in \S 2, we skip the proof.
By direct calculation, the characteristic polynomial for system \eqref{eq:1.14} is
\begin{eqnarray}\label{eq:3.1}
(\lambda-v)^2\big((\lambda-v)^{2}-c^{2}\big)=0,
\end{eqnarray}
which admits four roots, \emph{i.e.}, the eigenvalues of system \eqref{eq:1.14}:
\begin{eqnarray}\label{eq:3.2}
\lambda_{j}(U)=v+(-1)^{j}c  \ \ \mbox{for $j=1,4$}, \qquad\,\,\, \lambda_{i}(U)=v \ \ \mbox{for $i=2, 3$}.
\end{eqnarray}
The corresponding eigenvectors are
\begin{align}
&\mathbf{r}_{j}(U)=\frac{2}{\gamma+1}((-1)^{j}\frac{\rho}{c},\,(-1)^{j+1}\lambda_{j}(U),\,1,\,(-1)^{j}\rho c)^{\top} \quad\, \mbox{for $j=1,4$,}\quad \label{eq:3.3}\\[2pt]
&\mathbf{r}_{2}(U)=(0, 1, 0, 0)^{\top},\quad \ \mathbf{r}_{3}(U)=(1, 0, 0, 0)^{\top}, \label{eq:3.4}
\end{align}
which satisfy
\begin{equation}\label{eq:3.5}
\nabla_{U}\lambda_{j}(U)\cdot\mathbf{r}_{j}(U)\equiv 1  \ \, \mbox{for $j=1, 4$},\qquad\,
\nabla_{U}\lambda_{i}(U)\cdot\mathbf{r}_{i}(U)\equiv 0 \ \, \mbox{for $i=2,3$}.
\end{equation}

From \eqref{eq:3.5}, we know that the $1^{\rm st}$ and $4^{\rm th}$ characteristics fields are genuinely nonlinearity,
and the $2^{\rm nd}$ and $3^{\rm th}$ characteristics fields are linearly degenerate. Moreover, we have
\begin{eqnarray}\label{eq:3.5a}
\lambda^{(\tau)}_{j}(U,\tau^2)\big|_{\tau=0}=\lambda_{j}(U),\,\,\,\,
\mathbf{r}^{(\tau)}_{j}(U, \tau^2)\big|_{\tau=0}=\mathbf{r}_{j}(U) \quad\,\,\, \mbox{for $1 \leq j\leq 4$}.
\end{eqnarray}

\medskip
\subsection{Riemann problems and local interaction estimates for system \eqref{eq:1.14}}
In this subsection, we first consider the following Riemann problem.
\begin{equation}\label{eq:3.6}
\begin{cases}
\mbox{System}\, \eqref{eq:1.14},\\[5pt]
\left.U\right|_{x=\hat{x}_{0}}=
\begin{cases}
U_{a}, \quad  &y>\hat{y}_{0},\\[2pt]
U_{b}, \quad  &y<\hat{y}_{0}.
\end{cases}
\end{cases}
\end{equation}

Based on \eqref{eq:3.5}, we can follow the ideas in \cite{bressan,kong-yang, smoller} to show the solvability
of the Riemann problem \eqref{eq:3.6}.

\medskip
\noindent
{\bf Lemma A.1.}
{\it There exists a constant $\tilde{\epsilon}_{1}>0$ depending only on $\underline{U}$ such that,
for any given two constant states $U_{a}, U_{b}\in \mathcal{O}_{\tilde{\epsilon}_1}(\underline{U})$,
the Riemann problem \eqref{eq:3.6} admits a unique entropy solution, with five constant states denoted by $U_{k}, 0\leq k\leq 4$,
and separated by shock waves or rarefaction waves for the genuinely nonlinear characteristics fields and
by vortex sheets/entropy waves for the linearly degenerate characteristics fields.
Moreover, there exists a constant $\tilde{\delta}_{1}>0$ sufficiently small, depending only on $\underline{U}$, such that
the $j^{\rm th}$
physical admissible wave $\alpha_{j}$ in $\mathcal{O}_{\tilde{\epsilon}_1}(\underline{U})$
can be parameterized by $\sigma_{\alpha_{j}}$ as $\sigma_{\alpha_{j}}\mapsto \Phi_{j}(\sigma_{\alpha_{j}}; U_{b})$,
where $\Phi_{j}\in C^{2}\big((-\tilde{\delta}_{1},\tilde{\delta}_{1})\times O_{\tilde{\epsilon}_{1}}(\underline{U})\big)$ satisfies
\begin{align}
&U_{j}=\Phi_{j}(\sigma_{\alpha_j};U_{j-1}), \quad\,  U_{0}=U_{b},\quad\, U_{4}=U_{a},\label{eq:3.7}\\[1pt]
&\left. \frac{\partial\Phi_{j}}{\partial\sigma_{\alpha_{j}}}\right|_{\sigma_{\alpha_j}=0}=
\mathbf{r}_{j}(U_{b})\qquad \ \mbox{for $1\leq j\leq 4$}.\label{eq:3.8}
\end{align}
}

We call $\alpha_j$ the rarefaction wave if $\sigma_{\alpha_j}>0$, which is denoted by $\mathcal{R}_j(U_{b})\cap \mathcal{O}_{\tilde{\epsilon}_{1}}(\underline{U})$ for $j=1,4$.
We call $\alpha_j$ the shock wave if $\sigma_{\alpha_j}<0$, which is denoted by $\mathcal{S}_j(U_{b})\cap\mathcal{O}_{\tilde{\epsilon}_{1}}(\underline{U})$ for $j=1,4$.
We call $\alpha_k$ the vortex sheet/entropy wave if $\sigma_{\alpha_k}\neq0$, which is denoted by $\mathcal{C}_k(U_{b})\cap \mathcal{O}_{\tilde{\epsilon}_{1}}(\underline{U})$ for $k=2,3$.
Then, based on Lemma A.1
and by the standard method as done in \cite{bressan, smoller}, similar to Lemma \ref{lem:2.3},
we have the following local interaction estimates:

\medskip
\noindent
{\bf Lemma A.2.}
{\it There exists a constant $\tilde{\epsilon}_{2}\in (0,\tilde{\epsilon}_{1})$ sufficiently small, depending on $\underline{U}$, such that,
for given three constant states $U_a, U_m, U_b\in \mathcal{O}_{\tilde{\epsilon}_{2}}(\underline{U})$ with
\begin{align*}
&U_a=\Phi_{i}(\sigma_{\alpha_i}; U_m), \quad U_m=\Phi_{k}(\sigma_{\beta_j}; U_b),\\[2pt]
&U_a=\Phi(\boldsymbol{\sigma}_{\boldsymbol{\gamma}}; U_b),\quad
\boldsymbol{\sigma}_{\boldsymbol{\gamma}}=(\sigma_{\gamma_{1}},\sigma_{\gamma_{2}},
\sigma_{\gamma_{3}},\sigma_{\gamma_{4}}),
\end{align*}
then
\begin{eqnarray}\label{eq:3.9}
\sigma_{\gamma_{k}}=\delta_{ki}\sigma_{\alpha_k}+\delta_{kj}\sigma_{\beta_j}
+O(1)Q(\sigma_{\alpha_i},\sigma_{\beta_j}) \qquad\,\mbox{for $1\leq k\leq 4$};
\end{eqnarray}
if three is a non-physical wave $\gamma_{\mathcal{NP}}$ based on the construction, then
\begin{eqnarray}\label{eq:3.10}
\sigma_{\gamma_{\mathcal{NP}}}=O(1)Q(\sigma_{\alpha_i},\sigma_{\beta_j}),
\end{eqnarray}
where
\begin{equation}\label{eq:3.11}
Q(\sigma_{\alpha_i},\sigma_{\beta_j})=
\begin{cases}
0 \quad &\emph{if}\ i>j \quad \emph{or}\ i=j\ \emph{and}\
 \min\{\sigma_{\alpha_i},\sigma_{\beta_j}\}>0, \\[5pt]
|\sigma_{\alpha_i}||\sigma_{\beta_j}|\quad & \emph{otherwise},
\end{cases}
\end{equation}
and the bound of function $O(1)$ depends only on $\underline{U}$.
}

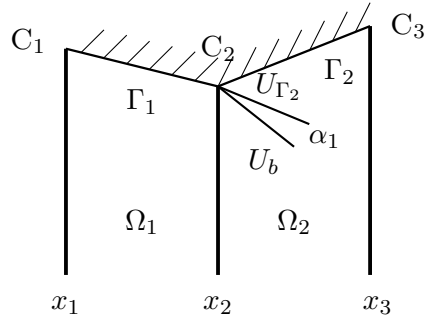
\begin{figure}[ht]
\begin{center}
\begin{tikzpicture}[scale=1.0]
\draw [line width=0.03cm](-5.0,1.5)--(-3,1)--(-1,1.8);

\draw [line width=0.05cm](-5.0,1.5)--(-5.0,-1.5);
\draw [line width=0.05cm](-3,1)--(-3,-1.5);
\draw [line width=0.05cm](-1.0,1.8)--(-1.0,-1.5);

\draw [thin](-4.8,1.45)--(-4.5, 1.75);
\draw [thin](-4.5,1.38)--(-4.2, 1.68);
\draw [thin](-4.2,1.30)--(-3.9, 1.60);
\draw [thin](-3.9, 1.23)--(-3.6,1.53);
\draw [thin](-3.6, 1.16)--(-3.3,1.46);
\draw [thin](-3.3, 1.08)--(-3.0,1.38);
\draw [thin](-3.0, 1.0)--(-2.8,1.4);
\draw [thin](-2.7,1.12)--(-2.5, 1.52);
\draw [thin](-2.4,1.23)--(-2.2, 1.63);
\draw [thin](-2.1,1.35)--(-1.9, 1.75);
\draw [thin](-1.8,1.47)--(-1.6, 1.87);
\draw [thin](-1.5, 1.59)--(-1.3,1.99);
\draw [thin](-1.2,1.71)--(-1.0, 2.11);

\draw [thick](-3,1)--(-1.8,0.5);
\draw [thick](-3,1)--(-2.0,0.2);

\node at (-5.5, 1.6) {$\textsc{C}_{1}$};
\node at (-3.0, 1.5) {$\textsc{C}_{2}$};
\node at (-0.5, 1.8) {$\textsc{C}_{3}$};
\node at (-1.6, 0.3) {$\alpha_1$};
\node at (-2.2, 1.0) {$U_{\Gamma_2}$};
\node at (-2.4, 0) {$U_{b}$};
\node at (1, 2) {$$};

\node at (-4.0, 0.8) {$\Gamma_{1}$};
\node at (-1.4, 1.2) {$\Gamma_{2}$};
\node at (-4.0, -0.8) {$\Omega_{1}$};
\node at (-2.0, -0.8) {$\Omega_{2}$};

\node at (-5.0, -1.9) {$x_{1}$};
\node at (-3.0, -1.9) {$x_{2}$};
\node at (-0.9, -1.9) {$x_{3}$};
\end{tikzpicture}
\caption{Mixed Riemann problem for system \eqref{eq:1.14}}\label{fig3.1}
\end{center}
\end{figure}

\medskip
Next, we consider the following mixed Riemann problem:
\begin{eqnarray}\label{eq:3.12}
\begin{cases}
{\rm System}~\eqref{eq:1.14} \quad & \mbox{in $\Omega_{2}$},\\[2pt]
U=U_b \quad & \mbox{on $\{x=x_2\}\cap\Omega_{2}$},\\[2pt]
v=\tan\theta_2 \quad & \mbox{on $\Gamma_{2}$},
\end{cases}
\end{eqnarray}
where $U_b$ is a constant state near $\underline{U}$ and satisfies
\begin{eqnarray}\label{eq:3.12b}
v_{b}=\tan\theta_{1},
\end{eqnarray}
and $\Omega_{k}$, $\theta_k$, and $x_k$ are given as in \eqref{eq-B1}--\eqref{eq-B2} for $k=1,2$;
see Fig. \ref{fig3.1}.
Then, similar to Lemma \ref{lem:2.5}, we have the following lemma on the solvability of the mixed Riemann problem \eqref{eq:3.12}:

\medskip
\noindent
{\bf Lemma A.3.}
{\it There exists a small constant $\tilde{\epsilon}_{3}>0$ depending only on $\underline{U}$ such that,
if $U_{b}\in \mathcal{O}_{\tilde{\epsilon}_{3}}(\underline{U})$ and $|\omega|+|\theta_{1}|<\tilde{\epsilon}_{3}$,
then problem \eqref{eq:3.12} admits a unique solution consisting of states $U_{\Gamma_{2}}$ and $U_{b}$,
which are connected by a weak $1^{\rm st}$ wave $\alpha_1$ with strength $\sigma_{\alpha_1}${\rm :}
\begin{eqnarray}\label{eq:3.13}
U_{\Gamma_{2}}=\Phi_{1}(\sigma_{\alpha_1}; U_{b}),
\end{eqnarray}
satisfying
\begin{eqnarray}\label{eq:3.14}
\sigma_{\alpha_1}=K_{b}\omega,
\end{eqnarray}
where $\omega=\theta_{2}-\theta_{1}$, and the bound of $K_{b}>0$ depends only on $\underline{U}$.
}

\begin{figure}[ht]
\begin{center}
\begin{tikzpicture}[scale=0.6]
\draw [line width=0.05cm] (-3.5,-3.8) --(-3.5,1.5);
\draw [line width=0.05cm] (2.5,-3.8) --(2.5,0.5);
\draw [line width=0.03cm](-3.5,1.5)--(2.5,0.5);

\draw [thin] (-3, 1.4) --(-2.6, 1.8);
\draw [thin] (-2.6, 1.35) --(-2.2, 1.75);
\draw [thin] (-2.2, 1.30) --(-1.8, 1.70);
\draw [thin] (-1.8, 1.23) --(-1.4, 1.63);
\draw [thin] (-1.4, 1.16) --(-1.0, 1.56);
\draw [thin] (-1.0, 1.10) --(-0.6, 1.50);
\draw [thin] (-0.6, 1.03) --(-0.2, 1.43);
\draw [thin] (-0.2, 0.97) --(0.2, 1.37);
\draw [thin] (0.2, 0.9) --(0.6, 1.30);
\draw [thin] (0.6, 0.83) --(1, 1.23);
\draw [thin] (1, 0.76) --(1.4, 1.16);
\draw [thin] (1.4, 0.67) --(1.8, 1.07);
\draw [thin] (1.8, 0.60) --(2.2, 1.0);
\draw [thin] (2.2, 0.55) --(2.6, 0.95);

\draw [thick](-2.5,-1.5)--(-0.5,1);
\draw [thick][red](-0.5,1)--(1.7,-1.0);

\node at (0.2, -2.8){$\Omega_{2}$};
\node at (3.3, 0.4){$\Gamma_{2}$};

\node at (-2.6, -1.9){$\alpha_{\ell}$};
\node at (2.1, -1.1){$\beta_{1}$};

\node at (-0.4, -1.0){$U_{b}$};
\node at (-2.2, 0.2){$U^{-}_{\Gamma_2}$};
\node at (1.7, -0.1){$U^{+}_{\Gamma_2}$};

\node at (-3.5, -4.5){$x_{2}$};
\node at (2.5, -4.5){$x_{3}$};
\end{tikzpicture}
\end{center}
\caption{Weak waves hit on the boundary and reflect}\label{fig9}
\end{figure}
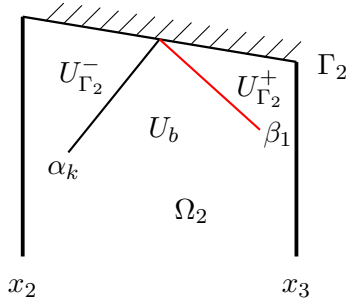

\medskip
Finally, similar to Lemma \ref{lem:2.6}, we consider the weak wave reflection on the boundary (see also Fig. \ref{fig9}).

\medskip
\noindent
{\bf Lemma A.4.}
{\it There is a constant $\tilde{\epsilon}_{4}>0$ depending only on $\underline{U}$ such that,
if $U^{-}_{\Gamma_2}$, $U_b\in O_{\tilde{\epsilon}_{4}}(\underline{U})$ are two constant states and satisfy
\begin{equation*}
v^{-}_{\Gamma_2}=\tan\theta_{2}, \qquad U^{-}_{\Gamma_2}=\Phi_{\ell}(\sigma_{\alpha_\ell}; U_b) \quad \mbox{for $\ell=2,3,4$},
\end{equation*}
then, for the constant state $U^{+}_{\Gamma_2}\in O_{\tilde{\epsilon}_{3}}(\underline{U})$ with
\begin{eqnarray*}
v^{+}_{\Gamma_2}=\tan\theta_{2}, \qquad U^{+}_{\Gamma_2}=\Phi_{1}(\sigma_{\beta_1};U_b),
\end{eqnarray*}
the following estimate holds{\rm :}
\begin{equation}\label{eq:3.15}
\sigma_{\beta_1}=K_{r,\ell} \sigma_{\alpha_\ell}\qquad \mbox{for {$2\leq \ell\leq 4$}},
\end{equation}
where the reflection coefficients $K_{r,\ell}$ are $C^{2}$-functions of $(\sigma_{\alpha_\ell}, U_b)$ and satisfy
\begin{equation}\label{eq:3.16}
K_{r,4}|_{\sigma_{\alpha_4}=0, U_{b}=\underline{U}}=1,\qquad
K_{r,\ell}|_{\sigma_{\alpha_\ell}=0, U_{b}=\underline{U}}=0  \quad \mbox{for $\ell=2,3$}.
\end{equation}
}

\medskip
\subsection{Entropy solutions of the initial-boundary value problem \eqref{eq:1.12}--\eqref{eq:1.15}}
In this subsection, we construct the approximate solutions of the initial-boundary value problem \eqref{eq:1.12}--\eqref{eq:1.15}
via the wave-front tracking scheme, which is similar to the argument in \S 2.3.
Then the entropy solutions are obtained by passing the limit.
As done in \S 2.3, we first choose a parameter $\nu>0$ and then approximate the initial data $U_{0}$ by $U^{\nu}_{0}$,
with finite discontinuities, satisfying \eqref{eq:2.28}.
Then we choose a mesh length $h$ and approximate the boundary function $g(x)$ by $g_{h}(x)$, with piecewise constant slope,
satisfying \eqref{eq:2.26}.
Thus, the corresponding domain $\Omega$ is approximated by $\Omega_{h}$ with $\Gamma_{h}$ as its boundary.
The approximate solution $U_{h,\nu}$ of the initial-boundary value problem \eqref{eq:1.12}--\eqref{eq:1.15},
via the wave-front tracking scheme corresponding to $h$ and $\nu$, is constructed in the following way:

As done in \S 2.3, we introduce two types of Riemann solvers, $(ARS)$ and $(SRS)$,
for which the non-physical wave is introduced for $(SRS)$.
Moreover, a threshold parameter $\varrho>0$ given in \S 2.3 is also introduced here.
Following the procedure in \S 2.3, if the product of the strengths of the two approaching wave-fronts
is larger than $\varrho$, then the accurate Riemann solver $(ARS)$ is used.
Otherwise, the simplified Riemann solver $(SRS)$ is used.
Now we solve the Riemann problem at each discontinuity via either  $(ARS)$ or $(SRS)$
and continue this procedure for $x>0$.

Let $\mathcal{J}=\mathcal{S}\cup\mathcal{C}\cup\mathcal{R}\cup\mathcal{NP}$ represent the wave-fronts in $U_{h, \nu}$.
In order to define the approximate solution $U_{h, \nu}$ globally for all $x>0$,
one needs to show that the total number of the wave-fronts $\mathcal{J}$ is finite.
This is done once the uniform bound on the total variation of $U_{h, \nu}$ is obtained.
To achieve this, we follow the argument in \S 2.3 by introducing a weighted Glimm-type functional.

First, we make some \emph{a priori} assumptions on the approximate solution $U_{h, \nu}$.
\begin{itemize}
\item[$\mathbf{(\widetilde{P1})}$]  The approximate solution $U_{h, \nu}$ has been defined
for $x<\hat{x}$ and satisfies
$U_{h, \nu}(\hat{x}-,\cdot)\in \mathcal{O}_{\tilde{\epsilon}_*}(\underline{U})$
for
$\tilde{\epsilon}_{*}\in(0, \min\{\tilde{\epsilon}_{2},\tilde{\epsilon}_{3}, \tilde{\epsilon}_{4}\})$;
\item[$\mathbf{(\widetilde{P2})}$]  For $x<\hat{x}$, the strength $\sigma_{\alpha}$ of each front $\alpha$ satisfies
$|\sigma_{\alpha}|\leq O(1)\nu^{-1}$.
\end{itemize}

\smallskip
Then, for $x>0$, we define the modified Glimm-type functional as
\begin{eqnarray}\label{3.17}
\mathcal{G}(x)=\mathcal{V}(x)+\mathcal{K}\mathcal{Q}(x),
\end{eqnarray}
where
\begin{align}
&\mathcal{V}(x)=\mathcal{V}_{1}(x)+\sum^{4}_{ i=2}\mathcal{K}_{i}\mathcal{V}_{i}(x)
+\mathcal{V}_{\mathcal{NP}}(x)+\mathcal{K}_{c}\mathcal{V}_{c}(x),\label{3.17a}\\[2pt]
&\mathcal{V}_{i}(x)=\sum_{\alpha_i\in \mathcal{J}}|\sigma_{\alpha_i}|\,\,\,\, \mbox{for $1\leq i\leq 4$},  \qquad
\mathcal{V}_{\mathcal{NP}}(x)=\sum_{\alpha\in \mathcal{NP}}|\sigma_{\alpha}|,
    \label{eq:3.17b}\\[2pt]
&\mathcal{V}_{c}(x)=\sum_{k>[\frac{x}{h}]}|\omega_{k}|,\quad \mathcal{Q}(x)=\sum_{(\alpha_i,\beta_j)\in \mathcal{A}}|\sigma_{\alpha_i}||\sigma_{\beta_j}|.\label{eq:3.17c}
\end{align}
Here $\mathcal{A}$ stands for the set of all approaching waves; that is, for the $i^{\rm th}$ wave $\alpha_i$ and $j^{\rm th}$ wave $\beta_j$
with strengths $\sigma_{\alpha_i}$ and $\sigma_{\beta_j}$ and
located at $y_{\alpha_i}$ and $y_{\beta_j}$, respectively, so that $y_{\alpha_i}<y_{\beta_j}$ and one of the following conditions holds:

\begin{enumerate}
\item[\rm (i)] $\alpha_i\in \mathcal{NP}$ and $\beta_j\in \mathcal{J}\setminus \mathcal{NP}$,

\item[\rm (ii)] $i>j$,

\item[\rm (iii)] $i=j$ and $\min\{\sigma_{\alpha_i}, \sigma_{\beta_j}\}<0$.
\end{enumerate}
Now, following a similar argument to the proof of Lemma \ref{lem:2.7}, we obtain

\medskip
\noindent
{\bf Lemma A.5.}
{\it There exist constants $\mathcal{K}_{i}$ for $i=2,3,4$, $\mathcal{K}_{c}$, and $\mathcal{K}$ depending only
on $\underline{U}$, and a small constant $\tilde{\varepsilon}_{0}>0$ depending only on $\tilde{\epsilon}_{*}$ and $\underline{U}$
such that, if $\mathcal{G}(\hat{x}-)<\varepsilon$ for some $\varepsilon\in(0, \tilde{\varepsilon}_{0})$, then
\begin{eqnarray*}
\mathcal{G}(\hat{x}+)<\mathcal{G}(\hat{x}-) \qquad \mbox{for $\hat{x}>0$},
\end{eqnarray*}
and $U_{h, \nu}(\hat{x}+,\cdot)\in \mathcal{O}_{\tilde{\epsilon}_*}(\underline{U})$.
Moreover, at $x=\hat{x}+$, strength $\sigma_{\alpha}$ of each front $\alpha\in \mathcal{J}$ satisfies $|\sigma_{\alpha}|\leq O(1)\nu^{-1}$.
}

Based on Lemma A.5
and by an induction procedure as in \cite{bressan, kuang-zhao},
we can obtain the uniform $BV\cap L^{\infty}$ bound of the approximate solution $U_{h,\nu}$.

\medskip
\noindent
{\bf Proposition A.1.}
{\it Under assumptions $(\mathbf{U_0})$ and $(\mathbf{g})$, there exists a constant $\tilde{\varepsilon}_{1}=\tilde{\varepsilon}_{1}(\tilde{\epsilon},\tilde{\varepsilon}_{0} )>0$
depending only on $\underline{U}$ such that, for some $\varepsilon\in (0, \tilde{\varepsilon}_{1})$, if the initial data $U_{0}$ and
the boundary function $g(x)$ satisfy \eqref{eq:2.36},
then the approximate solution $U_{h,\nu}$ generated by the wave-front tracking scheme can be defined globally for all $x>0$ with
the properties that $U_{h,\nu}\in (BV_{\rm loc}\cap L^{1}_{\rm loc})(\Omega_{h})$,
\begin{align}
&\sup_{x>0}\big\|U_{h,\nu}(x,\cdot)-\underline{U}\big\|_{BV((-\infty, g_{h}(x)))}\nonumber\\
&\, \leq {C_{A,1}}\big(\|U_{0}-\underline{U}\|_{BV(\mathcal{I})}+|g'(0)|+\|g'\|_{BV(\mathbb{R}_{+})}\big),\label{eq:3.19}
\end{align}
and, for any $\tilde{x}', \tilde{x}''>0$,
\begin{align}
\big\|U_{h,\nu}(\tilde{x}',\cdot+g_{h}(\tilde{x}'))-U_{h,\nu}(\tilde{x}'',\cdot+g_{h}(\tilde{x}''))\big\|_{L^{1}((-\infty, 0))}\leq {C_{A,2}}|\tilde{x}'-\tilde{x}''|,
\label{eq:3.20}
\end{align}
where constants {$C_{A,1}>0$ and $C_{A,2}>0$} depend only on $\underline{U}$.
The strength of each rarefaction-front is small{\rm :}
\begin{eqnarray}\label{eq:3.21}
|\sigma_{\alpha}|\leq {C_{A,3}}\nu^{-1}\qquad \mbox{for $\alpha\in\mathcal{R}$}.
\end{eqnarray}
The total strength of nonphysical waves is small, \emph{i.e.}, there exists a positive threshold $\varrho=\varrho_{\nu}>0$
with $\varrho_{\nu}\rightarrow 0$ as $\nu\rightarrow \infty$
such that
\begin{eqnarray}\label{eq:3.22}
\sum_{\alpha\in\mathcal{NP}}|\sigma_{\alpha}|\leq {C_{A,4}}2^{-\nu},
\end{eqnarray}
where constants {$C_{A,3}>0$ and $C_{A,4}>0$} depend only on $\underline{U}$.
}

By Proposition A.1
and  Theorem 2.4 in \cite{bressan}
and following the arguments in \cite{chen-kuang-zhang,hu-kuang,kuang-zhao,smoller,zhang-1},
we can obtain the compactness for the convergence of the approximate solutions $\{U_{h,\nu}\}$.

\medskip
\noindent
{\bf Proposition A.1.}
{\it
Under assumptions $(\mathbf{U_0})$ and $(\mathbf{g})$,
there exists a subsequence $\{U_{h_k,\nu_k}\}$ obtained
by the wave-front tracking scheme such that $U_{h_k,\nu_k}\rightarrow U$ in $L^{1}_{\rm loc}(\Omega)$ as $h_k\rightarrow 0$ and $\nu_k\rightarrow \infty$
for $k\rightarrow \infty$ such that
$U\in (BV_{loc}\cap L^{1}_{\rm loc})(\Omega)$ is the global entropy solution of the initial-boundary value problem \eqref{eq:1.12}--\eqref{eq:1.15}
in the sense of {\rm Definition \ref{def:1.1}} for $\tau=0$ satisfying
\begin{align*}
\sup_{x>0}\big\|U(x,\cdot)-\underline{U}\big\|_{BV((-\infty, g(x)))}\leq {C_{A,5}}\big(\|U_{0}-\underline{U}\|_{BV(\mathcal{I})}+|g'(0)|+\|g'\|_{BV(\mathbb{R}_{+})}\big),
\end{align*}
{where $C_{A,5}>0$ depends only on $\underline{U}$.}
}

\bigskip
\bigskip
\medskip
\noindent
{\bf Acknowledgements.}
The research of Gui-Qiang G. Chen was supported in part by the UK Engineering and Physical Sciences Research Council Awards
EP/L015811/1, EP/V008854/1, and EP/V051121/1.
The research of Jie Kuang was supported in part by the NSFC Projects under Grant No. 11801549, No. 11971024, and No. 12271507, and the Multidisciplinary Interdisciplinary Cultivation Project No. S21S6401 from
Innovation Academy for Precision Measurement Science and Technology, Chinese Academy of Sciences.
The research of Wei Xiang was supported in part by the Research Grants Council of the HKSAR, China (Project
No. CityU 11304820, No. CityU 11300021, No. CityU 11311722 and No. CityU 11305523), and in part by the Research Center for Nonlinear
Analysis of the Hong Kong Polytechnic University.
The research of Yongqian Zhang was supported in part by the NSFC Projects under Grant No. 12271507, No. 11421061, No. 11031001, and No. 11121101,
the 111 Project B08018 (China) and the Shanghai Natural Science Foundation 15ZR1403900.


\bigskip
\begin{thebibliography}{10}

\bibitem{amadori} D.~Amadori,
{Initial boundary value problem for nonlinear systems of conservation laws},
Nonlinear Differ. Equ. Appl., 4 (1997), 1--42.


\bibitem{anderson} J.~Anderson,
\textit{Hypersonic and High-Temperature Gas Dynamics}, Second Edition, AIAA Education Series, Reston, 2006.


\bibitem{bressan} A.~Bressan,
\textit{Hyperbolic Systems of Conservation Laws. The One-Dimensional Cauchy Problem}, Oxford University Press: Oxford, 2000.


\bibitem{bressan-liu-yang} A.~Bressan, T.-P.~Liu, and T.~Yang,
{$L^{1}$ stability estimates for $n \times n$ conservation laws},
Arch. Ration. Mech. Anal., 149 (1999), 1--22.


\bibitem{chen-christoforou-zhang-1} G.-Q.~Chen, C.~Christoforou, and Y.~Zhang,
{Dependence of entropy solutions with large oscillations to the Euler equations on the nonlinear flux functions},
Indiana Univ. Math. J., 56 (2007), 2535--2568.


\bibitem{chen-christoforou-zhang-2}
G.-Q.~Chen, C.~Christoforou, and Y.~Zhang,
{Continuous dependence of entropy solutions to the Euler equations on the adiabatic exponent and Mach number},
Arch. Ration. Mech. Anal., 189 (2008), 97--130.


\bibitem{ChenFeldman}
G.-Q.~Chen and M.~Feldman,
\textit{Mathematics of Shock Reflection-Diffraction and von Neumann's Conjectures}. Research Monograph,
Annals of Mathematics Studies, 197, Princeton University Press: Princeton, 2018.

\bibitem{ChenFeldmanXiang}
G.-Q.~Chen, M.~Feldman, and W.~Xiang,
{Convexity of self-similar transonic shocks and free boundaries for the Euler equations for potential flow},
Arch. Ration. Mech. Anal., 238 (2020), 47--124.


\bibitem{chen-kuang-zhang} G.-Q.~Chen, J.~Kuang, and Y.~Zhang,
{Two-dimensional steady supersonic exothermically reacting Euler flow past Lipschitz bending walls},
SIAM J. Math. Anal., 49 (2017), 818--873.


\bibitem{chen-kuang-zhang-2} G.-Q.~Chen, J.~Kuang, and Y.~Zhang,
{Stability of conical shocks in the three-dimensional steady supersonic isothermal flows past Lipschitz perturbed cones},
SIAM J. Math. Anal., 53 (2021), 2811--2862.


\bibitem{chen-kuang-xiang-zhang} G.-Q.~Chen, J.~Kuang, W. Xiang and Y.~Zhang,
{Convergence rate of the hypersonic similarity for two-dimensional steady potential flows with large data}, arXiv:2405.04720v1, Preprint, 2024.


\bibitem{chen-li} G.-Q.~Chen and T.-H.~Li,
{Well-posedness for two-dimensional steady supersonic Euler flows past a Lipschitz wedge},
 J. Differential Equations, 244 (2008), 1521--1550.


\bibitem{chen-xiang-zhang} G.-Q.~Chen, W.~Xiang, and Y.~Zhang,
{Weakly nonlinear geometric optics for hyperbolic systems of conservation laws},
Comm. Partial Differential Equations, 38 (2015), 1936--1970.


\bibitem{chen-zhang-zhu-1} G.-Q.~Chen, Y. Zhang, and D.-W.~Zhu,
{Existence and stability of supersonic Euler flows past Lipschitz wedges},
 Arch. Ration. Mech. Anal., 181 (2006), 261--310.


\bibitem{chen-zhang-zhu-2} G.-Q.~Chen, Y.~Zhang, and D.-W.~Zhu,
{Stability of compressible vertex sheets in steady supersonic Euler flows over Lipschitz walls},
SIAM. J. Math. Anal., 38 (2007), 1660--1693.


\bibitem{chen-xin-yin} S.-X. Chen, Z. Xin, and H. Yin,
{Global shock waves for the supersonic flow past a perturbed cone},
Commun. Math. Phys., 228 (2002), 47--84.


\bibitem{colombo-guerra} R.~M.~Colombo and G.~Guerra,
{On general balance laws with boundary},
J. Differential Equations, 248 (2010), 1017--1043.


\bibitem{courant-friedrich} R.~Courant and K.O.~Friedrichs,
\textit{Supersonic Flow and Shock Waves}, Interscience Publishers Inc., New York, 1948.

\bibitem{Dafermos}
C.~M. Dafermos,
\textit{\sl Hyperbolic Conservation Laws in Continuum Physics}, Fourth Edition,
Springer-Verlag: Berlin, 2016.


\bibitem{donadello-marson} C.~Donadello and A.~Marson,
{Stability of front tracking solutions to the initial and boundary value problem for systems of conservation laws},
Nonlinear Differ. Equ. Appl., 14 (2007), 569--592.


\bibitem{dyke} M. Van~Dyke,
{A study of hypersonic small disturbance theory}, NACA Rept., 1194, April, 1954.


\bibitem{hu-zhang} D.~Hu and Y.~Zhang,
{Global conic shock wave for the steady supersonic flow past a curved cone},
SIAM J. Math. Anal., 51 (2019), 2072--2389.


\bibitem{huang-kuang-wang-xiang} F.~Huang, J.~Kuang, D.~Wang, and  W.~Xiang,
{Stability of supersonic contact discontinuity for 2-D steady compressible Euler flows in a finitely long nozzles},
J. Differential Equations, 266 (2019), 4337--4376.


\bibitem{hu-kuang} K.~Hu and J.~Kuang,
{Global well-posedness of shock front solutions to two-dimensional piston problem for combustion Euler flows},
SIAM J. Math. Anal., 55 (2023), 2042--2110.


\bibitem{jin-qu-yuan-1} Y.~Jin, A.~Qu, and H.~Yuan,
{On two-dimensional steady hypersonic-limit Euler flows passing ramps and radon measure solutions of compressible Euler equations},
Commun. Math. Sci., 20 (2022), 1331--1361.


\bibitem{jin-qu-yuan-2} Y.~Jin, A.~Qu, and H.~Yuan,
{Radon measure solutions for steady compressible hypersonic-limit Euler flows passing cylindrically symmetric conical bodies},
Comm. Pure Appl. Anal., 20 (2021), 2665--2685.

\bibitem{kong-yang} D.-X.~Kong and T.~Yang,
{A note on ``well-posedness theory for hyperbolic conservation laws''},
Appl. Math. Lett., 16 (2003), 143--146.


\bibitem{kuang-xiang-zhang-1} J.~Kuang, W.~Xiang, and Y.~Zhang,
{Hypersonic similarity for the two dimensional steady potential flow with large data},
Ann. Inst. H. Poincar\'{e} Anal. Non Lin\'{e}aire, 37 (2020), 1379--1423.


\bibitem{kuang-xiang-zhang-2} J.~Kuang, W.~Xiang, and Y.~Zhang,
{Convergence rate of hypersonic similarity for steady potential flows over two-dimensional Lipschitz wedge},
Calc. Var. \& PDEs, 62 (2023), art. no. 106.


\bibitem{kuang-zhao} J.~Kuang and Q.~Zhao,
{Global existence and stability of shock front solution to the 1-D piston problem for exothermically reacting Euler equations},
J. Math. Fulid Mech., 22 (2020), 42pp.

\bibitem{landau-lifschitz} L.~Landau and E.~Lifschitz,
\textit{Fluid Mechanics}, Second Edition, Elsevier Ltd.: Singapore, 2004.


\bibitem{li-witt-yin} J.~Li, I.~Witt, and H.~Yin,
{On the global existence and stability of a multi-dimensional supersonic conic shock waves},
Commun. Math. Phys., 329 (2014), 609--640.


\bibitem{lien-liu} W.-C. Lien and T.-P. Liu,
{Nonlinear stability of a self-similar 3-D gas flow},
Commun. Math. Phys., 304 (1999), 524--549.


\bibitem{qu-wang-yuan} A.~Qu, L. Wang, and H.~Yuan,
{Radon measure solutions for steady hypersonic-limit Euler flows passing two-dimensional finite non-symmetric obstacles and interactions of free concentration layers},
Commun. Math. Sci., 19 (2021), 875--901.


\bibitem{qu-yuan} A.~Qu and H.~Yuan,
{Radon measure solutions for steady compressible Euler equations of hypersonic-limit conical flows and Newton's sine-squared law},
J. Differential Equations, 269 (2020), 495--522.


\bibitem{qu-yuan-zhao} A.~Qu, H.~Yuan, and Q.~Zhao,
{Hypersonic limit of two-dimensional steady compressible Euler flows passing a straight wedge},
Z. Angew. Math. Mech., 100 (2020), 14pp.


\bibitem{tsien} H.-S.~Tsien,
{Similarity laws of hypersonic flows},
J. Math. Phys., 25 (1946), 247--251.


\bibitem{smoller} J.~Smoller,
\textit{Shock Waves and Reaction-Diffusion Equations},
Second Edition, Springer-Verlag, Inc.: New York, 1994.


\bibitem{wang-zhang} Z.~Wang and Y.~Zhang,
{Steady supersonic flow past a curved cone},
J. Differential Equations, 247 (2009), 1817--1850.


\bibitem{xiang-zhang-zhao} W.~Xiang, Y.~Zhang, and Q.~Zhao,
{Two-dimensional steady supersonic exothermically reacting Euler flows with strong contact discontinuity over a Lipschitz wall},
Interfaces Free Bound., 20 (2018), 437--481.


\bibitem{zhang-1} Y.~Zhang,
{Global existence of steady supersonic potential flow past a curved wedge with piecewise smooth boundary},
SIAM J. Math. Anal., 31 (1999), 166--183.


\bibitem{zhang-2} Y.~Zhang,
{Steady supersonic flow past an almost straight wedge with large vertex angle},
J.Differential Equations, 192 (2003), 1--46.

\bibitem{zhang-3} Y.~Zhang,
{On the irrotational approximation to steady supersonic flow},
Z. Angew. Math. Phys., 58 (2007), 209--223.
\end{thebibliography}
\end{document}